%% file: qbx-fmm.tex
\documentclass{elsarticle_nonatbib}

\newif\ifpdfa

\ifpdfa
\usepackage[a-3u]{pdfx}
\fi
\let\epsilon=\varepsilon
\usepackage{algorithm}
\usepackage{algorithmic}
\usepackage{amsthm}
\usepackage{amssymb}
\usepackage{booktabs}
\usepackage{comment}
\usepackage{etoolbox}
\usepackage{float}
\usepackage{graphicx}
\usepackage{needspace}
\usepackage{textcomp}
\usepackage{siunitx}
\usepackage{tikz}
\usepackage[margin=1in]{geometry}
\usepackage{multirow}
\usepackage{hyperref}
\usepackage{mathtools}
\usepackage{adjustbox}
\usepackage{paralist}
\usepackage{microtype}

\def\today{March 4, 2020}

\usepackage[utf8]{inputenc}
\DeclareUnicodeCharacter{2212}{-}

\usetikzlibrary{calc}

\usepackage{import}

\usepackage[
  natbib = true,
    backend=biber,
    isbn=false,
    url=false,
    doi=true,
    eprint=false,
    style=numeric,
    sorting=nyt,
    sortcites = true
]{biblatex}
\bibliography{qbx-fmm.bib}

\renewcommand{\Re}{\mathop{\mathrm{Re}}%
}
\newtheorem{theorem}{Theorem}
\newtheorem{lemma}[theorem]{Lemma}
\newtheorem{proposition}[theorem]{Proposition}
\newtheorem{remark}[theorem]{Remark}
\newtheorem{definition}{Definition}

\newcommand{\ilist}[2]{
  \ifstrequal{#1}{1}{U_{#2}}{
  \ifstrequal{#1}{2}{V_{#2}}{
  \ifstrequal{#1}{3}{W_{#2}}{
  \ifstrequal{#1}{3close}{W^\mathrm{close}_{#2}}{
  \ifstrequal{#1}{3far}{W^\mathrm{far}_{#2}}{
  \ifstrequal{#1}{4}{X_{#2}}{
  \ifstrequal{#1}{4close}{X^\mathrm{close}_{#2}}{
  \ifstrequal{#1}{4far}{X^\mathrm{far}_{#2}}{}
}}%
}}%
}}%
}}

\def\Mpole{\mathsf{M}}
\def\Locfar{\mathsf{L}^{\text{far}}%
}
\def\Potnear{\mathsf{P}^{\text{near}}%
}
\def\PotW#1{\mathsf{P}^{\ilist{3}{}}_{#1}}
\def\Lqbxnear{\mathsf{L}^{q,\text{near}}%
}
\def\LqbxW#1{\mathsf{L}^{q,\ilist{3}{}}_{#1}}
\def\Lqbxfar{\mathsf{L}^{q,\text{far}}%
}

\newlength{\longrowlength}
\newlength{\figurewidth}
\newcommand{\nmax}{{n_{\mathrm{max}}%
}}

\newcommand{\pqbx}{{p_\mathrm{QBX}%
}}
\newcommand{\pfmm}{{p_\mathrm{FMM}%
}}
\newcommand{\pquad}{{p_\mathrm{quad}%
}}

\newcommand{\homebox}[1]{{b_#1}%
}
\newcommand{\closedbox}{\overline {B_\infty}}

\newcommand{\ancestors}{\mathsf{Ancestors}}
\newcommand{\descendants}{\mathsf{Descendants}}
\newcommand{\tcr}{\mathsf{TCR}}
\newcommand{\parent}{\mathsf{Parent}}
\newcommand{\adequatesep}{\prec}

\newcommand{\inlbfrac}[2]{(#1/#2)}

\newcommand{\algbrand}{GIGAQBX}
\newcommand{\belowtableskip}{\vspace*{1ex}}
\newcommand{\intertableskip}{\vspace*{4ex}}

\newenvironment{accfigadjust}{
    \begin{adjustbox}{clip,trim=0 0.4cm 0 0.4cm}
  }
  {
    \end{adjustbox}
  }

\begin{document}

\def\papertitle{A Fast Algorithm with Error Bounds for Quadrature~by~Expansion}

\title{\papertitle}

\auth[uiuc]{Matt Wala\corref{MW}}%
\ead{wala1@illinois.edu}
\auth[uiuc]{Andreas Kl\"ockner}%
\ead{andreask@illinois.edu}
\address[uiuc]{Department of Computer Science,
  University of Illinois at Urbana-Champaign, 201 North Goodwin Ave,
  Urbana, IL 61801}
\cortext[MW]{Corresponding author}

\begin{abstract}
Quadrature by Expansion (QBX) is a quadrature method for approximating
the value of the singular integrals encountered in the evaluation of layer potentials. It
exploits the smoothness of the layer potential by forming
locally-valid expansions which are then evaluated to compute the near
or on-surface value of the potential. Recent work towards coupling of a
Fast Multipole Method (FMM) to QBX yielded a first step towards the
rapid evaluation of such integrals (and the solution of related
integral equations), albeit with only empirically understood error
behavior. In this paper, we improve upon this approach with a modified
algorithm for which we give a comprehensive analysis of error and cost in the case
of the Laplace equation in two dimensions.  For the same levels of
(user-specified) accuracy, the new algorithm empirically has cost-per-accuracy
comparable to prior approaches.  We provide experimental results to demonstrate
scalability and numerical accuracy.
\end{abstract}



\maketitle

\section{Introduction}%
\label{sec:intro}
Integral equation methods for the solution of boundary value problems of partial
differential equations offer a number of numerically attractive properties,
including boundary-only discretizations for homogeneous problems,
seamless treatment of exterior problems, and mesh-independent conditioning.
However, their effective numerical realization presents a number of
technical challenges. A key prerequisite for the use of these methods
is the scalable and accurate evaluation of layer potentials on, near,
and away from the source layer. This in turn involves \emph{singular
and near-singular quadrature} and so-called \emph{fast algorithms}
(like the Fast Multipole Method) to facilitate the evaluation of
$O(N^2)$ interactions in linear or near-linear time. To maintain
accuracy and efficiency, both aspects need to be well-integrated, and,
as a unit, have well-understood error behavior.

The single layer and double layer potential integral operators $\mathcal{S}$,
$\mathcal{D}$ for the Laplace equation with boundary density
function $\mu(y): \Gamma \to \mathbb{R}$ are defined as
\begin{align}
  \label{eqn:slp-definition}
  \mathcal{S} \mu(x)
  &= -\frac{1}{2 \pi} \int_{\Gamma} \mu(y) \log{|y - x|} \, ds(y),\\
  \label{eqn:dlp-definition}
  \mathcal{D} \mu(x)
  &= -\frac{1}{2 \pi} \int_{\Gamma} \mu(y) \, \hat n \cdot \nabla_y \log{|y - x|} \, ds(y).
\end{align}
The capability of evaluating such layer potentials can be used for the solution
of homogeneous PDE boundary value problems. We demonstrate this by the following
example. Consider, for
specificity, the exterior Neumann problem in two dimensions for the Laplace
equation:
\begin{align}
  \notag
  \triangle u(x) &= 0\quad (x\in \mathbb R^2\setminus \Omega),\\
  \label{eq:laplace-bc}
  ( \hat n(x) \cdot \nabla u(y)) &\to g\quad (x\in \partial\Omega, y\to x_+),\\
  \notag
  u(x)&\to 0 \quad(x\to\infty),
\end{align}
where $\Omega$ is a bounded domain with a smooth boundary, $\hat n(x)$ is the
unit normal to $\partial \Omega$ at $x$, and $\lim_{y\to x_+}$ denotes a limit
approaching the boundary from the exterior of $\Omega$. Then, by
choosing $\Gamma=\partial \Omega$ and representing the solution $u$ in terms of a single layer
potential $u(x)=\mathcal S\mu(x)$ with an unknown density function $\mu$, we obtain that
the Laplace PDE and the far-field boundary condition are satisfied by $u$.
The remaining Neumann boundary condition becomes, by way of the well-known
jump relations for layer potentials (see~\cite[Theorem 6.28]{kress:2014:integral-equations})
the integral equation of the second kind
\begin{equation}
  -\frac \mu2 + \mathcal S'\mu = g.
  \label{eq:ext-neumann-ie}
\end{equation}
The boundary $\Gamma$ and density $\mu$ may then be discretized using piecewise
polynomials and, using the action of the normal derivative $\mathcal S'$ of $\mathcal S$ as supplied by our method
below, solved for the unknown density $\mu$. Once $\mu$ is known, the
representation of $u$ in terms of the single-layer potential~\eqref{eqn:slp-definition}
may be evaluated anywhere in $\mathbb R^2\setminus\Omega$, again using the
method described below, to obtain the sought solution $u$ of the boundary value
problem.

Quadrature by Expansion (`QBX') is an approach to singular and
near-singular quadrature in the setting of layer potential evaluation that is kernel- and
dimension-independent. QBX makes use of the fact that the layer potential is
analytic and accurately resolved via regular quadrature methods for smooth
functions (such as Gaussian quadrature)
when the target point is sufficiently far away from the surface.
Accuracy in the near or on-surface
regime is recovered through analytic continuation by ways of, e.g.\ a Taylor
(`local') expansion of the potential about a center in the well-resolved regime.

Since the earliest days of the numerical use of integral equation
methods~\cite{atkinson1976survey}, acceleration of the otherwise quadratic (in
the number of degrees of freedom) runtime of the associated matrix-vector
products has been a concern. If no acceleration is used, the integrals
of~\eqref{eqn:slp-definition} and~\eqref{eqn:dlp-definition} must be evaluated
from scratch for each of $O(N)$ target points, where each such integration
involves evaluation of the integrand at $O(N)$ quadrature nodes. Acceleration
approaches range from custom methods based on the hierarchical decomposition
of curves~\cite{rokhlin:1985:fmm-without-tree} to evaluation
methods~\cite{christlieb2006grid} based on
Barnes-Hut-style~\cite{barnes1986hierarchical} tree codes. In these methods, the coexistence of
quadrature and acceleration is a pervasive concern. When used to solve PDE BVPs
as described above, layer potential evaluation may be viewed as two distinct
tasks: First, evaluation of the potential on the surface itself as required for
the solution of the integral equation to obtain the density, and, second,
evaluation of the potential in the volume to obtain the actual solution of the boundary value problem.
Kussmaul-Martensen quadrature~\cite{kussmaul1969numerisches,martensen1963methode}
as a representative of singularity subtraction techniques,
the polar coordinate transform~\cite{kress:2014:integral-equations}
as a representative of singularity cancellation techniques,
or Generalized Gaussian quadrature~\cite{yarvin_generalized_1998} are examples
of a quadrature scheme only suited to the evaluation of (weakly) singular on-surface
layer potential integrals.

Meanwhile, the evaluation of layer potentials in the volume is in principle
straightforward as no singular integrals are involved, for example by adaptive
quadrature. Careful management of accuracy that avoids dramatic performance
degradation however is less straightforward to achieve~\cite{helsing_2008a,
biros_embedded_2004}. QBX, as a quadrature
scheme, unifies on-surface and off-surface
evaluation~\cite{barnett:2014:close-eval} with only minor accommodations, and
we would like to retain this feature of the method in its accelerated version.

Beyond this overview, we will not attempt to review the vast
literature on singular quadrature (e.g.~\cite{helsing_2008b,
lowengrub_1993,goodman_1990,haroldson_1998,mayo_1985,beale_lai_2001,davis_1984,hackbusch_sauter_1994,graglia_2008,
jarvenpaa_2003,schwab_1992,khayat_2005,bruno_2001,ying_2006,bremer_nonlinear_2010,farina_2001,strain_1995,johnson_1989,
sidi_1988,carley_2007,atkinson_1995,lyness_numerical_1967,chapko_numerical_2000,hao_high-order_2014}).
We instead refer the reader to~\cite{klockner:2013:qbx} for a rough overview.
Here, we will continue by
focusing instead on approaches to combining singular quadrature with a fast
algorithm into a single scheme.

The use of hierarchically-based fast algorithms for the evaluation of layer
potentials considerably predates the Fast Multipole Method itself, such as
Rokhlin's early work aimed squarely at accelerated
quadrature~\cite{rokhlin:1985:fmm-without-tree}. Within the framework of the
Fast Multipole Method, quadrature methods that require special treatment of near
interactions often proceed by replacing the direct interactions (of `List 1' in
the FMM) with their own procedure. Unfortunately, no guarantees of geometric
separation between source and target can be derived from membership in List 1, and so methods
requiring this may subtract out unwanted interactions already mediated by the
FMM, at an additional computational cost.

Within the realm of the acceleration of QBX, early
work~\cite{klockner:2013:qbx,epstein:2013:qbx-error-est} remarked on the
apparent ease with which QBX might be integrated into Fast Multipole-type
algorithms, by slightly modifying the algorithm to yield local expansions
containing contributions from the entire source geometry in what has come to be
called \emph{global QBX}.  First steps towards the realization of such an
integration were soon made, first in unpublished work. These early attempts were
plagued by uncontrolled and poorly-understood accuracy issues. An initial
approach to recovering accuracy through an increase of the FMM
order~\cite{rachh:2017:qbx-fmm,rachh2015integral} succeeded, but provided only
empirical evidence for its accuracy.  We refer to this order-increase scheme as
the `conventional QBX FMM' throughout this text. QBKIX~\cite{rahimian:2017:qbkix}
(Quadrature by Kernel-Independent Expansion) emerged as a related, global
QBX-based numerical method that is built upon the machinery of
kernel-independent Fast Multipole Methods~\cite{ying2004kernel}.

QBX may also be operated as a local correction applied in the near-field of an
FMM, as
described above. These schemes, broadly classified as \emph{local
QBX}~\cite{rachh2015integral,rachh2018fast}, are algorithmically much simpler
than global QBX since a fast algorithm for point potential evaluation may be used largely unmodified. However,
to allow the transition between QBX-regularized near-field and point-potential
far-field to occur without loss of accuracy, schemes based on local QBX
generically require very high QBX expansion orders, which in turn requires a
large amount of oversampling.  Recent work~\cite{siegel2017local} has been
seeking to mitigate the computational cost of this effect.

This contribution is concerned with presenting a version of a global QBX-based
FMM coupling that provides \emph{rigorous error bounds}, thus providing one
approach for a compatible coupling of a singular quadrature rule with acceleration.
To accomplish this, we make substantial modifications to the Fast
Multipole Algorithm itself. We list these comprehensively in Section~\ref{sec:algorithm}.
Some versions of some of these modifications have been discussed in the
literature, though in contexts unrelated to layer potential evaluation.
For example, we restrict the set of allowable multipole-to-local
translations to be between boxes separated by a distance of at least twice their
own size. Greengard already discusses the possibility of such a modification, in
the context of a three-dimensional generalization of the FMM, in his thesis
work~\cite{greengard:1988:thesis}. We also introduce the notion of sizing for
targets, to accommodate the unique requirements of global QBX centers. A related
need emerged in the chemical physics community, where Coulomb interactions
between (extent-bearing) `clouds' of charge need to be
evaluated~\cite{white_continuous_1994}, though the algorithm ultimately
constructed is substantially different from ours.

We refer to our algorithm as \emph{\algbrand} (for `GeometrIc
Global Accelerated QBX'), to contrast with prior versions of the scheme.
In this paper, we take the point of view that the cumulative error
in an accelerated QBX scheme effectively splits into three additive components:
\begin{equation}
\label{eqn:error-splitting}
|\text{accelerated QBX error}| \leq |\text{truncation error}| + |\text{quadrature error}| +
|\text{FMM error}|.
\end{equation}
Here, \emph{truncation error} refers to the analytical truncation error, and
\emph{quadrature error} is error in evaluating the QBX-regularized integral using
quadrature. The \emph{FMM error} refers to the FMM's achieved accuracy in approximating
the output of unaccelerated QBX (see Section~\ref{sec:motivation} for
details). We choose to split error contributions in this way, rather than relying on direct
approximation of the layer potential by the FMM\@. This interpretation allows us to rely on
the existing body of work establishing bounds on the truncation and
quadrature components of the error
(e.g.~\cite{epstein:2013:qbx-error-est,rachh:2017:qbx-fmm,afklinteberg:2016:quadrature-est,
afklinteberg:2017:adaptive-qbx}).  A complicating factor for the
FMM error analysis is that traditional FMM error estimates
apply only to the approximation of potentials at point-shape targets, whereas
our version of these estimates must account for the approximation of a \emph{QBX
local expansion} and \emph{its} accuracy when evaluated as an approximation the
potential. The FMM error in this setting is not well-studied.


\section{Overview and Motivation}%
\label{sec:motivation}
In this section, after reviewing the basic operating principle and convergence theory
of QBX in Section~\ref{sec:qbx}, we summarize recent progress made
towards making global, unaccelerated QBX geometrically robust~\cite{rachh:2017:qbx-fmm}
which we continue to leverage (see Section~\ref{sec:qbx-geometry-preprocessing}).
We summarize the approach taken toward
acceleration in the same article for comparison and to contrast with
our version in Section~\ref{sec:qbx-fast-alg}. We assume that the
reader has some familiarity with Fast Multipole methods but not necessarily the
details of~\cite{rachh:2017:qbx-fmm}. The major
difference between our new algorithm and the algorithm
of~\cite{rachh:2017:qbx-fmm} occurs in the handling of QBX centers within the
FMM\@. The motivation for how we handle QBX centers is detailed in
Section~\ref{sec:qbx-experimental-accuracy}.

FMMs conventionally evaluate potentials from point sources. In the remainder of
this paper, we will often refer to that family of algorithms as \emph{point
  FMMs}, especially to distinguish from our intended task, which is the
evaluation of potentials originating from a layer $\Gamma$, i.e.\ layer
potentials.

\subsection{Quadrature by Expansion}%
\label{sec:qbx}

We consider the problem of evaluating the Laplace single layer
potential over a smooth simple closed curve $\Gamma$ in a region
that is close to the boundary itself (which may include points on the boundary).
This is done solely to simplify and focus the discussion; as with many
other parts of the QBX literature, our approach generalizes
straightforwardly to many more layer potentials and kernels.
We discretize the integrals of~\eqref{eqn:slp-definition} and~\eqref{eqn:dlp-definition}
by subdividing $\Gamma$ into disjoint pieces $\Gamma_k$,
each parametrized by $\gamma_k:[-1,1] \to \Gamma_k$,
and by using Gaussian quadrature of a fixed node count $\pquad$ on each piece. Because of the
nearly or weakly singular behavior of the integrand, Gaussian
quadrature will not yield accurate results when the target point $x$
is near the boundary $\Gamma$. The precise
width of this neighborhood depends on both the quadrature node count $\pquad$ and the
panel length $h$ but it is usually on the order of a panel length.
For a thorough experimental account of this phenomenon,
see the first section of~\cite{klockner:2013:qbx}. By
contrast, when $x$ is far enough away from the boundary, the integrand
is smooth and the integral may be easily evaluated to high accuracy
with a smooth quadrature rule.

Roughly speaking, QBX extends the `high accuracy' quadrature region by making use of
the fact that the potential is \emph{analytic} in the set $\mathbb{R}^2
\setminus \Gamma$. QBX proceeds in two stages:

\emph{First stage.} First, Taylor expansion centers are placed away from the boundary $\Gamma$ in the
high accuracy region. Let $c$ be the expansion center associated with a
target/evaluation point $x$.
For simplicity, in the remainder of the paper we will identify $\mathbb{R}^2$
with $\mathbb{C}$ and make use of the complex logarithm, which satisfies $\Re \log(z) = \log |z|$
for all $z \in \mathbb{C} \setminus \{0\}$.
We truncate the Taylor expansion and write the $\pqbx$-th order
expansion as follows, interchanging the order of summation and integration:
\begin{align*}
\mathcal{S} \mu(x)
\approx -\frac{1}{2 \pi} \Re \sum_{n=0}^\pqbx
  \left[ \frac{1}{n!} \int_{\Gamma} \mu(y)
  \left(\frac{d^n}{dc^n} \log(y - c) \right) \, ds(y) \right]
{(x-c)}^n.
\end{align*}

\emph{Second stage.} The occurring (nonsingular) integrals are discretized
using (in our case Gaussian) quadrature. Let ${\{y_i\}}_{i=1}^N$ be the
set of quadrature nodes with weights and arc length elements
${\{w_i\}}_{i=1}^N$. Then the formula for QBX is
\begin{equation}
\label{eq:qbx}
\mathcal{S} \mu(x)
\approx
\mathcal{S}_{\text{QBX}(\pqbx, N)} \mu(x)
:=
- \frac{1}{2 \pi} \Re \sum_{n=0}^\pqbx
  \left[ \frac{1}{n!} Q_N\left\{ \int_\Gamma \mu(y)
  \left(\frac{d^n}{dc^n} \log(y - c) \right)
  ds(y)
  \right\}
  \right]
{(x-c)}^n,
\end{equation}
where $Q_N$ denotes the approximate computation of an integral
of a smooth function through the application of a
quadrature rule:
\[
  Q_N\left\{\int_\Gamma f(y)ds(y)\right\}= \sum_{i=1}^N w_i f(y_i).
\]

For purposes of solving integral equations, one is mostly interested in the case
where $x \in \Gamma$. QBX handles this case with high order accuracy, provided
that $\mu$ and $\Gamma$ are smooth~\cite{epstein:2013:qbx-error-est}. More precisely, we state
the error estimates in the case where the quadrature is composite Gaussian
quadrature over panels of equal length $h$. Let $r$ be the expansion radius for
the Taylor expansion. Let $\overline{B}(c,r)$ denote the closed disk of radius $r$ centered
at $c$. Assume that $x \in (\overline{B}(c, r) \cap \Gamma)$ and that $\Gamma \setminus
\{x\}$ does not intersect $\overline{B}(c, r)$. Let $\pquad$ denote the node count of the Gaussian
quadrature. Then it can be shown that the error components scale as
\begin{align}
  |\text{truncation error}|
  & \leq
    \frac{1}{2 \pi} \left| \sum_{n=\pqbx+1}^\infty
     \left[ \frac{1}{n!} \int_{\Gamma} \mu(y)
      \left(\frac{d^n}{dc^n} \log(y - c) \right) \, ds(y) \right]
    {(x-c)}^n
    \right| \nonumber \\
  & \le C_1(\pqbx,\Gamma) \| \mu \|_{C^p(\Gamma)} r^{\pqbx+1} \log
  {\frac{1}{r}}, \label{eq:qbx-truncation-estimate} \\
  |\text{quadrature error}|
  & \leq
  \frac{1}{2 \pi} \left| \sum_{n=0}^\pqbx
  \left[ \frac{1}{n!}
    \left(\int_{\Gamma} - Q_N\left\{\int_{\Gamma}\right\}\right)
    \mu(y)
    \left(\frac{d^n}{dc^n} \log(y - c) \right) ds(y) \right]
       {(x-c)}^n
       \right| \nonumber \\
  & \le C_2(\pquad,\pqbx,\Gamma) {\left(\frac{h}{4r}\right)}^{2\pquad} \| \mu
       \|_{C^{2\pquad}(\Gamma)}. \label{eq:qbx-quadrature-estimate}
\end{align}
If we choose $r=\alpha h$ with a proportionality factor $\alpha$
so that $h/(4r)<1$, we obtain a scheme that is accurate of order $\pqbx+1$
in the mesh spacing $h$ up to controlled precision
${1/(4\alpha)}^{2\pquad}$.  See~\cite{epstein:2013:qbx-error-est} for more
details about the error estimates for QBX\@.
This choice reveals $\pquad$ as a free parameter that governs the
controlled precision term ${1/(4\alpha)}^{2\pquad}$, where it is mostly the
requirement of resolving the high derivatives of the kernel occurring
in~\eqref{eq:qbx} that governs the magnitude of $\pquad$.
\subsection{Ensuring Accuracy}%
\label{sec:qbx-fast-alg}
The assumptions required by QBX convergence theory will not
necessarily be met by input geometries supplied by a user. In
addition, quadrature resolution and placement of centers need
to be carefully controlled so as to retain bounds on quadrature
and truncation error.

An efficient algorithm that accomplishes this is the main contribution
of~\cite{rachh:2017:qbx-fmm}. We briefly and informally review this
procedure.
\subsubsection{Preprocessing the Geometry}%
\label{sec:qbx-geometry-preprocessing}
To ensure the accuracy of QBX (independently of any fast
algorithm), we check for the following situations which may result
in inaccuracy:
\begin{itemize}
\item \emph{When a QBX disk intersects the source curve at more points than the
  target point}. This comes from the analytical requirement that the QBX
  expansion disk must not be obstructed by any piece of the source curve.

\item \emph{When resolution and separation of source geometry from a
  QBX expansion center do not guarantee accurate coefficient
  computation}. This issue occurs in the presence of varying panel sizes.
  Depending on the quadrature rule and the panel size, the
  quadrature
  contribution from the singular integrand may not be resolved adequately if the
  source geometry is too close to the QBX center. This happens when a large
  panel is close to a small panel: quadrature from the large panel may not
  adequately resolve the integrand when evaluated at a QBX center near the small
  panel.
\end{itemize}

Because of the assumption of smooth, non-self-intersecting source
geometry $\Gamma$, both of these sources of error can be controlled by
\emph{iterative refinement}, such as by repeated bisection of panels.
In the first case, bisecting the source panel of the `disturbed' center,
by the proportionality $h=\alpha r$,  draws the expansion center closer to the panel and
hence, if applied often enough, the resulting expansion disk will eventually avoid
the conflicting geometry. In the second case, the offending source
panels can be iteratively refined to grow the region in which accurate
coefficient quadrature is achieved to include the target center.
Both of these checks can be efficiently implemented using a mechanism
termed \emph{area queries} in~\cite{rachh:2017:qbx-fmm}.

Once the geometry has been processed to ensure that there are no obstructions to
accurate quadrature and control of truncation error, it remains to choose a
quadrature node count.  The resulting $\pquad$ typically exceeds what might be
required to resolve the density and the geometry by some factor. Especially when
solving integral equations, it is thus natural to maintain density and geometry
at a suitable resolution and `upsample' them to $\pquad$ nodes for QBX
computation. $\pquad$ can be empirically estimated, as in Table~1
of~\cite{rachh:2017:qbx-fmm}, or adaptively determined based on analytic
knowledge~\cite{afklinteberg:2016:quadrature-est}.

\subsubsection{Placing Centers and Identifying QBX Targets}%
\label{sec:qbx-center-placement}

Evaluation of the potential at a target point requires special
treatment (e.g.\ by QBX) when it is so close to a source panel that the
underlying (here, Gaussian)
quadrature cannot resolve the integrand for the target. Any such
targets need to be identified and associated with a QBX center whose expansion disk
contains it. Like the geometry processing tasks of Section~\ref{sec:qbx-geometry-preprocessing},
both \emph{identifying} a target that is too close to the source and
\emph{finding} a QBX center for the target can be accomplished
efficiently using area queries as described in~\cite{rachh:2017:qbx-fmm}.

On the question of center placement, the most straightforward approach (and the
one used here) is to place QBX centers at $t \pm
\frac{h_k}{2}\hat n(t)$, where $t\in\Gamma$ is a target point, $h_k$ is the length of the panel $\Gamma_k$
\[
  h_k:=\frac 12 \int_{-1}^1 |\gamma_k'(t)| dt
\]
containing $t$, and $\hat n_t$ is the unit normal at $t$. This ensures that the QBX disks cover
most of the area near the (smooth) source curve. There may be gaps in coverage
where a target point needing close evaluation does not fall inside a QBX disk;
this occurs during volume evaluation of the layer potential for target points
\begin{tikzpicture}[baseline={(0,0)},scale=0.8]
  \draw [draw=none,use as bounding box] (-1.6,0) rectangle (1.6,0.1);
  \draw (-0.7,1.5) ++(35-90:1.5) arc (35-90:-35-90:1.5);
  \draw (0.7,1.5) ++(35-90:1.5) arc (35-90:-35-90:1.5);
  \draw (-1.5,0) -- (1.5,0);
  \fill (0, 0.07) circle (1pt);
\end{tikzpicture}
very close to the surface. We currently treat such gaps
by refining the source geometry until all required target points
are covered, at considerable computational expense.
Empirically, we have observed that simply associating targets with QBX
centers even if they fall outside their closest QBX center's expansion
disk by a given factor, possibly by up to 20\%, leads to little or no
observable loss in accuracy, though such use is not covered by
theoretical guarantees. Improvements on either strategy are the
subject of future investigation.

\subsection{Evaluating the Potential with an FMM}%
\label{sec:qbx-fmm}

If QBX is implemented following (\ref{eq:qbx}) directly, then
a quadratically-scaling computational cost $O(NM)$ is incurred by
evaluating the contribution of the $N$ source points at each of the $M$ targets.
Making use of the point-discrete form of the quadrature-discretized
sources and interpreting the summation in
(\ref{eq:qbx}) as the evaluation of a local expansion of a potential due to $N$ source charges
provides an avenue by which QBX can be accelerated to an $O(N + M)$ scheme by ways
of a variant of the Fast Multipole Method (FMM).

To achieve this, our work follows the strategy used
in~\cite{rachh:2017:qbx-fmm} which treats a QBX center as a special kind of FMM
target at which the FMM \emph{forms
a local expansion} instead of \emph{evaluating} the potential.
These expansion center targets participate in the FMM algorithm in much the same way
point-shape targets do; in particular, they are each `owned' by a box in the FMM
tree.
In principle, the only additional capability required of a Taylor global QBX FMM
is the accurate evaluation of local expansion coefficients, i.e.\ higher-order partial
derivatives of the potential. The ability to compute
one or two derivatives of the potential is a common feature in production FMM codes.
Such numerical differentiation is commonly associated
with some loss of accuracy.
Since the coefficient order (and hence the number of derivatives) in QBX can be
substantial, there is the possibility of substantial loss of accuracy.
This notion is empirically confirmed in~\cite{rachh:2017:qbx-fmm}, and
an empirically-determined increase in the FMM order is suggested as
a remedy. They key contribution of this article is to (a) introduce a
modified algorithm that does not require an artificial order increase
and (b) provide convergence theory that gives concrete error bounds for
layer potential approximated in this manner.

Recall that, in an adaptive point-evaluation
FMM~\cite{carrier:1988:adaptive-fmm}, the potential at a target point
is computed from three parts (each of which is often the sum
of further contributions), summarized in the first two columns of
Table~\ref{tab:pot-contributions}.
The conventional QBX FMM~\cite{rachh:2017:qbx-fmm} replaces each of these
with a translation to a local expansion as appropriate, summarized in
the last column of Table~\ref{tab:pot-contributions}.

\begin{table}[t]
\centering
\begin{tabular}{lcc}
  \toprule
  Interaction [List] & Point FMM~\cite{carrier:1988:adaptive-fmm} & Conv.\ QBX FMM~\cite{rachh:2017:qbx-fmm}\\
  \midrule
  Near neighbor boxes [1/U] & point $\to$ point eval. & point $\to$ local transl.\\
  Sep.\ smaller mpoles [3/W] & mpole $\to$ point eval. & mpole $\to$ local transl.\\
  Far field & local $\to$ point eval. & local $\to$ local transl.\\
  \bottomrule
\end{tabular}
\belowtableskip%
\caption{%
  Contributions to the potential in a point-evaluation FMM and
  the conventional QBX FMM of~\cite{rachh:2017:qbx-fmm}. `Mpole'
  and `local' refer to multipole and local expansions respectively,
  and `eval.' and `transl.' refers to expansion evaluation and
  translation respectively.
}%
\label{tab:pot-contributions}
\end{table}

\subsection{Accuracy of Using Translated Local Expansions for QBX}%
\label{sec:qbx-experimental-accuracy}

\begin{figure}[ht]
  \begin{minipage}[t]{.48\linewidth}
    \begin{accfigadjust}
      \input{list2.pgf}%
    \end{accfigadjust}
    \caption{%
      Accuracy  of obtaining a QBX expansion by multipole-to-local translation
      (vs.\ direct computation) for an interaction that  may be encountered in
      `List 2' of an FMM\@.
    }%
    \label{fig:list2-sfw-accuracy}
  \end{minipage}
  \hfill
  \begin{minipage}[t]{.48\linewidth}
    \begin{accfigadjust}
      \input{list3.pgf}%
    \end{accfigadjust}
    \caption{%
      Accuracy  of obtaining a QBX expansion by multipole-to-target translation
      (vs.\ direct computation) for an interaction that  may be encountered in
      `List 3' of an FMM\@.
    }%
    \label{fig:list3-sfw-accuracy}
  \end{minipage}
\end{figure}

By \emph{accuracy}, here and in Section~\ref{sec:error-estimates}, we mean the
fast algorithm's accuracy in approximating the terms of~\eqref{eq:qbx}.  This
notion is distinct from (though closely related to) the accuracy of the
underlying point FMM\@. The main difference is that the point FMM approximates a
potential but the QBX FMM approximates the local expansion of a potential.

A key detail not explicitly considered in the modifications of
Table~\ref{tab:pot-contributions} is that QBX expansion disks,
unlike target points, have an extent. Since no geometric constraints
are imposed, some expansion disks will almost inevitably cross box boundaries.
Using the above notion of accuracy, it is easy to imagine that this
might have an adverse influence on the accuracy of the computed QBX
expansion, owing to either reduced separation from source
boxes or larger separation of evaluation points from expansion centers
than allowed by FMM separation criteria.

To frame the discussion and give the reader an intuitive sense of this
issue, this section presents numerical examples of interactions that
\emph{may plausibly occur} in the conventional QBX FMM which lead to large losses
in accuracy. We also give an intuitive idea of how our method avoids these errors.
We defer a precise statement of the algorithm and a proof of its
accuracy to Sections~\ref{sec:error-estimates}
and~\ref{sec:algorithm}. We consider a number of different types
of interactions occurring in an FMM, and we demonstrate the
possibility of inaccuracy in each.  We consider the potential originating
from a single source of unit strength and a single
expansion center separated by a reference distance.  All expansions
have order 8 unless stated otherwise.

Figure~\ref{fig:list2-sfw-accuracy} portrays an interaction from a
point source (blue, right) to a QBX expansion center (orange, middle). A black circle
indicates the size of the QBX expansion disk about the center. In a typical
usage scenario of QBX, the source may contribute to an expansion of the layer
potential about the center, which is then evaluated back at the source, for,
say, the computation of the one-sided limit of the layer potential at the source. We now
consider an evaluation scenario in which FMM acceleration mediates this
interaction through a multipole (shown as `$m_8$') and a local expansion (shown as
`$\ell_8$'). While we have chosen this placement to be `adversarial' (i.e.\ to lead to
a large loss in accuracy), the scenario is permissible under the
rules of the conventional QBX FMM\@, since all required conditions are met:
The source point lies inside the box for which the multipole expansion is
formed, the source and target box are `well-separated', and the target QBX
expansion center lies inside of the target box.

A sufficient criterion to assure that
the accelerated and unaccelerated version of the scheme yield the same potential is
that the QBX local expansions computed directly and by ways of
intermediate expansions evaluate to the same potential, up to FMM accuracy.
The false-color plot of Figure~\ref{fig:list2-sfw-accuracy} shows the magnitude
of the difference between those two expansions in the scenario.
In the point FMM,
a coarse estimate of
multipole-to-local (`M2L' for short) accuracy for eighth-order expansions
evaluated within the target box gives
\begin{equation}
{\left(
\frac{\text{dist}(\text{box center}, \text{furthest target})}
     {\text{dist}(\text{box center}, \text{closest source})}
\right)}^{p+1}
\leq
{\left( \frac{\sqrt{2}}{4-\sqrt{2}} \right)^{9}\approx 4.4 \cdot 10^{-3}.}
\label{eq:rough-m2l-error}
\end{equation}
We have not yet demonstrated the applicability of such an estimate to the QBX
case (cf. Lemma~\ref{lem:m2l2qbxl}), but the data show that, for evaluation when
the source point is also the target, the expansion computed through intervening
multipole and local expansions misses this accuracy goal by a noticeable margin.

Analogously inaccurate approximation of QBX interactions for evaluation back at
the source may occur not just in multipole-to-local, but also in other types of
interactions in the FMM\@.  In Figure~\ref{fig:list3-sfw-accuracy}, the
false-color plot again shows the magnitude of the difference between the
potential obtained from evaluating the QBX expansion computed directly from the
source and the QBX expansion obtained indirectly by ways of intermediate
expansions, this time from a multipole expansion associated with a small box
containing the source point directly to QBX expansion center `target' within a
larger target box. Such an interaction may occur through \emph{List 3} in the
conventional QBX FMM\@.  Figure~\ref{fig:list4-sfw-accuracy} similarly shows
a source-to-local interaction of the type one might encounter in a \emph{List 4}
of the conventional QBX FMM\@.

\begin{figure}[ht]
  \begin{minipage}[t]{.48\linewidth}
    \begin{accfigadjust}
      \input{list2-less-bump.pgf}%
    \end{accfigadjust}
    \caption{%
      Accuracy  of obtaining a QBX expansion by multipole-to-local translation
      (vs.\ direct computation) for an interaction that  may be encountered in
      `List 2' of an FMM\@. In this experiment, the QBX order is lower than
      the order of the intermediate multipole and local (FMM) expansions.
    }%
    \label{fig:list2-less-bumped}
  \end{minipage}
  \hfill
  \begin{minipage}[t]{.48\linewidth}
    \begin{accfigadjust}
      \input{list2-bump.pgf}%
    \end{accfigadjust}
    \caption{%
      Accuracy  of obtaining a QBX expansion by multipole-to-local translation
      (vs.\ direct computation) for an interaction that  may be encountered in
      `List 2' of an FMM\@. In this experiment, the QBX order is lower than
      the order of the intermediate (FMM) multipole and local  expansions.
      Compared with Figure~\ref{fig:list2-less-bumped}, this experiment explores the
      effect of increasing the order of the intermediate (FMM) multipole and
      local expansions.
    }%
    \label{fig:list2-bumped}
  \end{minipage}
\end{figure}

The experiments described so far still paint an incomplete picture of the
translation process involved in accelerating QBX\@. For a more complete
understanding, consider that the FMM
order affects the accuracy of the potential by ways of the multipole-to-local
error, in a form like~\eqref{eq:rough-m2l-error}, whereas the QBX order controls
$h$-convergence as in~\eqref{eq:qbx-truncation-estimate} up to controlled precision,
as in~\eqref{eq:qbx-quadrature-estimate}.
As a result, the QBX order is typically lower than the FMM order, a fact that is
not reflected in our experiments thus far. Figure~\ref{fig:list2-less-bumped}
shows the result of a first experiment that takes this into consideration.
Denote the lower-order QBX expansion obtained directly
from the source, here of order $q$, by $\ell_{q,\text{direct}}$.
Further, denote the local expansion of order $q$ centered at the same location,
obtained through a multipole-to-local chain of order $p$ as pictured by $\ell_{q,\text{M2L}(p)}$.
Then the top and outer parts of Figure~\ref{fig:list2-less-bumped} show
$|\ell_{3,\text{direct}}-\ell_{8,M2L(8)}|$, while the bottom part shows
$|\ell_{3,\text{direct}}-\ell_{3,M2L(8)}|$.
A first observation from this experiment is that
$\ell_{3,M2L(8)}$ better approximates $\ell_{3,\text{direct}}$
than $\ell_{8,M2L(8)}$. While atypical from the point of view of conventional M2L error
estimation theory (where high order entails higher accuracy), this is also not
entirely surprising, as the translation chain is bound to approximate
lower-order coefficients more accurately than higher-order ones. In other words,
simply truncating $\ell_{8,M2L(8)}$ leads to higher accuracy. While this
argument is intuitively immediately appealing, we are not aware of any estimates
that would aid in quantifying the effect. Next, we observed in our earlier
experiments that M2L-mediated expansions did not achieve `conventional' M2L
accuracy for QBX evaluation at the source point. Based on the results of
our latest experiment, we still cannot confidently assert that these tolerances
are being met here. We can however predict that, as long as the order of the final
M2L-mediated (`QBX') local expansion is being kept constant, increasing the
intervening M2L orders should improve the approximation of the individual coefficients
of $\ell_{3,\text{direct}}$. This prediction is borne out by
the experiment of
Figure~\ref{fig:list2-bumped} which analogously to Figure~\ref{fig:list2-less-bumped} compares
$|\ell_{3,\text{direct}}-\ell_{15,M2L(15)}|$ with
$|\ell_{3,\text{direct}}-\ell_{3,M2L(15)}|$.

\begin{figure}[ht]
  \begin{minipage}[t]{.48\linewidth}
    \begin{accfigadjust}
      \input{list4.pgf}%
    \end{accfigadjust}
    \caption{%
      Accuracy experiments for QBX-FMM coupling with interactions as found in an FMM\@.
      QBX FMM error for a point-to-local (`List 4') interaction.
    }%
    \label{fig:list4-sfw-accuracy}
  \end{minipage}
  \hfill
  \begin{minipage}[t]{.48\linewidth}
    \begin{accfigadjust}
      \input{list2-zero-stickout.pgf}%
    \end{accfigadjust}
    \caption{%
      Accuracy experiments for QBX-FMM coupling with interactions as found in an FMM\@.
      QBX FMM error for a List 2
      interaction, with the expansion confined to a small region extending beyond the
      box containing the center.%
    }%
    \label{fig:list2-stickout}
  \end{minipage}
\end{figure}

This is the basic mechanism by which the conventional QBX FMM of~\cite{rachh:2017:qbx-fmm}
achieves accuracy. Looking ahead to the results obtained, Table~\ref{tab:starfish-accuracy-old} summarizes
the accuracy achieved in a verification of Green's formula
$\mathcal S(\partial_n u)-\mathcal D(u)=u/2$ for a harmonic $u$ across a range of FMM orders $\pfmm$
and $\pqbx$ for a reasonably simple test geometry by that scheme. We find these results
unsatisfactory for two reasons: First, the evidence supporting the attained
accuracy, while surprisingly robust across geometries in practice, is empirical.
Second, considering the results shown in Table~\ref{tab:starfish-accuracy-old}
for high QBX orders $\pqbx$ (and thus high relative accuracies),
the required FMM order quickly becomes very large if high accuracy is
desired. In the remainder of this contribution, we pursue a different
strategy that addresses both of these issues.

\subsection{Improved Accuracy Through a Geometric Criterion}%
\label{sec:fmm-alg-tweaks}

Perhaps the foremost problem with the above translation schemes in the context
of the conventional error estimates for multipole and local expansions is that
they permit---for QBX purposes---inaccurate near-field contributions mediated through
multipole and local expansions to enter the QBX local expansion.
As shown, increasing the order of those expansions can (empirically) mitigate
this circumstance. We prefer to rework the fast algorithm so as to prevent those
contributions in the first place. Roughly speaking, this requires separating
contributions in the `far field' \emph{of the QBX local expansion disk} (not just the box) from the inaccurate
`near-field' ones. As we will show, the far-field ones may be computed with
intervening translations without endangering accuracy. We have shown above that
intervening translations can considerably damage accuracy for the near-field
evaluation. It is useful to realize that the Fast Multipole algorithm
already contains a mechanism for handling this type of issue; its chief purpose,
after all, is to separate a far-field that is easily approximated from a
near-field that does not tolerate approximation. It is thus natural to seek to
broaden the FMM's notion of a near-field so as to respect the needs of QBX\@.
Wishing to avoid the scenarios that led to loss of accuracy above, we begin with
the coarse notion that we wish to avoid expansion-`accelerated' contributions to
QBX local expansions which would not meet the same accuracy target as the FMM
itself.

A first algorithmic variant that provides sufficiently strong accuracy
guarantees is nearly immediate: One may require that the entirety of a QBX expansion disk be
contained inside some FMM (potentially non-leaf-level) target box. From there, it is at least intuitively
plausible that the conventional FMM interaction patterns and their associated
error estimates might generalize to guarantee accurate multipole-/local-mediated far-field contributions
to QBX local expansions inside each box. This relatively simple generalization
of the FMM already represents a somewhat large algorithmic change: While in the
original FMM, target particles can only occur in leaf boxes, confining a center to its box
entails that \emph{non-leaf} boxes may also contain QBX expansion disk targets.
We call this restriction a \emph{target confinement rule}, named this way because the
target QBX disks are confined to the inside of a box.
This modification achieves the desired accuracy (rigorously, as we will
show in Section~\ref{sec:error-estimates}). Unfortunately, it is unsuitable in practice
because it no longer has linear complexity---neither in theory nor in practice.
In fact, the restriction  may lead QBX expansion disk targets
that overlap the boundaries of boxes near the root of the tree to exist at
near-root levels. Such QBX expansion disks, of which there could be a large
number, interact with nearly the entire geometry without the benefit of
multipole acceleration.

A second algorithmic variant that remedies this is again virtually immediate:
Let QBX expansion disks with a center inside a target box not be confined to the
strict extent of their containing box, but instead allow them to extend beyond
it by a constant factor of the box size, called the \emph{target
confinement factor} (`\emph{TCF}').  Intuitively, this ensures that each
expansion disk may propagate down the tree (away from the root) until it reaches
a box whose size is commensurate with the disk's own diameter. It is perhaps
plausible that such a scheme might no longer be subject to superlinear
complexity. However, the price for the lower cost is that obtaining guaranteed
accuracy requires a more complicated algorithm than in the previous
case, which we may describe as having a target confinement factor of zero.

Figure~\ref{fig:list2-stickout} provides a graphical representation. The
larger target confinement region is shown with a dashed line. It extends
beyond the boundaries of the FMM box, which are drawn using a solid line.  The
figure also shows a computational experiment analogous to the one of
Figure~\ref{fig:list2-sfw-accuracy} demonstrating that, at least in this situation, mediated
expansions accurately approximate the directly-obtained QBX expansion if they
are contained in the target confinement region.

As we will see, modifying the FMM algorithm to retain its benign characteristics
in terms of accuracy and cost under this modification presents a considerable set of challenges. At the
heart of this modification process is the choice of the target confinement
factor, which represents the main control point for the cost-accuracy trade-off
inherent in our algorithm.  To illustrate: a larger TCF may
result in worse convergence factors for nearly all FMM interactions, while
yielding smaller cost by allowing QBX expansion disks to settle closer to
the leaves of the tree.
To obtain good convergence factors in two and three
dimensions, we have chosen to modify the basic notion of `well-separated-ness'
inherent in the FMM, from, roughly speaking, `1-away' to `2-away',  similar to
the three-dimensional FMM of~\cite{greengard:1988:thesis}.
Similarly, we had to considerably rework
the criteria for interaction lists of well-separated smaller and bigger boxes
(`List 3' and `List 4').

The purpose of the remainder of this paper is to make rigorous the heuristic
arguments of the previous paragraphs. In Section~\ref{sec:error-estimates}, we
present a novel, more versatile version of the expansion translation error
estimates of~\cite{greengard:1987:fmm} that allow us to estimate the accuracy
achieved by a chain of translation operators in the presence of varying
expansion orders. In Section~\ref{sec:algorithm}, we precisely state our
algorithm and provide a complexity analysis that provides a set of benign
conditions under which linear complexity is retained. We further point out how
the analysis of Section~\ref{sec:error-estimates} can be used to understand the
accuracy of the full algorithm. We close with a comprehensive set of accuracy
and complexity experiments in Section~\ref{sec:results}.



\section{Analysis on Translation Operators with Varying Orders}%
\label{sec:error-estimates}


To arrive at a better, more quantitative understanding of the accuracy of the
translation chains examined in the previous section, we prove new
results that explore the accuracy behavior of FMM translation operators
when the expansion order varies throughout the chain. We are particularly
interested in the effect of truncation of an `upstream' expansion on the
accuracy of `downstream' expansion coefficients. The lemmas proven below, while
useful in our analysis of the \algbrand~FMM, are entirely independent of our
particular usage scenario and may prove useful in other settings.  For
simplicity and conciseness, we prove these results in the setting of the Laplace
equation in two dimensions.

\subsection{Analytical Preliminaries}

First, we recall standard facts regarding multipole and local expansions. For
proofs of these facts, we refer the reader to~\cite{greengard:1987:fmm}.

The \emph{multipole expansion} centered at
the origin due to a unit strength source $s_1$ at $z_0$ takes the form
\begin{equation}
\label{eqn:multipole}
\phi(z) = \log(z - z_0) = a_0 \log(z) + \sum_{k=1}^\infty \frac{a_k}{z^k}
\end{equation}
with $a_0 = 1$ and $a_k = -z_0^k / k$ for $k > 0$. This series converges for
$|z| > |z_0|$, where $R = |z_0|$ is called the radius of the multipole
expansion. The multipole expansion (\ref{eqn:multipole}) can also be truncated
to $(p + 1)$ terms, which we term a \emph{$p$-th order expansion}.

Two important operations on the (truncated or non-truncated) expansion
(\ref{eqn:multipole}) are (1) shifting the center of the expansion and (2)
conversion to a local expansion. The center of the multipole expansion
(\ref{eqn:multipole}) may be shifted to a new center $y$, obtaining another
multipole expansion, with coefficients ${(\alpha_k)}_{k=0}^\infty$, given by
\begin{equation}
  \label{eqn:m2m}
    \alpha_m =
    \begin{cases}
      a_0 & m = 0, \\
      %
      - \frac{a_0 (-y)^m}{m} +
      \sum_{k=1}^m \binom{m - 1}{k - 1} a_k (-y)^{m-k}
      & m > 0.
    \end{cases}
\end{equation}
The resulting expansion converges for $|z| > R + |y|$.  The multipole
expansion~\eqref{eqn:multipole} can also be converted to a local (Taylor)
expansion, centered at $y$ for $|y| > R$, with coefficients
$(\beta_k)_{k=0}^\infty$ given by
\begin{equation}
  \label{eqn:m2l}
    \beta_m =
    \begin{cases}
      %
      a_0 \log(y) + \sum_{k=1}^\infty \frac{a_k}{y^k} & m = 0, \\
      (-1)^m \left(
      \frac{a_0}{m y^m}
      +
      \sum_{k=1}^\infty \binom{m+k-1}{k-1} \frac{a_k}{y^{m+k}}
      \right) & m > 0.
    \end{cases}
\end{equation}
The series $\sum_{k=0}^\infty \beta_k(z - y)^k$ converges when $|z - y| < R -
|y|$.

The local expansion centered at the origin of the potential
due to a unit strength source at $z_0$ is the Taylor expansion
\begin{equation}
  \label{eqn:local}
  \phi(z) = \log(z - z_0) = \sum_{k=0}^\infty b_k z^k.
\end{equation}
This converges for $|z| < |z_0|$. The coefficients ${(b_k)}_{k=0}^\infty$ are
given by $ b_0 = \log(-z_0)$ and $b_m = -1/(mz_0^m)$ for $m > 0$. The main
operation on expansions of the type~\eqref{eqn:local} of importance to this
discussion is shifting the center of the local expansion. The center of a $p$-th
order local expansion of the form~\eqref{eqn:local} may be shifted to another
center $y$, with new coefficients $(\beta_m)_{m=1}^p$, given by
\begin{equation}
  \label{eqn:l2l}
  \beta_m = \sum_{k=m}^p \binom{k}{m} b_k y^{k-m}.
\end{equation}

\subsection{Error Estimates for Chained Translations}

\colorlet{zlocalcolor}{blue}
\colorlet{Gammalocalcolor}{blue!50!red}
\colorlet{z0multipolecolor}{red}

\begin{figure}
  \begin{minipage}[t]{0.48\textwidth}
  \centering
  \begin{tikzpicture}[scale=0.8]
    \tikzset{%
      disk/.style={draw, circle, thick, inner sep=0},
      shorterarrow/.style={shorten <=1pt, shorten >=1pt}
    }

    \node [disk, minimum size=3cm](Gamma) at (0,0) {};
    \node [below left] at (Gamma.south west) {$\overline B(0,r)$};
    \draw [fill] (Gamma) circle (1pt) node [below right] {$0$};
    \draw [<->, shorterarrow] (Gamma.center) -- (Gamma.west)
        node [below, midway, xshift=2pt] {$r$};

    \path (Gamma) ++ (-30:1.25cm) node [disk, color=zlocalcolor, minimum size=1cm](lexp) {};
    \path (lexp) ++ (20:0.4cm) node (y) {};
    \draw [fill] (lexp) circle (1pt) node [below] {$z$};
    \draw [fill] (y) circle (1pt) node [below] {$y$};
    \draw [shorterarrow] (-25:3cm)
        node [color=zlocalcolor] {\textrm{local}}
        edge[color=zlocalcolor,out=135,in=45,->] (lexp.north east);

    \path (Gamma) ++ (135:4cm) node [disk, color=z0multipolecolor, minimum size=1.5cm](gamma) {};
    \node [below left] at (gamma.south) {$\overline B(z_0, \lambda r)$};
    \path (gamma) ++(25:0.75cm) coordinate (source);
    \draw [fill] (gamma) circle (1pt) node [below] {$z_0$};
    \draw [fill] (source) circle (1pt) node [below] {src.};
    \draw [<->, shorterarrow] (gamma) -- (Gamma) node [right, midway] {$cr$};
    \draw [<->, shorterarrow] (gamma.center) -- (gamma.north west)
        node [right, midway, xshift=2pt] {$\lambda r$};
    \draw [shorterarrow] (90:3.25cm)
        node [color=z0multipolecolor] {\textrm{multipole}}
        edge[color=z0multipolecolor,out=-135,in=-35,->] (gamma.east);
  \end{tikzpicture}

  \caption{%
    \label{fig:m2qbxl}%
    Obtaining the local expansion of a point potential using an intermediate
    multipole expansion. The local expansion of the potential due to the source
    charge is formed by first forming a multipole expansion inside
    $\overline B(z_0,\lambda r)$ and
    then shifting to $z$.  This provides the setting for Lemma~\ref{lem:m2qbxl}.
  }
  \end{minipage}
  \hfill
  \begin{minipage}[t]{0.48\textwidth}
  \centering
  \begin{tikzpicture}[scale=0.8]
    \tikzset{%
      disk/.style={draw, circle, thick, inner sep=0},
      shorterarrow/.style={shorten <=1pt, shorten >=1pt}
    }

    \node [disk, color=Gammalocalcolor, minimum size=3cm](Gamma) at (0,0) {};
    \node [below left] at (Gamma.south west) {$\overline B(0,r)$};
    \draw [fill] (Gamma) circle (1pt) node [below right] {$0$};
    \draw [<->, shorterarrow] (Gamma.center) -- (Gamma.west)
        node [below, midway, xshift=2pt] {$r$};
    \draw [shorterarrow] (20:3cm)
        node [color=Gammalocalcolor] {\textrm{local}} edge[color=Gammalocalcolor,out=90,in=45,->]
        (Gamma.north east);

    \path (Gamma) ++ (-30:1.25cm) node [disk, color=zlocalcolor, minimum size=1cm](lexp) {};
    \path (lexp) ++ (20:0.4cm) node (y) {};
    \draw [fill] (lexp) circle (1pt) node [below] {$z$};
    \draw [fill] (y) circle (1pt) node [below] {$y$};
    \draw [shorterarrow] (-25:3.5cm)
        node [color=zlocalcolor] {\textrm{shifted local}}
        edge[color=zlocalcolor,out=135,in=45,->] (lexp.north east);

    \path (Gamma) ++ (135:4cm) node [draw=none,fill=none](gamma) {};
    \draw [fill] (gamma) circle (1pt) node [right] {$z_0$: source};
    \draw [<->, shorterarrow] (gamma) -- (Gamma) node [right, midway] {$cr$};
  \end{tikzpicture}

  \caption{%
    \label{fig:l2qbxl}%
    Obtaining the local expansion of a point potential using an intermediate
    local expansion.  The local expansion of the potential due to the source
    charge is formed inside the disk $\overline B(0,r)$ and then shifted to the
    center $z$.  This provides the setting for Lemma~\ref{lem:l2qbxl}.  }
  \end{minipage}
\end{figure}

Given our intended usage pattern in accelerated QBX, we are interested in the
following types of translation chains:
\begin{enumerate}
  \item Source $\to$ Multipole$(p)$ $\to$ Local$(q)$ (Lemma~\ref{lem:m2qbxl})
  \item Source $\to$ Local$(p)$ $\to$ Local$(q)$ (Lemma~\ref{lem:l2qbxl})
  \item Source $\to$ Multipole$(p)$ $\to$ Local$(p)$ $\to$ Local$(q)$
   (Lemma~\ref{lem:m2l2qbxl})
\end{enumerate}
The main distinction among these we encounter is whether whether the interaction is
mediated through an intermediate multipole or local
expansion, or both. The list above shows, abstractly, the order of each
expansion through the values $p$ and $q$. In our envisioned usage scenario, $q$
represents the order of the final QBX local expansion and will generically be
lower than $p$.  The reader familiar with conventional adaptive FMMs
(e.g.~\cite{carrier:1988:adaptive-fmm}) may discover a direct correspondence of
these types of translation chains and the various interaction lists used in
those algorithms.

Without loss of generality, we may assume that an interaction goes through at
most a single intermediate multipole expansion and intermediate local expansion,
occupying a single level of the FMM's hierarchy. This is due to the fact
that, absent additional truncation, the FMM `forgets' intermediate
translations in the following way: the value of a local expansion shifted
downward through a
sequence of local-to-local (\ref{eqn:l2l}) translations only depends on the
source and the \emph{initial} local expansion center. Similarly, the value of a
multipole expansion shifted upward through a sequence of multipole-to-multipole
(\ref{eqn:m2m}) translations only depends on the source and the \emph{final}
multipole expansion center. (See~\cite[Lemma 2.3 and Lemma 2.5]{greengard:1987:fmm}.)

We recall a technique from complex analysis for bounding the $n$-th derivative
of a complex analytic function. The proof can be found in~\cite[IV.2.14 on page~73]{conway:1978:complex-variables}.
\begin{proposition}%
\label{prop:cauchy-derivative-bound}%
Let $U \subseteq \mathbb{C}$ be open and let $\phi : U \to \mathbb{C}$ be a
complex analytic function. Let $z \in U, r > 0$ and suppose that $\overline B(z,
r) \subseteq U$. Then for all $n \geq 0$
\[
|\phi^{(n)}(z)| \leq \frac{n!}{r^n} \left( \max_{w \in {\overline B}(z, r)} |\phi(w)| \right).
\]
\end{proposition}

\begin{remark}
Although Lemmas~\ref{lem:m2qbxl},~\ref{lem:l2qbxl}, and~\ref{lem:m2l2qbxl} are
stated for a single source charge of unit strength, the statements can be
straightforwardly generalized for
an ensemble of $m$ charges of strengths $q_1, \ldots, q_m$, with the error bound
scaled by $\sum_{k=1}^m |q_k|$.
\end{remark}



See Figure~\ref{fig:m2qbxl} for context on the following lemma.
\begin{lemma}[Truncating a mediating multipole to $p$-th order
on a $q$-th order local]%
\label{lem:m2qbxl}%
Let $\lambda, c, r > 0$.  Suppose that a single unit strength charge is placed
in the closed disk $\overline B(z_0, \lambda r)$ with radius $\lambda r$ and center $z_0$, such that
$|z_0| \geq (c + 1 + \lambda)r$.  The corresponding multipole expansion with
coefficients ${{(a_k)}}_{k=0}^\infty$ converges in the closed disk $\overline B(0,r)$ of
radius $r$ centered at the origin.

Suppose that $y, z \in \overline B(0,r)$. Then if $|z|<r$ and $|y - z| \leq r - |z|$, the
potential due to the charge is described by a power series
\[ \phi(y) = \sum_{k=0}^{\infty} \beta_k {(y-z)}^k. \]
Fix the intermediate multipole order $p \geq 0$. For $n \geq 0$, let
$\tilde{\beta}_n$ be the $n$-th coefficient of the local expansion centered at
$z$ obtained by translating the $p$-th order multipole expansion of $\phi$:
\[
\tilde{\beta}_n = \frac{1}{n!} \frac{d^n}{dz^n}
\left(a_0 \log{(z - z_0)} + \sum_{k=1}^p \frac{a_k}{{(z-z_0)}^k} \right).
\]
Define $\omega = 1/(1 + \frac{c}{\lambda})$.  Fix the local
expansion order $q \geq 0$. Then
\[
\left| \sum_{k=0}^q \beta_k{(y-z)}^k
- \sum_{k=0}^q \tilde{\beta}_k {(y-z)}^k \right|
\leq
\left( \frac{q+1}{p+1} \right)
\left( \frac{\omega^{p+1}}{1 - \omega} \right).
\]
\end{lemma}
\begin{proof}
We may write $\beta_n - \tilde{\beta}_n$ as
\[
\beta_n - \tilde{\beta}_n
=
\frac{1}{n!}
\frac{d^n}{dz^n}
R_p(z)
\]
where the function $R_p: (\mathbb{C} \setminus \overline B(z_0, \lambda r)) \to \mathbb{C}$, defined
as
\[
R_p(z) = \sum_{k=p+1}^{\infty} \frac{a_k}{{(z - z_0)}^k}
\]
is what remains after truncating the multipole expansion of $\phi$ to $p$-th
order.

We bound the $n$-th derivative of $R_p$ at $z$. $R_p$ is complex analytic and its
domain contains the closed disk $\{ w : |w-z| \leq r - |z| \}$, so by
Proposition~\ref{prop:cauchy-derivative-bound}
\begin{equation}
\label{eqn:deriv-cauchy-estimate}
|R_p^{(n)}(z)| \leq
\frac{n!}{{(r - |z|)}^n}
\left(\max_{y \in \overline B(z,r-|z|)} |R_p(y)| \right).
\end{equation}
Recall (e.g., from~\eqref{eqn:multipole}) that the multipole coefficients
${(a_k)}_{k=1}^\infty$ satisfy
\[ |a_k| \leq \frac{{(\lambda r)}^k}{k}, \quad k > 0. \]
Using Figure~\ref{fig:m2qbxl} and noting $z\ne z_0$,
when $|y - z| \le r - |z|$ and $k > 0$, we have
\[
\left| \frac{a_k}{{(y - z_0)}^k} \right|
\leq \frac{{(\lambda r)}^k}{k{(cr + \lambda r)}^k}
= \frac{\omega^k}{k}.
\]
Noting $\omega<1$, we find
\begin{equation}
\label{eqn:max-truncation-error-on-disk}
\max_{y \in \overline B(z,r-|z|)}
|R_p(y)|
\leq \sum_{k=p+1}^{\infty} \frac{\omega^k}{k}
\leq \frac{1}{p+1} \left( \frac{\omega^{p+1}}{1 - \omega} \right).
\end{equation}
Combining (\ref{eqn:deriv-cauchy-estimate}) and
(\ref{eqn:max-truncation-error-on-disk}) yields
\begin{equation}
\label{eqn:m2qbxl-coeff-estimate}
|\beta_n - \tilde{\beta}_n|
= \left| \frac{R_p(z)}{n!} \right|
\leq \frac{1}{p+1}
\left( \frac{1}{{(r-|z|)}^n} \right)
\left( \frac{\omega^{p+1}}{1-\omega} \right).
\end{equation}
From the triangle inequality and (\ref{eqn:m2qbxl-coeff-estimate}),
we obtain the claim:
\[
\left| \sum_{k=0}^q \beta_k{(y-z)}^k
- \sum_{k=0}^q \tilde{\beta}_k {(y-z)}^k \right|
\leq \sum_{k=0}^q
\left(
  \frac{1}{p+1}
  \cdot \frac{\omega^{p+1}}{1-\omega}
  \cdot \frac{1}{{(r-|z|)}^k}
\right) {(r-|z|)}^k
= \left(\frac{q+1}{p+1}\right)
\left( \frac{\omega^{p+1}}{1-\omega} \right).
\]
\end{proof}

See Figure~\ref{fig:l2qbxl} for context on the following lemma.
\begin{lemma}[Truncating a mediating local to $p$-th order
on a $q$-th order local]%
\label{lem:l2qbxl}
Let $c, r > 0$. Suppose that a single unit strength charge is placed at $z_0$,
with $|z_0| \geq (c + 1)r$. Consider the closed disk $\overline B(0,r)$ of radius $r$
centered at the origin. Suppose that $y, z \in \overline B(0,r)$.  If $|z| < r$ and $|y -
z| \leq r - |z|$, the potential $\phi$ due to the charge is described by a power
series
\[ \phi(y) = \sum_{l=0}^{\infty} \beta_l {(y-z)}^l. \]
Fix the intermediate local order $p \ge 0$.  For $n \geq 0$, let
$\tilde{\beta}_n$ be the $n$-th coefficient of a local expansion centered at $z$
obtained by translating a $p$-th order local expansion of $\phi$ centered at the
origin:
\[
\tilde{\beta}_n = \frac{1}{n!} \frac{d^n}{dz^n}
\left(\sum_{k=0}^{p} \frac{\phi^{(k)}(0)}{k!} z^k \right).
\]
Fix the local expansion order $q \geq 0$. Define $\alpha = 1/(1 + c)$. Then
\[
\left| \sum_{k=0}^q \beta_k{(y-z)}^k
- \sum_{k=0}^q \tilde{\beta}_k {(y-z)}^k \right|
\leq \left(\frac{q+1}{p+1}\right)
\left( \frac{\alpha^{p+1}}{1 - \alpha} \right).
\]
\end{lemma}
\begin{proof}
This lemma may be proved with an argument almost identical to the proof of
Lemma~\ref{lem:m2qbxl}, so we only sketch the proof. We have
\(
\beta_n - \tilde{\beta}_n =
(1/n!)
\inlbfrac{d^n}{dz^n} R_p(z)
\)
where the complex analytic function $R_p: \overline B(0,r) \to \mathbb{C}$ given by
\( R_p(z) = \sum_{k=p+1}^\infty -z^k/(kz_0^k) \)
is the Taylor remainder of the Taylor series for $\phi$ (cf.~\eqref{eqn:local})
centered at $0$ and evaluated at $z$.

Noting $\alpha<1$, applying
Proposition~\ref{prop:cauchy-derivative-bound} to $R_p$ yields
\(
|\beta_n - \tilde{\beta}_n|
\leq 1/(p+1)
\left( 1/{(r-|z|)}^n \right)
\left( \alpha^{p+1}/(1-\alpha) \right).
\)
From this and the triangle inequality, the claim follows.
\end{proof}

Once again, see Figure~\ref{fig:m2qbxl} for context on the following lemma.
\begin{lemma}[Truncating mediating multipole and local
to $p$-th order on a $q$-th order local]%
\label{lem:m2l2qbxl}
Let $c$, $\lambda$, $r$, ${(a_k)}_{k=0}^\infty$, $\overline B(0,r)$, $\overline B(z_0, \lambda r)$, and $\phi$
be as in Lemma~\ref{lem:m2qbxl}.
Let $y,z \in \overline B(0,r)$. Then if $|z| < r$ and $|y - z| \leq r - |z|$, the
potential due to the charge is described by a power series
\[ \phi(y) = \sum_{k=0}^{\infty} \beta_k {(y-z)}^k. \]
Fix the intermediate multipole and local order $p \geq 0$. For $n \geq 0$, let
$\tilde{\zeta}_n$ be the $n$-th coefficient of the local expansion at the origin
of the potential arising from the $p$-th order multipole expansion at $z_0$
of $\phi$:
\[
\tilde{\zeta}_n = \frac{1}{n!}
\left. \frac{d^n}{dz^n} \right|_{z=0}
\left(a_0 \log{(z-z_0)} + \sum_{k=1}^p \frac{a_k}{{(z-z_0)}^k} \right).
\]
Also, let $\tilde{\beta}_n$ be the $n$-th coefficient of the local expansion
at $z$
of the potential arising from the $p$-th
order local expansion with coefficients ${(\tilde{\zeta}_k)}_{k=0}^p$:
\[
\tilde{\beta}_n = \frac{1}{n!} \frac{d^n}{dz^n}
\left(\sum_{k=0}^p \tilde{\zeta}_k z^k \right).
\]
Fix the final local expansion order $q \geq 0$.
Define $\alpha = 1/(1+c)$ and $\omega = 1/(1+\frac{c}{\lambda})$. Then
\[
\left| \sum_{k=0}^q \tilde{\beta}_k {(y-z)}^k -
\sum_{k=0}^q \beta_k {(y-z)}^k \right|
\leq
(q+1)
\left(\frac{\omega^{p+1}}{1 - \omega}\right)
+
\left(\frac{q+1}{p+1}\right)
\left(\frac{\alpha^{p+1}}{1 - \alpha} \right).
\]
\end{lemma}
\begin{proof}
For $n \geq 0$, define $\tau_n$ as the $n$-th coefficient
of the $q$-th order local expansion at $z$ of the
potential arising from the $p$-th order local expansion
of the source potential $\phi$ at the origin:
\[
\tau_n = \frac{1}{n!} \frac{d^n}{dz^n}
\left(\sum_{k=0}^p \frac{\phi^{(k)}(0)}{k!} z^k \right).
\]
From the triangle inequality,
\[
\left| \sum_{k=0}^q \tilde{\beta}_k {(y-z)}^k - \sum_{k=0}^q \beta_k{(y-z)}^k \right|
\leq
\left| \sum_{k=0}^q \tilde{\beta}_k {(y-z)}^k
- \sum_{k=0}^q \tau_k {(y-z)}^k \right|
+
\left|
\sum_{k=0}^q \tau_k {(y-z)}^k
-
\sum_{k=0}^q \beta_k {(y-z)}^k
\right|.
\]
Realizing that the expansion with coefficients ${(\tau_k)}_{k=0}^q$ is
the result of $p$-th order truncation of an intermediate local expansion,
we can apply Lemma~\ref{lem:l2qbxl} to obtain
\begin{equation}
\label{eqn:m2l-left-equality}
\left|
\sum_{k=0}^q \tau_k {(y-z)}^k
-\sum_{k=0}^q \beta_k {(y-z)}^k
\right|
\leq \left(\frac{q+1}{p+1}\right) \left(\frac{\alpha^{p+1}}{1-\alpha}\right).
\end{equation}

To estimate
\[
\left| \sum_{k=0}^q \tilde{\beta}_k {(y-z)}^k
- \sum_{k=0}^q \tau_k {(y-z)}^k \right|,
\]
write
\[
\tau_n - \tilde{\beta}_n =
\frac{1}{n!}
\frac{d^n}{dz^n}
R_p(z)
\]
for the complex analytic function
\[
R_p(z) = \sum_{k=0}^p \left(\frac{\phi^{(k)}{(0)}}{k!} - \tilde{\zeta}_k\right) z^k.
\]
Then we have from Proposition~\ref{prop:cauchy-derivative-bound} that
\begin{equation}
\label{eqn:m2l-deriv-estimate}
|R_p^{(n)}(z)| \leq
\frac{n!}{{(r-|z|)}^n}
\left(\max_{y \in \overline B(z,r-|z|)} |R_p(y)| \right).
\end{equation}
Realizing that $R_p$ embodies the difference between a direct local expansion of
the source and one mediated by the $p$-truncated multipole expansion with
coefficients ${(\tilde{\zeta}_k)}_{k=0}^p$, both centered at the origin, we may apply
Lemma~\ref{lem:m2qbxl} to find that for $y \in \overline B(z,r-|z|) \subseteq
\overline B(0,r)$
\begin{equation}
\label{eqn:m2l-max-boundary}
|R_p(y)| \leq \frac{\omega^{p+1}}{1-\omega}.
\end{equation}
Combining (\ref{eqn:m2l-deriv-estimate}) and (\ref{eqn:m2l-max-boundary}) we
obtain
\[
|\tau_n - \tilde{\beta}_n| =
\left| \frac{R_p^{(n)}(z)}{n!} \right|
\leq \frac{1}{{(r-|z|)}^n} \left( \frac{\omega^{p+1}}{1-\omega} \right).
\]
This implies
\begin{equation}
\label{eqn:m2l-right-equality}
\left| \sum_{k=0}^q \tilde{\beta}_k {(y-z)}^k
- \sum_{k=0}^q \tau_k {(y-z)}^k \right|
=
\left| \sum_{k=0}^q (\tilde{\beta}_k - \tau_k) {(y-z)}^k \right|
\le (q+1) \left( \frac{\omega^{p+1}}{1-\omega} \right).
\end{equation}
The claim follows from combining (\ref{eqn:m2l-left-equality}) and
(\ref{eqn:m2l-right-equality}).
\end{proof}



\section{The \algbrand\ Algorithm}%
\label{sec:algorithm}

The algorithm in this section is a hierarchical fast algorithm modeled on the
adaptive Fast Multipole Method~\cite{carrier:1988:adaptive-fmm}.
Section~\ref{sec:fmm-alg-tweaks}
provided a glimpse of the differences between the conventional point FMM and our
modified version. For the benefit of readers familiar with point FMMs, these
modifications in brief amount to:
\begin{itemize}
  \item \textbf{Targets} (in the form of QBX centers) \textbf{may have an non-zero size or `extent'}. This extent is
    considered during tree construction. Interactions to the target from sources not well-separated from its
    radius are evaluated without expansion-based acceleration.
  \item To retain efficiency in the presence of this constraint, \textbf{each box has an associated
    `target confinement region'} (`TCR') which protrudes beyond the box by a fixed multiple
    of the box size. Targets (with their full extent) must fit within that
    region of a box to be eligible for inclusion in that box. If they do not fit, they will
    remain in a (larger) ancestor box.
  \item As a result of the prior point, \textbf{targets may occur in non-leaf boxes}. Interaction
    list generation must be modified to permit this.
  \item To avoid degradation of the expansion convergence factors in the
    presence of the target confinement region larger than the box, we
    \textbf{modify the basic recursion structure to consider a neighborhood two
    boxes wide} when measured from the target box instead of the classical point
    FMM's one-wide region.
  \item To further retain convergence in the presence of the TCR, we \textbf{divide} some
    of the conventional \textbf{interaction lists (particularly, List 3 and 4) into `close'
    and `far' parts}, based on whether expansion-accelerated evaluation provides
    sufficient accuracy. The `close' sub-lists are then evaluated directly,
    without acceleration.
\end{itemize}
The remainder of this section is devoted to a precise statement of the modified
algorithm as well as an analysis of its complexity.

The input to the algorithm consists of:
\begin{inparaenum}[(a)]
\item a curve $\Gamma$ discretized into panels equipped with a piecewise Gaussian quadrature rule;
\item a chosen accuracy $\epsilon > 0$;
\item a density $\mu$ with values at the points of the discretized geometry; and
\item a set of (potentially on-surface) target points at which the potential is to be
  evaluated.
\end{inparaenum}
The global accuracy parameter $\epsilon$ is used to determine the order and
truncation parameters $\pqbx$, $\pfmm$, and $\pquad$, which we describe in
Section~\ref{sec:accuracy-pars}. The geometry and targets should be preprocessed
according to Section~\ref{sec:input-geometry-preprocessing}.

\subsection{Choice of Algorithm Parameters}%
\label{sec:accuracy-pars}

The splitting (\ref{eqn:error-splitting}) of the error into FMM error,
truncation error, and quadrature error, allows us to control the error
components separately, so that in total they do not exceed the allowed precision
$\epsilon$.

The error estimates in Section~\ref{sec:error-estimates} can be used to
guarantee that the FMM error component is of order $\approx
{\left(\frac{1}{2}\right)}^{\pfmm+1}$ (cf. Theorem~\ref{thm:gigaqbx-accuracy}). Thus
we may set $\pfmm \approx \lvert \log_2 \epsilon \rvert$.

The QBX order $\pqbx$ controls the truncation error component and can be set
independently of the FMM order. Unlike the algorithm
of~\cite{rachh:2017:qbx-fmm}, our algorithm does not require an artificial order
increase to the FMM order to maintain accuracy depending on the value of
$\pqbx$. Using
(\ref{eq:qbx-truncation-estimate}) as the truncation error estimate, we see that
the truncation error should be $O(r^{\pqbx + 1})$. This means the
choice of $\pqbx$ depends on the expansion radius, and hence the length of the associated panel.

Finally, the quadrature error depends chiefly on the node count $\pquad$ of the (upsampled)
Gaussian quadrature. This error typically decays quickly in comparison
to the other error components. For instance, assuming the QBX centers are
placed a distance of $h/2$ from the panels, where $h$ is the panel width, then
the estimate (\ref{eq:qbx-quadrature-estimate}) and the surrounding discussion
imply the convergence factor for the quadrature error is approximately
${(1/2)}^{2\pquad}$. For calculations on curves in the plane, the rapid increase in accuracy
with $\pquad$ makes it expedient (if not necessarily efficient) to choose a
generic, high value (e.g. $\pquad=64$), ensuring the smallness of the quadrature error term.
The contributions~\cite{afklinteberg:2017:adaptive-qbx,afklinteberg:2016:quadrature-est}
provide precise means of estimating this error contribution.

\subsection{Preparing Geometry and Targets}%
\label{sec:input-geometry-preprocessing}

As with the algorithm in~\cite{rachh:2017:qbx-fmm}, a number of preprocessing
steps are required on the inputs, which we include in this section by reference.
The motivation and postconditions of these steps are described in
Section~\ref{sec:qbx-fast-alg}. Detailed algorithms can be found
in~\cite{rachh:2017:qbx-fmm}.

Specifically, we require that $\Gamma$ has been refined to satisfy Conditions~1,~3,~and~4
of~\cite{rachh:2017:qbx-fmm},
to control for quadrature and truncation error. Additionally,
preprocessing needs to ensure that geometry and density are upsampled to the
chosen quadrature node count $\pquad$.  Lastly, our algorithm expects that targets
needing QBX-based evaluation have been associated to expansion centers according to
the algorithm in Section 6 of~\cite{rachh:2017:qbx-fmm}.

Unless otherwise noted, we make the same parameter choices as that contribution,
concerning, e.g., center placement (cf.\ also
Section~\ref{sec:qbx-center-placement}) and oversampling.
\subsection{Tree and Interaction Lists}%
\label{sec:interaction-lists}
The inputs to our algorithm give rise to a variety of entities in the plane,
specifically source quadrature nodes, QBX centers, and target points not needing
QBX-based evaluation. We will refer to these generically as `particles'. We
disregard target points requiring QBX-based potential evaluation at this stage
because their potential evaluation needs can be met simply by evaluating the
local expansion that was obtained at the end of the algorithm at the target's
associated QBX center.

Our algorithm is based on a quadtree whose axis-aligned root box includes all
these particles as well as all the entirety of each placed expansion disk.  Each
box (even a non-leaf box) may `own' a subset of particles. The quadtree is
formed by repeatedly subdividing boxes, starting with the root box. A box is
subdivided if it owns more than $\nmax$ particles eligible to be owned by its
child boxes. If a QBX disk does not fully fit within the target confinement
region (see below) of the subdivided box, it is not eligible to be owned by the
child box, and its center remains owned by the parent box.
\subsubsection{Notation}
In the context of the quadtree described above, we introduce the following
notation:
We will use $\closedbox(r, c)$ to denote the closed $\ell^\infty$ ball
(i.e., square) of radius $r$ centered at $c$.

Let $b$ be a box in the quadtree. We will use $|b|$ to denote the radius of $b$
(i.e., half its width).
The \emph{target confinement region} (`TCR', also $\tcr(b)$) of a box $b$ with center $c$ is
$\closedbox(|b|(1 + t_f), c)$, where $t_f$ is the \emph{target confinement
factor} (`TCF').
We assume $0\le t_f<1$. A typical value is $t_f=0.9$ (cf.
Theorem~\ref{thm:gigaqbx-accuracy}).

The \emph{$k$-near neighborhood} of a box $b$ with center $c$ is the region
$\closedbox(|b|(1 + 2k), c)$.
The \emph{$k$-colleagues} of a box $b$ are boxes of the same level
as $b$ that are contained inside the $k$-near neighborhood of $b$.
$T_b$ denotes the set of $2$-colleagues of a box $b$.
Two boxes at the same level are \emph{$k$-well-separated} if
they are not $k$-colleagues.
The parent of $b$ is denoted $\parent(b)$. The
set of ancestors is $\ancestors(b)$. The set of descendants is
$\descendants(b)$. $\ancestors$ and $\descendants$ are also defined in the
natural way for sets of boxes.
A box owning a point or QBX center target is called a \emph{target box}. A box
owning a source quadrature node is called a \emph{source box}. Ancestors of
target boxes are called \emph{target-ancestor boxes}.

\begin{definition}[Adequate separation relation, $\adequatesep$]
  We define a relation $\adequatesep$ on the set
  of boxes and target confinement regions, with $a \adequatesep b$ read as `$a$
  is adequately separated from $b$, relative to the size of $a$'.

  We write $a \adequatesep b$ for boxes $a$ and $b$ if
  the $\ell^\infty$ distance between $a$ and $b$ is at least $2|a|$,
  i.e.\ the $\ell^\infty$ distance between the centers of $a$ and $b$ is at least
  $3|a|+|b|$.

  We write $a \adequatesep \tcr(b)$ for boxes $a$ and $b$ if
  the $\ell^\infty$ distance between $a$ and $\tcr(b)$ is at least $2|a|$,
  i.e.\ the $\ell^\infty$ distance between the centers of $a$ and $b$ is at least
  $3|a|+|b|(1+t_f)$.

  We write $\tcr(a) \adequatesep b$ for boxes $a$ and $b$ if
  the $\ell^\infty$ distance between $\tcr(a)$ and $b$ is at least $2|a|(1+t_f)$,
  i.e.\ the $\ell^\infty$ distance between the centers of $a$ and $b$ is at least
  $3|a|(1+t_f)+|b|$.
\end{definition}

Because the size of the TCR is proportional to the box size,
$\parent(a)\adequatesep b$ implies $a\adequatesep b$. We refer to
this property as the `\emph{monotonicity}' of `$\adequatesep$'.

\subsubsection{Interaction Lists}

The core function of the FMM is to convey interactions between boxes by ways of
multipole and local expansions. It is common for implementations to store lists of source boxes, one per
expansion/interaction type and target or target-ancestor box.
These lists are called \emph{interaction lists}.

Roughly, the FMM proceeds by obtaining multipole expansions of the sources in each box,
propagating them upwards in the tree (towards larger boxes), then
using multipole-to-local translation to convert those to local expansions where
allowable. These local expansions are then propagated down the tree and evaluated, yielding an
approximation of the far field of the box. The near field is evaluated directly,
completing the evaluation of the potential from all source boxes at each target
box. In adaptive trees (like ours), it cannot be assumed that all subtrees have
the same number of levels; additional interaction lists were introduced
in~\cite{carrier:1988:adaptive-fmm} to deal with the arising special cases.

We motivate and define the interaction lists used in our implementation
in this section.  Building on these, a precise, step-by-step statement of our
version of the FMM can be found in Section~\ref{sec:algorithm-statement}.
For a target or target-ancestor box $b$, the interaction lists $\ilist{1}{b}, \ilist{2}{b}, \ilist{3}{b},
\ilist{3close}{b}, \ilist{3far}{b}, \ilist{4}{b}, \ilist{4close}{b},
\ilist{4far}{b}$ are sets of boxes defined as follows:

\textbf{List 1} ($\ilist{1}{b}$) enumerates interactions from boxes adjacent to $b$
for which no acceleration scheme is used. This includes the interaction of the
box with itself and, since target boxes can be non-leaf boxes, also
interactions with $b$'s descendants.

\begin{definition}[List 1, $\ilist{1}{b}$]\label{def:list-1}
For a target box $b$, $\ilist{1}{b}$ consists of all leaf boxes from among
$\descendants(b) \cup \{b\}$ and the set of boxes adjacent to $b$.
\end{definition}

\textbf{List 2} ($\ilist{2}{b}$) enumerates interactions from boxes of the same
size/level as $b$ with separation to $b$ sufficient to satisfy the assumptions
for required error bounds on multipole-to-local translation.

\begin{definition}[List 2, $\ilist{2}{b}$]\label{def:list-2}
For a target or target-ancestor box $b$, $\ilist{2}{b}$ consists of the children of the
$2$-colleagues of $b$'s parent that are $2$-well-separated from $b$.
\end{definition}

\textbf{List 3} ($\ilist{3}{b}$) enumerates interactions between non-adjacent,
not 2-well-separated sources/target box pairs in which the target box $b$ is too
large (considering its separation) to receive the contribution of the source box
through multipole-to-local translation. These interactions are typically
conveyed through evaluation of the source box's multipole expansion and are
implied to cover any children of the source box. Of the descendants of the
2-colleagues of $b$, the boxes of List 3 are the first ones to become
non-adjacent to $b$ as one descends the tree towards the leaves.

\begin{definition}[List 3, $\ilist{3}{b}$]\label{def:list-3}
For a target box $b$, a box $d \in \descendants(T_b)$ is in $\ilist{3}{b}$ if
$d$ is not adjacent to $b$ and, for all $w \in \ancestors(d) \cap
\descendants(T_b)$, $w$ is adjacent to $b$.
\end{definition}

The following observations are immediate:
\begin{inparaenum}[(a)]
\item List 3 of $b$ contains the immediate children of any $2$-colleagues of
  $b$ not adjacent to $b$.
\item Any box in $\ilist{3}{b}$ is strictly smaller than $b$.
\item Any box $d \in \ilist{3}{b}$ is separated from $b$ by at least the width
  of $d$.
\end{inparaenum}

$\ilist{3}{b}$ specifies no geometric relationship of its constituent boxes to
$b$'s TCR\@. As a result, $\ilist{3}{b}$ never occurs explicitly in our
algorithm. We merely use $\ilist{3}{b}$ as a stepping stone to define
two sub-lists, $\ilist{3far}{b}$ and $\ilist{3close}{b}$ (`List 3 far' and
`List 3 close') whose definitions
take into account the existence of the TCR\@. Considering
Figure~\ref{fig:list3-sfw-accuracy}, some elements of $\ilist{3}{b}$ may be too close
to $b$ for evaluation of the source multipole to deliver the required
accuracy. Interactions between such boxes and $b$ may be handled via direct evaluation.
Since direct evaluation, unlike multipole evaluation, does not include
information from children, child boxes of too-close boxes must also be
considered. Boxes sufficiently far from $b$ make up $\ilist{3far}{b}$ (`List 3
far'), while close leaf (source) boxes comprise $\ilist{3close}{b}$.
Because of monotonicity, children of boxes in $\ilist{3far}{b}$ also
satisfy the TCR separation requirement.
As an easy consequence, observe that while
$\ilist{3close}{b}\cup \ilist{3far}{b}\subseteq \ilist{3}{b}$
does not hold in general,
$\ilist{3close}{b}\cup \ilist{3far}{b}\subseteq
\descendants(\ilist{3}{b}) \cup \ilist{3}{b}$ is generally true.

\begin{definition}[List 3 close, $\ilist{3close}{b}$]\label{def:list-3-close}
For a target box $b$, a leaf box $d$ is said to be in $\ilist{3close}{b}$
if $d \in \descendants(\ilist{3}{b}) \cup \ilist{3}{b}$ such that
$d \not \adequatesep \tcr(b)$.
\end{definition}

\begin{definition}[List 3 far, $\ilist{3far}{b}$]\label{def:list-3-far}
For a target box $b$, a box $d$ is in $\ilist{3far}{b}$
if $d \in \descendants(\ilist{3}{b}) \cup \ilist{3}{b}$ such that $d \adequatesep
\tcr(b)$ and, for all $w \in \ancestors(d) \cap (\descendants(\ilist{3}{b})
\cup \ilist{3}{b})$, $w \not \adequatesep \tcr(b)$.
\end{definition}

\textbf{List 4} ($\ilist{4}{b}$) enumerates interactions between non-adjacent,
not 2-well-separated source/target box pairs in which the source/leaf box $d$ is
too large (considering its separation) to transmit its contribution to the
target box $b$ through multipole-to-local translation. These
interactions are typically conveyed through formation of a local expansions
from the source box's sources. Since this local expansion can then participate
in the downward propagation, the interaction from $d$ does not also need to be
conveyed to $b$'s children by way of List 4. List 4 consists of non-adjacent 2-colleagues of $b$
or 2-colleagues of its ancestors.

\begin{definition}[List 4, $\ilist{4}{b}$]\label{def:list-4}
For a target or target-ancestor box $b$, a source/leaf box $d$ is
in $\ilist{4}{b}$ if $d$ is a $2$-colleague of $b$ and $d$ is not adjacent to
$b$. \emph{Additionally}, a leaf box $d$ is in List 4 of $b$ if $d$ is a
$2$-colleague of some ancestor of $b$ and $d$ is adjacent to $\parent(b)$ but
not $b$ itself.
\end{definition}

The following observations are immediate:
\begin{inparaenum}[(a)]
\item Any box in $\ilist{4}{b}$ is at least the width of $b$.
\item Any box in $\ilist{4}{b}$ is separated from $b$ by at least the width of
  $b$.
\item For any $d \in \ilist{4}{b}$, either $b \in \ilist{3}{d}$ or $d$ is a
  $2$-colleague of $b$.
\end{inparaenum}

Again, $\ilist{4}{b}$ specifies no geometric relationship of its constituent boxes to
$b$'s TCR\@. As a result, $\ilist{4}{b}$ never occurs explicitly in our
algorithm. We merely use $\ilist{4}{b}$ as a stepping stone to define
two sub-lists, $\ilist{4far}{b}$ and $\ilist{4close}{b}$
(`List 4 far' and `List 4 close') whose definitions
take into account the existence of the TCR\@. Considering
Figure~\ref{fig:list4-sfw-accuracy}, some elements of $\ilist{4}{b}$ may be too close
to $b$ to allow the resulting local expansion to deliver the required
accuracy. Interactions between such boxes and $b$ may be handled via direct evaluation.
If a source box $d$ meets the separation requirement of $\parent(b)$,
by monotonicity it will meet the separation requirements of $b$ and its
descendants. Hence it will enter the downward propagation at $\parent(b)$ and thus
need not be part of either $\ilist{4far}{b}$ or $\ilist{4close}{b}$.
Conversely, if $d\in \ilist{4}{b}$ while \emph{not} meeting the
separation requirement, it will need to be added to List 4 close of $b$
and its descendants down to the level at which it meets the requirement, at
which point its contribution enters the downward propagation via
$\ilist{4far}{b}$.

\begin{definition}[List 4 close, $\ilist{4close}{b}$]\label{def:list-4-close}
A box $d$ is in $\ilist{4close}{b}$ if for some $w \in \ancestors(b) \cup \{b\}$
we have $d \in \ilist{4}{w}$ and furthermore $\tcr(b) \not \adequatesep d$.
\end{definition}

\begin{definition}[List 4 far, $\ilist{4far}{b}$]\label{def:list-4-far}
A box $d \in \ilist{4}{b}$ is in List 4 far if $\tcr(b) \adequatesep d$.
Furthermore, if $b$ has a parent, a box $d \in \ilist{4close}{\parent(b)}$ is in
List 4 far if $\tcr(b) \adequatesep d$.
\end{definition}

\begin{remark}%
  \label{rem:far-to-close-mpole-optimization}
  In some cases, it is computationally cheaper to choose a direct interaction
  instead of an indirect one (see~\cite[Section
    2.4]{dehnen:2002:gravitational-fmm} for an example of a treecode
  optimization based on this observation). The accuracy of the
  algorithm is not impacted negatively by this choice. For example,
  a multipole-to-QBX local interaction can be more expensive than
  the corresponding direct interaction if only a small number of sources
  contribute to the multipole expansion. Such a situation can occur for a box
  $b' \in \ilist{3far}{b}$. In this case we remove $b'$ from $\ilist{3far}{b}$
  and place its leaf descendants in $\ilist{3close}{b}$. We make use of this
  possibility in Section~\ref{sec:complexity}.
\end{remark}

\subsection{Formal Statement of the Algorithm}%
\label{sec:algorithm-statement}
We use the following notation:
$\Potnear_b(t)$ denotes the potential at a target point $t$ due to all sources in $\ilist{1}{b} \cup
\ilist{3close}{b} \cup \ilist{4close}{b}$,
$\PotW{b}(t)$ denotes the potential at a target $t$ due to all sources in $\ilist{3far}{b}$,
$\Lqbxnear_{b}(t)$ denotes the (QBX) local expansion of the potential at
target/center $t$ due to all sources in $\ilist{1}{b} \cup \ilist{3close}{b} \cup
\ilist{4close}{b}$,
$\LqbxW{b}(t)$ denotes the (QBX) local expansion at target/center $t$ due to all
sources in $\ilist{3far}{b}$, and
$\Lqbxfar{b}(t)$ denotes the local expansion at target/center $t$ due to all
sources not in $\ilist{1}{b} \cup \ilist{3}{b} \cup \ilist{4close}{b}$.

\needspace{3\baselineskip}
\noindent\rule{\textwidth}{0.5pt}\\
\textbf{Algorithm:} \algbrand\ Fast Multipole Method\\
\rule[1.25ex]{\textwidth}{0.5pt}\\[-3ex]
  \def\algstage#1#2{\vspace{2ex}\begin{minipage}{0.97\textwidth}\STATE{\textit{#1}}#2\end{minipage}}
  \begin{algorithmic}
    \small%
    \REQUIRE{The maximum number of FMM targets/sources $\nmax$ per
      box for quadtree refinement and a target confinement factor $t_f$ are
      chosen.}
    \REQUIRE{Based on the precision $\epsilon$ to be achieved, a QBX
      order $\pqbx$, an FMM order $\pfmm$, and an oversampled quadrature node
      count
      $\pquad$ are chosen in accordance with Section~\ref{sec:accuracy-pars}.}
    \REQUIRE{The input geometry and targets are preprocessed according
      to Section~\ref{sec:input-geometry-preprocessing}.}
    \ENSURE{An accurate approximation to the potential
      $\mathcal{S}(\mu)$ at all target points is computed.}

    \algstage{Stage 1: Build tree}
    {%
    \STATE{Create a quadtree on the computational domain containing all sources,
      targets, and expansion centers.}
    \REPEAT{}
    \STATE{Subdivide each box containing more than $n_\text{max}$ particles into
      four children, pruning any empty child boxes. If an expansion center
      cannot be placed in a child box with target confinement factor $t_f$ due to its radius,
      it remains in the parent box.}
    \UNTIL{each box no longer needs to be subdivided or an iteration produced only
    empty child boxes}
    }

    \algstage{Stage 2: Form multipoles}{%
    \FORALL{boxes $b$}
    \STATE{Form a $\pfmm$-th order multipole expansion $\Mpole_b$ centered at $b$ due to
      sources owned by $b$.}
    \ENDFOR{}
    \FORALL{boxes $b$ in postorder}
    \STATE{For each child of $b$, shift the center of the multipole expansion at
      the child to $b$. Add the resulting expansions to $\Mpole_b$.}
    \ENDFOR{}
    }

    \algstage{Stage 3: Evaluate direct interactions}{%
    \FORALL{boxes $b$}
    \STATE{For each conventional target $t$ owned by $b$, add to $\Potnear_b(t)$
      the contribution due to the interactions from sources owned by boxes in
      $\ilist{1}{b}$ to $t$.}
    \ENDFOR{}
    \FORALL{boxes $b$}
    \STATE{For each expansion center target $t$ owned by $b$, add to the expansion
      coefficients $\Lqbxnear_{b}(t)$, the
      contribution due to the interactions from $\ilist{1}{b}$ to $t$.}
    \ENDFOR{}
    }

    \algstage{Stage 4: Translate multipoles to local expansions}{%
    \FORALL{boxes $b$}
    \STATE{For each box $d \in \ilist{2}{b}$, translate the multipole expansion
      $\Mpole_{d}$ to a local expansion centered at $b$. Add the resulting
      expansions to obtain $\Locfar_b$.}
    \ENDFOR{}
    }

    \algstage{Stage 5(a): Evaluate direct interactions due to $\ilist{3close}{b}$}{%
    \STATE{Repeat Stage 3 with $\ilist{3close}{b}$ instead of $\ilist{1}{b}$.}
    }

    \algstage{Stage 5(b): Evaluate multipoles due to $\ilist{3far}{b}$}{%
    \FORALL{boxes $b$}
    \STATE{For each conventional target $t$ owned by $b$, evaluate the multipole
      expansion $\Mpole_{d}$ of each box $d \in \ilist{3far}{b}$
      to obtain $\PotW{b}(t)$.}
    \ENDFOR{}
    \FORALL{boxes $b$}
    \STATE{For each expansion center target $t$ owned by $b$, compute the expansion
      coefficients $\LqbxW{b}(t)$, due to
      the multipole expansion $\Mpole_{d}$ of each box $d \in \ilist{3far}{b}$.}
    \ENDFOR{}
    }

    \algstage{Stage 6(a): Evaluate direct interactions due to $\ilist{4close}{b}$}{%
    \STATE{Repeat Stage 3 with $\ilist{4close}{b}$ instead of $\ilist{1}{b}$.}
    }

    \algstage{Stage 6(b): Form locals due to $\ilist{4far}{b}$}{%
    \FORALL{boxes $b$}
    \STATE{Convert the field of every particle owned by boxes in $\ilist{4far}{b}$
      to a local expansion about $b$. Add to $\Locfar_b$.}
    \ENDFOR{}
    }

    \algstage{Stage 7: Propagate local expansions downward}{%
    \FORALL{boxes $b$ in preorder}
    \STATE{For each child $d$ of $b$, shift the center of the local expansions
      $\Locfar_b$ to the child. Add the resulting expansions to
      $\Locfar_d$ respectively.}
    \ENDFOR{}
    }

    \algstage{Stage 8: Evaluate final potential at targets}{%
    \FORALL{boxes $b$}
    \STATE{For each conventional target $t$ owned by $b$, evaluate $\Locfar_b(t)$.}
    \STATE{Add $\Potnear_b(t), \PotW{b}(t), \Locfar_b(t)$ to obtain the
      potential at $t$.}
    \ENDFOR{}
    \FORALL{boxes $b$}
    \STATE{For each expansion center target $t$ owned by $b$, translate $\Locfar_b$ to
      $t$, obtaining $\Lqbxfar_b(t)$.
      }
    \STATE{Add $\Lqbxnear_{b}(t), \LqbxW{b}(t), \Lqbxfar_b(t)$ to obtain the
      QBX local expansion at $t$.}
    \ENDFOR{}
    }

  \end{algorithmic}
  \vspace{1ex}
\rule{\textwidth}{0.5pt}

\definecolor{darkergreen}{rgb}{0.0, 0.5, 0.0}
\colorlet{srccolor}{darkergreen}
\colorlet{furthestcolor}{cyan}
\colorlet{closestcolor}{purple}

\def\sqrtwo{1.4142135623730951}
\def\tcf{0.5}

\begin{figure}[ht]
  \begin{minipage}[b]{0.48\textwidth}
  \begin{tikzpicture}[scale=0.8]
    \coordinate (box1ctr) at (0,0);
    \coordinate (box2ctr) at (6,0);

    \draw [srccolor,thick] (box1ctr) ++(-1,-1) rectangle ++(2,2);
    \draw [srccolor,thick,dashed] (box1ctr) circle (\sqrtwo);

    \draw [thick] (box2ctr) ++(-1,-1) rectangle ++(2,2);
    \draw [dashed,thick] (box2ctr) ++(-1-\tcf,-1-\tcf) rectangle ++(2+2*\tcf,2+2*\tcf);
    \draw [|<->|,thick,furthestcolor] (box2ctr) -- ++(-1-\tcf/\sqrtwo,1+\tcf/\sqrtwo)
      node [pos=0.4, anchor=west, xshift=2pt]
      {$r_t(t_f+\sqrt{2})$};
    \draw [dashed,thick,furthestcolor] (box2ctr) ++(-1,1) circle (\tcf);

    \draw [|<->|,thick] (box2ctr) -- ++(0,-1) node [pos=0.5,anchor=east] {$r_t$};
    \draw [dotted] (box1ctr) ++(2,0) ++(-1,-1) rectangle ++(2,2);
    \draw [dotted] (box1ctr) ++(4,0) ++(-1,-1) rectangle ++(2,2);
    \draw [|<->|,thick,closestcolor] (box2ctr) -- ($ (box1ctr) + (\sqrtwo, 0) $)
      node [pos=0.7, anchor=south] {$\geq (6 - \sqrt{2}) r_t$};
  \end{tikzpicture}
  \vspace{0.5cm}
  \[
  \frac{\text{\color{furthestcolor}furthest target}}{\text{\color{closestcolor}closest source}}
  \leq
  \frac{\sqrt{2} + t_f}{6 - \sqrt{2}}
  \]
  \caption{%
    Separation criteria for List 2, with convergence factor calculation for a
    multipole-to-local-to-QBX local interaction.  The target box with radius
    $r_t$ is on the right.  See Lemma~\ref{lem:m2l2qbxl}.
  }%
  \label{fig:list2-sep-criteria}
  \end{minipage}
  \hfill
  \begin{minipage}[b]{0.38\textwidth}
  \centering
  \def\tcf{0.5}

  \begin{tikzpicture}
    \draw [step=1cm, thick] (-1,-1) grid (1,1);
    \draw [dashed,thick] (-1-\tcf,-1-\tcf) rectangle (1+\tcf,1+\tcf);
    \draw [dashed,thick] (-\tcf/2,-\tcf/2) rectangle (1+\tcf/2,1+\tcf/2);

    \draw [<->,thick,purple] (1,0) -- (1+\tcf/2,0) node[pos=2,anchor=west]
        {$\frac{t_f}{2}|b|$};
    \draw [<->,thick,purple] (1,-1) -- (1+\tcf,-1) node[pos=1,anchor=west] {$t_f|b|$};

    \coordinate (ctr) at (0.9, 0.9);
    \fill [thick,color=purple] (ctr) circle (1pt);
    \draw [thick,color=purple,dashed] (ctr) circle (\tcf/2 + 0.3);
  \end{tikzpicture}

  \vspace{0.5cm}

  \caption{\label{fig:neighborhood-of-suspended-center}%
    The expansion disk of a suspended center must have radius at least $t_f / 2$
    times the radius of the box $b$ that owns the center.}
  \end{minipage}
\end{figure}

\begin{figure}[ht]
  \begin{minipage}[b]{.48\linewidth}
  \centering
  \begin{tikzpicture}[scale=0.8]
    \def\srcr{0.75};

    \coordinate (box1ctr) at (0,0);
    \coordinate (box2ctr) at (4,0);

    \draw [srccolor,thick] (box1ctr) ++(-\srcr,-\srcr) rectangle ++(2*\srcr,2*\srcr);
    \draw [srccolor,thick,dashed] (box1ctr) circle (\sqrtwo*\srcr);

    \draw [thick] (box2ctr) ++(-1,-1) rectangle ++(2,2);
    \draw [dashed,thick] (box2ctr) ++(-1-\tcf,-1-\tcf) rectangle ++(2+2*\tcf,2+2*\tcf);

    \draw [|<->|,thick,closestcolor]
      (box1ctr) -- ($ (box2ctr) - (1+\tcf, 0) $)
        node [anchor=south, pos=0.7] {$\geq 3 r_s$};

    \draw [|<->|,thick] (box1ctr) -- ++(0, \srcr) node[pos=0.5,anchor=west] {$r_s$};
    \draw [|<->|,thick,furthestcolor]
      (box1ctr) -- (-\srcr,-\srcr) node[pos=0.5,anchor=west] {$\sqrt{2} r_s$};
  \end{tikzpicture}
  \vspace{0.5cm}
  \[
  \displaystyle
  \frac{\text{\color{furthestcolor}furthest source}}
  {\text{\color{closestcolor}closest target}}
  \leq
  \frac{\sqrt{2}}{3}
  \]
  \caption{%
    Separation criteria for List 3 far, with convergence factor calculation
    for a multipole-to-QBX local interaction.  The source box with radius $r_s$
    is on the left.  See Lemma~\ref{lem:m2qbxl}.
  }%
  \label{fig:list3-sep-criteria}
  \end{minipage}
  \hfill
  \begin{minipage}[b]{.48\linewidth}
  \centering
  \begin{tikzpicture}
    \def\tgtr{0.75};

    \coordinate (box1ctr) at (0,0);
    \coordinate (box2ctr) at (4,0);

    \draw [srccolor,thick] (box1ctr) ++(-1,-1) rectangle ++(2,2);

    \draw [thick] (box2ctr) ++(-\tgtr,-\tgtr) rectangle ++(2*\tgtr,2*\tgtr);
    \draw [dashed,thick] (box2ctr) ++(-\tgtr-\tgtr*\tcf,-\tgtr-\tgtr*\tcf)
      rectangle ++(2*\tgtr+2*\tgtr*\tcf,2*\tgtr+2*\tgtr*\tcf);

    \draw [|<->|,thick] (box2ctr) -- ++(0,-\tgtr) node [pos=0.5,anchor=east] {$r_t$};
    \draw [|<->|,thick,furthestcolor] (box2ctr)
      -- ++(-\tgtr-\tgtr*\tcf/\sqrtwo,\tgtr+\tgtr*\tcf/\sqrtwo)
      node [pos=0.55, anchor=west, xshift=3pt] {$r_t(t_f + \sqrt{2})$};
    \draw [dashed,thick,furthestcolor] (box2ctr) ++(-\tgtr,\tgtr) circle (\tgtr*\tcf);

    \draw [|<->|,thick,closestcolor] (box2ctr) -- ($ (box1ctr) + (1,0) $)
      node [anchor=north, pos=0.7] {$\geq 3 r_t(1+t_f)$};
  \end{tikzpicture}
  \vspace{0.5cm}
  \[
  \displaystyle
  \frac{\text{\color{furthestcolor}furthest target}}
  {\text{\color{closestcolor}closest source}}
  \leq
  \frac{\sqrt{2}+t_f}{3(1+t_f)}
  \leq \frac{\sqrt{2}}{3}
  \]
  \caption{%
    Separation criteria for List 4 far, with convergence factor calculation
    for a local-to-QBX local interaction. The target box with radius $r_t$ is on
    the right. See Lemma~\ref{lem:l2qbxl}.
  }%
  \label{fig:list4-sep-criteria}
  \end{minipage}
\end{figure}

\subsection{Accuracy of the Computed Potential}%
\label{sec:accuracy-statement}

Section~\ref{sec:error-estimates} contains the information necessary to derive
an accuracy estimate for the \algbrand~FMM\@. The interaction lists are designed
so that each interaction mediated through them has a provable convergence
factor. As we shall see in the next theorem, only the List 2 convergence factor
depends on $t_f$, and the convergence factor for Lists 3 and 4 far is fixed at
$\sqrt{2}/3$. Thus, the overall accuracy of the \algbrand~FMM is
primarily determined by the choice of $t_f$.

\begin{theorem}%
\label{thm:gigaqbx-accuracy}
Fix a target confinement factor $0 \le t_f < 6 - 2 \sqrt 2$
and define $\alpha = (t_f + \sqrt{2})/(6 - \sqrt{2}) < 1$.
There exists a constant $C$ such that for every target
point $x\in \mathbb R^2$
\[
  \left|
  \mathcal{S}_{\text{QBX}(\pqbx, N)} \mu(x)
  -
  \mathcal G_{\pfmm}\left[\mathcal{S}_{\text{QBX}(\pqbx, N)} \mu(x)\right]
  \right|
  \le
  \frac{1}{1 -\alpha} C A (\pqbx+1)
  \max{\left(\frac{\sqrt{2}}{3}, \alpha \right)}^{\pfmm+1},
\]
where $\mathcal G_{\pfmm}[\cdot]$ denotes approximation by the \algbrand\ FMM
of order $\pfmm$ and
\[
A = \sum_{i=0}^{N} |w_i \mu (y_i)|,
\]
with the $\{w_i\}$ the quadrature weights and the $\{y_i\}\subset \Gamma$
the quadrature nodes.
In particular, for $t_f \leq 3 - 3/\sqrt{2} \approx 0.87$, we obtain
\[
  \left|
  \mathcal{S}_{\text{QBX}(\pqbx, N)} \mu(x)
  -
  \mathcal G_{\pfmm}\left[\mathcal{S}_{\text{QBX}(\pqbx, N)} \mu(x)\right]
  \right|
  \le
  C A (\pqbx+1) {\left( \frac{1}{2} \right)}^{\pfmm}.
\]
C is independent of $\mu$, $t_f$, $\pqbx$, $\pfmm$, $\pquad$,
and of the curve $\Gamma$ and its discretization.
\end{theorem}

\begin{proof}
The proof of the statement results from applying the error estimates of
Section~\ref{sec:error-estimates} to the geometric situations
resulting from the definitions of the interaction lists in
Section~\ref{sec:interaction-lists}.
We fix a target point $x$. Without loss of generality, we assume that $x$ is
associated with a QBX center. Let $c$ be the QBX center associated with $x$. Every
source point $y_i$ contributing to the summation (\ref{eq:qbx}) contributes via
either $\Lqbxnear_{b}(c)$, $\LqbxW{b}(c)$, or $\Lqbxfar_b(c)$, where $b$ is the
box that owns $c$.
For $\Lqbxnear_{b}(c)$, the contribution must arrive via a direct
interaction. This contribution incurs no error.

For $\LqbxW{b}(c)$, the contribution must arrive via a $\ilist{3far}{b}$
interaction. The contribution from all $\ilist{3far}{b}$ interactions incurs
an error of at most
\[
\frac{3A}{3-\sqrt{2}} \left(\frac{\pqbx+1}{\pfmm+1}
\right) {\left( \frac{\sqrt{2}}{3} \right)}^{\pfmm+1}.
\]
See Figure~\ref{fig:list3-sep-criteria} and Lemma~\ref{lem:m2qbxl}.
For $\Lqbxfar_b(c)$, the contribution must arrive via a $\ilist{2}{b'}$ or
$\ilist{4far}{b'}$ interaction, where $b'$ is either $b$ or an ancestor of
$b$.
The contribution from all $\ilist{4far}{b'}$ interactions incurs an error of at
most
\[
\frac{3A}{3-\sqrt{2}}
  {\left(\frac{\pqbx+1}{\pfmm+1} \right) \left( \frac{\sqrt{2}}{3}
  \right)}^{\pfmm+1}.
\]
See Figure~\ref{fig:list4-sep-criteria} and Lemma~\ref{lem:l2qbxl}.
The contribution from all $\ilist{2}{b'}$ interactions incurs an error of at most
\[
\frac{A}{1-\alpha} (\pqbx + 1) \left(1 + \frac{1}{\pfmm+1} \right) \alpha^{\pfmm + 1}.
\]
See Figure~\ref{fig:list2-sep-criteria} and Lemma~\ref{lem:m2l2qbxl}.
Figures~\ref{fig:list2-sep-criteria},~\ref{fig:list3-sep-criteria},
and~\ref{fig:list4-sep-criteria} reinterpret the convergence factor
geometrically in terms of ratios involving distances to sources and distances to
targets. These are equivalent to the definitions of the convergence factors
encountered in
Lemmas~\ref{lem:m2qbxl},~\ref{lem:l2qbxl}, and~\ref{lem:m2l2qbxl}.
Combining the estimates above yields the final error estimate.
\end{proof}

The analysis in Section~\ref{sec:error-estimates} that leads to the bound in
Theorem~\ref{thm:gigaqbx-accuracy} is not sharp due to various mathematical
simplifications, as we will see in Section~\ref{sec:results}.
Techniques similar to the ones of~\cite{petersen:1995:fmm-error-est}
may lead to sharper bounds.

\subsection{Complexity}%
\label{sec:complexity}

%
%
%
%
%
%
%
%
%
%
%
%
%
%
%
%

\begin{table}[h]
  \centering
  \begin{tabular}{ccp{0.5\textwidth}}
    \toprule
    Stage & Modeled Operation Count & Note\\

    \midrule

    Stage 1 & $NL$ & There are $N$ total particles with at most $L$ levels of
    refinement. \\

    Stage 2 & $N_S \pfmm + N_B \pfmm^2$ & $N_S \pfmm$ for forming multipoles and
    the rest for shifting multipoles upward, with each shift costing $\pfmm^2$ \\

    Stage 3 & $9 N \nmax \pqbx + N_C M_C \pqbx$ &
    Lemma~\ref{lem:list1-complexity} \\

    Stage 4 & $75N_B \pfmm^2$ & Lemma~\ref{lem:list2-complexity} \\

    Stage 5 & $\begin{aligned}[t] N_C M_C \pqbx + 64 N_C \pfmm \pqbx + \\8 N_S L
      \nmax \pqbx \end{aligned}$ & Lemma~\ref{lem:list3-complexity} \\

    Stage 6 & $63 N_B \nmax \pfmm + 42 N_C \nmax \pqbx$ &
    Lemma~\ref{lem:list4-complexity} \\

    Stage 7 & $4N_B \pfmm^2$ & The cost of shifting a local expansion
    downward is $\pfmm^2$. There are at most $4$ children per box. \\

    Stage 8 & $N_C \pfmm \pqbx$ & Cost of translating the box local
    expansions to $N_C$ centers. \\

    \bottomrule
  \end{tabular}
  \belowtableskip%
  \caption{Complexity of each stage of the \algbrand\ Algorithm.}%
  \label{tab:complexity-analysis}
\end{table}

The purpose of this section is demonstrate that the \algbrand\ algorithm has a
running time that is, roughly speaking, linear in the size of the input, assuming, roughly, that
the number of sources within a neighborhood of each QBX disk is constant.  Let
$S$ be a set of source points with $N_S = |S|$ (as obtained from the
discretization of a curve $\Gamma$ in accordance with
Section~\ref{sec:motivation}), and let $C$ be a set of expansion centers with
$N_C = |C|$. Let $N = N_S + N_C$. Let the quadtree have $N_B$ boxes and $L$
levels.

As a simplifying assumption, we restrict the complexity analysis to the case
that the set of targets at which the potential is to be evaluated is covered by QBX
expansion disks, or, in other words, we eliminate from consideration any targets
whose potential can be evaluated through the conventional FMM algorithm without
the use of QBX expansions.

We provide worst case running time bounds that apply to all particle
distributions (although they do not imply linear complexity for all particle
distributions). In particular, we do not attempt to obtain tight bounds on the
leading constant terms from the complexity analysis. To complement the
theoretical analysis with a more precise cost of each stage of the algorithm, we
offer an empirical cost model in Section~\ref{sec:op-counts-results}.

From the standpoint of complexity analysis, perhaps the most significant
difference between a point FMM and the \algbrand~FMM is that in the
\algbrand~algorithm it is no longer the case that there always are at most
$\nmax$ particles per box. For the point FMM, this feature allows for bounding
the number of near-neighborhood interactions between two
boxes~\cite{carrier:1988:adaptive-fmm}. In the \algbrand~FMM, it is possible for
any number of QBX centers to cluster inside a box in the tree due to target
confinement restrictions, and so another technique is needed to count the
near-neighbor interactions.

Because we cannot rely on there being at most $\nmax$ centers per box, our
analysis takes into account whether a QBX center is \emph{suspended} in an upper
level of the tree or not.  A QBX center that is owned by a leaf box and \emph{could} be
owned by a hypothetical child of the leaf box is called \emph{leaf-settled}. A
QBX center that is not leaf-settled is called \emph{suspended}. There are never
more than $\nmax$ leaf-settled QBX centers in a box. For suspended centers, we
make use of Proposition~\ref{prop:suspended-center-nn-interaction-cost} below,
which relates the size of the `near neighborhood' of a suspended center to the
size of the near neighborhood of its owner box.

A summary of the complexity results for each stage is given in
Table~\ref{tab:complexity-analysis}. The rest of this section provides the
details of this complexity analysis. We provide our complexity analysis in terms
of `modeled floating point operations'.  This means that while we include
constants throughout, depending on interpretation, these constants may omit a
flop-related constant factor independent of problem parameters when we felt that
no information was gained from including it for added realism. For instance, we
model the cost of evaluating an expansion of order $p$ as $p$, when more
realistic operation counts might range from $2p+1$ if multiplications and
additions are counted, to $p+1$ if a fused-multiply-add operation is assumed, to
yet different counts if the computation of powers is taken into account.

\subsubsection{Near Neighborhoods of Suspended QBX Centers}


We start with the following basic observation about suspended QBX centers.

\begin{proposition}%
\label{prop:qbx-nn-to-box-nn}%
Let $c$ be a suspended QBX center of radius $r_c$ owned by the box
$\homebox{c}$.  Then the closed square $\closedbox \left( 8r_c/t_f, c \right)$ is,
geometrically, a superset of the $1$-near neighborhood of $\homebox{c}$.
\end{proposition}

\begin{proof}
Because $c$ is suspended, $c$ cannot be placed in any (hypothetical) child of
$\homebox{c}$ because it will not fit in the target confinement region.
Since a child has radius $\frac{1}{2}|\homebox{c}|$, it follows that $r_c
> \frac{t_f}{2} |\homebox{c}|$. This situation is illustrated in
Figure~\ref{fig:neighborhood-of-suspended-center}.

Regardless of where $c$ is located in $\homebox{c}$, $\closedbox(c,
4|\homebox{c}|)$ must contain the $1$-near neighborhood of $\homebox{c}$.  The
claim follows since $\frac{8}{t_f} r_c > 4 |\homebox{c}|$.
\end{proof}

Based on this proposition, we define the following parameter. Consider the input
to the \algbrand\ algorithm, with a set of centers $C$ and a set of
sources $S$. Then $M_C$ is defined as
\[
M_C = \frac{1}{N_C} \sum_{c \in C} \left|S \cap \closedbox \left(c, \frac{8}{t_f}
r_c\right) \right|
\]
where $r_c$ denotes the radius of the center $c$. In other words, $M_C$ is the
average number of sources that intersect with a square of radius $\frac{8}{t_f}
r_c$ surrounding a QBX center $c$.

The running time of \algbrand\ algorithm depends non-trivially on the
particle distribution in the tree. However, $M_C$ is a geometry-dependent
parameter that is independent of the tree structure, and can be used to provide
a worst case bound on the $1$-near neighborhood interaction cost for a suspended
center. Moreover, if the geometry is smooth and refined in such a way that the
panel sizes are locally uniform, $M_C$ will not depend on the
total number of particles. We give some values of $M_C$ for actual geometries in
Section~\ref{sec:results}.

The main utility of $M_C$ is in the following proposition.

\begin{proposition}%
\label{prop:suspended-center-nn-interaction-cost}%
The number of source-center pairs $(s, c) \in S \times C$, such that $c$ is a
suspended center and $s$ is in the $1$-near neighborhood of the box that owns
$c$, is at most $N_C M_C$.
\end{proposition}

\begin{proof}
This follows immediately from Proposition~\ref{prop:qbx-nn-to-box-nn} and the
definition of $M_C$.
\end{proof}

\subsubsection{Detailed Complexity Analysis}


In this section, we will use $\homebox{x}$ to refer to the box that owns
particle $x$ and $B$ to refer to the set of boxes in the quadtree.

\begin{proposition}%
\label{prop:num-bigger-nn-leaves}%
Let $b$ be an arbitrary box. Then there are at most $9$ \emph{leaf} boxes at
least as large as $b$ that intersect the $1$-near neighborhood of $b$.
\end{proposition}

\begin{proof}
Each such leaf box can be mapped injectively to $b$ or one of the $8$ colleagues
of $b$ it contains.
\end{proof}

\begin{lemma}%
\label{lem:list1-complexity}
The amount of work done in Stage 3 (direct evaluation of the potential from List
1 source boxes) is at most $$9 N \nmax \pqbx + N_C M_C \pqbx.$$
\end{lemma}


\begin{proof}
We model the cost of all Stage 3 interactions as $\pqbx|U|$, where $U = \left\{ (s,c) \in S
\times C \mid \homebox{s} \in \ilist{1}{\homebox{c}} \right\}$. $U$ may be
written as the disjoint union $U = U_\mathrm{big} \cup U_\mathrm{small}$ where
$U_\mathrm{big}$ is the set of pairs $(s, c)$, such that $|\homebox{s}| \ge
|\homebox{c}|$.

We bound $|U_\mathrm{big}|$ as follows. Consider a center $c$. Then there are at
most $9$ leaf boxes larger than $\homebox{c}$ that are in or adjacent to
$\homebox{c}$, by Proposition~\ref{prop:num-bigger-nn-leaves}.  Therefore there
are at most $9 \nmax$ sources $s$ such that $(s,c) \in U_\mathrm{big}$.  Thus
$|U_\mathrm{big}| \leq 9 N_C \nmax$.

Now we bound $|U_\mathrm{small}|$. We can group the particle-center
interactions $(s, c)$ in $U_\mathrm{small}$ according to whether $c$ is
suspended or leaf settled. We consider these two cases separately. Suppose that
$c$ is suspended. Then any $s$ such that $(s, c) \in U_\mathrm{small}$ must be
in the $1$-near neighborhood of $\homebox{c}$.  By
Proposition~\ref{prop:suspended-center-nn-interaction-cost}, the number of such $(s,c)$ pairs is
at most $M_C N_C$.  Now, to consider settled centers, let $s$ be a source. From
Proposition~\ref{prop:num-bigger-nn-leaves}, the number of leaf boxes larger
than $\homebox{s}$ and adjacent to it is at most $9$. Therefore there are at
most $9 \nmax$ leaf settled centers $c$ such that $(s, c) \in U_\mathrm{small}$.

It follows that $|U_\mathrm{small}| \leq M_C N_C + 9 N_S \nmax$. Therefore the
total cost of Stage 3 is at most $9 N \nmax \pqbx + M_C N_C \pqbx$.
\end{proof}


\begin{lemma}%
\label{lem:list2-complexity}
The amount of work done in Stage 4 (translation of multipole to local
expansions) is at most $75 N_B \pfmm^2$.
\end{lemma}


\begin{proof}
For each box $b$, $|V_b| \leq 10^2 - 5^2 = 75$ because there are at most $10^2$
children of the parent of $b$ and its $2$-colleagues, and $b$ and its $2$-colleagues
cannot be in $V_b$. There are $N_B$ boxes, and each multipole to local
translation has a modeled cost of $\pfmm^2$ operations.
\end{proof}


\begin{lemma}%
\label{lem:list3-complexity}
Assume that $0 \leq t_f < 1$. The amount of work done in Stage 5 (handling Lists
3 close and far) is at most $$N_C M_C \pqbx + 64 N_C \pfmm \pqbx + 8 N_S L \nmax
\pqbx.$$
\end{lemma}


\begin{proof}
In order to simplify the analysis, we assume that, for every box $b$,
\emph{$\ilist{3far}{b}$ is disjoint from the $1$-near neighborhood of $b$}.
In other words, for our cost analysis, we treat $\ilist{3}{b}$-type
interactions from within the $1$-near neighborhood of $b$ as being mediated
through $\ilist{3close}{b}$.  This assumption is conservative and will not lead
to an underestimation of the cost because (1) direct interactions through
$\ilist{3close}{b}$ do not subsume interactions from their children and as such
are more numerous and (2) any multipole evaluation that turns out to be more
expensive than direct evaluation of the interaction from that box and its
children may be replaced by the latter. (In fact, this latter strategy is
a viable, if minor, cost optimization. See
Remark~\ref{rem:far-to-close-mpole-optimization} for details.)

We model the cost of Stage 5 as $\pfmm \pqbx |W_\mathrm{far}| + \pqbx
|W_\mathrm{close}|$, where
\[
  W_\mathrm{far} = \left\{ (b,c) \in B \times C \mid b \in
  \ilist{3far}{\homebox{c}} \right\},
  \quad\text{and}\quad %
  W_\mathrm{close} = \left\{ (s, c) \in S \times C \mid \homebox{s} \in
  \ilist{3close}{\homebox{c}} \right\}.
\]

We bound $|W_\mathrm{far}|$ first. Let $c$ be an expansion center.  By our
assumption on $\ilist{3far}{b}$ and the 1-near-neighborhood,
every box $b$ with $(b, c) \in W_\mathrm{far}$ is a
descendant of one of the $5^2 - 3^2 = 16$ boxes that are $2$-colleagues of
$\homebox{c}$ that are not adjacent to $\homebox{c}$. Since $0 \leq t_f < 1$,
the direct descendants of $b_c$'s 2-colleagues satisfy
Definition~\ref{def:list-3-far}, and no children of the direct descendants
can be in $\ilist{3far}{b_c}$. Thus there are at most $16 \times 4$ boxes $b$ such that
$(b, c) \in W_\mathrm{far}$. In other words, $|W_\mathrm{far}| \leq 64 N_C$.

We handle $|W_\mathrm{close}|$ next.
Because $0 \leq t_f < 1$, for every center $c$, every box in
$\ilist{3close}{\homebox{c}}$ is contained inside the $1$-neighborhood of
$\homebox{c}$.
As in the proof of Lemma~\ref{lem:list1-complexity}, we will consider separately
the cases that the center is suspended or leaf settled.
In the case that the center is suspended,
Proposition~\ref{prop:suspended-center-nn-interaction-cost} bounds the number of
all such $(s, c)$ pairs by $M_C N_C$.  For the leaf settled case, let $s$ be a
source particle. Consider a pair $(s, c) \in W_\mathrm{close}$. Then an ancestor
of $\homebox{s}$ is adjacent to $\homebox{c}$. Since there are $L$ levels, there
are at most $8 L$ boxes adjacent to ancestors of $\homebox{s}$.  So there are at
most $8 L \nmax$ leaf settled centers $c$ such that $(s, c) \in
W_\mathrm{close}$.  It follows that $|W_\mathrm{close}| \leq N_C M_C + 8 N_S L
\nmax$.

Therefore the total cost of Stage 5 is at most $N_C M_C \pqbx + 64 N_C
\pfmm \pqbx + 8 N_S L \nmax \pqbx$.
\end{proof}


\begin{remark}%
\label{rem:level-restriction}
The factor of $N_S L$ in the cost estimate of Lemma~\ref{lem:list3-complexity}
implies that the cost of Stage 5 of the algorithm has a worst-case dependence on
the number of particles times the number of levels in the tree. Since there are
$\Omega(\log N)$ levels, this could lead to $\Omega(N \log N)$ algorithmic
scaling. In practice, we have not observed this for the particle distributions
we have tried and we expect $\Omega(N \log N)$ behavior to be uncommon in
geometries used for layer potential evaluation. It is conceivable that a sharper
analysis might be able to eliminate the factor.

In fact, if we assume the quadtree to be \emph{level-restricted}, i.e.\ if we
assume that adjacent leaf boxes' levels differ by at most one, then for every
leaf box $b$, it can be shown that $|\ilist{3close}{b} \cup \ilist{3far}{b}|$ is
at most a (dimension-dependent) constant. Because of this, every center in a
leaf-settled box will interact via $\ilist{3close}{b}$ with at most a constant
number of source particles, and via $\ilist{3far}{b}$ with at most a constant
number of boxes. This allows us to remove the dependence on $L$ in the
complexity estimates.

Furthermore, using a level-restricted quadtree does not affect the asymptotic
cost of any other stage of the algorithm. Starting with an arbitrary adaptive
quadtree, the cost of converting it into a level-restricted quadtree is $O(N_B)$
and the resulting tree has $O(N_B)$ boxes.
See~\cite[Theorem 1]{moore:1995:cost-of-balancing-generalized-quadtrees}.
\end{remark}

\begin{lemma}%
  \label{lem:list4-complexity}
  The amount of work done in Stage 6 (handling Lists 4
  close and far) is at most
  \[
    63 N_B \nmax \pfmm + 42 N_C \nmax \pqbx.
  \]
\end{lemma}


\begin{proof}
To aid with the later analysis, we first bound $|\ilist{4}{b}|$.
Let $b \in B$. By definition of $\ilist{4}{b}$, any box $b' \in \ilist{4}{b}$ is
not adjacent to $b$, and must be adjacent to the parent of $b$ and at least as large
as the parent of $b$, or a $2$-colleague of $b$. Of the former, it is an easy
result based on the same argument as Proposition~\ref{prop:num-bigger-nn-leaves}
that there are at most $5$
\begin{tikzpicture}[scale=0.08]
  \draw [fill=gray] (0,0) rectangle ++(1,1);
  \draw (1,-1) rectangle ++(1,1);
  \draw (1,0) rectangle ++(1,1);
  \draw (1,1) rectangle ++(1,1);
  \draw (0,1) rectangle ++(1,1);
  \draw (-1,1) rectangle ++(1,1);
\end{tikzpicture}
such boxes.
There are at most 16 $2$-colleagues of $b$
not adjacent to $b$. Thus $|\ilist{4}{b}| \leq 21$.

We begin by bounding the size of $\ilist{4close}{b}$. Recall from
Definition~\ref{def:list-4-close} that $\ilist{4close}{b}$ is a subset of the
List 4's of $b$ and its ancestors. In fact, as we now show,
$\ilist{4close}{b}\subseteq \ilist{4}{b} \cup \ilist{4}{\parent(b)}$.

For any ancestor $b'$ of $b$ that is
$k$ levels above $b$, $b$ is separated by at least a distance of $2^{k+1}|b|$
from any box in $\ilist{4}{b'}$.
Now suppose that $b'$ is an ancestor of $b$ that is $k \geq 2$ levels above
$b$. Let $e \in \ilist{4}{b'}$. Then $d(b, e) \geq 8|b|$. Furthermore,
for any point $t\in\tcr(b)$,
$d(t, b) < |b|$ since $0 \leq t_f < 1$. From the reverse triangle
inequality,
\[ 4|b| < 7|b| < |d(t, b) - d(b, e)| \leq d(t, e). \]
This means that $\tcr(b) \adequatesep e$, i.e. $e\not\in \ilist{4close}{b}$. Since $e$ was taken from the List 4 of
a grandparent of $b$ or higher, $\ilist{4close}{b}$ must be a subset of
$\ilist{4}{b} \cup \ilist{4}{\parent(b)}$. By the earlier argument,
$|\ilist{4close}{b}| \leq 42$.

Recall from Definition~\ref{def:list-4-far} that $\ilist{4far}{b}$ is a subset of $\ilist{4}{b}\cup
\ilist{4close}{\parent(b)}$. It follows that $|\ilist{4far}{b}| \leq 21 + 42 =
63$.
Therefore the total cost of Stage 6 is at most is at most $63 N_B \nmax \pfmm +
42 N_C \nmax \pqbx$.
\end{proof}


The next theorem summarizes the contents of this section. It is useful to
consider the tree build phase (Stage 1) and the evaluation phase (the remaining
stages) as conceptually distinct phases of the algorithm, with their own cost
analysis. In the context of solving integral equations, the tree build typically only
needs to be done once, while evaluation may need to be done many times, such as in the inner
iteration of an iterative method like GMRES~\cite{saad_gmres_1986}. The tree build has a
straightforward cost of $O(NL)$. The cost of the evaluation phase is more
complicated to analyze.

As is the case for most adaptive point FMMs, linear running time is achievable
under some set of additional assumptions on the particle distribution. For
instance, in the paper~\cite{carrier:1988:adaptive-fmm}, the algorithm is shown
to be linear-time only for particle distributions with at most $\lvert \log_2
\epsilon \rvert$ levels in the tree. Unlike~\cite{carrier:1988:adaptive-fmm}, we
choose not to make any assumptions on the number of levels of the tree.  We
instead establish that our algorithm (with a level-restricted quadtree) has a
cost at most \emph{linear in the number of boxes $N_B$}. Under the additional
assumption of $N_B = O(N)$, cost linear in $N_B$ implies that cost is
also linear in the number of particles $N$. This assumption is weaker than
limiting the number of levels of the tree. Recent work on adaptive
FMMs~\cite{pouransari:2015:adaptive-fractal-fmm} seeks to show that the adaptive
FMM has time complexity linear in the number of particles irrespective of the
particle distribution, by some modification to the basic FMM algorithm. Our
variant of the FMM should be amenable to these modifications should the need
arise.

\begin{theorem}%
\label{thm:gigaqbx-cost}
(a) The cost of the tree build phase of the \algbrand~FMM is $O(NL)$.

(b) Assume that $\pfmm = O(\lvert \log \epsilon \rvert)$, and that $\pqbx \le \pfmm$. For a
fixed value of $\nmax$, the modeled cost of the evaluation stage of the
\algbrand~FMM is $O((N + N_B) \lvert \log \epsilon \rvert^2 + NL \lvert \log \epsilon \rvert + N_C M_C
\lvert \log \epsilon \rvert)$. Using a level-restricted quadtree, the modeled cost is $O((N
+ N_B) \lvert \log \epsilon \rvert^2 + N_C M_C \lvert \log \epsilon \rvert)$. If the particle
distribution satisfies $N_B = O(N)$ and $M_C = O(1)$, the worst-case modeled
cost using a level-restricted quadtree is linear in $N$.
\end{theorem}

\begin{proof}
The proof of this theorem is evident from adding up the cost of the
individual stages of the algorithm as given in
Table~\ref{tab:complexity-analysis}. The cost of Stage 1 (the tree build phase)
is immediate from the table. The cost of the remaining stages may be obtained by
bounding $\pqbx$ and $\pfmm$ by $O(\lvert \log \epsilon \rvert)$. The factor of $L$ in Stage
5 may be eliminated by using a level-restricted quadtree, cf.
Remark~\ref{rem:level-restriction}.
\end{proof}




\section{Results}%
\label{sec:results}

\def\starfishquadorder{8}
\def\starfishovsmp{4}
\def\starfishtotalquadorder{32}
\def\starfishinitialpanelsperarm{50}
\def\starfishgreentestarms{65}
\def\starfishgreentestpanels{3250}
\def\complexitygigaqbxavgerrorthree{\num{4.97e-05}}
\def\complexitygigaqbxavgerrorseven{\num{2.97e-06}}
\def\complexityqbxfmmavgerrorthree{\num{5.37e-05}}
\def\complexityqbxfmmavgerrorseven{\num{3.08e-06}}
\def\nmaxgigaqbx{64}
\def\nmaxqbxfmm{128}


\begin{figure}
  \centering
  \input{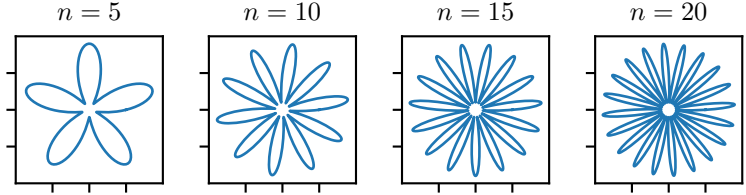}%
  \caption{%
    A subset of the `starfish' test geometries used to obtain many
    of the results of Section~\ref{sec:results}.}%
  \label{fig:starfish}
\end{figure}

In this section, we illustrate the numerical accuracy and cost scaling of the
\algbrand\ FMM
algorithm. We also perform a cost comparison of this
algorithm with the prior algorithm of~\cite{rachh:2017:qbx-fmm}.

The experimental setup is as follows. We use a family of test geometries,
parameterized by $n \in \mathbb{N}$, that form a `starfish' curve $\gamma_n :
[0,1] \to \mathbb{R}^2$ whose parametrization is given by
\begin{equation}
\label{eqn:starfish}
\gamma_n(t) =
\left(1 + 0.8 \sin(2 \pi n t) \right)
\begin{pmatrix}
  \cos(2\pi t) \\
  \sin(2 \pi t)
\end{pmatrix}.
\end{equation}
With increasing $n$, the starfish geometry has a larger number of more
closely-spaced `arms'. See Figure~\ref{fig:starfish} for graphical
renditions of some of these geometries. We make use of these geometries
because we have empirically found them to present a demanding scenario for layer
potential evaluation, with varying feature and panel sizes, close spacing
of unconnected parts of the geometry, and large (and scalable) overall size.
We have found these characteristics to present an adequate challenge both the
accuracy and the scalability of a layer potential evaluation code in a way
that is representative of smooth source geometries `in the wild'.


\subsection{Accuracy}%
\label{sec:accuracy-results}

{
  \renewrobustcmd{\bfseries}{\fontseries{b}\selectfont}
  \newcommand{\converged}[1]{\bfseries #1}
  \sisetup{
    table-format = 1.2e-1,
    table-number-alignment = center,
    table-sign-exponent = true,
    round-mode = places,
    round-precision = 2,
    detect-weight = true,
    mode = text,
  }
  \begin{table}
    \centering
\begin{tabular}{S[table-format = 1e-1, round-precision = 0]cSSSS}
\toprule
{$(1/2)^{\pfmm+1}$} & {$\pfmm$} & {$\pqbx=3$} & {$\pqbx=5$} & {$\pqbx=7$} & {$\pqbx=9$}\\
\midrule
{0} & (direct) & 4.347051e-06 & 6.208251e-07 & 1.046014e-07 & 5.707204e-08\\
6.250000e-02 & 3 & 4.191135e-03 & 4.205747e-03 & 4.205641e-03 & 4.205793e-03\\
1.562500e-02 & 5 & 2.790554e-04 & 3.423915e-04 & 3.432574e-04 & 3.433527e-04\\
4.882812e-04 & 10 & \converged{4.350459e-06} & 1.358643e-06 & 1.710896e-06 & 1.756206e-06\\
1.525879e-05 & 15 & \converged{4.347050e-06} & \converged{6.208327e-07} & \converged{1.045694e-07} & 5.739338e-08\\
4.768372e-07 & 20 & \converged{4.347051e-06} & \converged{6.208251e-07} & \converged{1.046015e-07} & \converged{5.707204e-08}\\
\bottomrule
\end{tabular}
    \belowtableskip%
    \caption{$\ell^\infty$ error in Green's formula $\mathcal S(\partial_n
      u)-\mathcal D(u)=u/2$, scaled by $1/\|u\|_\infty$, for the
      $\starfishgreentestarms$-armed starfish
      $\gamma_{65}$, using the \algbrand\ algorithm.  $\pfmm$ denotes the FMM
      order and $\pqbx$ the QBX order.  The geometry was discretized with
      $\starfishgreentestpanels$ Gauss-Legendre panels, with $33$ nodes per
      panel. Idealized point FMM error ${(1/2)}^{\pfmm+1}$ included for
      comparison. Entries in bold indicate that the FMM error is negligible.}%
    \label{tab:starfish-accuracy}
    \intertableskip%
    \begin{tabular}{S[table-format = 1e-1, round-precision = 0]cSSSS}
      \toprule
      {$(1/2)^{\pfmm+1}$} & {$\pfmm$} & {$\pqbx=3$} & {$\pqbx=5$} & {$\pqbx=7$} & {$\pqbx=9$}\\
      \midrule
      {0} & (direct) & 4.347051e-06 & 6.208251e-07 & 1.046014e-07 & 5.707204e-08\\
      6.250000e-02 & 3 & 2.546447e-02 & 2.963941e-02 & 4.074078e-02 & 5.766060e-02\\
      1.562500e-02 & 5 & 6.942927e-03 & 1.613796e-02 & 2.289611e-02 & 3.102830e-02\\
      4.882812e-04 & 10 & 4.952200e-04 & 1.749411e-03 & 5.803410e-03 & 9.479241e-03\\
      1.525879e-05 & 15 & 1.582358e-05 & 1.845764e-04 & 6.402331e-04 & 3.171478e-03\\
      4.768372e-07 & 20 & \converged{4.346750e-06} & 1.314485e-05 & 8.993484e-05 & 5.012005e-04\\
      \bottomrule
    \end{tabular}
    \belowtableskip%
    \caption{Analogous data to Table~\ref{tab:starfish-accuracy}
      for the conventional QBX FMM algorithm of~\cite{rachh:2017:qbx-fmm}.
    }%
    \label{tab:starfish-accuracy-old}
  \end{table}
}

We test the accuracy of the algorithm of this paper in a sequence of
experiments. To assess accuracy, we employ \emph{Green's formula} on
the source geometry. Let $\Gamma$ be the boundary of the domain. Let $u$ be a
harmonic function defined inside the domain and extending smoothly to the
boundary. Because $u$ extends smoothly to the boundary of the domain, the normal
derivative $\partial_n u$ at the boundary is well-defined. Then Green's formula
(e.g.~\cite[Theorem 6.5]{kress:2014:integral-equations}) states that for $x\in\Gamma$,
\[ \mathcal{S} (\partial_n u)(x) - \mathcal{D}(u)(x) = \frac{u(x)}{2}. \]
We use the residual in Green's formula as a convenient proxy for the accuracy
attained in the evaluation of the layer potential evaluation as well as the
overall accuracy attainable in application problems, in particular in the
context of the solution of boundary value problems.
We conducted two sets of experiments in this section. The first set is summarized
in Table~\ref{tab:bvp-green-accuracy}, which presents data to support the
assertion that the residual in Green's formula is a reasonable proxy for the
error in the solution of boundary value problems.
In the second set of experiments, we measured the residual in the
evaluation of Green's formula for a complicated geometry. We let $u$ be
the potential due to a charge located outside
$\Gamma$ at $(2,1)$. We compute approximations to $\mathcal{S} (\partial_n u) -
\mathcal{D}(u)$ and report the error in the discrete infinity norm.
The error reported is the absolute error scaled by $1/\|u\|_\infty$ (so that it is
relative to the magnitude of $u$). We use
the starfish curve with $n = \starfishgreentestarms$ and test with various
combinations of QBX order $\pqbx$ and FMM order $\pfmm$. $\gamma_{65}$ was
discretized with $\starfishgreentestpanels$ Gauss-Legendre panels with
$\pgfmathtruncatemacro{\pgmres}{\starfishquadorder+1}\pgmres$ nodes oversampled
to panels with
$\pgfmathtruncatemacro{\pgmres}{\starfishtotalquadorder+1}\pgmres$ nodes
(cf. Section~\ref{sec:qbx}). The curve was subsequently refined according to the
refinement criteria of Section~\ref{sec:qbx-geometry-preprocessing}.

Table~\ref{tab:starfish-accuracy}
shows the results of these experiments for the \algbrand\ FMM, varying $\pqbx$ across columns
and $\pfmm$ across rows.
The error incurred in unaccelerated QBX is shown in the first row of results.
This value represents a lower bound on the accuracy of the scheme for errors of
a given QBX order (as shown within a column); no error obtained with
acceleration (as shown in the remaining rows) will be meaningfully smaller.
Any error beyond the value in the first row is necessarily attributable
to the effects of acceleration.
We show table entries in bold if they do not significantly exceed the
error value for unaccelerated QBX, indicating that the error contribution
of FMM acceleration is negligible.

We choose the target confinement factor as $t_f=0.9$.  For that value of $t_f$,
Theorem~\ref{thm:gigaqbx-accuracy} roughly establishes $\|u\|_\infty {(1/2)}^{\pfmm+1}$ as
a bound on the absolute error incurred by acceleration, neglecting a factor of
$(\pqbx+1)$ and a number of other factors that do not vary across the entries of
the table.
We show ${(1/2)}^{\pfmm+1}$ in the left column of the table.
We find that the results support $\|u\|_\infty {(1/2)}^{\pfmm+1}$ as an asymptotic upper
bound on the error, and, in turn, the assertion that the error in the potential computed
via the \algbrand\ FMM is bounded simply by
\begin{equation}
  | \text{unaccelerated QBX error}| + \|u\|_\infty {(1/2)}^{\pfmm+1},
  \label{eq:gigaqbx-sum-bound}
\end{equation}
consistent with~\eqref{eqn:error-splitting}.
In addition, we observe that the bound~\eqref{eq:gigaqbx-sum-bound} lends itself to the
simple interpretation that the additional error in the potential incurred from
due to \algbrand\ FMM acceleration is asymptotically (in $\pfmm$) the same as the error incurred in the
evaluation of a point potential in the adaptive FMM of~\cite{carrier:1988:adaptive-fmm}.

Table~\ref{tab:starfish-accuracy} also allows us to assess the sharpness
of the analysis underpinning Theorem~\ref{thm:gigaqbx-accuracy}.
In the regime where the error is dominated by
the contributions of FMM acceleration (the upper part of the table),
we observe a match between the bound and the behavior of the error in asymptotic
behavior, although concrete error values are overestimated by around two orders
of magnitude.

Table~\ref{tab:starfish-accuracy-old} shows an analogous set of results for the
conventional QBX FMM of~\cite{rachh:2017:qbx-fmm}.
We find that the conventional QBX FMM is also able to match the error achieved
by unaccelerated QBX, albeit at considerably higher $\pfmm$ than our scheme.
The relationship between
$\pqbx$, $\pfmm$ and the error is more complicated than the simple
bound of~\eqref{eq:gigaqbx-sum-bound}. As a matter of fact,
only empirical error data were shown in~\cite{rachh:2017:qbx-fmm}.
Most poignantly perhaps, for the conventional QBX FMM, the error contribution
due to acceleration
is \emph{not} bounded by the FMM error incurred in a corresponding point FMM, and
the error appears to degrade with increasing QBX order as $\pfmm$ is held fixed.
Our scheme exhibits neither of these two issues.

The difference in behavior between the two schemes is easily explained.
The proofs of Lemma~\ref{lem:m2qbxl} and Lemma~\ref{lem:l2qbxl}
give bounds on the error of individual expansion coefficients.
Careful study shows that in a generic FMM translation operator, the higher order
coefficients are approximated less accurately than the lower order coefficients,
e.g.\ in formula~(\ref{eqn:m2qbxl-coeff-estimate}). While this issue in
principle applies to both versions of the scheme, the additional geometric
restrictions in our version mitigate the impact of this phenomenon by
controlling the amplification of this error by a geometric condition.

\begin{table}[t]
  \begin{center}
  \renewrobustcmd{\bfseries}{\fontseries{b}\selectfont}
  \newcommand{\converged}[1]{\bfseries #1}
  \sisetup{
    table-format = 1.2e-1,
    table-number-alignment = center,
    table-sign-exponent = true,
    round-mode = places,
    round-precision = 2,
    detect-weight = true,
    mode = text,
  }
\begin{tabular}{S[table-format = 1e-1, round-precision = 0]cScScScSc}
\toprule
{$(1/2)^{\pfmm+1}$} & {$\pfmm$} & {$\pqbx=3$} & \#it & {$\pqbx=5$} & \#it & {$\pqbx=7$} & \#it & {$\pqbx=9$} & \#it\\
\midrule
6.250000e-02 & 3 & 5.318344e-03 & 197 & 4.462246e-03 & 198 & 4.488637e-03 & 198 & 4.494761e-03 & 198\\
 &  & 2.542463e-03 &  & 2.480026e-03 &  & 2.477633e-03 &  & 2.477992e-03 & \\
\cmidrule{1-10} 1.562500e-02 & 5 & \converged{2.547243e-03} & 183 & 3.047912e-04 & 179 & 2.722968e-04 & 181 & 2.637143e-04 & 184\\
 &  & 2.379409e-04 &  & 2.545438e-04 &  & 2.546119e-04 &  & 2.546730e-04 & \\
\cmidrule{1-10} 4.882812e-04 & 10 & \converged{2.707797e-03} & 183 & \converged{5.283574e-05} & 179 & 1.508989e-05 & 180 & 1.644718e-06 & 183\\
 &  & \converged{1.110507e-04} &  & \converged{2.076393e-05} &  & \converged{5.176632e-06} &  & \converged{2.490217e-06} & \\
\cmidrule{1-10} 1.525879e-05 & 15 & \converged{2.707631e-03} & 183 & \converged{5.282786e-05} & 179 & \converged{1.502215e-05} & 180 & \converged{8.925412e-07} & 183\\
 &  & \converged{1.110462e-04} &  & \converged{2.076447e-05} &  & \converged{5.176010e-06} &  & \converged{2.496164e-06} & \\
\cmidrule{1-10} 4.768372e-07 & 20 & \converged{2.707631e-03} & 183 & \converged{5.282862e-05} & 179 & \converged{1.502090e-05} & 180 & \converged{8.949539e-07} & 183\\
 &  & \converged{1.110462e-04} &  & \converged{2.076447e-05} &  & \converged{5.176010e-06} &  & \converged{2.496163e-06} & \\
\bottomrule
\end{tabular}
  \belowtableskip%
  \caption{%
    A comparison of relative $\ell^\infty$ errors attained at a set of target
    points in the solution of an exterior Neumann boundary value
    problem~\eqref{eq:laplace-bc} using the integral
    equation~\eqref{eq:ext-neumann-ie} (top row of each segment) with errors
    attained in the residual of Green's formula on $\gamma_{25}$ (bottom row of
    each segment). Discretization parameters for both problem types as well as
    the procedure for obtaining the residual in Green's formula are as in
    Table~\ref{tab:starfish-accuracy}.  Iteration counts for unpreconditioned
    GMRES are shown in the columns labeled `\#it'.  The discrete linear system
    used the weighting technique of~\cite{bremer_nystrom_2011}.  To `manufacture' a
    reference solution of the BVP, point potentials were evaluated originating
    from sources at locations $0.75{[\cos\alpha_i,\sin\alpha_i]}^T$ with
    `charges' randomly assigned according to a standard normal distribution.  The angles
    $\alpha_i$ are given by $\alpha_i=\pi/2+2\pi i/25$ ($i\in\{0,\dots,24\}$).
    The $\ell^\infty$ norm of the vector of differences between the
    `manufactured' potentials and the potentials from the BVP solve at the
    target points at locations $1.5{[\cos(\pi+\alpha_i),\sin(\pi+\alpha_i)]}^T$
    was computed and is shown in the table.
  }%
  \label{tab:bvp-green-accuracy}
  \end{center}%
\end{table}

\begin{table}
  \newcommand{\cellcenter}[1]{\multicolumn{1}{c}{#1}}
  \begin{minipage}[t]{0.63\textwidth}
    \vspace{0pt} 
    \centering
\begin{tabular}{rrrrrrrr}
\toprule
\cellcenter{\multirow{2}{*}{$n$}} & \cellcenter{\multirow{2}{*}{$N_S$}} & \cellcenter{\multirow{2}{*}{$M_C$}} & \multicolumn{5}{c}{Percentiles}\\
\cmidrule(lr){4-8} &  &  & \cellcenter{20\%} & \cellcenter{40\%} & \cellcenter{60\%} & \cellcenter{80\%} & \cellcenter{100\%}\\
\midrule
5 & 9735 & 551.8 & 335.0 & 391.0 & 492.0 & 793.2 & 1656.0\\
15 & 52965 & 965.2 & 552.0 & 674.0 & 867.0 & 1228.0 & 7671.0\\
25 & 122925 & 988.2 & 600.0 & 670.0 & 804.0 & 1189.0 & 6516.0\\
35 & 255255 & 867.6 & 601.0 & 656.0 & 752.0 & 1070.0 & 3653.0\\
45 & 392040 & 882.6 & 602.0 & 659.0 & 741.0 & 1113.0 & 3944.0\\
55 & 555390 & 932.1 & 607.0 & 673.0 & 829.0 & 1176.0 & 4135.0\\
65 & 789360 & 921.2 & 604.0 & 659.0 & 801.0 & 1173.0 & 3616.0\\
\bottomrule
\end{tabular}
    \belowtableskip%
    \caption{Values of the parameter $M_C$ for various starfish geometries
      $\gamma_n$
      parametrized by $n$. $N_S$ denotes the number of source quadrature points.
      Shown here are percentiles
      for the distribution of the number of particles in a square of radius
      $8/t_f$ around each QBX center. Here, $t_f = 0.9$. $M_C$ denotes the
      empirical mean of the distribution.
      }%
    \label{tab:mc-results}
  \end{minipage}
  \hfill
  \begin{minipage}[t]{0.35\textwidth}
    \vspace{0pt} 
    \centering
    \begin{tabular}{lc}
      \toprule
      List & Cost \\
      \midrule
      $\ilist{1}{b}$ & $\pqbx n_s n_t$ \\
      $\ilist{2}{b}$ & $\pfmm^2$ \\
      $\ilist{3close}{b}$ & $\pqbx n_s n_t$ \\
      $\ilist{3far}{b}$ & $\pfmm \pqbx n_t$ \\
      $\ilist{4close}{b}$ & $\pqbx n_s n_t$ \\
      $\ilist{4far}{b}$ & $\pfmm n_s$ \\
      \bottomrule
    \end{tabular}
    \belowtableskip%
    \caption{Cost per interaction list entry modeled in
      Figures~\ref{fig:complexity-gigaqbx}
      and~\ref{fig:complexity-qbxfmm}, i.e.\ for a single (source box, target box) interaction
      list pair.  $\pfmm$ = FMM order and $\pqbx$ = QBX order. $n_s$ = number of sources
      in the source box and $n_t$ = number of QBX centers in the target box.}%
    \label{tab:operation-costs}
  \end{minipage}
\end{table}

\begin{figure}[t]
  \centering

  \input{complexity-gigaqbx.pgf}
  \caption{Modeled operation counts for the \algbrand~FMM for evaluating the single
    layer potential on a sequence of `starfish' geometries of increasing particle
    count. The operations are counted according to the model presented in
    Table~\ref{tab:operation-costs}. Here, $\nmax = \nmaxgigaqbx$ and $t_f =
    0.9$. The mean $\ell^\infty$ error in Green's identity across all runs,
    scaled by $1/\|u\|_\infty$, was
    \complexitygigaqbxavgerrorthree~for $\pqbx = 3$ and
    \complexitygigaqbxavgerrorseven~for $\pqbx = 7$.}%
  \label{fig:complexity-gigaqbx}

  \vspace{1ex}
  \input{complexity-qbxfmm.pgf}
  \caption{Modeled operation counts for the conventional QBX FMM
    of~\cite{rachh:2017:qbx-fmm}
    for evaluating the single layer
    potential on a sequence of `starfish' geometries of increasing particle count. The
    operations are counted according to the model presented in
    Table~\ref{tab:operation-costs}. Here, $\nmax = \nmaxqbxfmm$. The mean
    $\ell^\infty$ error in Green's identity across all runs,
    scaled by $1/\|u\|_\infty$, was
    \complexityqbxfmmavgerrorthree~for $\pqbx = 3$ and
    \complexityqbxfmmavgerrorseven~for $\pqbx = 7$.}%
  \label{fig:complexity-qbxfmm}
\end{figure}

\subsection{Cost and Scalability}%
\label{sec:scaling-results}
Having established that the accuracy of layer potentials evaluated \algbrand\ FMM
can be understood with the help of easy-to-use estimates and that high
levels of accuracy can be achieved, we seek to evaluate several aspects of the
computational cost of our algorithm. First and foremost, we examine the
scaling behavior of the scheme to large problem sizes.
Next, we briefly highlight the cost-accuracy trade-off encountered.
Lastly, since our scheme competes with the conventional QBX FMM, we give a
cost comparison between the two approaches.

For the remainder of this section, we use the same family of `starfish' geometries
from~\eqref{eqn:starfish} already familiar to the reader from our accuracy experiments.
More specifically, for a fixed value of the `arm count'
$n$, we begin with the curve $\gamma_n$ discretized into
$\starfishinitialpanelsperarm n$ panels equispaced in the parameter domain with
$\pgfmathtruncatemacro{\pgmres}{\starfishquadorder+1}\pgmres$ nodes per panel,
which was upsampled to
$\pgfmathtruncatemacro{\pgmres}{\starfishtotalquadorder+1}\pgmres$ nodes per panel. We
use values of $n$ ranging from $5$ to $65$ in increments of $10$. Additional refinement
in accordance with Section~\ref{sec:qbx-geometry-preprocessing} was applied if necessary. Ultimately, this family
of geometries ranged in size from about $1.7 \cdot 10^4$ to about $1.4 \cdot 10^6$ particles,
where by `particle' we mean a class of entities including QBX centers,
source quadrature nodes, and targets. We choose to employ this family of geometries with increasing complexity
over, say, a simpler, growing grid of identical geometries because we expect the
resulting scalability data to be credibly applicable to most other scenarios,
including those of the growing grid.

\subsubsection{Factors Influencing Computational Scalability}%
\label{sec:neighborhood-sizes-results}
Ideally, we would like to retain linear scaling of computational cost
with the size of the geometries, as measured in the number of source quadrature
points. Following the discussion of Section~\ref{sec:complexity},
it is not obvious that such scaling necessarily occurs.  Recall
the definition of the model parameter $M_C$, the main use of which is to provide
a worst case bound on the number of direct interactions between source particles
and \emph{suspended} QBX centers. The cost of these interactions in the
\algbrand~FMM is always bounded from above by $O(N_C M_C)$, where $N_C$ is the
number of centers.  For worst-case particle distributions, this cost is
unavoidably quadratic, because $M_C$ can be as large as $N_C$.
Linear scaling of the method will only be seen if $M_C$ does not change
substantially across different-sized geometries.  We expect particle
distributions to which our method is applied to originate from discretizations
of smooth, non-self-intersecting curves, and these are significantly more
regular than an arbitrary particle distribution. Consequently it is conceivable
that we will observe behavior considerably more benign than the worst case.

Although $M_C$ does not depend on the tree, it is nevertheless not immediately
obvious how one might derive a meaningful a-priori bound for $M_C$ for general geometries that
may `loop back' on themselves in the way that (say) the starfish geometries do, bringing
QBX centers into the proximity of source geometry non-adjacent to their `parent'
geometry. To empirically determine the behavior of $M_C$, we wrote a program that
counts the number of source particles within $\closedbox \left( 8r_c/t_f, c \right)$
for each QBX center $c$. $M_C$ is the mean of these counts.
Recall Proposition~\ref{prop:qbx-nn-to-box-nn}, which states that this region
is a superset of the $1$-near neighborhood of $\homebox{c}$, which in turn
represents the region with which a center may need to interact directly.
The results are presented in Table~\ref{tab:mc-results}, including means and
percentiles for the distribution of source particle counts.

As can be seen in the table, the distribution of particles seems to be
heavy-tailed, but with a mean value ($M_C$) of at most 1030 particles, which
does not appear be growing as the number of source particles increases. These
data are consistent with the observation that $M_C$ should not depend on the
number of particles for smooth geometries of adequate refinement.

\subsubsection{Experimental Results on Scaling and Comparative Cost}%
\label{sec:op-counts-results}

In this section, we illustrate the cost of the algorithm with the help of operation
counts. To give a machine-independent understanding of the computational cost of
our algorithm, we modeled computational cost by attributing an operation count
to each entry in the interaction lists.  The cost we attributed to an entry in
each type of interaction lists is summarized in Table~\ref{tab:operation-costs}.
Similarly to the approach of Section~\ref{sec:complexity}, the operation counts
thus obtained are intended to roughly correspond with the number of floating
point operations required.

We chose two FMM/QBX order pairs at which to gather this data for the
\algbrand~FMM, namely $(\pqbx,\pfmm) = (3, 10)$ and $(\pqbx,\pfmm) = (7, 15)$.
These values yield an average of roughly five and six digits of accuracy,
respectively. We show modeled operation counts across a number of `arm counts'
of the `starfish' geometries, as described.  The results are shown in graphical
form in Figure~\ref{fig:complexity-gigaqbx}. In addition to the cost for each
type interaction list, we also show an overall operation count summing the other
contributions, labeled `all'.

The costs in Figure~\ref{fig:complexity-gigaqbx} include the performance
optimization mentioned in Remark~\ref{rem:far-to-close-mpole-optimization}. Our
implementation used a $\ilist{3far}{b}$ interaction only with source boxes
having a cumulative source particle count of $15$ or more. In every case, the
improvement in cumulative operation counts due to this optimization was no more
than $1\%$.

Before entering into a discussion of this data, we introduce a second set of
data for comparison, based on the conventional QBX FMM
of~\cite{rachh:2017:qbx-fmm}.  We applied the same cost model to the QBX FMM in
order to perform an approximate comparison of the cost of the two algorithms. In
order to make the comparison meaningful, we compare the computational cost of
the FMMs for achieving a similar level of accuracy on the Green's identity test
(Section~\ref{sec:accuracy-results}) with a fixed QBX order. Experiments showed
that the (higher) FMM order values of $(\pqbx, \pfmm) = (3,15)$ and
$(\pqbx,\pfmm)=(7, 30)$ resulted in accuracies matching the above for the
conventional QBX FMM\@. We show graphs of computational cost across geometry sizes
analogous to the earlier ones for this data set in
Figure~\ref{fig:complexity-qbxfmm}.

We have tuned the user-chosen parameters for both algorithms to minimize their cost as measured by our
model. (This process is also known as `balancing' an FMM, since it
tends to balance various contributions to the cost.) The main parameter amenable
to such optimization is $\nmax$, the maximum number of particles per box.
We observed that $\nmax$ has different impact on
the performance for the two algorithms. Roughly, the \algbrand~FMM will benefit
from a smaller $\nmax$, as this can potentially decrease the number of direct
interactions. In contrast, the conventional QBX FMM benefits from a larger $\nmax$. The main
reason this is the case is that this reduces the number of boxes/levels in the
tree, and hence the number of multipole-to-local translations. We also observed
a noticeable degradation of accuracy for small $\nmax$ in the conventional QBX FMM, which we
believe may be related to the effect of $\nmax$ on source/target separation. As
a result, we found $\nmax = \nmaxgigaqbx$ for the \algbrand~FMM and $\nmax =
\nmaxqbxfmm$ for the conventional QBX FMM\@ to yield near-minimal modeled cost. We used $t_f
= 0.9$ for the \algbrand~FMM\@.

The linear scaling of both schemes is evident from the slope of the graphs,
with one decade of geometry growth (indicated by the vertical grid lines)
leading to one decade of cost growth (indicated by horizontal grid lines).
As is typical for schemes based on the FMM, the overall cost is dominated by
multipole-to-local translations (List 2/$\ilist{2}{b}$) and direct interactions (List 1/$\ilist{1}{b}$).
Additionally, in the \algbrand~FMM, List 4 close ($\ilist{4close}{b}$; which consists of direct
interactions just like List 1) is also a significant contributor to the cost.

The overall operation counts for the two schemes are roughly comparable, with
$\pqbx = 7$ more closely matching than $\pqbx = 3$. For $\pqbx = 3$, the
\algbrand~FMM has on average $1.43 \times$ as many modeled operations as the QBX
FMM, but for $\pqbx = 7$, it has about $1.10 \times$ as many. In terms of actual
wall times for evaluating the single layer potential, our implementation of the
\algbrand~FMM is on average $17\%$ slower than conventional QBX FMM for $\pqbx
= 3$ and $12\%$ slower for $\pqbx = 7$. As the QBX order increases, we expect the
\algbrand~FMM to maintain its competitiveness, particularly considering the
rapid growth of the FMM orders required to maintain accuracy in the conventional
QBX FMM\@.

Another factor worth highlighting is that the \algbrand~FMM is composed of a
larger number of simpler operations on, typically, lower-order expansions (thus
with relatively short chains of dependent computations within one translation
operation), while the conventional QBX FMM uses fewer higher-complexity operations to translate
expansions of higher order. While, according to our results above, these costs are similar for
sequential execution, we expect that the \algbrand~FMM will be able to make
better use of massively parallel computational resources.

A number of limitations of this study are evident. Both schemes stand
to benefit from standard FMM optimizations that have not been applied, such as,
for instance, using translation operators with asympotically improved
costs~\cite{hrycak:1998:improved-2d-fmm}. Additionally, tuning for the specific
hardware was was not applied when measuring the wall time. Nevertheless, the
results in this section suggest that the two schemes are competitive in terms of
cost.


\section{Conclusions and Future Work}%
\label{sec:conclusions}


In this paper, we have presented a fast algorithm for Quadrature by Expansion for
which we have also supplied analytical accuracy estimates. This algorithm
is compatible with (and builds upon) previous work designed to control for
truncation and quadrature error in QBX~\cite{rachh:2017:qbx-fmm}.
Unlike this and other work on accelerated global QBX, we are able to prove
\emph{strong accuracy guarantees} for our fast high order QBX scheme while
retaining a cost comparable with or cheaper than previous schemes.
We have demonstrated the viability of our approach
through numerical experiments. Lastly, we have provided a set of sufficient conditions
under which the algorithm exhibits linear scaling and also shown that in
practice the algorithm scales linearly on complicated geometries.

Traditional hierarchical algorithms developed for $n$-body problems have
considered \emph{point} (i.e., zero-dimensional) sources and targets. An
important feature of our work is the recognition that local expansions behave
like `targets with extent' from the point of view of the accuracy of
translation operators. It is possible to view this work purely in
this context, removed from QBX\@: As a fast algorithm that permits targets with
extent.

Many exciting avenues for future work open up building upon this contribution:
First, the cost estimates of Section~\ref{sec:complexity} are inherently
pessimistic because they do not leverage very much information about the
particle distribution. Furthermore, they are also conservative when it comes to
constants. Sharpening these estimates, perhaps with a more detailed
understanding of the typical particle distribution of a source curve, will help
provide a better understanding of the cost of our algorithm along with ideas to
reduce said cost.  Another direction of work is to apply the techniques of the
error analysis used in this paper in order to understand analytically the
accuracy behavior for the global QBX FMM of~\cite{rachh:2017:qbx-fmm}.
Preliminary results along these lines are encouraging.  Lastly, we are in the
process of extending the FMM developed in this paper to more kernels in two and
three dimensions, with the goal of providing black-box, fast, and accurate layer
potential evaluation for any kernel for which FMM translation infrastructure is
available.


\section*{Acknowledgments}
The authors' research was supported by the National Science Foundation under
grants DMS-1418961 and DMS-1654756 and by the University of Illinois.
Part of the work was performed while the authors were participating in
the HKUST-ICERM workshop `Integral Equation Methods, Fast
Algorithms and Their Applications to Fluid Dynamics and Materials
Science' held in 2017.

We thank Hao~Gao for bringing to our attention the software issue leading to the
miscounting of List~4~close in the numerical results (see~\ref{sec:corrigendum}).

\appendix

\section{Software and Reproducibility}%
\label{sec:software}

To allow other researchers to precisely verify our claimed results, we have
prepared a Docker image from which all results from Section~\ref{sec:results}
can be automatically reproduced with included software and scripts. This is
available at \url{https://doi.org/10.5281/zenodo.3483367}. Alternatively, the
code for the experiments is available
at~\url{https://doi.org/10.5281/zenodo.3483391}.

\section{Corrigendum to `\papertitle'}%
\label{sec:corrigendum}

Since the publication of `\papertitle', it has come to our
attention that the code used to obtain the numerical results reported in the
paper contains a number of errors. Additionally, we found that one portion of
the description of how the initial assumptions of the algorithm in the paper are
assured requires a correction. We have updated this manuscript appropriately.
In this appendix we summarize the necessary
corrections for the affected portions of the paper. In our judgment, none of
these changes presented here affect the conclusions of the paper.

\subsection*{Numerical Experiments}%

The errors in the code for the numerical experiments affect
Tables~\ref{tab:starfish-accuracy},~\ref{tab:bvp-green-accuracy},~and~\ref{tab:mc-results};
Figures~\ref{fig:complexity-gigaqbx}~and~\ref{fig:complexity-qbxfmm}; and some
of the values reported in Section~\ref{sec:op-counts-results}. We next list these
errors and discuss their implications on the results.

\paragraph*{Incorrect TCR and Separation Criteria.}

Due to an oversight, the experiments involving the GIGAQBX~algorithm used an
incorrectly shaped target confinement region. The experiments should have used a
`box-shaped,' or $\ell^\infty$~TCR, but due to an incorrect parameter setting
the experiments were run with a `disk-shaped,' or $\ell^2$~TCR, a modification
which is described in a later paper
(see~\url{https://arxiv.org/abs/1805.06106}). The criterion for `adequate
separation' of boxes did not follow the paper definition and was corrected. This
issue affects both reported accuracy and costs of the GIGAQBX~scheme. Compared
with the $\ell^2$~TCR, the use of the $\ell^\infty$~TCR results in an algorithm
with slightly higher accuracy and computational cost when other parameters are
held the same.

\paragraph*{Inconsistent Version of Algorithm.}

The results in Table~\ref{tab:bvp-green-accuracy} were obtained using a different
geometric refinement procedure than that used in the rest of the paper. This
inconsistency has been addressed by this correction and is the reason that that
the corrected values of the minimum attainable Green error in
Table~\ref{tab:bvp-green-accuracy} are somewhat greater than the originally
reported values.

\paragraph*{Issues in Cost Comparison.}

Two errors were present in the code used for the cost comparison study in
Section~\ref{sec:op-counts-results}. First, the cost of the stage associated
with List~4~close was undercounted, affecting the results in
Figure~\ref{fig:complexity-gigaqbx} under the label `$\ilist{4close}{b}$', but
not the results of Figure~\ref{fig:complexity-qbxfmm}. Second,
Stages~2,~7,~and~8 of the algorithm were overcounted. This latter overcounting
affects the total cost (labeled `all') in both figures, but not any other data
points. After correcting these errors, we put in place a verification scheme
that compares the computed modeled cost with independently obtained reference
results.  Taking into account these corrections, our results show that across
these geometries the GIGAQBX~FMM is about 9\% more expensive than originally
reported, while the QBX~FMM is about 6\% less expensive than originally
reported. The text in Section~\ref{sec:op-counts-results} was updated to reflect
these changes.

\paragraph*{Incorrectly Counted Particle Distribution.}%

Finally, an error in the computation of particle distances led to an overcount
of the values in Table~\ref{tab:mc-results}. The values presented in this table
report statistics about the particle distribution of the discretized geometries
used in the numerical experiments. The corrected version shows that the
particles are slightly less densely distributed, while still supporting our
claim that the parameter $M_C$ does not vary substantially across geometries.

\subsection*{Refinement Procedure}%

The GIGAQBX FMM described requires that the geometry discretization satisfies
certain assumptions informally stated in
Section~\ref{sec:qbx-geometry-preprocessing} to control for sources of
quadrature and truncation error. Section~\ref{sec:input-geometry-preprocessing}
stated that the geometric refinement procedure from Section~5
of~\cite{rachh:2017:qbx-fmm} was required to process the geometry to satisfy
these assumptions.  The description in
Section~\ref{sec:input-geometry-preprocessing} should merely have required that
Conditions~1,~3,~and~4 of~\cite{rachh:2017:qbx-fmm} are satisfied without
requiring a specific procedure. In our numerical experiments, we used an
improved version of the algorithm from~\cite{rachh:2017:qbx-fmm}. We refer to
our code (see~\ref{sec:software}) for details.  We report on an evolved
version of this algorithm in Section~3 of~\url{https://arxiv.org/abs/1805.06106}.

\printbibliography{}

\end{document}

%% file: list2.pgf
\begingroup%
\makeatletter%
\begin{pgfpicture}%
\pgfpathrectangle{\pgfpointorigin}{\pgfqpoint{3.080000in}{2.310000in}}%
\pgfusepath{use as bounding box, clip}%
\begin{pgfscope}%
\pgfsetbuttcap%
\pgfsetmiterjoin%
\definecolor{currentfill}{rgb}{1.000000,1.000000,1.000000}%
\pgfsetfillcolor{currentfill}%
\pgfsetlinewidth{0.000000pt}%
\definecolor{currentstroke}{rgb}{1.000000,1.000000,1.000000}%
\pgfsetstrokecolor{currentstroke}%
\pgfsetdash{}{0pt}%
\pgfpathmoveto{\pgfqpoint{0.000000in}{0.000000in}}%
\pgfpathlineto{\pgfqpoint{3.080000in}{0.000000in}}%
\pgfpathlineto{\pgfqpoint{3.080000in}{2.310000in}}%
\pgfpathlineto{\pgfqpoint{0.000000in}{2.310000in}}%
\pgfpathclose%
\pgfusepath{fill}%
\end{pgfscope}%
\begin{pgfscope}%
\pgfsetbuttcap%
\pgfsetmiterjoin%
\definecolor{currentfill}{rgb}{1.000000,1.000000,1.000000}%
\pgfsetfillcolor{currentfill}%
\pgfsetlinewidth{0.000000pt}%
\definecolor{currentstroke}{rgb}{0.000000,0.000000,0.000000}%
\pgfsetstrokecolor{currentstroke}%
\pgfsetstrokeopacity{0.000000}%
\pgfsetdash{}{0pt}%
\pgfpathmoveto{\pgfqpoint{0.392475in}{0.225139in}}%
\pgfpathlineto{\pgfqpoint{2.279142in}{0.225139in}}%
\pgfpathlineto{\pgfqpoint{2.279142in}{2.111806in}}%
\pgfpathlineto{\pgfqpoint{0.392475in}{2.111806in}}%
\pgfpathclose%
\pgfusepath{fill}%
\end{pgfscope}%
\begin{pgfscope}%
\pgfpathrectangle{\pgfqpoint{0.392475in}{0.225139in}}{\pgfqpoint{1.886667in}{1.886667in}} %
\pgfusepath{clip}%
\pgfsys@transformshift{0.392475in}{0.225139in}%
\pgftext[left,bottom]{\pgfimage[interpolate=true,width=1.888333in,height=1.888333in]{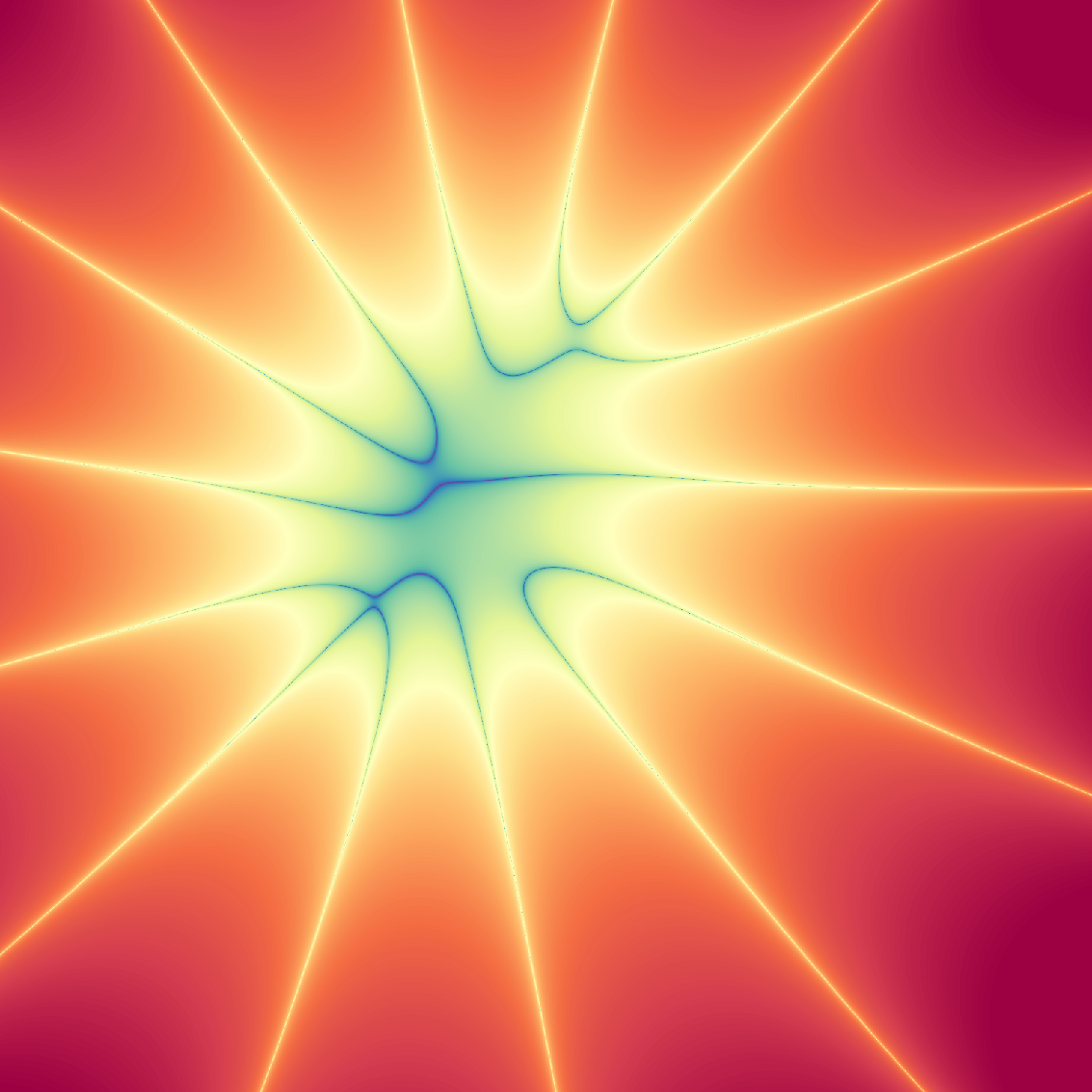}}%
\end{pgfscope}%
\begin{pgfscope}%
\pgfpathrectangle{\pgfqpoint{0.392475in}{0.225139in}}{\pgfqpoint{1.886667in}{1.886667in}} %
\pgfusepath{clip}%
\pgfsetbuttcap%
\pgfsetmiterjoin%
\pgfsetlinewidth{0.803000pt}%
\definecolor{currentstroke}{rgb}{0.000000,0.000000,0.000000}%
\pgfsetstrokecolor{currentstroke}%
\pgfsetdash{}{0pt}%
\pgfpathmoveto{\pgfqpoint{0.958475in}{1.168472in}}%
\pgfpathlineto{\pgfqpoint{1.335808in}{1.168472in}}%
\pgfpathlineto{\pgfqpoint{1.335808in}{1.545806in}}%
\pgfpathlineto{\pgfqpoint{0.958475in}{1.545806in}}%
\pgfpathclose%
\pgfusepath{stroke}%
\end{pgfscope}%
\begin{pgfscope}%
\pgfpathrectangle{\pgfqpoint{0.392475in}{0.225139in}}{\pgfqpoint{1.886667in}{1.886667in}} %
\pgfusepath{clip}%
\pgfsetbuttcap%
\pgfsetmiterjoin%
\pgfsetlinewidth{0.803000pt}%
\definecolor{currentstroke}{rgb}{0.000000,0.000000,0.000000}%
\pgfsetstrokecolor{currentstroke}%
\pgfsetdash{}{0pt}%
\pgfpathmoveto{\pgfqpoint{1.713142in}{1.168472in}}%
\pgfpathlineto{\pgfqpoint{2.090475in}{1.168472in}}%
\pgfpathlineto{\pgfqpoint{2.090475in}{1.545806in}}%
\pgfpathlineto{\pgfqpoint{1.713142in}{1.545806in}}%
\pgfpathclose%
\pgfusepath{stroke}%
\end{pgfscope}%
\begin{pgfscope}%
\pgfpathrectangle{\pgfqpoint{0.392475in}{0.225139in}}{\pgfqpoint{1.886667in}{1.886667in}} %
\pgfusepath{clip}%
\pgfsetbuttcap%
\pgfsetmiterjoin%
\pgfsetlinewidth{0.803000pt}%
\definecolor{currentstroke}{rgb}{0.000000,0.000000,0.000000}%
\pgfsetstrokecolor{currentstroke}%
\pgfsetdash{}{0pt}%
\pgfpathmoveto{\pgfqpoint{1.335808in}{0.413805in}}%
\pgfpathcurveto{\pgfqpoint{1.535948in}{0.413805in}}{\pgfqpoint{1.727918in}{0.493322in}}{\pgfqpoint{1.869438in}{0.634842in}}%
\pgfpathcurveto{\pgfqpoint{2.010959in}{0.776363in}}{\pgfqpoint{2.090475in}{0.968332in}}{\pgfqpoint{2.090475in}{1.168472in}}%
\pgfpathcurveto{\pgfqpoint{2.090475in}{1.368612in}}{\pgfqpoint{2.010959in}{1.560582in}}{\pgfqpoint{1.869438in}{1.702102in}}%
\pgfpathcurveto{\pgfqpoint{1.727918in}{1.843622in}}{\pgfqpoint{1.535948in}{1.923139in}}{\pgfqpoint{1.335808in}{1.923139in}}%
\pgfpathcurveto{\pgfqpoint{1.135668in}{1.923139in}}{\pgfqpoint{0.943699in}{1.843622in}}{\pgfqpoint{0.802178in}{1.702102in}}%
\pgfpathcurveto{\pgfqpoint{0.660658in}{1.560582in}}{\pgfqpoint{0.581142in}{1.368612in}}{\pgfqpoint{0.581142in}{1.168472in}}%
\pgfpathcurveto{\pgfqpoint{0.581142in}{0.968332in}}{\pgfqpoint{0.660658in}{0.776363in}}{\pgfqpoint{0.802178in}{0.634842in}}%
\pgfpathcurveto{\pgfqpoint{0.943699in}{0.493322in}}{\pgfqpoint{1.135668in}{0.413805in}}{\pgfqpoint{1.335808in}{0.413805in}}%
\pgfpathclose%
\pgfusepath{stroke}%
\end{pgfscope}%
\begin{pgfscope}%
\pgfpathrectangle{\pgfqpoint{0.392475in}{0.225139in}}{\pgfqpoint{1.886667in}{1.886667in}} %
\pgfusepath{clip}%
\pgfsetbuttcap%
\pgfsetroundjoin%
\definecolor{currentfill}{rgb}{0.121569,0.466667,0.705882}%
\pgfsetfillcolor{currentfill}%
\pgfsetlinewidth{1.003750pt}%
\definecolor{currentstroke}{rgb}{0.121569,0.466667,0.705882}%
\pgfsetstrokecolor{currentstroke}%
\pgfsetdash{}{0pt}%
\pgfsys@defobject{currentmarker}{\pgfqpoint{-0.041667in}{-0.041667in}}{\pgfqpoint{0.041667in}{0.041667in}}{%
\pgfpathmoveto{\pgfqpoint{0.000000in}{-0.041667in}}%
\pgfpathcurveto{\pgfqpoint{0.011050in}{-0.041667in}}{\pgfqpoint{0.021649in}{-0.037276in}}{\pgfqpoint{0.029463in}{-0.029463in}}%
\pgfpathcurveto{\pgfqpoint{0.037276in}{-0.021649in}}{\pgfqpoint{0.041667in}{-0.011050in}}{\pgfqpoint{0.041667in}{0.000000in}}%
\pgfpathcurveto{\pgfqpoint{0.041667in}{0.011050in}}{\pgfqpoint{0.037276in}{0.021649in}}{\pgfqpoint{0.029463in}{0.029463in}}%
\pgfpathcurveto{\pgfqpoint{0.021649in}{0.037276in}}{\pgfqpoint{0.011050in}{0.041667in}}{\pgfqpoint{0.000000in}{0.041667in}}%
\pgfpathcurveto{\pgfqpoint{-0.011050in}{0.041667in}}{\pgfqpoint{-0.021649in}{0.037276in}}{\pgfqpoint{-0.029463in}{0.029463in}}%
\pgfpathcurveto{\pgfqpoint{-0.037276in}{0.021649in}}{\pgfqpoint{-0.041667in}{0.011050in}}{\pgfqpoint{-0.041667in}{0.000000in}}%
\pgfpathcurveto{\pgfqpoint{-0.041667in}{-0.011050in}}{\pgfqpoint{-0.037276in}{-0.021649in}}{\pgfqpoint{-0.029463in}{-0.029463in}}%
\pgfpathcurveto{\pgfqpoint{-0.021649in}{-0.037276in}}{\pgfqpoint{-0.011050in}{-0.041667in}}{\pgfqpoint{0.000000in}{-0.041667in}}%
\pgfpathclose%
\pgfusepath{stroke,fill}%
}%
\begin{pgfscope}%
\pgfsys@transformshift{2.090475in}{1.168472in}%
\pgfsys@useobject{currentmarker}{}%
\end{pgfscope}%
\end{pgfscope}%
\begin{pgfscope}%
\pgfpathrectangle{\pgfqpoint{0.392475in}{0.225139in}}{\pgfqpoint{1.886667in}{1.886667in}} %
\pgfusepath{clip}%
\pgfsetrectcap%
\pgfsetroundjoin%
\pgfsetlinewidth{1.505625pt}%
\definecolor{currentstroke}{rgb}{1.000000,0.498039,0.054902}%
\pgfsetstrokecolor{currentstroke}%
\pgfsetdash{}{0pt}%
\pgfpathmoveto{\pgfqpoint{1.335808in}{1.168472in}}%
\pgfusepath{stroke}%
\end{pgfscope}%
\begin{pgfscope}%
\pgfpathrectangle{\pgfqpoint{0.392475in}{0.225139in}}{\pgfqpoint{1.886667in}{1.886667in}} %
\pgfusepath{clip}%
\pgfsetbuttcap%
\pgfsetroundjoin%
\definecolor{currentfill}{rgb}{1.000000,0.498039,0.054902}%
\pgfsetfillcolor{currentfill}%
\pgfsetlinewidth{1.003750pt}%
\definecolor{currentstroke}{rgb}{1.000000,0.498039,0.054902}%
\pgfsetstrokecolor{currentstroke}%
\pgfsetdash{}{0pt}%
\pgfsys@defobject{currentmarker}{\pgfqpoint{-0.041667in}{-0.041667in}}{\pgfqpoint{0.041667in}{0.041667in}}{%
\pgfpathmoveto{\pgfqpoint{0.000000in}{-0.041667in}}%
\pgfpathcurveto{\pgfqpoint{0.011050in}{-0.041667in}}{\pgfqpoint{0.021649in}{-0.037276in}}{\pgfqpoint{0.029463in}{-0.029463in}}%
\pgfpathcurveto{\pgfqpoint{0.037276in}{-0.021649in}}{\pgfqpoint{0.041667in}{-0.011050in}}{\pgfqpoint{0.041667in}{0.000000in}}%
\pgfpathcurveto{\pgfqpoint{0.041667in}{0.011050in}}{\pgfqpoint{0.037276in}{0.021649in}}{\pgfqpoint{0.029463in}{0.029463in}}%
\pgfpathcurveto{\pgfqpoint{0.021649in}{0.037276in}}{\pgfqpoint{0.011050in}{0.041667in}}{\pgfqpoint{0.000000in}{0.041667in}}%
\pgfpathcurveto{\pgfqpoint{-0.011050in}{0.041667in}}{\pgfqpoint{-0.021649in}{0.037276in}}{\pgfqpoint{-0.029463in}{0.029463in}}%
\pgfpathcurveto{\pgfqpoint{-0.037276in}{0.021649in}}{\pgfqpoint{-0.041667in}{0.011050in}}{\pgfqpoint{-0.041667in}{0.000000in}}%
\pgfpathcurveto{\pgfqpoint{-0.041667in}{-0.011050in}}{\pgfqpoint{-0.037276in}{-0.021649in}}{\pgfqpoint{-0.029463in}{-0.029463in}}%
\pgfpathcurveto{\pgfqpoint{-0.021649in}{-0.037276in}}{\pgfqpoint{-0.011050in}{-0.041667in}}{\pgfqpoint{0.000000in}{-0.041667in}}%
\pgfpathclose%
\pgfusepath{stroke,fill}%
}%
\begin{pgfscope}%
\pgfsys@transformshift{1.335808in}{1.168472in}%
\pgfsys@useobject{currentmarker}{}%
\end{pgfscope}%
\end{pgfscope}%
\begin{pgfscope}%
\pgfsetrectcap%
\pgfsetmiterjoin%
\pgfsetlinewidth{0.803000pt}%
\definecolor{currentstroke}{rgb}{0.000000,0.000000,0.000000}%
\pgfsetstrokecolor{currentstroke}%
\pgfsetdash{}{0pt}%
\pgfpathmoveto{\pgfqpoint{0.392475in}{0.225139in}}%
\pgfpathlineto{\pgfqpoint{0.392475in}{2.111806in}}%
\pgfusepath{stroke}%
\end{pgfscope}%
\begin{pgfscope}%
\pgfsetrectcap%
\pgfsetmiterjoin%
\pgfsetlinewidth{0.803000pt}%
\definecolor{currentstroke}{rgb}{0.000000,0.000000,0.000000}%
\pgfsetstrokecolor{currentstroke}%
\pgfsetdash{}{0pt}%
\pgfpathmoveto{\pgfqpoint{2.279142in}{0.225139in}}%
\pgfpathlineto{\pgfqpoint{2.279142in}{2.111806in}}%
\pgfusepath{stroke}%
\end{pgfscope}%
\begin{pgfscope}%
\pgfsetrectcap%
\pgfsetmiterjoin%
\pgfsetlinewidth{0.803000pt}%
\definecolor{currentstroke}{rgb}{0.000000,0.000000,0.000000}%
\pgfsetstrokecolor{currentstroke}%
\pgfsetdash{}{0pt}%
\pgfpathmoveto{\pgfqpoint{0.392475in}{0.225139in}}%
\pgfpathlineto{\pgfqpoint{2.279142in}{0.225139in}}%
\pgfusepath{stroke}%
\end{pgfscope}%
\begin{pgfscope}%
\pgfsetrectcap%
\pgfsetmiterjoin%
\pgfsetlinewidth{0.803000pt}%
\definecolor{currentstroke}{rgb}{0.000000,0.000000,0.000000}%
\pgfsetstrokecolor{currentstroke}%
\pgfsetdash{}{0pt}%
\pgfpathmoveto{\pgfqpoint{0.392475in}{2.111806in}}%
\pgfpathlineto{\pgfqpoint{2.279142in}{2.111806in}}%
\pgfusepath{stroke}%
\end{pgfscope}%
\begin{pgfscope}%
\pgfsetroundcap%
\pgfsetroundjoin%
\pgfsetlinewidth{0.803000pt}%
\definecolor{currentstroke}{rgb}{0.000000,0.000000,0.000000}%
\pgfsetstrokecolor{currentstroke}%
\pgfsetdash{}{0pt}%
\pgfpathmoveto{\pgfqpoint{1.307385in}{1.196896in}}%
\pgfpathquadraticcurveto{\pgfqpoint{1.274621in}{1.229659in}}{\pgfqpoint{1.233074in}{1.271207in}}%
\pgfusepath{stroke}%
\end{pgfscope}%
\begin{pgfscope}%
\pgfsetroundcap%
\pgfsetroundjoin%
\definecolor{currentfill}{rgb}{0.000000,0.000000,0.000000}%
\pgfsetfillcolor{currentfill}%
\pgfsetlinewidth{0.803000pt}%
\definecolor{currentstroke}{rgb}{0.000000,0.000000,0.000000}%
\pgfsetstrokecolor{currentstroke}%
\pgfsetdash{}{0pt}%
\pgfpathmoveto{\pgfqpoint{1.287743in}{1.255821in}}%
\pgfpathlineto{\pgfqpoint{1.307385in}{1.196896in}}%
\pgfpathlineto{\pgfqpoint{1.248459in}{1.216537in}}%
\pgfpathlineto{\pgfqpoint{1.287743in}{1.255821in}}%
\pgfpathclose%
\pgfusepath{stroke,fill}%
\end{pgfscope}%
\begin{pgfscope}%
\pgfsetroundcap%
\pgfsetroundjoin%
\pgfsetlinewidth{0.803000pt}%
\definecolor{currentstroke}{rgb}{0.000000,0.000000,0.000000}%
\pgfsetstrokecolor{currentstroke}%
\pgfsetdash{}{0pt}%
\pgfpathmoveto{\pgfqpoint{1.275213in}{1.357139in}}%
\pgfpathquadraticcurveto{\pgfqpoint{1.521536in}{1.357139in}}{\pgfqpoint{1.780281in}{1.357139in}}%
\pgfusepath{stroke}%
\end{pgfscope}%
\begin{pgfscope}%
\pgfsetroundcap%
\pgfsetroundjoin%
\definecolor{currentfill}{rgb}{0.000000,0.000000,0.000000}%
\pgfsetfillcolor{currentfill}%
\pgfsetlinewidth{0.803000pt}%
\definecolor{currentstroke}{rgb}{0.000000,0.000000,0.000000}%
\pgfsetstrokecolor{currentstroke}%
\pgfsetdash{}{0pt}%
\pgfpathmoveto{\pgfqpoint{1.330769in}{1.329361in}}%
\pgfpathlineto{\pgfqpoint{1.275213in}{1.357139in}}%
\pgfpathlineto{\pgfqpoint{1.330769in}{1.384917in}}%
\pgfpathlineto{\pgfqpoint{1.330769in}{1.329361in}}%
\pgfpathclose%
\pgfusepath{stroke,fill}%
\end{pgfscope}%
\begin{pgfscope}%
\pgftext[x=1.147142in,y=1.357139in,,]{\rmfamily\fontsize{10.000000}{12.000000}\selectfont \(\displaystyle \ell_{8}\)}%
\end{pgfscope}%
\begin{pgfscope}%
\pgfsetroundcap%
\pgfsetroundjoin%
\pgfsetlinewidth{0.803000pt}%
\definecolor{currentstroke}{rgb}{0.000000,0.000000,0.000000}%
\pgfsetstrokecolor{currentstroke}%
\pgfsetdash{}{0pt}%
\pgfpathmoveto{\pgfqpoint{2.019680in}{1.239267in}}%
\pgfpathquadraticcurveto{\pgfqpoint{2.050686in}{1.208262in}}{\pgfqpoint{2.090475in}{1.168472in}}%
\pgfusepath{stroke}%
\end{pgfscope}%
\begin{pgfscope}%
\pgfsetroundcap%
\pgfsetroundjoin%
\definecolor{currentfill}{rgb}{0.000000,0.000000,0.000000}%
\pgfsetfillcolor{currentfill}%
\pgfsetlinewidth{0.803000pt}%
\definecolor{currentstroke}{rgb}{0.000000,0.000000,0.000000}%
\pgfsetstrokecolor{currentstroke}%
\pgfsetdash{}{0pt}%
\pgfpathmoveto{\pgfqpoint{2.039322in}{1.180341in}}%
\pgfpathlineto{\pgfqpoint{2.019680in}{1.239267in}}%
\pgfpathlineto{\pgfqpoint{2.078606in}{1.219625in}}%
\pgfpathlineto{\pgfqpoint{2.039322in}{1.180341in}}%
\pgfpathclose%
\pgfusepath{stroke,fill}%
\end{pgfscope}%
\begin{pgfscope}%
\pgftext[x=1.901808in,y=1.357139in,,]{\rmfamily\fontsize{10.000000}{12.000000}\selectfont \(\displaystyle m_{8}\)}%
\end{pgfscope}%
\begin{pgfscope}%
\pgfsetroundcap%
\pgfsetroundjoin%
\pgfsetlinewidth{0.803000pt}%
\definecolor{currentstroke}{rgb}{0.000000,0.000000,0.000000}%
\pgfsetstrokecolor{currentstroke}%
\pgfsetdash{}{0pt}%
\pgfpathmoveto{\pgfqpoint{1.832873in}{0.934005in}}%
\pgfpathquadraticcurveto{\pgfqpoint{1.951401in}{1.041888in}}{\pgfqpoint{2.060742in}{1.141410in}}%
\pgfusepath{stroke}%
\end{pgfscope}%
\begin{pgfscope}%
\pgfsetroundcap%
\pgfsetroundjoin%
\definecolor{currentfill}{rgb}{0.000000,0.000000,0.000000}%
\pgfsetfillcolor{currentfill}%
\pgfsetlinewidth{0.803000pt}%
\definecolor{currentstroke}{rgb}{0.000000,0.000000,0.000000}%
\pgfsetstrokecolor{currentstroke}%
\pgfsetdash{}{0pt}%
\pgfpathmoveto{\pgfqpoint{2.000959in}{1.124557in}}%
\pgfpathlineto{\pgfqpoint{2.060742in}{1.141410in}}%
\pgfpathlineto{\pgfqpoint{2.038355in}{1.083472in}}%
\pgfpathlineto{\pgfqpoint{2.000959in}{1.124557in}}%
\pgfpathclose%
\pgfusepath{stroke,fill}%
\end{pgfscope}%
\begin{pgfscope}%
\pgftext[x=1.524475in,y=0.791139in,left,base]{\rmfamily\fontsize{10.000000}{12.000000}\selectfont source}%
\end{pgfscope}%
\begin{pgfscope}%
\pgftext[x=1.109408in,y=0.791139in,,]{\rmfamily\fontsize{10.000000}{12.000000}\selectfont \(\displaystyle \ell_{8}\)}%
\end{pgfscope}%
\begin{pgfscope}%
\pgfpathrectangle{\pgfqpoint{2.372270in}{0.225139in}}{\pgfqpoint{0.094333in}{1.886667in}} %
\pgfusepath{clip}%
\pgfsetbuttcap%
\pgfsetmiterjoin%
\definecolor{currentfill}{rgb}{1.000000,1.000000,1.000000}%
\pgfsetfillcolor{currentfill}%
\pgfsetlinewidth{0.010037pt}%
\definecolor{currentstroke}{rgb}{1.000000,1.000000,1.000000}%
\pgfsetstrokecolor{currentstroke}%
\pgfsetdash{}{0pt}%
\pgfpathmoveto{\pgfqpoint{2.372270in}{0.225139in}}%
\pgfpathlineto{\pgfqpoint{2.372270in}{0.232509in}}%
\pgfpathlineto{\pgfqpoint{2.372270in}{2.104436in}}%
\pgfpathlineto{\pgfqpoint{2.372270in}{2.111806in}}%
\pgfpathlineto{\pgfqpoint{2.466603in}{2.111806in}}%
\pgfpathlineto{\pgfqpoint{2.466603in}{2.104436in}}%
\pgfpathlineto{\pgfqpoint{2.466603in}{0.232509in}}%
\pgfpathlineto{\pgfqpoint{2.466603in}{0.225139in}}%
\pgfpathclose%
\pgfusepath{stroke,fill}%
\end{pgfscope}%
\begin{pgfscope}%
\pgfsys@transformshift{2.371667in}{0.225000in}%
\pgftext[left,bottom]{\pgfimage[interpolate=true,width=0.095000in,height=1.886667in]{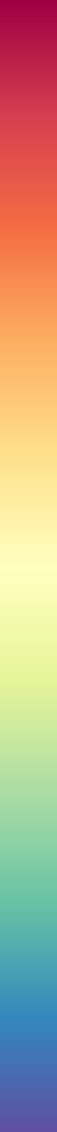}}%
\end{pgfscope}%
\begin{pgfscope}%
\pgfsetbuttcap%
\pgfsetroundjoin%
\definecolor{currentfill}{rgb}{0.000000,0.000000,0.000000}%
\pgfsetfillcolor{currentfill}%
\pgfsetlinewidth{0.803000pt}%
\definecolor{currentstroke}{rgb}{0.000000,0.000000,0.000000}%
\pgfsetstrokecolor{currentstroke}%
\pgfsetdash{}{0pt}%
\pgfsys@defobject{currentmarker}{\pgfqpoint{0.000000in}{0.000000in}}{\pgfqpoint{0.048611in}{0.000000in}}{%
\pgfpathmoveto{\pgfqpoint{0.000000in}{0.000000in}}%
\pgfpathlineto{\pgfqpoint{0.048611in}{0.000000in}}%
\pgfusepath{stroke,fill}%
}%
\begin{pgfscope}%
\pgfsys@transformshift{2.466603in}{0.225139in}%
\pgfsys@useobject{currentmarker}{}%
\end{pgfscope}%
\end{pgfscope}%
\begin{pgfscope}%
\pgftext[x=2.563825in,y=0.176944in,left,base]{\rmfamily\fontsize{10.000000}{12.000000}\selectfont −8}%
\end{pgfscope}%
\begin{pgfscope}%
\pgfsetbuttcap%
\pgfsetroundjoin%
\definecolor{currentfill}{rgb}{0.000000,0.000000,0.000000}%
\pgfsetfillcolor{currentfill}%
\pgfsetlinewidth{0.803000pt}%
\definecolor{currentstroke}{rgb}{0.000000,0.000000,0.000000}%
\pgfsetstrokecolor{currentstroke}%
\pgfsetdash{}{0pt}%
\pgfsys@defobject{currentmarker}{\pgfqpoint{0.000000in}{0.000000in}}{\pgfqpoint{0.048611in}{0.000000in}}{%
\pgfpathmoveto{\pgfqpoint{0.000000in}{0.000000in}}%
\pgfpathlineto{\pgfqpoint{0.048611in}{0.000000in}}%
\pgfusepath{stroke,fill}%
}%
\begin{pgfscope}%
\pgfsys@transformshift{2.466603in}{0.460972in}%
\pgfsys@useobject{currentmarker}{}%
\end{pgfscope}%
\end{pgfscope}%
\begin{pgfscope}%
\pgftext[x=2.563825in,y=0.412778in,left,base]{\rmfamily\fontsize{10.000000}{12.000000}\selectfont −7}%
\end{pgfscope}%
\begin{pgfscope}%
\pgfsetbuttcap%
\pgfsetroundjoin%
\definecolor{currentfill}{rgb}{0.000000,0.000000,0.000000}%
\pgfsetfillcolor{currentfill}%
\pgfsetlinewidth{0.803000pt}%
\definecolor{currentstroke}{rgb}{0.000000,0.000000,0.000000}%
\pgfsetstrokecolor{currentstroke}%
\pgfsetdash{}{0pt}%
\pgfsys@defobject{currentmarker}{\pgfqpoint{0.000000in}{0.000000in}}{\pgfqpoint{0.048611in}{0.000000in}}{%
\pgfpathmoveto{\pgfqpoint{0.000000in}{0.000000in}}%
\pgfpathlineto{\pgfqpoint{0.048611in}{0.000000in}}%
\pgfusepath{stroke,fill}%
}%
\begin{pgfscope}%
\pgfsys@transformshift{2.466603in}{0.696805in}%
\pgfsys@useobject{currentmarker}{}%
\end{pgfscope}%
\end{pgfscope}%
\begin{pgfscope}%
\pgftext[x=2.563825in,y=0.648611in,left,base]{\rmfamily\fontsize{10.000000}{12.000000}\selectfont −6}%
\end{pgfscope}%
\begin{pgfscope}%
\pgfsetbuttcap%
\pgfsetroundjoin%
\definecolor{currentfill}{rgb}{0.000000,0.000000,0.000000}%
\pgfsetfillcolor{currentfill}%
\pgfsetlinewidth{0.803000pt}%
\definecolor{currentstroke}{rgb}{0.000000,0.000000,0.000000}%
\pgfsetstrokecolor{currentstroke}%
\pgfsetdash{}{0pt}%
\pgfsys@defobject{currentmarker}{\pgfqpoint{0.000000in}{0.000000in}}{\pgfqpoint{0.048611in}{0.000000in}}{%
\pgfpathmoveto{\pgfqpoint{0.000000in}{0.000000in}}%
\pgfpathlineto{\pgfqpoint{0.048611in}{0.000000in}}%
\pgfusepath{stroke,fill}%
}%
\begin{pgfscope}%
\pgfsys@transformshift{2.466603in}{0.932639in}%
\pgfsys@useobject{currentmarker}{}%
\end{pgfscope}%
\end{pgfscope}%
\begin{pgfscope}%
\pgftext[x=2.563825in,y=0.884444in,left,base]{\rmfamily\fontsize{10.000000}{12.000000}\selectfont −5}%
\end{pgfscope}%
\begin{pgfscope}%
\pgfsetbuttcap%
\pgfsetroundjoin%
\definecolor{currentfill}{rgb}{0.000000,0.000000,0.000000}%
\pgfsetfillcolor{currentfill}%
\pgfsetlinewidth{0.803000pt}%
\definecolor{currentstroke}{rgb}{0.000000,0.000000,0.000000}%
\pgfsetstrokecolor{currentstroke}%
\pgfsetdash{}{0pt}%
\pgfsys@defobject{currentmarker}{\pgfqpoint{0.000000in}{0.000000in}}{\pgfqpoint{0.048611in}{0.000000in}}{%
\pgfpathmoveto{\pgfqpoint{0.000000in}{0.000000in}}%
\pgfpathlineto{\pgfqpoint{0.048611in}{0.000000in}}%
\pgfusepath{stroke,fill}%
}%
\begin{pgfscope}%
\pgfsys@transformshift{2.466603in}{1.168472in}%
\pgfsys@useobject{currentmarker}{}%
\end{pgfscope}%
\end{pgfscope}%
\begin{pgfscope}%
\pgftext[x=2.563825in,y=1.120278in,left,base]{\rmfamily\fontsize{10.000000}{12.000000}\selectfont −4}%
\end{pgfscope}%
\begin{pgfscope}%
\pgfsetbuttcap%
\pgfsetroundjoin%
\definecolor{currentfill}{rgb}{0.000000,0.000000,0.000000}%
\pgfsetfillcolor{currentfill}%
\pgfsetlinewidth{0.803000pt}%
\definecolor{currentstroke}{rgb}{0.000000,0.000000,0.000000}%
\pgfsetstrokecolor{currentstroke}%
\pgfsetdash{}{0pt}%
\pgfsys@defobject{currentmarker}{\pgfqpoint{0.000000in}{0.000000in}}{\pgfqpoint{0.048611in}{0.000000in}}{%
\pgfpathmoveto{\pgfqpoint{0.000000in}{0.000000in}}%
\pgfpathlineto{\pgfqpoint{0.048611in}{0.000000in}}%
\pgfusepath{stroke,fill}%
}%
\begin{pgfscope}%
\pgfsys@transformshift{2.466603in}{1.404306in}%
\pgfsys@useobject{currentmarker}{}%
\end{pgfscope}%
\end{pgfscope}%
\begin{pgfscope}%
\pgftext[x=2.563825in,y=1.356111in,left,base]{\rmfamily\fontsize{10.000000}{12.000000}\selectfont −3}%
\end{pgfscope}%
\begin{pgfscope}%
\pgfsetbuttcap%
\pgfsetroundjoin%
\definecolor{currentfill}{rgb}{0.000000,0.000000,0.000000}%
\pgfsetfillcolor{currentfill}%
\pgfsetlinewidth{0.803000pt}%
\definecolor{currentstroke}{rgb}{0.000000,0.000000,0.000000}%
\pgfsetstrokecolor{currentstroke}%
\pgfsetdash{}{0pt}%
\pgfsys@defobject{currentmarker}{\pgfqpoint{0.000000in}{0.000000in}}{\pgfqpoint{0.048611in}{0.000000in}}{%
\pgfpathmoveto{\pgfqpoint{0.000000in}{0.000000in}}%
\pgfpathlineto{\pgfqpoint{0.048611in}{0.000000in}}%
\pgfusepath{stroke,fill}%
}%
\begin{pgfscope}%
\pgfsys@transformshift{2.466603in}{1.640139in}%
\pgfsys@useobject{currentmarker}{}%
\end{pgfscope}%
\end{pgfscope}%
\begin{pgfscope}%
\pgftext[x=2.563825in,y=1.591944in,left,base]{\rmfamily\fontsize{10.000000}{12.000000}\selectfont −2}%
\end{pgfscope}%
\begin{pgfscope}%
\pgfsetbuttcap%
\pgfsetroundjoin%
\definecolor{currentfill}{rgb}{0.000000,0.000000,0.000000}%
\pgfsetfillcolor{currentfill}%
\pgfsetlinewidth{0.803000pt}%
\definecolor{currentstroke}{rgb}{0.000000,0.000000,0.000000}%
\pgfsetstrokecolor{currentstroke}%
\pgfsetdash{}{0pt}%
\pgfsys@defobject{currentmarker}{\pgfqpoint{0.000000in}{0.000000in}}{\pgfqpoint{0.048611in}{0.000000in}}{%
\pgfpathmoveto{\pgfqpoint{0.000000in}{0.000000in}}%
\pgfpathlineto{\pgfqpoint{0.048611in}{0.000000in}}%
\pgfusepath{stroke,fill}%
}%
\begin{pgfscope}%
\pgfsys@transformshift{2.466603in}{1.875972in}%
\pgfsys@useobject{currentmarker}{}%
\end{pgfscope}%
\end{pgfscope}%
\begin{pgfscope}%
\pgftext[x=2.563825in,y=1.827778in,left,base]{\rmfamily\fontsize{10.000000}{12.000000}\selectfont −1}%
\end{pgfscope}%
\begin{pgfscope}%
\pgfsetbuttcap%
\pgfsetroundjoin%
\definecolor{currentfill}{rgb}{0.000000,0.000000,0.000000}%
\pgfsetfillcolor{currentfill}%
\pgfsetlinewidth{0.803000pt}%
\definecolor{currentstroke}{rgb}{0.000000,0.000000,0.000000}%
\pgfsetstrokecolor{currentstroke}%
\pgfsetdash{}{0pt}%
\pgfsys@defobject{currentmarker}{\pgfqpoint{0.000000in}{0.000000in}}{\pgfqpoint{0.048611in}{0.000000in}}{%
\pgfpathmoveto{\pgfqpoint{0.000000in}{0.000000in}}%
\pgfpathlineto{\pgfqpoint{0.048611in}{0.000000in}}%
\pgfusepath{stroke,fill}%
}%
\begin{pgfscope}%
\pgfsys@transformshift{2.466603in}{2.111806in}%
\pgfsys@useobject{currentmarker}{}%
\end{pgfscope}%
\end{pgfscope}%
\begin{pgfscope}%
\pgftext[x=2.563825in,y=2.063611in,left,base]{\rmfamily\fontsize{10.000000}{12.000000}\selectfont 0}%
\end{pgfscope}%
\begin{pgfscope}%
\pgftext[x=2.796881in,y=1.168472in,,top,rotate=90.000000]{\rmfamily\fontsize{10.000000}{12.000000}\selectfont \(\displaystyle \log_{10} |\ell_{8,\mathrm{direct}} - \ell_{8,\mathrm{M2L}(8)}|\)}%
\end{pgfscope}%
\begin{pgfscope}%
\pgfsetbuttcap%
\pgfsetmiterjoin%
\pgfsetlinewidth{0.803000pt}%
\definecolor{currentstroke}{rgb}{0.000000,0.000000,0.000000}%
\pgfsetstrokecolor{currentstroke}%
\pgfsetdash{}{0pt}%
\pgfpathmoveto{\pgfqpoint{2.372270in}{0.225139in}}%
\pgfpathlineto{\pgfqpoint{2.372270in}{0.232509in}}%
\pgfpathlineto{\pgfqpoint{2.372270in}{2.104436in}}%
\pgfpathlineto{\pgfqpoint{2.372270in}{2.111806in}}%
\pgfpathlineto{\pgfqpoint{2.466603in}{2.111806in}}%
\pgfpathlineto{\pgfqpoint{2.466603in}{2.104436in}}%
\pgfpathlineto{\pgfqpoint{2.466603in}{0.232509in}}%
\pgfpathlineto{\pgfqpoint{2.466603in}{0.225139in}}%
\pgfpathclose%
\pgfusepath{stroke}%
\end{pgfscope}%
\end{pgfpicture}%
\makeatother%
\endgroup%

%% file: list3.pgf
\begingroup%
\makeatletter%
\begin{pgfpicture}%
\pgfpathrectangle{\pgfpointorigin}{\pgfqpoint{3.080000in}{2.310000in}}%
\pgfusepath{use as bounding box, clip}%
\begin{pgfscope}%
\pgfsetbuttcap%
\pgfsetmiterjoin%
\definecolor{currentfill}{rgb}{1.000000,1.000000,1.000000}%
\pgfsetfillcolor{currentfill}%
\pgfsetlinewidth{0.000000pt}%
\definecolor{currentstroke}{rgb}{1.000000,1.000000,1.000000}%
\pgfsetstrokecolor{currentstroke}%
\pgfsetdash{}{0pt}%
\pgfpathmoveto{\pgfqpoint{0.000000in}{0.000000in}}%
\pgfpathlineto{\pgfqpoint{3.080000in}{0.000000in}}%
\pgfpathlineto{\pgfqpoint{3.080000in}{2.310000in}}%
\pgfpathlineto{\pgfqpoint{0.000000in}{2.310000in}}%
\pgfpathclose%
\pgfusepath{fill}%
\end{pgfscope}%
\begin{pgfscope}%
\pgfsetbuttcap%
\pgfsetmiterjoin%
\definecolor{currentfill}{rgb}{1.000000,1.000000,1.000000}%
\pgfsetfillcolor{currentfill}%
\pgfsetlinewidth{0.000000pt}%
\definecolor{currentstroke}{rgb}{0.000000,0.000000,0.000000}%
\pgfsetstrokecolor{currentstroke}%
\pgfsetstrokeopacity{0.000000}%
\pgfsetdash{}{0pt}%
\pgfpathmoveto{\pgfqpoint{0.393078in}{0.225139in}}%
\pgfpathlineto{\pgfqpoint{2.279745in}{0.225139in}}%
\pgfpathlineto{\pgfqpoint{2.279745in}{2.111806in}}%
\pgfpathlineto{\pgfqpoint{0.393078in}{2.111806in}}%
\pgfpathclose%
\pgfusepath{fill}%
\end{pgfscope}%
\begin{pgfscope}%
\pgfpathrectangle{\pgfqpoint{0.393078in}{0.225139in}}{\pgfqpoint{1.886667in}{1.886667in}} %
\pgfusepath{clip}%
\pgfsys@transformshift{0.393078in}{0.225139in}%
\pgftext[left,bottom]{\pgfimage[interpolate=true,width=1.888333in,height=1.888333in]{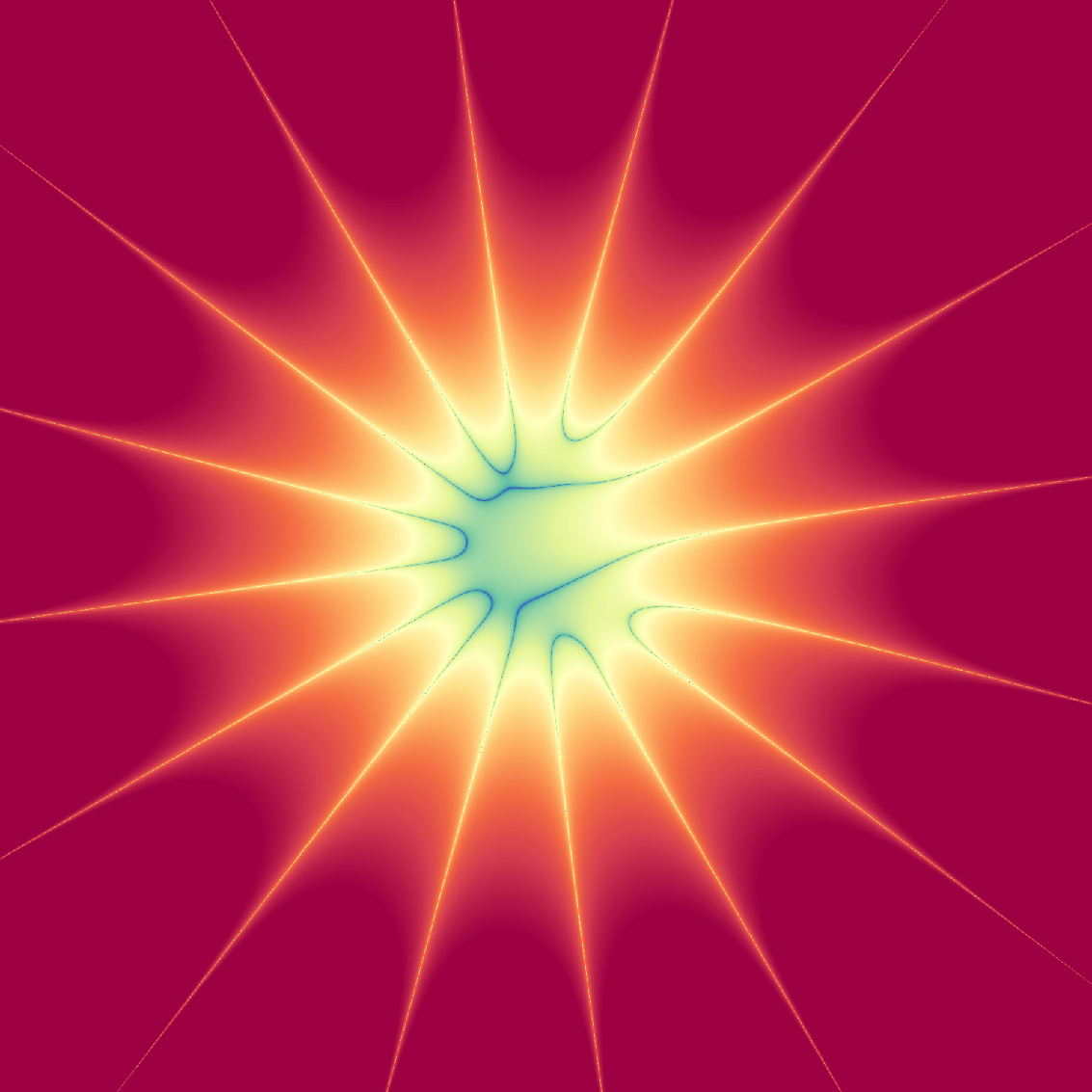}}%
\end{pgfscope}%
\begin{pgfscope}%
\pgfpathrectangle{\pgfqpoint{0.393078in}{0.225139in}}{\pgfqpoint{1.886667in}{1.886667in}} %
\pgfusepath{clip}%
\pgfsetbuttcap%
\pgfsetmiterjoin%
\pgfsetlinewidth{0.803000pt}%
\definecolor{currentstroke}{rgb}{0.000000,0.000000,0.000000}%
\pgfsetstrokecolor{currentstroke}%
\pgfsetdash{}{0pt}%
\pgfpathmoveto{\pgfqpoint{0.707523in}{1.168472in}}%
\pgfpathlineto{\pgfqpoint{1.336412in}{1.168472in}}%
\pgfpathlineto{\pgfqpoint{1.336412in}{1.797361in}}%
\pgfpathlineto{\pgfqpoint{0.707523in}{1.797361in}}%
\pgfpathclose%
\pgfusepath{stroke}%
\end{pgfscope}%
\begin{pgfscope}%
\pgfpathrectangle{\pgfqpoint{0.393078in}{0.225139in}}{\pgfqpoint{1.886667in}{1.886667in}} %
\pgfusepath{clip}%
\pgfsetbuttcap%
\pgfsetmiterjoin%
\pgfsetlinewidth{0.803000pt}%
\definecolor{currentstroke}{rgb}{0.000000,0.000000,0.000000}%
\pgfsetstrokecolor{currentstroke}%
\pgfsetdash{}{0pt}%
\pgfpathmoveto{\pgfqpoint{1.650856in}{1.168472in}}%
\pgfpathlineto{\pgfqpoint{1.965301in}{1.168472in}}%
\pgfpathlineto{\pgfqpoint{1.965301in}{1.482917in}}%
\pgfpathlineto{\pgfqpoint{1.650856in}{1.482917in}}%
\pgfpathclose%
\pgfusepath{stroke}%
\end{pgfscope}%
\begin{pgfscope}%
\pgfpathrectangle{\pgfqpoint{0.393078in}{0.225139in}}{\pgfqpoint{1.886667in}{1.886667in}} %
\pgfusepath{clip}%
\pgfsetbuttcap%
\pgfsetmiterjoin%
\pgfsetlinewidth{0.803000pt}%
\definecolor{currentstroke}{rgb}{0.000000,0.000000,0.000000}%
\pgfsetstrokecolor{currentstroke}%
\pgfsetdash{}{0pt}%
\pgfpathmoveto{\pgfqpoint{1.336412in}{0.539583in}}%
\pgfpathcurveto{\pgfqpoint{1.503195in}{0.539583in}}{\pgfqpoint{1.663170in}{0.605847in}}{\pgfqpoint{1.781103in}{0.723781in}}%
\pgfpathcurveto{\pgfqpoint{1.899037in}{0.841714in}}{\pgfqpoint{1.965301in}{1.001689in}}{\pgfqpoint{1.965301in}{1.168472in}}%
\pgfpathcurveto{\pgfqpoint{1.965301in}{1.335255in}}{\pgfqpoint{1.899037in}{1.495230in}}{\pgfqpoint{1.781103in}{1.613164in}}%
\pgfpathcurveto{\pgfqpoint{1.663170in}{1.731097in}}{\pgfqpoint{1.503195in}{1.797361in}}{\pgfqpoint{1.336412in}{1.797361in}}%
\pgfpathcurveto{\pgfqpoint{1.169628in}{1.797361in}}{\pgfqpoint{1.009654in}{1.731097in}}{\pgfqpoint{0.891720in}{1.613164in}}%
\pgfpathcurveto{\pgfqpoint{0.773786in}{1.495230in}}{\pgfqpoint{0.707523in}{1.335255in}}{\pgfqpoint{0.707523in}{1.168472in}}%
\pgfpathcurveto{\pgfqpoint{0.707523in}{1.001689in}}{\pgfqpoint{0.773786in}{0.841714in}}{\pgfqpoint{0.891720in}{0.723781in}}%
\pgfpathcurveto{\pgfqpoint{1.009654in}{0.605847in}}{\pgfqpoint{1.169628in}{0.539583in}}{\pgfqpoint{1.336412in}{0.539583in}}%
\pgfpathclose%
\pgfusepath{stroke}%
\end{pgfscope}%
\begin{pgfscope}%
\pgfpathrectangle{\pgfqpoint{0.393078in}{0.225139in}}{\pgfqpoint{1.886667in}{1.886667in}} %
\pgfusepath{clip}%
\pgfsetbuttcap%
\pgfsetroundjoin%
\definecolor{currentfill}{rgb}{0.121569,0.466667,0.705882}%
\pgfsetfillcolor{currentfill}%
\pgfsetlinewidth{1.003750pt}%
\definecolor{currentstroke}{rgb}{0.121569,0.466667,0.705882}%
\pgfsetstrokecolor{currentstroke}%
\pgfsetdash{}{0pt}%
\pgfsys@defobject{currentmarker}{\pgfqpoint{-0.041667in}{-0.041667in}}{\pgfqpoint{0.041667in}{0.041667in}}{%
\pgfpathmoveto{\pgfqpoint{0.000000in}{-0.041667in}}%
\pgfpathcurveto{\pgfqpoint{0.011050in}{-0.041667in}}{\pgfqpoint{0.021649in}{-0.037276in}}{\pgfqpoint{0.029463in}{-0.029463in}}%
\pgfpathcurveto{\pgfqpoint{0.037276in}{-0.021649in}}{\pgfqpoint{0.041667in}{-0.011050in}}{\pgfqpoint{0.041667in}{0.000000in}}%
\pgfpathcurveto{\pgfqpoint{0.041667in}{0.011050in}}{\pgfqpoint{0.037276in}{0.021649in}}{\pgfqpoint{0.029463in}{0.029463in}}%
\pgfpathcurveto{\pgfqpoint{0.021649in}{0.037276in}}{\pgfqpoint{0.011050in}{0.041667in}}{\pgfqpoint{0.000000in}{0.041667in}}%
\pgfpathcurveto{\pgfqpoint{-0.011050in}{0.041667in}}{\pgfqpoint{-0.021649in}{0.037276in}}{\pgfqpoint{-0.029463in}{0.029463in}}%
\pgfpathcurveto{\pgfqpoint{-0.037276in}{0.021649in}}{\pgfqpoint{-0.041667in}{0.011050in}}{\pgfqpoint{-0.041667in}{0.000000in}}%
\pgfpathcurveto{\pgfqpoint{-0.041667in}{-0.011050in}}{\pgfqpoint{-0.037276in}{-0.021649in}}{\pgfqpoint{-0.029463in}{-0.029463in}}%
\pgfpathcurveto{\pgfqpoint{-0.021649in}{-0.037276in}}{\pgfqpoint{-0.011050in}{-0.041667in}}{\pgfqpoint{0.000000in}{-0.041667in}}%
\pgfpathclose%
\pgfusepath{stroke,fill}%
}%
\begin{pgfscope}%
\pgfsys@transformshift{1.965301in}{1.168472in}%
\pgfsys@useobject{currentmarker}{}%
\end{pgfscope}%
\end{pgfscope}%
\begin{pgfscope}%
\pgfpathrectangle{\pgfqpoint{0.393078in}{0.225139in}}{\pgfqpoint{1.886667in}{1.886667in}} %
\pgfusepath{clip}%
\pgfsetrectcap%
\pgfsetroundjoin%
\pgfsetlinewidth{1.505625pt}%
\definecolor{currentstroke}{rgb}{1.000000,0.498039,0.054902}%
\pgfsetstrokecolor{currentstroke}%
\pgfsetdash{}{0pt}%
\pgfpathmoveto{\pgfqpoint{1.336412in}{1.168472in}}%
\pgfusepath{stroke}%
\end{pgfscope}%
\begin{pgfscope}%
\pgfpathrectangle{\pgfqpoint{0.393078in}{0.225139in}}{\pgfqpoint{1.886667in}{1.886667in}} %
\pgfusepath{clip}%
\pgfsetbuttcap%
\pgfsetroundjoin%
\definecolor{currentfill}{rgb}{1.000000,0.498039,0.054902}%
\pgfsetfillcolor{currentfill}%
\pgfsetlinewidth{1.003750pt}%
\definecolor{currentstroke}{rgb}{1.000000,0.498039,0.054902}%
\pgfsetstrokecolor{currentstroke}%
\pgfsetdash{}{0pt}%
\pgfsys@defobject{currentmarker}{\pgfqpoint{-0.041667in}{-0.041667in}}{\pgfqpoint{0.041667in}{0.041667in}}{%
\pgfpathmoveto{\pgfqpoint{0.000000in}{-0.041667in}}%
\pgfpathcurveto{\pgfqpoint{0.011050in}{-0.041667in}}{\pgfqpoint{0.021649in}{-0.037276in}}{\pgfqpoint{0.029463in}{-0.029463in}}%
\pgfpathcurveto{\pgfqpoint{0.037276in}{-0.021649in}}{\pgfqpoint{0.041667in}{-0.011050in}}{\pgfqpoint{0.041667in}{0.000000in}}%
\pgfpathcurveto{\pgfqpoint{0.041667in}{0.011050in}}{\pgfqpoint{0.037276in}{0.021649in}}{\pgfqpoint{0.029463in}{0.029463in}}%
\pgfpathcurveto{\pgfqpoint{0.021649in}{0.037276in}}{\pgfqpoint{0.011050in}{0.041667in}}{\pgfqpoint{0.000000in}{0.041667in}}%
\pgfpathcurveto{\pgfqpoint{-0.011050in}{0.041667in}}{\pgfqpoint{-0.021649in}{0.037276in}}{\pgfqpoint{-0.029463in}{0.029463in}}%
\pgfpathcurveto{\pgfqpoint{-0.037276in}{0.021649in}}{\pgfqpoint{-0.041667in}{0.011050in}}{\pgfqpoint{-0.041667in}{0.000000in}}%
\pgfpathcurveto{\pgfqpoint{-0.041667in}{-0.011050in}}{\pgfqpoint{-0.037276in}{-0.021649in}}{\pgfqpoint{-0.029463in}{-0.029463in}}%
\pgfpathcurveto{\pgfqpoint{-0.021649in}{-0.037276in}}{\pgfqpoint{-0.011050in}{-0.041667in}}{\pgfqpoint{0.000000in}{-0.041667in}}%
\pgfpathclose%
\pgfusepath{stroke,fill}%
}%
\begin{pgfscope}%
\pgfsys@transformshift{1.336412in}{1.168472in}%
\pgfsys@useobject{currentmarker}{}%
\end{pgfscope}%
\end{pgfscope}%
\begin{pgfscope}%
\pgfsetrectcap%
\pgfsetmiterjoin%
\pgfsetlinewidth{0.803000pt}%
\definecolor{currentstroke}{rgb}{0.000000,0.000000,0.000000}%
\pgfsetstrokecolor{currentstroke}%
\pgfsetdash{}{0pt}%
\pgfpathmoveto{\pgfqpoint{0.393078in}{0.225139in}}%
\pgfpathlineto{\pgfqpoint{0.393078in}{2.111806in}}%
\pgfusepath{stroke}%
\end{pgfscope}%
\begin{pgfscope}%
\pgfsetrectcap%
\pgfsetmiterjoin%
\pgfsetlinewidth{0.803000pt}%
\definecolor{currentstroke}{rgb}{0.000000,0.000000,0.000000}%
\pgfsetstrokecolor{currentstroke}%
\pgfsetdash{}{0pt}%
\pgfpathmoveto{\pgfqpoint{2.279745in}{0.225139in}}%
\pgfpathlineto{\pgfqpoint{2.279745in}{2.111806in}}%
\pgfusepath{stroke}%
\end{pgfscope}%
\begin{pgfscope}%
\pgfsetrectcap%
\pgfsetmiterjoin%
\pgfsetlinewidth{0.803000pt}%
\definecolor{currentstroke}{rgb}{0.000000,0.000000,0.000000}%
\pgfsetstrokecolor{currentstroke}%
\pgfsetdash{}{0pt}%
\pgfpathmoveto{\pgfqpoint{0.393078in}{0.225139in}}%
\pgfpathlineto{\pgfqpoint{2.279745in}{0.225139in}}%
\pgfusepath{stroke}%
\end{pgfscope}%
\begin{pgfscope}%
\pgfsetrectcap%
\pgfsetmiterjoin%
\pgfsetlinewidth{0.803000pt}%
\definecolor{currentstroke}{rgb}{0.000000,0.000000,0.000000}%
\pgfsetstrokecolor{currentstroke}%
\pgfsetdash{}{0pt}%
\pgfpathmoveto{\pgfqpoint{0.393078in}{2.111806in}}%
\pgfpathlineto{\pgfqpoint{2.279745in}{2.111806in}}%
\pgfusepath{stroke}%
\end{pgfscope}%
\begin{pgfscope}%
\pgfsetroundcap%
\pgfsetroundjoin%
\pgfsetlinewidth{0.803000pt}%
\definecolor{currentstroke}{rgb}{0.000000,0.000000,0.000000}%
\pgfsetstrokecolor{currentstroke}%
\pgfsetdash{}{0pt}%
\pgfpathmoveto{\pgfqpoint{1.374545in}{1.181183in}}%
\pgfpathquadraticcurveto{\pgfqpoint{1.527774in}{1.232260in}}{\pgfqpoint{1.692787in}{1.287264in}}%
\pgfusepath{stroke}%
\end{pgfscope}%
\begin{pgfscope}%
\pgfsetroundcap%
\pgfsetroundjoin%
\definecolor{currentfill}{rgb}{0.000000,0.000000,0.000000}%
\pgfsetfillcolor{currentfill}%
\pgfsetlinewidth{0.803000pt}%
\definecolor{currentstroke}{rgb}{0.000000,0.000000,0.000000}%
\pgfsetstrokecolor{currentstroke}%
\pgfsetdash{}{0pt}%
\pgfpathmoveto{\pgfqpoint{1.436034in}{1.172399in}}%
\pgfpathlineto{\pgfqpoint{1.374545in}{1.181183in}}%
\pgfpathlineto{\pgfqpoint{1.418466in}{1.225104in}}%
\pgfpathlineto{\pgfqpoint{1.436034in}{1.172399in}}%
\pgfpathclose%
\pgfusepath{stroke,fill}%
\end{pgfscope}%
\begin{pgfscope}%
\pgfsetroundcap%
\pgfsetroundjoin%
\pgfsetlinewidth{0.803000pt}%
\definecolor{currentstroke}{rgb}{0.000000,0.000000,0.000000}%
\pgfsetstrokecolor{currentstroke}%
\pgfsetdash{}{0pt}%
\pgfpathmoveto{\pgfqpoint{1.925941in}{1.207831in}}%
\pgfpathquadraticcurveto{\pgfqpoint{1.941229in}{1.192544in}}{\pgfqpoint{1.965301in}{1.168472in}}%
\pgfusepath{stroke}%
\end{pgfscope}%
\begin{pgfscope}%
\pgfsetroundcap%
\pgfsetroundjoin%
\definecolor{currentfill}{rgb}{0.000000,0.000000,0.000000}%
\pgfsetfillcolor{currentfill}%
\pgfsetlinewidth{0.803000pt}%
\definecolor{currentstroke}{rgb}{0.000000,0.000000,0.000000}%
\pgfsetstrokecolor{currentstroke}%
\pgfsetdash{}{0pt}%
\pgfpathmoveto{\pgfqpoint{1.945583in}{1.148906in}}%
\pgfpathlineto{\pgfqpoint{1.925941in}{1.207831in}}%
\pgfpathlineto{\pgfqpoint{1.984867in}{1.188189in}}%
\pgfpathlineto{\pgfqpoint{1.945583in}{1.148906in}}%
\pgfpathclose%
\pgfusepath{stroke,fill}%
\end{pgfscope}%
\begin{pgfscope}%
\pgftext[x=1.808078in,y=1.325694in,,]{\rmfamily\fontsize{10.000000}{12.000000}\selectfont \(\displaystyle m_{8}\)}%
\end{pgfscope}%
\begin{pgfscope}%
\pgfsetroundcap%
\pgfsetroundjoin%
\pgfsetlinewidth{0.803000pt}%
\definecolor{currentstroke}{rgb}{0.000000,0.000000,0.000000}%
\pgfsetstrokecolor{currentstroke}%
\pgfsetdash{}{0pt}%
\pgfpathmoveto{\pgfqpoint{1.793142in}{0.997754in}}%
\pgfpathquadraticcurveto{\pgfqpoint{1.869357in}{1.073332in}}{\pgfqpoint{1.936751in}{1.140162in}}%
\pgfusepath{stroke}%
\end{pgfscope}%
\begin{pgfscope}%
\pgfsetroundcap%
\pgfsetroundjoin%
\definecolor{currentfill}{rgb}{0.000000,0.000000,0.000000}%
\pgfsetfillcolor{currentfill}%
\pgfsetlinewidth{0.803000pt}%
\definecolor{currentstroke}{rgb}{0.000000,0.000000,0.000000}%
\pgfsetstrokecolor{currentstroke}%
\pgfsetdash{}{0pt}%
\pgfpathmoveto{\pgfqpoint{1.877744in}{1.120768in}}%
\pgfpathlineto{\pgfqpoint{1.936751in}{1.140162in}}%
\pgfpathlineto{\pgfqpoint{1.916862in}{1.081319in}}%
\pgfpathlineto{\pgfqpoint{1.877744in}{1.120768in}}%
\pgfpathclose%
\pgfusepath{stroke,fill}%
\end{pgfscope}%
\begin{pgfscope}%
\pgftext[x=1.493634in,y=0.854028in,left,base]{\rmfamily\fontsize{10.000000}{12.000000}\selectfont source}%
\end{pgfscope}%
\begin{pgfscope}%
\pgftext[x=1.147745in,y=0.854028in,,]{\rmfamily\fontsize{10.000000}{12.000000}\selectfont \(\displaystyle \ell_{8}\)}%
\end{pgfscope}%
\begin{pgfscope}%
\pgfpathrectangle{\pgfqpoint{2.372900in}{0.225139in}}{\pgfqpoint{0.094333in}{1.886667in}} %
\pgfusepath{clip}%
\pgfsetbuttcap%
\pgfsetmiterjoin%
\definecolor{currentfill}{rgb}{1.000000,1.000000,1.000000}%
\pgfsetfillcolor{currentfill}%
\pgfsetlinewidth{0.010037pt}%
\definecolor{currentstroke}{rgb}{1.000000,1.000000,1.000000}%
\pgfsetstrokecolor{currentstroke}%
\pgfsetdash{}{0pt}%
\pgfpathmoveto{\pgfqpoint{2.372900in}{0.225139in}}%
\pgfpathlineto{\pgfqpoint{2.372900in}{0.232509in}}%
\pgfpathlineto{\pgfqpoint{2.372900in}{2.104436in}}%
\pgfpathlineto{\pgfqpoint{2.372900in}{2.111806in}}%
\pgfpathlineto{\pgfqpoint{2.467233in}{2.111806in}}%
\pgfpathlineto{\pgfqpoint{2.467233in}{2.104436in}}%
\pgfpathlineto{\pgfqpoint{2.467233in}{0.232509in}}%
\pgfpathlineto{\pgfqpoint{2.467233in}{0.225139in}}%
\pgfpathclose%
\pgfusepath{stroke,fill}%
\end{pgfscope}%
\begin{pgfscope}%
\pgfsys@transformshift{2.373333in}{0.225000in}%
\pgftext[left,bottom]{\pgfimage[interpolate=true,width=0.093333in,height=1.886667in]{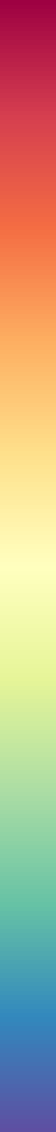}}%
\end{pgfscope}%
\begin{pgfscope}%
\pgfsetbuttcap%
\pgfsetroundjoin%
\definecolor{currentfill}{rgb}{0.000000,0.000000,0.000000}%
\pgfsetfillcolor{currentfill}%
\pgfsetlinewidth{0.803000pt}%
\definecolor{currentstroke}{rgb}{0.000000,0.000000,0.000000}%
\pgfsetstrokecolor{currentstroke}%
\pgfsetdash{}{0pt}%
\pgfsys@defobject{currentmarker}{\pgfqpoint{0.000000in}{0.000000in}}{\pgfqpoint{0.048611in}{0.000000in}}{%
\pgfpathmoveto{\pgfqpoint{0.000000in}{0.000000in}}%
\pgfpathlineto{\pgfqpoint{0.048611in}{0.000000in}}%
\pgfusepath{stroke,fill}%
}%
\begin{pgfscope}%
\pgfsys@transformshift{2.467233in}{0.225139in}%
\pgfsys@useobject{currentmarker}{}%
\end{pgfscope}%
\end{pgfscope}%
\begin{pgfscope}%
\pgftext[x=2.564455in,y=0.176944in,left,base]{\rmfamily\fontsize{10.000000}{12.000000}\selectfont −8}%
\end{pgfscope}%
\begin{pgfscope}%
\pgfsetbuttcap%
\pgfsetroundjoin%
\definecolor{currentfill}{rgb}{0.000000,0.000000,0.000000}%
\pgfsetfillcolor{currentfill}%
\pgfsetlinewidth{0.803000pt}%
\definecolor{currentstroke}{rgb}{0.000000,0.000000,0.000000}%
\pgfsetstrokecolor{currentstroke}%
\pgfsetdash{}{0pt}%
\pgfsys@defobject{currentmarker}{\pgfqpoint{0.000000in}{0.000000in}}{\pgfqpoint{0.048611in}{0.000000in}}{%
\pgfpathmoveto{\pgfqpoint{0.000000in}{0.000000in}}%
\pgfpathlineto{\pgfqpoint{0.048611in}{0.000000in}}%
\pgfusepath{stroke,fill}%
}%
\begin{pgfscope}%
\pgfsys@transformshift{2.467233in}{0.460972in}%
\pgfsys@useobject{currentmarker}{}%
\end{pgfscope}%
\end{pgfscope}%
\begin{pgfscope}%
\pgftext[x=2.564455in,y=0.412778in,left,base]{\rmfamily\fontsize{10.000000}{12.000000}\selectfont −7}%
\end{pgfscope}%
\begin{pgfscope}%
\pgfsetbuttcap%
\pgfsetroundjoin%
\definecolor{currentfill}{rgb}{0.000000,0.000000,0.000000}%
\pgfsetfillcolor{currentfill}%
\pgfsetlinewidth{0.803000pt}%
\definecolor{currentstroke}{rgb}{0.000000,0.000000,0.000000}%
\pgfsetstrokecolor{currentstroke}%
\pgfsetdash{}{0pt}%
\pgfsys@defobject{currentmarker}{\pgfqpoint{0.000000in}{0.000000in}}{\pgfqpoint{0.048611in}{0.000000in}}{%
\pgfpathmoveto{\pgfqpoint{0.000000in}{0.000000in}}%
\pgfpathlineto{\pgfqpoint{0.048611in}{0.000000in}}%
\pgfusepath{stroke,fill}%
}%
\begin{pgfscope}%
\pgfsys@transformshift{2.467233in}{0.696805in}%
\pgfsys@useobject{currentmarker}{}%
\end{pgfscope}%
\end{pgfscope}%
\begin{pgfscope}%
\pgftext[x=2.564455in,y=0.648611in,left,base]{\rmfamily\fontsize{10.000000}{12.000000}\selectfont −6}%
\end{pgfscope}%
\begin{pgfscope}%
\pgfsetbuttcap%
\pgfsetroundjoin%
\definecolor{currentfill}{rgb}{0.000000,0.000000,0.000000}%
\pgfsetfillcolor{currentfill}%
\pgfsetlinewidth{0.803000pt}%
\definecolor{currentstroke}{rgb}{0.000000,0.000000,0.000000}%
\pgfsetstrokecolor{currentstroke}%
\pgfsetdash{}{0pt}%
\pgfsys@defobject{currentmarker}{\pgfqpoint{0.000000in}{0.000000in}}{\pgfqpoint{0.048611in}{0.000000in}}{%
\pgfpathmoveto{\pgfqpoint{0.000000in}{0.000000in}}%
\pgfpathlineto{\pgfqpoint{0.048611in}{0.000000in}}%
\pgfusepath{stroke,fill}%
}%
\begin{pgfscope}%
\pgfsys@transformshift{2.467233in}{0.932639in}%
\pgfsys@useobject{currentmarker}{}%
\end{pgfscope}%
\end{pgfscope}%
\begin{pgfscope}%
\pgftext[x=2.564455in,y=0.884444in,left,base]{\rmfamily\fontsize{10.000000}{12.000000}\selectfont −5}%
\end{pgfscope}%
\begin{pgfscope}%
\pgfsetbuttcap%
\pgfsetroundjoin%
\definecolor{currentfill}{rgb}{0.000000,0.000000,0.000000}%
\pgfsetfillcolor{currentfill}%
\pgfsetlinewidth{0.803000pt}%
\definecolor{currentstroke}{rgb}{0.000000,0.000000,0.000000}%
\pgfsetstrokecolor{currentstroke}%
\pgfsetdash{}{0pt}%
\pgfsys@defobject{currentmarker}{\pgfqpoint{0.000000in}{0.000000in}}{\pgfqpoint{0.048611in}{0.000000in}}{%
\pgfpathmoveto{\pgfqpoint{0.000000in}{0.000000in}}%
\pgfpathlineto{\pgfqpoint{0.048611in}{0.000000in}}%
\pgfusepath{stroke,fill}%
}%
\begin{pgfscope}%
\pgfsys@transformshift{2.467233in}{1.168472in}%
\pgfsys@useobject{currentmarker}{}%
\end{pgfscope}%
\end{pgfscope}%
\begin{pgfscope}%
\pgftext[x=2.564455in,y=1.120278in,left,base]{\rmfamily\fontsize{10.000000}{12.000000}\selectfont −4}%
\end{pgfscope}%
\begin{pgfscope}%
\pgfsetbuttcap%
\pgfsetroundjoin%
\definecolor{currentfill}{rgb}{0.000000,0.000000,0.000000}%
\pgfsetfillcolor{currentfill}%
\pgfsetlinewidth{0.803000pt}%
\definecolor{currentstroke}{rgb}{0.000000,0.000000,0.000000}%
\pgfsetstrokecolor{currentstroke}%
\pgfsetdash{}{0pt}%
\pgfsys@defobject{currentmarker}{\pgfqpoint{0.000000in}{0.000000in}}{\pgfqpoint{0.048611in}{0.000000in}}{%
\pgfpathmoveto{\pgfqpoint{0.000000in}{0.000000in}}%
\pgfpathlineto{\pgfqpoint{0.048611in}{0.000000in}}%
\pgfusepath{stroke,fill}%
}%
\begin{pgfscope}%
\pgfsys@transformshift{2.467233in}{1.404306in}%
\pgfsys@useobject{currentmarker}{}%
\end{pgfscope}%
\end{pgfscope}%
\begin{pgfscope}%
\pgftext[x=2.564455in,y=1.356111in,left,base]{\rmfamily\fontsize{10.000000}{12.000000}\selectfont −3}%
\end{pgfscope}%
\begin{pgfscope}%
\pgfsetbuttcap%
\pgfsetroundjoin%
\definecolor{currentfill}{rgb}{0.000000,0.000000,0.000000}%
\pgfsetfillcolor{currentfill}%
\pgfsetlinewidth{0.803000pt}%
\definecolor{currentstroke}{rgb}{0.000000,0.000000,0.000000}%
\pgfsetstrokecolor{currentstroke}%
\pgfsetdash{}{0pt}%
\pgfsys@defobject{currentmarker}{\pgfqpoint{0.000000in}{0.000000in}}{\pgfqpoint{0.048611in}{0.000000in}}{%
\pgfpathmoveto{\pgfqpoint{0.000000in}{0.000000in}}%
\pgfpathlineto{\pgfqpoint{0.048611in}{0.000000in}}%
\pgfusepath{stroke,fill}%
}%
\begin{pgfscope}%
\pgfsys@transformshift{2.467233in}{1.640139in}%
\pgfsys@useobject{currentmarker}{}%
\end{pgfscope}%
\end{pgfscope}%
\begin{pgfscope}%
\pgftext[x=2.564455in,y=1.591944in,left,base]{\rmfamily\fontsize{10.000000}{12.000000}\selectfont −2}%
\end{pgfscope}%
\begin{pgfscope}%
\pgfsetbuttcap%
\pgfsetroundjoin%
\definecolor{currentfill}{rgb}{0.000000,0.000000,0.000000}%
\pgfsetfillcolor{currentfill}%
\pgfsetlinewidth{0.803000pt}%
\definecolor{currentstroke}{rgb}{0.000000,0.000000,0.000000}%
\pgfsetstrokecolor{currentstroke}%
\pgfsetdash{}{0pt}%
\pgfsys@defobject{currentmarker}{\pgfqpoint{0.000000in}{0.000000in}}{\pgfqpoint{0.048611in}{0.000000in}}{%
\pgfpathmoveto{\pgfqpoint{0.000000in}{0.000000in}}%
\pgfpathlineto{\pgfqpoint{0.048611in}{0.000000in}}%
\pgfusepath{stroke,fill}%
}%
\begin{pgfscope}%
\pgfsys@transformshift{2.467233in}{1.875972in}%
\pgfsys@useobject{currentmarker}{}%
\end{pgfscope}%
\end{pgfscope}%
\begin{pgfscope}%
\pgftext[x=2.564455in,y=1.827778in,left,base]{\rmfamily\fontsize{10.000000}{12.000000}\selectfont −1}%
\end{pgfscope}%
\begin{pgfscope}%
\pgfsetbuttcap%
\pgfsetroundjoin%
\definecolor{currentfill}{rgb}{0.000000,0.000000,0.000000}%
\pgfsetfillcolor{currentfill}%
\pgfsetlinewidth{0.803000pt}%
\definecolor{currentstroke}{rgb}{0.000000,0.000000,0.000000}%
\pgfsetstrokecolor{currentstroke}%
\pgfsetdash{}{0pt}%
\pgfsys@defobject{currentmarker}{\pgfqpoint{0.000000in}{0.000000in}}{\pgfqpoint{0.048611in}{0.000000in}}{%
\pgfpathmoveto{\pgfqpoint{0.000000in}{0.000000in}}%
\pgfpathlineto{\pgfqpoint{0.048611in}{0.000000in}}%
\pgfusepath{stroke,fill}%
}%
\begin{pgfscope}%
\pgfsys@transformshift{2.467233in}{2.111806in}%
\pgfsys@useobject{currentmarker}{}%
\end{pgfscope}%
\end{pgfscope}%
\begin{pgfscope}%
\pgftext[x=2.564455in,y=2.063611in,left,base]{\rmfamily\fontsize{10.000000}{12.000000}\selectfont 0}%
\end{pgfscope}%
\begin{pgfscope}%
\pgftext[x=2.797511in,y=1.168472in,,top,rotate=90.000000]{\rmfamily\fontsize{10.000000}{12.000000}\selectfont \(\displaystyle \log_{10} |\ell_{8,\mathrm{direct}} - \ell_{8,\mathrm{M}(8)}|\)}%
\end{pgfscope}%
\begin{pgfscope}%
\pgfsetbuttcap%
\pgfsetmiterjoin%
\pgfsetlinewidth{0.803000pt}%
\definecolor{currentstroke}{rgb}{0.000000,0.000000,0.000000}%
\pgfsetstrokecolor{currentstroke}%
\pgfsetdash{}{0pt}%
\pgfpathmoveto{\pgfqpoint{2.372900in}{0.225139in}}%
\pgfpathlineto{\pgfqpoint{2.372900in}{0.232509in}}%
\pgfpathlineto{\pgfqpoint{2.372900in}{2.104436in}}%
\pgfpathlineto{\pgfqpoint{2.372900in}{2.111806in}}%
\pgfpathlineto{\pgfqpoint{2.467233in}{2.111806in}}%
\pgfpathlineto{\pgfqpoint{2.467233in}{2.104436in}}%
\pgfpathlineto{\pgfqpoint{2.467233in}{0.232509in}}%
\pgfpathlineto{\pgfqpoint{2.467233in}{0.225139in}}%
\pgfpathclose%
\pgfusepath{stroke}%
\end{pgfscope}%
\end{pgfpicture}%
\makeatother%
\endgroup%

%% file: list2-less-bump.pgf
\begingroup%
\makeatletter%
\begin{pgfpicture}%
\pgfpathrectangle{\pgfpointorigin}{\pgfqpoint{3.080000in}{2.310000in}}%
\pgfusepath{use as bounding box, clip}%
\begin{pgfscope}%
\pgfsetbuttcap%
\pgfsetmiterjoin%
\definecolor{currentfill}{rgb}{1.000000,1.000000,1.000000}%
\pgfsetfillcolor{currentfill}%
\pgfsetlinewidth{0.000000pt}%
\definecolor{currentstroke}{rgb}{1.000000,1.000000,1.000000}%
\pgfsetstrokecolor{currentstroke}%
\pgfsetdash{}{0pt}%
\pgfpathmoveto{\pgfqpoint{0.000000in}{0.000000in}}%
\pgfpathlineto{\pgfqpoint{3.080000in}{0.000000in}}%
\pgfpathlineto{\pgfqpoint{3.080000in}{2.310000in}}%
\pgfpathlineto{\pgfqpoint{0.000000in}{2.310000in}}%
\pgfpathclose%
\pgfusepath{fill}%
\end{pgfscope}%
\begin{pgfscope}%
\pgfsetbuttcap%
\pgfsetmiterjoin%
\definecolor{currentfill}{rgb}{1.000000,1.000000,1.000000}%
\pgfsetfillcolor{currentfill}%
\pgfsetlinewidth{0.000000pt}%
\definecolor{currentstroke}{rgb}{0.000000,0.000000,0.000000}%
\pgfsetstrokecolor{currentstroke}%
\pgfsetstrokeopacity{0.000000}%
\pgfsetdash{}{0pt}%
\pgfpathmoveto{\pgfqpoint{0.392475in}{0.225139in}}%
\pgfpathlineto{\pgfqpoint{2.279142in}{0.225139in}}%
\pgfpathlineto{\pgfqpoint{2.279142in}{2.111806in}}%
\pgfpathlineto{\pgfqpoint{0.392475in}{2.111806in}}%
\pgfpathclose%
\pgfusepath{fill}%
\end{pgfscope}%
\begin{pgfscope}%
\pgfpathrectangle{\pgfqpoint{0.392475in}{0.225139in}}{\pgfqpoint{1.886667in}{1.886667in}} %
\pgfusepath{clip}%
\pgfsys@transformshift{0.392475in}{0.225139in}%
\pgftext[left,bottom]{\pgfimage[interpolate=true,width=1.888333in,height=1.888333in]{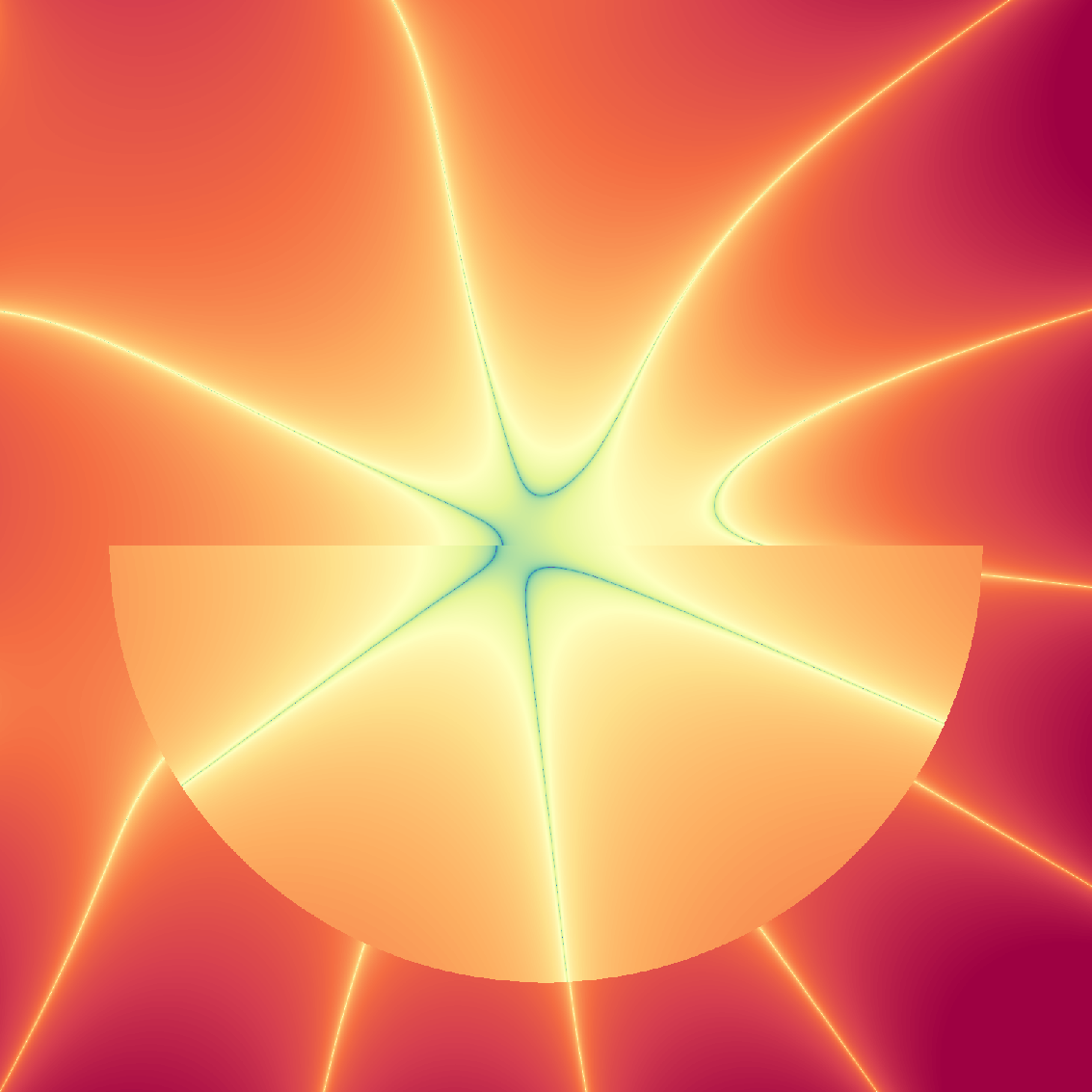}}%
\end{pgfscope}%
\begin{pgfscope}%
\pgfpathrectangle{\pgfqpoint{0.392475in}{0.225139in}}{\pgfqpoint{1.886667in}{1.886667in}} %
\pgfusepath{clip}%
\pgfsetbuttcap%
\pgfsetmiterjoin%
\pgfsetlinewidth{0.803000pt}%
\definecolor{currentstroke}{rgb}{0.000000,0.000000,0.000000}%
\pgfsetstrokecolor{currentstroke}%
\pgfsetdash{}{0pt}%
\pgfpathmoveto{\pgfqpoint{0.958475in}{1.168472in}}%
\pgfpathlineto{\pgfqpoint{1.335808in}{1.168472in}}%
\pgfpathlineto{\pgfqpoint{1.335808in}{1.545806in}}%
\pgfpathlineto{\pgfqpoint{0.958475in}{1.545806in}}%
\pgfpathclose%
\pgfusepath{stroke}%
\end{pgfscope}%
\begin{pgfscope}%
\pgfpathrectangle{\pgfqpoint{0.392475in}{0.225139in}}{\pgfqpoint{1.886667in}{1.886667in}} %
\pgfusepath{clip}%
\pgfsetbuttcap%
\pgfsetmiterjoin%
\pgfsetlinewidth{0.803000pt}%
\definecolor{currentstroke}{rgb}{0.000000,0.000000,0.000000}%
\pgfsetstrokecolor{currentstroke}%
\pgfsetdash{}{0pt}%
\pgfpathmoveto{\pgfqpoint{1.713142in}{1.168472in}}%
\pgfpathlineto{\pgfqpoint{2.090475in}{1.168472in}}%
\pgfpathlineto{\pgfqpoint{2.090475in}{1.545806in}}%
\pgfpathlineto{\pgfqpoint{1.713142in}{1.545806in}}%
\pgfpathclose%
\pgfusepath{stroke}%
\end{pgfscope}%
\begin{pgfscope}%
\pgfpathrectangle{\pgfqpoint{0.392475in}{0.225139in}}{\pgfqpoint{1.886667in}{1.886667in}} %
\pgfusepath{clip}%
\pgfsetbuttcap%
\pgfsetmiterjoin%
\pgfsetlinewidth{0.803000pt}%
\definecolor{currentstroke}{rgb}{0.000000,0.000000,0.000000}%
\pgfsetstrokecolor{currentstroke}%
\pgfsetdash{}{0pt}%
\pgfpathmoveto{\pgfqpoint{1.335808in}{0.413805in}}%
\pgfpathcurveto{\pgfqpoint{1.535948in}{0.413805in}}{\pgfqpoint{1.727918in}{0.493322in}}{\pgfqpoint{1.869438in}{0.634842in}}%
\pgfpathcurveto{\pgfqpoint{2.010959in}{0.776363in}}{\pgfqpoint{2.090475in}{0.968332in}}{\pgfqpoint{2.090475in}{1.168472in}}%
\pgfpathcurveto{\pgfqpoint{2.090475in}{1.368612in}}{\pgfqpoint{2.010959in}{1.560582in}}{\pgfqpoint{1.869438in}{1.702102in}}%
\pgfpathcurveto{\pgfqpoint{1.727918in}{1.843622in}}{\pgfqpoint{1.535948in}{1.923139in}}{\pgfqpoint{1.335808in}{1.923139in}}%
\pgfpathcurveto{\pgfqpoint{1.135668in}{1.923139in}}{\pgfqpoint{0.943699in}{1.843622in}}{\pgfqpoint{0.802178in}{1.702102in}}%
\pgfpathcurveto{\pgfqpoint{0.660658in}{1.560582in}}{\pgfqpoint{0.581142in}{1.368612in}}{\pgfqpoint{0.581142in}{1.168472in}}%
\pgfpathcurveto{\pgfqpoint{0.581142in}{0.968332in}}{\pgfqpoint{0.660658in}{0.776363in}}{\pgfqpoint{0.802178in}{0.634842in}}%
\pgfpathcurveto{\pgfqpoint{0.943699in}{0.493322in}}{\pgfqpoint{1.135668in}{0.413805in}}{\pgfqpoint{1.335808in}{0.413805in}}%
\pgfpathclose%
\pgfusepath{stroke}%
\end{pgfscope}%
\begin{pgfscope}%
\pgfpathrectangle{\pgfqpoint{0.392475in}{0.225139in}}{\pgfqpoint{1.886667in}{1.886667in}} %
\pgfusepath{clip}%
\pgfsetbuttcap%
\pgfsetroundjoin%
\definecolor{currentfill}{rgb}{0.121569,0.466667,0.705882}%
\pgfsetfillcolor{currentfill}%
\pgfsetlinewidth{1.003750pt}%
\definecolor{currentstroke}{rgb}{0.121569,0.466667,0.705882}%
\pgfsetstrokecolor{currentstroke}%
\pgfsetdash{}{0pt}%
\pgfsys@defobject{currentmarker}{\pgfqpoint{-0.041667in}{-0.041667in}}{\pgfqpoint{0.041667in}{0.041667in}}{%
\pgfpathmoveto{\pgfqpoint{0.000000in}{-0.041667in}}%
\pgfpathcurveto{\pgfqpoint{0.011050in}{-0.041667in}}{\pgfqpoint{0.021649in}{-0.037276in}}{\pgfqpoint{0.029463in}{-0.029463in}}%
\pgfpathcurveto{\pgfqpoint{0.037276in}{-0.021649in}}{\pgfqpoint{0.041667in}{-0.011050in}}{\pgfqpoint{0.041667in}{0.000000in}}%
\pgfpathcurveto{\pgfqpoint{0.041667in}{0.011050in}}{\pgfqpoint{0.037276in}{0.021649in}}{\pgfqpoint{0.029463in}{0.029463in}}%
\pgfpathcurveto{\pgfqpoint{0.021649in}{0.037276in}}{\pgfqpoint{0.011050in}{0.041667in}}{\pgfqpoint{0.000000in}{0.041667in}}%
\pgfpathcurveto{\pgfqpoint{-0.011050in}{0.041667in}}{\pgfqpoint{-0.021649in}{0.037276in}}{\pgfqpoint{-0.029463in}{0.029463in}}%
\pgfpathcurveto{\pgfqpoint{-0.037276in}{0.021649in}}{\pgfqpoint{-0.041667in}{0.011050in}}{\pgfqpoint{-0.041667in}{0.000000in}}%
\pgfpathcurveto{\pgfqpoint{-0.041667in}{-0.011050in}}{\pgfqpoint{-0.037276in}{-0.021649in}}{\pgfqpoint{-0.029463in}{-0.029463in}}%
\pgfpathcurveto{\pgfqpoint{-0.021649in}{-0.037276in}}{\pgfqpoint{-0.011050in}{-0.041667in}}{\pgfqpoint{0.000000in}{-0.041667in}}%
\pgfpathclose%
\pgfusepath{stroke,fill}%
}%
\begin{pgfscope}%
\pgfsys@transformshift{2.090475in}{1.168472in}%
\pgfsys@useobject{currentmarker}{}%
\end{pgfscope}%
\end{pgfscope}%
\begin{pgfscope}%
\pgfpathrectangle{\pgfqpoint{0.392475in}{0.225139in}}{\pgfqpoint{1.886667in}{1.886667in}} %
\pgfusepath{clip}%
\pgfsetrectcap%
\pgfsetroundjoin%
\pgfsetlinewidth{1.505625pt}%
\definecolor{currentstroke}{rgb}{1.000000,0.498039,0.054902}%
\pgfsetstrokecolor{currentstroke}%
\pgfsetdash{}{0pt}%
\pgfpathmoveto{\pgfqpoint{1.335808in}{1.168472in}}%
\pgfusepath{stroke}%
\end{pgfscope}%
\begin{pgfscope}%
\pgfpathrectangle{\pgfqpoint{0.392475in}{0.225139in}}{\pgfqpoint{1.886667in}{1.886667in}} %
\pgfusepath{clip}%
\pgfsetbuttcap%
\pgfsetroundjoin%
\definecolor{currentfill}{rgb}{1.000000,0.498039,0.054902}%
\pgfsetfillcolor{currentfill}%
\pgfsetlinewidth{1.003750pt}%
\definecolor{currentstroke}{rgb}{1.000000,0.498039,0.054902}%
\pgfsetstrokecolor{currentstroke}%
\pgfsetdash{}{0pt}%
\pgfsys@defobject{currentmarker}{\pgfqpoint{-0.041667in}{-0.041667in}}{\pgfqpoint{0.041667in}{0.041667in}}{%
\pgfpathmoveto{\pgfqpoint{0.000000in}{-0.041667in}}%
\pgfpathcurveto{\pgfqpoint{0.011050in}{-0.041667in}}{\pgfqpoint{0.021649in}{-0.037276in}}{\pgfqpoint{0.029463in}{-0.029463in}}%
\pgfpathcurveto{\pgfqpoint{0.037276in}{-0.021649in}}{\pgfqpoint{0.041667in}{-0.011050in}}{\pgfqpoint{0.041667in}{0.000000in}}%
\pgfpathcurveto{\pgfqpoint{0.041667in}{0.011050in}}{\pgfqpoint{0.037276in}{0.021649in}}{\pgfqpoint{0.029463in}{0.029463in}}%
\pgfpathcurveto{\pgfqpoint{0.021649in}{0.037276in}}{\pgfqpoint{0.011050in}{0.041667in}}{\pgfqpoint{0.000000in}{0.041667in}}%
\pgfpathcurveto{\pgfqpoint{-0.011050in}{0.041667in}}{\pgfqpoint{-0.021649in}{0.037276in}}{\pgfqpoint{-0.029463in}{0.029463in}}%
\pgfpathcurveto{\pgfqpoint{-0.037276in}{0.021649in}}{\pgfqpoint{-0.041667in}{0.011050in}}{\pgfqpoint{-0.041667in}{0.000000in}}%
\pgfpathcurveto{\pgfqpoint{-0.041667in}{-0.011050in}}{\pgfqpoint{-0.037276in}{-0.021649in}}{\pgfqpoint{-0.029463in}{-0.029463in}}%
\pgfpathcurveto{\pgfqpoint{-0.021649in}{-0.037276in}}{\pgfqpoint{-0.011050in}{-0.041667in}}{\pgfqpoint{0.000000in}{-0.041667in}}%
\pgfpathclose%
\pgfusepath{stroke,fill}%
}%
\begin{pgfscope}%
\pgfsys@transformshift{1.335808in}{1.168472in}%
\pgfsys@useobject{currentmarker}{}%
\end{pgfscope}%
\end{pgfscope}%
\begin{pgfscope}%
\pgfsetrectcap%
\pgfsetmiterjoin%
\pgfsetlinewidth{0.803000pt}%
\definecolor{currentstroke}{rgb}{0.000000,0.000000,0.000000}%
\pgfsetstrokecolor{currentstroke}%
\pgfsetdash{}{0pt}%
\pgfpathmoveto{\pgfqpoint{0.392475in}{0.225139in}}%
\pgfpathlineto{\pgfqpoint{0.392475in}{2.111806in}}%
\pgfusepath{stroke}%
\end{pgfscope}%
\begin{pgfscope}%
\pgfsetrectcap%
\pgfsetmiterjoin%
\pgfsetlinewidth{0.803000pt}%
\definecolor{currentstroke}{rgb}{0.000000,0.000000,0.000000}%
\pgfsetstrokecolor{currentstroke}%
\pgfsetdash{}{0pt}%
\pgfpathmoveto{\pgfqpoint{2.279142in}{0.225139in}}%
\pgfpathlineto{\pgfqpoint{2.279142in}{2.111806in}}%
\pgfusepath{stroke}%
\end{pgfscope}%
\begin{pgfscope}%
\pgfsetrectcap%
\pgfsetmiterjoin%
\pgfsetlinewidth{0.803000pt}%
\definecolor{currentstroke}{rgb}{0.000000,0.000000,0.000000}%
\pgfsetstrokecolor{currentstroke}%
\pgfsetdash{}{0pt}%
\pgfpathmoveto{\pgfqpoint{0.392475in}{0.225139in}}%
\pgfpathlineto{\pgfqpoint{2.279142in}{0.225139in}}%
\pgfusepath{stroke}%
\end{pgfscope}%
\begin{pgfscope}%
\pgfsetrectcap%
\pgfsetmiterjoin%
\pgfsetlinewidth{0.803000pt}%
\definecolor{currentstroke}{rgb}{0.000000,0.000000,0.000000}%
\pgfsetstrokecolor{currentstroke}%
\pgfsetdash{}{0pt}%
\pgfpathmoveto{\pgfqpoint{0.392475in}{2.111806in}}%
\pgfpathlineto{\pgfqpoint{2.279142in}{2.111806in}}%
\pgfusepath{stroke}%
\end{pgfscope}%
\begin{pgfscope}%
\pgfsetroundcap%
\pgfsetroundjoin%
\pgfsetlinewidth{0.803000pt}%
\definecolor{currentstroke}{rgb}{0.000000,0.000000,0.000000}%
\pgfsetstrokecolor{currentstroke}%
\pgfsetdash{}{0pt}%
\pgfpathmoveto{\pgfqpoint{1.307385in}{1.196896in}}%
\pgfpathquadraticcurveto{\pgfqpoint{1.274621in}{1.229659in}}{\pgfqpoint{1.233074in}{1.271207in}}%
\pgfusepath{stroke}%
\end{pgfscope}%
\begin{pgfscope}%
\pgfsetroundcap%
\pgfsetroundjoin%
\definecolor{currentfill}{rgb}{0.000000,0.000000,0.000000}%
\pgfsetfillcolor{currentfill}%
\pgfsetlinewidth{0.803000pt}%
\definecolor{currentstroke}{rgb}{0.000000,0.000000,0.000000}%
\pgfsetstrokecolor{currentstroke}%
\pgfsetdash{}{0pt}%
\pgfpathmoveto{\pgfqpoint{1.287743in}{1.255821in}}%
\pgfpathlineto{\pgfqpoint{1.307385in}{1.196896in}}%
\pgfpathlineto{\pgfqpoint{1.248459in}{1.216537in}}%
\pgfpathlineto{\pgfqpoint{1.287743in}{1.255821in}}%
\pgfpathclose%
\pgfusepath{stroke,fill}%
\end{pgfscope}%
\begin{pgfscope}%
\pgfsetroundcap%
\pgfsetroundjoin%
\pgfsetlinewidth{0.803000pt}%
\definecolor{currentstroke}{rgb}{0.000000,0.000000,0.000000}%
\pgfsetstrokecolor{currentstroke}%
\pgfsetdash{}{0pt}%
\pgfpathmoveto{\pgfqpoint{1.275213in}{1.357139in}}%
\pgfpathquadraticcurveto{\pgfqpoint{1.521536in}{1.357139in}}{\pgfqpoint{1.780281in}{1.357139in}}%
\pgfusepath{stroke}%
\end{pgfscope}%
\begin{pgfscope}%
\pgfsetroundcap%
\pgfsetroundjoin%
\definecolor{currentfill}{rgb}{0.000000,0.000000,0.000000}%
\pgfsetfillcolor{currentfill}%
\pgfsetlinewidth{0.803000pt}%
\definecolor{currentstroke}{rgb}{0.000000,0.000000,0.000000}%
\pgfsetstrokecolor{currentstroke}%
\pgfsetdash{}{0pt}%
\pgfpathmoveto{\pgfqpoint{1.330769in}{1.329361in}}%
\pgfpathlineto{\pgfqpoint{1.275213in}{1.357139in}}%
\pgfpathlineto{\pgfqpoint{1.330769in}{1.384917in}}%
\pgfpathlineto{\pgfqpoint{1.330769in}{1.329361in}}%
\pgfpathclose%
\pgfusepath{stroke,fill}%
\end{pgfscope}%
\begin{pgfscope}%
\pgftext[x=1.147142in,y=1.357139in,,]{\rmfamily\fontsize{10.000000}{12.000000}\selectfont \(\displaystyle \ell_{8}\)}%
\end{pgfscope}%
\begin{pgfscope}%
\pgfsetroundcap%
\pgfsetroundjoin%
\pgfsetlinewidth{0.803000pt}%
\definecolor{currentstroke}{rgb}{0.000000,0.000000,0.000000}%
\pgfsetstrokecolor{currentstroke}%
\pgfsetdash{}{0pt}%
\pgfpathmoveto{\pgfqpoint{2.019680in}{1.239267in}}%
\pgfpathquadraticcurveto{\pgfqpoint{2.050686in}{1.208262in}}{\pgfqpoint{2.090475in}{1.168472in}}%
\pgfusepath{stroke}%
\end{pgfscope}%
\begin{pgfscope}%
\pgfsetroundcap%
\pgfsetroundjoin%
\definecolor{currentfill}{rgb}{0.000000,0.000000,0.000000}%
\pgfsetfillcolor{currentfill}%
\pgfsetlinewidth{0.803000pt}%
\definecolor{currentstroke}{rgb}{0.000000,0.000000,0.000000}%
\pgfsetstrokecolor{currentstroke}%
\pgfsetdash{}{0pt}%
\pgfpathmoveto{\pgfqpoint{2.039322in}{1.180341in}}%
\pgfpathlineto{\pgfqpoint{2.019680in}{1.239267in}}%
\pgfpathlineto{\pgfqpoint{2.078606in}{1.219625in}}%
\pgfpathlineto{\pgfqpoint{2.039322in}{1.180341in}}%
\pgfpathclose%
\pgfusepath{stroke,fill}%
\end{pgfscope}%
\begin{pgfscope}%
\pgftext[x=1.901808in,y=1.357139in,,]{\rmfamily\fontsize{10.000000}{12.000000}\selectfont \(\displaystyle m_{8}\)}%
\end{pgfscope}%
\begin{pgfscope}%
\pgfsetroundcap%
\pgfsetroundjoin%
\pgfsetlinewidth{0.803000pt}%
\definecolor{currentstroke}{rgb}{0.000000,0.000000,0.000000}%
\pgfsetstrokecolor{currentstroke}%
\pgfsetdash{}{0pt}%
\pgfpathmoveto{\pgfqpoint{1.832873in}{0.934005in}}%
\pgfpathquadraticcurveto{\pgfqpoint{1.951401in}{1.041888in}}{\pgfqpoint{2.060742in}{1.141410in}}%
\pgfusepath{stroke}%
\end{pgfscope}%
\begin{pgfscope}%
\pgfsetroundcap%
\pgfsetroundjoin%
\definecolor{currentfill}{rgb}{0.000000,0.000000,0.000000}%
\pgfsetfillcolor{currentfill}%
\pgfsetlinewidth{0.803000pt}%
\definecolor{currentstroke}{rgb}{0.000000,0.000000,0.000000}%
\pgfsetstrokecolor{currentstroke}%
\pgfsetdash{}{0pt}%
\pgfpathmoveto{\pgfqpoint{2.000959in}{1.124557in}}%
\pgfpathlineto{\pgfqpoint{2.060742in}{1.141410in}}%
\pgfpathlineto{\pgfqpoint{2.038355in}{1.083472in}}%
\pgfpathlineto{\pgfqpoint{2.000959in}{1.124557in}}%
\pgfpathclose%
\pgfusepath{stroke,fill}%
\end{pgfscope}%
\begin{pgfscope}%
\pgftext[x=1.524475in,y=0.791139in,left,base]{\rmfamily\fontsize{10.000000}{12.000000}\selectfont source}%
\end{pgfscope}%
\begin{pgfscope}%
\pgftext[x=0.769808in,y=0.942072in,,]{\rmfamily\fontsize{10.000000}{12.000000}\selectfont \(\displaystyle \ell_{3}\)}%
\end{pgfscope}%
\begin{pgfscope}%
\pgftext[x=0.769808in,y=1.394872in,,]{\rmfamily\fontsize{10.000000}{12.000000}\selectfont \(\displaystyle \ell_{8}\)}%
\end{pgfscope}%
\begin{pgfscope}%
\pgfpathrectangle{\pgfqpoint{2.372270in}{0.225139in}}{\pgfqpoint{0.094333in}{1.886667in}} %
\pgfusepath{clip}%
\pgfsetbuttcap%
\pgfsetmiterjoin%
\definecolor{currentfill}{rgb}{1.000000,1.000000,1.000000}%
\pgfsetfillcolor{currentfill}%
\pgfsetlinewidth{0.010037pt}%
\definecolor{currentstroke}{rgb}{1.000000,1.000000,1.000000}%
\pgfsetstrokecolor{currentstroke}%
\pgfsetdash{}{0pt}%
\pgfpathmoveto{\pgfqpoint{2.372270in}{0.225139in}}%
\pgfpathlineto{\pgfqpoint{2.372270in}{0.232509in}}%
\pgfpathlineto{\pgfqpoint{2.372270in}{2.104436in}}%
\pgfpathlineto{\pgfqpoint{2.372270in}{2.111806in}}%
\pgfpathlineto{\pgfqpoint{2.466603in}{2.111806in}}%
\pgfpathlineto{\pgfqpoint{2.466603in}{2.104436in}}%
\pgfpathlineto{\pgfqpoint{2.466603in}{0.232509in}}%
\pgfpathlineto{\pgfqpoint{2.466603in}{0.225139in}}%
\pgfpathclose%
\pgfusepath{stroke,fill}%
\end{pgfscope}%
\begin{pgfscope}%
\pgfsys@transformshift{2.371667in}{0.225000in}%
\pgftext[left,bottom]{\pgfimage[interpolate=true,width=0.095000in,height=1.886667in]{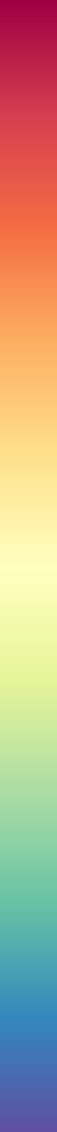}}%
\end{pgfscope}%
\begin{pgfscope}%
\pgfsetbuttcap%
\pgfsetroundjoin%
\definecolor{currentfill}{rgb}{0.000000,0.000000,0.000000}%
\pgfsetfillcolor{currentfill}%
\pgfsetlinewidth{0.803000pt}%
\definecolor{currentstroke}{rgb}{0.000000,0.000000,0.000000}%
\pgfsetstrokecolor{currentstroke}%
\pgfsetdash{}{0pt}%
\pgfsys@defobject{currentmarker}{\pgfqpoint{0.000000in}{0.000000in}}{\pgfqpoint{0.048611in}{0.000000in}}{%
\pgfpathmoveto{\pgfqpoint{0.000000in}{0.000000in}}%
\pgfpathlineto{\pgfqpoint{0.048611in}{0.000000in}}%
\pgfusepath{stroke,fill}%
}%
\begin{pgfscope}%
\pgfsys@transformshift{2.466603in}{0.225139in}%
\pgfsys@useobject{currentmarker}{}%
\end{pgfscope}%
\end{pgfscope}%
\begin{pgfscope}%
\pgftext[x=2.563825in,y=0.176944in,left,base]{\rmfamily\fontsize{10.000000}{12.000000}\selectfont −8}%
\end{pgfscope}%
\begin{pgfscope}%
\pgfsetbuttcap%
\pgfsetroundjoin%
\definecolor{currentfill}{rgb}{0.000000,0.000000,0.000000}%
\pgfsetfillcolor{currentfill}%
\pgfsetlinewidth{0.803000pt}%
\definecolor{currentstroke}{rgb}{0.000000,0.000000,0.000000}%
\pgfsetstrokecolor{currentstroke}%
\pgfsetdash{}{0pt}%
\pgfsys@defobject{currentmarker}{\pgfqpoint{0.000000in}{0.000000in}}{\pgfqpoint{0.048611in}{0.000000in}}{%
\pgfpathmoveto{\pgfqpoint{0.000000in}{0.000000in}}%
\pgfpathlineto{\pgfqpoint{0.048611in}{0.000000in}}%
\pgfusepath{stroke,fill}%
}%
\begin{pgfscope}%
\pgfsys@transformshift{2.466603in}{0.460972in}%
\pgfsys@useobject{currentmarker}{}%
\end{pgfscope}%
\end{pgfscope}%
\begin{pgfscope}%
\pgftext[x=2.563825in,y=0.412778in,left,base]{\rmfamily\fontsize{10.000000}{12.000000}\selectfont −7}%
\end{pgfscope}%
\begin{pgfscope}%
\pgfsetbuttcap%
\pgfsetroundjoin%
\definecolor{currentfill}{rgb}{0.000000,0.000000,0.000000}%
\pgfsetfillcolor{currentfill}%
\pgfsetlinewidth{0.803000pt}%
\definecolor{currentstroke}{rgb}{0.000000,0.000000,0.000000}%
\pgfsetstrokecolor{currentstroke}%
\pgfsetdash{}{0pt}%
\pgfsys@defobject{currentmarker}{\pgfqpoint{0.000000in}{0.000000in}}{\pgfqpoint{0.048611in}{0.000000in}}{%
\pgfpathmoveto{\pgfqpoint{0.000000in}{0.000000in}}%
\pgfpathlineto{\pgfqpoint{0.048611in}{0.000000in}}%
\pgfusepath{stroke,fill}%
}%
\begin{pgfscope}%
\pgfsys@transformshift{2.466603in}{0.696805in}%
\pgfsys@useobject{currentmarker}{}%
\end{pgfscope}%
\end{pgfscope}%
\begin{pgfscope}%
\pgftext[x=2.563825in,y=0.648611in,left,base]{\rmfamily\fontsize{10.000000}{12.000000}\selectfont −6}%
\end{pgfscope}%
\begin{pgfscope}%
\pgfsetbuttcap%
\pgfsetroundjoin%
\definecolor{currentfill}{rgb}{0.000000,0.000000,0.000000}%
\pgfsetfillcolor{currentfill}%
\pgfsetlinewidth{0.803000pt}%
\definecolor{currentstroke}{rgb}{0.000000,0.000000,0.000000}%
\pgfsetstrokecolor{currentstroke}%
\pgfsetdash{}{0pt}%
\pgfsys@defobject{currentmarker}{\pgfqpoint{0.000000in}{0.000000in}}{\pgfqpoint{0.048611in}{0.000000in}}{%
\pgfpathmoveto{\pgfqpoint{0.000000in}{0.000000in}}%
\pgfpathlineto{\pgfqpoint{0.048611in}{0.000000in}}%
\pgfusepath{stroke,fill}%
}%
\begin{pgfscope}%
\pgfsys@transformshift{2.466603in}{0.932639in}%
\pgfsys@useobject{currentmarker}{}%
\end{pgfscope}%
\end{pgfscope}%
\begin{pgfscope}%
\pgftext[x=2.563825in,y=0.884444in,left,base]{\rmfamily\fontsize{10.000000}{12.000000}\selectfont −5}%
\end{pgfscope}%
\begin{pgfscope}%
\pgfsetbuttcap%
\pgfsetroundjoin%
\definecolor{currentfill}{rgb}{0.000000,0.000000,0.000000}%
\pgfsetfillcolor{currentfill}%
\pgfsetlinewidth{0.803000pt}%
\definecolor{currentstroke}{rgb}{0.000000,0.000000,0.000000}%
\pgfsetstrokecolor{currentstroke}%
\pgfsetdash{}{0pt}%
\pgfsys@defobject{currentmarker}{\pgfqpoint{0.000000in}{0.000000in}}{\pgfqpoint{0.048611in}{0.000000in}}{%
\pgfpathmoveto{\pgfqpoint{0.000000in}{0.000000in}}%
\pgfpathlineto{\pgfqpoint{0.048611in}{0.000000in}}%
\pgfusepath{stroke,fill}%
}%
\begin{pgfscope}%
\pgfsys@transformshift{2.466603in}{1.168472in}%
\pgfsys@useobject{currentmarker}{}%
\end{pgfscope}%
\end{pgfscope}%
\begin{pgfscope}%
\pgftext[x=2.563825in,y=1.120278in,left,base]{\rmfamily\fontsize{10.000000}{12.000000}\selectfont −4}%
\end{pgfscope}%
\begin{pgfscope}%
\pgfsetbuttcap%
\pgfsetroundjoin%
\definecolor{currentfill}{rgb}{0.000000,0.000000,0.000000}%
\pgfsetfillcolor{currentfill}%
\pgfsetlinewidth{0.803000pt}%
\definecolor{currentstroke}{rgb}{0.000000,0.000000,0.000000}%
\pgfsetstrokecolor{currentstroke}%
\pgfsetdash{}{0pt}%
\pgfsys@defobject{currentmarker}{\pgfqpoint{0.000000in}{0.000000in}}{\pgfqpoint{0.048611in}{0.000000in}}{%
\pgfpathmoveto{\pgfqpoint{0.000000in}{0.000000in}}%
\pgfpathlineto{\pgfqpoint{0.048611in}{0.000000in}}%
\pgfusepath{stroke,fill}%
}%
\begin{pgfscope}%
\pgfsys@transformshift{2.466603in}{1.404306in}%
\pgfsys@useobject{currentmarker}{}%
\end{pgfscope}%
\end{pgfscope}%
\begin{pgfscope}%
\pgftext[x=2.563825in,y=1.356111in,left,base]{\rmfamily\fontsize{10.000000}{12.000000}\selectfont −3}%
\end{pgfscope}%
\begin{pgfscope}%
\pgfsetbuttcap%
\pgfsetroundjoin%
\definecolor{currentfill}{rgb}{0.000000,0.000000,0.000000}%
\pgfsetfillcolor{currentfill}%
\pgfsetlinewidth{0.803000pt}%
\definecolor{currentstroke}{rgb}{0.000000,0.000000,0.000000}%
\pgfsetstrokecolor{currentstroke}%
\pgfsetdash{}{0pt}%
\pgfsys@defobject{currentmarker}{\pgfqpoint{0.000000in}{0.000000in}}{\pgfqpoint{0.048611in}{0.000000in}}{%
\pgfpathmoveto{\pgfqpoint{0.000000in}{0.000000in}}%
\pgfpathlineto{\pgfqpoint{0.048611in}{0.000000in}}%
\pgfusepath{stroke,fill}%
}%
\begin{pgfscope}%
\pgfsys@transformshift{2.466603in}{1.640139in}%
\pgfsys@useobject{currentmarker}{}%
\end{pgfscope}%
\end{pgfscope}%
\begin{pgfscope}%
\pgftext[x=2.563825in,y=1.591944in,left,base]{\rmfamily\fontsize{10.000000}{12.000000}\selectfont −2}%
\end{pgfscope}%
\begin{pgfscope}%
\pgfsetbuttcap%
\pgfsetroundjoin%
\definecolor{currentfill}{rgb}{0.000000,0.000000,0.000000}%
\pgfsetfillcolor{currentfill}%
\pgfsetlinewidth{0.803000pt}%
\definecolor{currentstroke}{rgb}{0.000000,0.000000,0.000000}%
\pgfsetstrokecolor{currentstroke}%
\pgfsetdash{}{0pt}%
\pgfsys@defobject{currentmarker}{\pgfqpoint{0.000000in}{0.000000in}}{\pgfqpoint{0.048611in}{0.000000in}}{%
\pgfpathmoveto{\pgfqpoint{0.000000in}{0.000000in}}%
\pgfpathlineto{\pgfqpoint{0.048611in}{0.000000in}}%
\pgfusepath{stroke,fill}%
}%
\begin{pgfscope}%
\pgfsys@transformshift{2.466603in}{1.875972in}%
\pgfsys@useobject{currentmarker}{}%
\end{pgfscope}%
\end{pgfscope}%
\begin{pgfscope}%
\pgftext[x=2.563825in,y=1.827778in,left,base]{\rmfamily\fontsize{10.000000}{12.000000}\selectfont −1}%
\end{pgfscope}%
\begin{pgfscope}%
\pgfsetbuttcap%
\pgfsetroundjoin%
\definecolor{currentfill}{rgb}{0.000000,0.000000,0.000000}%
\pgfsetfillcolor{currentfill}%
\pgfsetlinewidth{0.803000pt}%
\definecolor{currentstroke}{rgb}{0.000000,0.000000,0.000000}%
\pgfsetstrokecolor{currentstroke}%
\pgfsetdash{}{0pt}%
\pgfsys@defobject{currentmarker}{\pgfqpoint{0.000000in}{0.000000in}}{\pgfqpoint{0.048611in}{0.000000in}}{%
\pgfpathmoveto{\pgfqpoint{0.000000in}{0.000000in}}%
\pgfpathlineto{\pgfqpoint{0.048611in}{0.000000in}}%
\pgfusepath{stroke,fill}%
}%
\begin{pgfscope}%
\pgfsys@transformshift{2.466603in}{2.111806in}%
\pgfsys@useobject{currentmarker}{}%
\end{pgfscope}%
\end{pgfscope}%
\begin{pgfscope}%
\pgftext[x=2.563825in,y=2.063611in,left,base]{\rmfamily\fontsize{10.000000}{12.000000}\selectfont 0}%
\end{pgfscope}%
\begin{pgfscope}%
\pgftext[x=2.796881in,y=1.168472in,,top,rotate=90.000000]{\rmfamily\fontsize{10.000000}{12.000000}\selectfont \(\displaystyle \log_{10} |\ell_{3,\mathrm{direct}} - \ell_{\leftarrow 3;\rightarrow 8,\mathrm{M2L}(8)}|\)}%
\end{pgfscope}%
\begin{pgfscope}%
\pgfsetbuttcap%
\pgfsetmiterjoin%
\pgfsetlinewidth{0.803000pt}%
\definecolor{currentstroke}{rgb}{0.000000,0.000000,0.000000}%
\pgfsetstrokecolor{currentstroke}%
\pgfsetdash{}{0pt}%
\pgfpathmoveto{\pgfqpoint{2.372270in}{0.225139in}}%
\pgfpathlineto{\pgfqpoint{2.372270in}{0.232509in}}%
\pgfpathlineto{\pgfqpoint{2.372270in}{2.104436in}}%
\pgfpathlineto{\pgfqpoint{2.372270in}{2.111806in}}%
\pgfpathlineto{\pgfqpoint{2.466603in}{2.111806in}}%
\pgfpathlineto{\pgfqpoint{2.466603in}{2.104436in}}%
\pgfpathlineto{\pgfqpoint{2.466603in}{0.232509in}}%
\pgfpathlineto{\pgfqpoint{2.466603in}{0.225139in}}%
\pgfpathclose%
\pgfusepath{stroke}%
\end{pgfscope}%
\end{pgfpicture}%
\makeatother%
\endgroup%

%% file: list2-bump.pgf
\begingroup%
\makeatletter%
\begin{pgfpicture}%
\pgfpathrectangle{\pgfpointorigin}{\pgfqpoint{3.080000in}{2.310000in}}%
\pgfusepath{use as bounding box, clip}%
\begin{pgfscope}%
\pgfsetbuttcap%
\pgfsetmiterjoin%
\definecolor{currentfill}{rgb}{1.000000,1.000000,1.000000}%
\pgfsetfillcolor{currentfill}%
\pgfsetlinewidth{0.000000pt}%
\definecolor{currentstroke}{rgb}{1.000000,1.000000,1.000000}%
\pgfsetstrokecolor{currentstroke}%
\pgfsetdash{}{0pt}%
\pgfpathmoveto{\pgfqpoint{0.000000in}{0.000000in}}%
\pgfpathlineto{\pgfqpoint{3.080000in}{0.000000in}}%
\pgfpathlineto{\pgfqpoint{3.080000in}{2.310000in}}%
\pgfpathlineto{\pgfqpoint{0.000000in}{2.310000in}}%
\pgfpathclose%
\pgfusepath{fill}%
\end{pgfscope}%
\begin{pgfscope}%
\pgfsetbuttcap%
\pgfsetmiterjoin%
\definecolor{currentfill}{rgb}{1.000000,1.000000,1.000000}%
\pgfsetfillcolor{currentfill}%
\pgfsetlinewidth{0.000000pt}%
\definecolor{currentstroke}{rgb}{0.000000,0.000000,0.000000}%
\pgfsetstrokecolor{currentstroke}%
\pgfsetstrokeopacity{0.000000}%
\pgfsetdash{}{0pt}%
\pgfpathmoveto{\pgfqpoint{0.449561in}{0.240209in}}%
\pgfpathlineto{\pgfqpoint{2.279142in}{0.240209in}}%
\pgfpathlineto{\pgfqpoint{2.279142in}{2.069791in}}%
\pgfpathlineto{\pgfqpoint{0.449561in}{2.069791in}}%
\pgfpathclose%
\pgfusepath{fill}%
\end{pgfscope}%
\begin{pgfscope}%
\pgfpathrectangle{\pgfqpoint{0.449561in}{0.240209in}}{\pgfqpoint{1.829581in}{1.829581in}} %
\pgfusepath{clip}%
\pgfsys@transformshift{0.449561in}{0.240209in}%
\pgftext[left,bottom]{\pgfimage[interpolate=true,width=1.830000in,height=1.830000in]{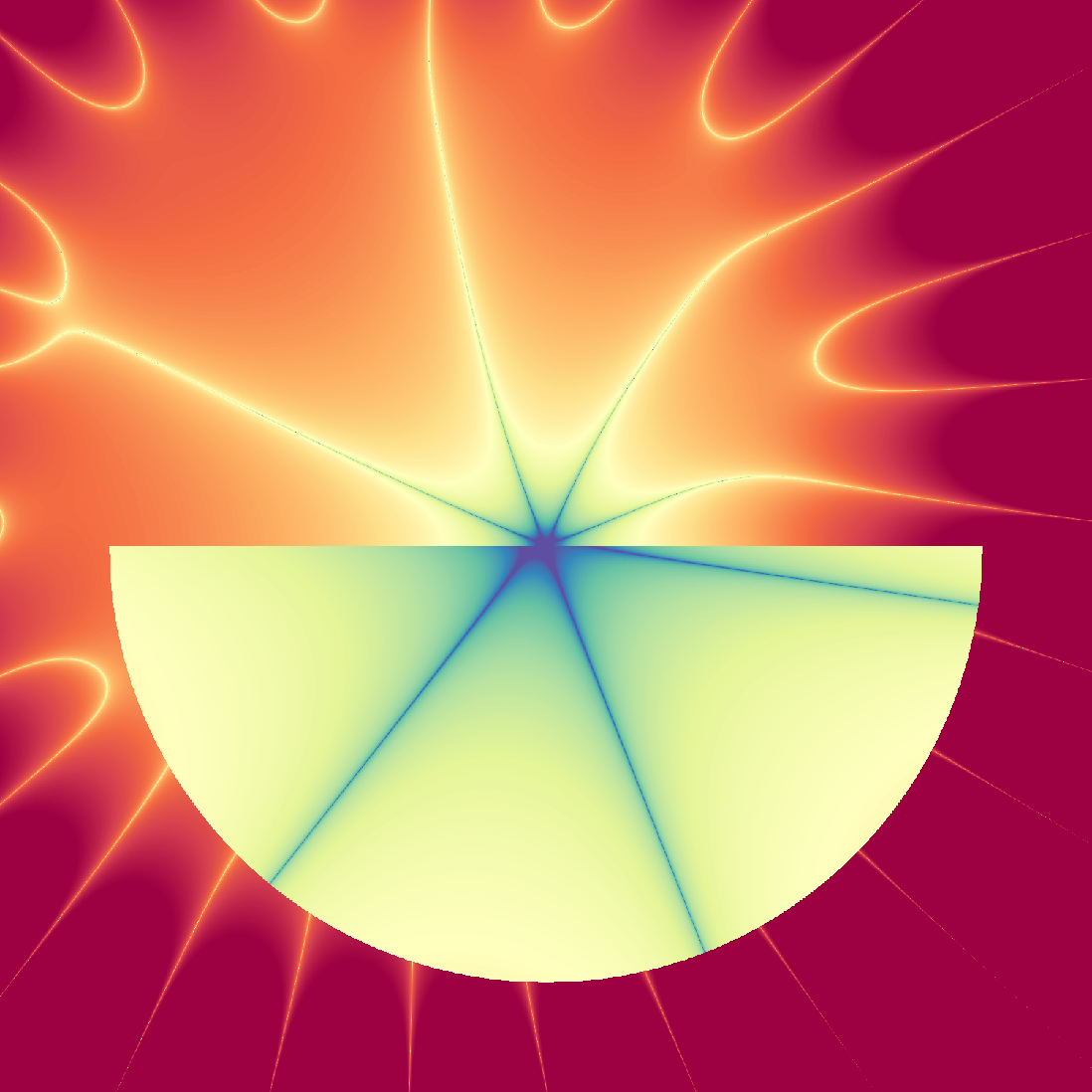}}%
\end{pgfscope}%
\begin{pgfscope}%
\pgfpathrectangle{\pgfqpoint{0.449561in}{0.240209in}}{\pgfqpoint{1.829581in}{1.829581in}} %
\pgfusepath{clip}%
\pgfsetbuttcap%
\pgfsetmiterjoin%
\pgfsetlinewidth{0.803000pt}%
\definecolor{currentstroke}{rgb}{0.000000,0.000000,0.000000}%
\pgfsetstrokecolor{currentstroke}%
\pgfsetdash{}{0pt}%
\pgfpathmoveto{\pgfqpoint{0.998435in}{1.155000in}}%
\pgfpathlineto{\pgfqpoint{1.364351in}{1.155000in}}%
\pgfpathlineto{\pgfqpoint{1.364351in}{1.520916in}}%
\pgfpathlineto{\pgfqpoint{0.998435in}{1.520916in}}%
\pgfpathclose%
\pgfusepath{stroke}%
\end{pgfscope}%
\begin{pgfscope}%
\pgfpathrectangle{\pgfqpoint{0.449561in}{0.240209in}}{\pgfqpoint{1.829581in}{1.829581in}} %
\pgfusepath{clip}%
\pgfsetbuttcap%
\pgfsetmiterjoin%
\pgfsetlinewidth{0.803000pt}%
\definecolor{currentstroke}{rgb}{0.000000,0.000000,0.000000}%
\pgfsetstrokecolor{currentstroke}%
\pgfsetdash{}{0pt}%
\pgfpathmoveto{\pgfqpoint{1.730267in}{1.155000in}}%
\pgfpathlineto{\pgfqpoint{2.096184in}{1.155000in}}%
\pgfpathlineto{\pgfqpoint{2.096184in}{1.520916in}}%
\pgfpathlineto{\pgfqpoint{1.730267in}{1.520916in}}%
\pgfpathclose%
\pgfusepath{stroke}%
\end{pgfscope}%
\begin{pgfscope}%
\pgfpathrectangle{\pgfqpoint{0.449561in}{0.240209in}}{\pgfqpoint{1.829581in}{1.829581in}} %
\pgfusepath{clip}%
\pgfsetbuttcap%
\pgfsetmiterjoin%
\pgfsetlinewidth{0.803000pt}%
\definecolor{currentstroke}{rgb}{0.000000,0.000000,0.000000}%
\pgfsetstrokecolor{currentstroke}%
\pgfsetdash{}{0pt}%
\pgfpathmoveto{\pgfqpoint{1.364351in}{0.423168in}}%
\pgfpathcurveto{\pgfqpoint{1.558435in}{0.423168in}}{\pgfqpoint{1.744597in}{0.500278in}}{\pgfqpoint{1.881835in}{0.637516in}}%
\pgfpathcurveto{\pgfqpoint{2.019073in}{0.774755in}}{\pgfqpoint{2.096184in}{0.960916in}}{\pgfqpoint{2.096184in}{1.155000in}}%
\pgfpathcurveto{\pgfqpoint{2.096184in}{1.349084in}}{\pgfqpoint{2.019073in}{1.535245in}}{\pgfqpoint{1.881835in}{1.672484in}}%
\pgfpathcurveto{\pgfqpoint{1.744597in}{1.809722in}}{\pgfqpoint{1.558435in}{1.886832in}}{\pgfqpoint{1.364351in}{1.886832in}}%
\pgfpathcurveto{\pgfqpoint{1.170267in}{1.886832in}}{\pgfqpoint{0.984106in}{1.809722in}}{\pgfqpoint{0.846868in}{1.672484in}}%
\pgfpathcurveto{\pgfqpoint{0.709629in}{1.535245in}}{\pgfqpoint{0.632519in}{1.349084in}}{\pgfqpoint{0.632519in}{1.155000in}}%
\pgfpathcurveto{\pgfqpoint{0.632519in}{0.960916in}}{\pgfqpoint{0.709629in}{0.774755in}}{\pgfqpoint{0.846868in}{0.637516in}}%
\pgfpathcurveto{\pgfqpoint{0.984106in}{0.500278in}}{\pgfqpoint{1.170267in}{0.423168in}}{\pgfqpoint{1.364351in}{0.423168in}}%
\pgfpathclose%
\pgfusepath{stroke}%
\end{pgfscope}%
\begin{pgfscope}%
\pgfpathrectangle{\pgfqpoint{0.449561in}{0.240209in}}{\pgfqpoint{1.829581in}{1.829581in}} %
\pgfusepath{clip}%
\pgfsetbuttcap%
\pgfsetroundjoin%
\definecolor{currentfill}{rgb}{0.121569,0.466667,0.705882}%
\pgfsetfillcolor{currentfill}%
\pgfsetlinewidth{1.003750pt}%
\definecolor{currentstroke}{rgb}{0.121569,0.466667,0.705882}%
\pgfsetstrokecolor{currentstroke}%
\pgfsetdash{}{0pt}%
\pgfsys@defobject{currentmarker}{\pgfqpoint{-0.041667in}{-0.041667in}}{\pgfqpoint{0.041667in}{0.041667in}}{%
\pgfpathmoveto{\pgfqpoint{0.000000in}{-0.041667in}}%
\pgfpathcurveto{\pgfqpoint{0.011050in}{-0.041667in}}{\pgfqpoint{0.021649in}{-0.037276in}}{\pgfqpoint{0.029463in}{-0.029463in}}%
\pgfpathcurveto{\pgfqpoint{0.037276in}{-0.021649in}}{\pgfqpoint{0.041667in}{-0.011050in}}{\pgfqpoint{0.041667in}{0.000000in}}%
\pgfpathcurveto{\pgfqpoint{0.041667in}{0.011050in}}{\pgfqpoint{0.037276in}{0.021649in}}{\pgfqpoint{0.029463in}{0.029463in}}%
\pgfpathcurveto{\pgfqpoint{0.021649in}{0.037276in}}{\pgfqpoint{0.011050in}{0.041667in}}{\pgfqpoint{0.000000in}{0.041667in}}%
\pgfpathcurveto{\pgfqpoint{-0.011050in}{0.041667in}}{\pgfqpoint{-0.021649in}{0.037276in}}{\pgfqpoint{-0.029463in}{0.029463in}}%
\pgfpathcurveto{\pgfqpoint{-0.037276in}{0.021649in}}{\pgfqpoint{-0.041667in}{0.011050in}}{\pgfqpoint{-0.041667in}{0.000000in}}%
\pgfpathcurveto{\pgfqpoint{-0.041667in}{-0.011050in}}{\pgfqpoint{-0.037276in}{-0.021649in}}{\pgfqpoint{-0.029463in}{-0.029463in}}%
\pgfpathcurveto{\pgfqpoint{-0.021649in}{-0.037276in}}{\pgfqpoint{-0.011050in}{-0.041667in}}{\pgfqpoint{0.000000in}{-0.041667in}}%
\pgfpathclose%
\pgfusepath{stroke,fill}%
}%
\begin{pgfscope}%
\pgfsys@transformshift{2.096184in}{1.155000in}%
\pgfsys@useobject{currentmarker}{}%
\end{pgfscope}%
\end{pgfscope}%
\begin{pgfscope}%
\pgfpathrectangle{\pgfqpoint{0.449561in}{0.240209in}}{\pgfqpoint{1.829581in}{1.829581in}} %
\pgfusepath{clip}%
\pgfsetrectcap%
\pgfsetroundjoin%
\pgfsetlinewidth{1.505625pt}%
\definecolor{currentstroke}{rgb}{1.000000,0.498039,0.054902}%
\pgfsetstrokecolor{currentstroke}%
\pgfsetdash{}{0pt}%
\pgfpathmoveto{\pgfqpoint{1.364351in}{1.155000in}}%
\pgfusepath{stroke}%
\end{pgfscope}%
\begin{pgfscope}%
\pgfpathrectangle{\pgfqpoint{0.449561in}{0.240209in}}{\pgfqpoint{1.829581in}{1.829581in}} %
\pgfusepath{clip}%
\pgfsetbuttcap%
\pgfsetroundjoin%
\definecolor{currentfill}{rgb}{1.000000,0.498039,0.054902}%
\pgfsetfillcolor{currentfill}%
\pgfsetlinewidth{1.003750pt}%
\definecolor{currentstroke}{rgb}{1.000000,0.498039,0.054902}%
\pgfsetstrokecolor{currentstroke}%
\pgfsetdash{}{0pt}%
\pgfsys@defobject{currentmarker}{\pgfqpoint{-0.041667in}{-0.041667in}}{\pgfqpoint{0.041667in}{0.041667in}}{%
\pgfpathmoveto{\pgfqpoint{0.000000in}{-0.041667in}}%
\pgfpathcurveto{\pgfqpoint{0.011050in}{-0.041667in}}{\pgfqpoint{0.021649in}{-0.037276in}}{\pgfqpoint{0.029463in}{-0.029463in}}%
\pgfpathcurveto{\pgfqpoint{0.037276in}{-0.021649in}}{\pgfqpoint{0.041667in}{-0.011050in}}{\pgfqpoint{0.041667in}{0.000000in}}%
\pgfpathcurveto{\pgfqpoint{0.041667in}{0.011050in}}{\pgfqpoint{0.037276in}{0.021649in}}{\pgfqpoint{0.029463in}{0.029463in}}%
\pgfpathcurveto{\pgfqpoint{0.021649in}{0.037276in}}{\pgfqpoint{0.011050in}{0.041667in}}{\pgfqpoint{0.000000in}{0.041667in}}%
\pgfpathcurveto{\pgfqpoint{-0.011050in}{0.041667in}}{\pgfqpoint{-0.021649in}{0.037276in}}{\pgfqpoint{-0.029463in}{0.029463in}}%
\pgfpathcurveto{\pgfqpoint{-0.037276in}{0.021649in}}{\pgfqpoint{-0.041667in}{0.011050in}}{\pgfqpoint{-0.041667in}{0.000000in}}%
\pgfpathcurveto{\pgfqpoint{-0.041667in}{-0.011050in}}{\pgfqpoint{-0.037276in}{-0.021649in}}{\pgfqpoint{-0.029463in}{-0.029463in}}%
\pgfpathcurveto{\pgfqpoint{-0.021649in}{-0.037276in}}{\pgfqpoint{-0.011050in}{-0.041667in}}{\pgfqpoint{0.000000in}{-0.041667in}}%
\pgfpathclose%
\pgfusepath{stroke,fill}%
}%
\begin{pgfscope}%
\pgfsys@transformshift{1.364351in}{1.155000in}%
\pgfsys@useobject{currentmarker}{}%
\end{pgfscope}%
\end{pgfscope}%
\begin{pgfscope}%
\pgfsetrectcap%
\pgfsetmiterjoin%
\pgfsetlinewidth{0.803000pt}%
\definecolor{currentstroke}{rgb}{0.000000,0.000000,0.000000}%
\pgfsetstrokecolor{currentstroke}%
\pgfsetdash{}{0pt}%
\pgfpathmoveto{\pgfqpoint{0.449561in}{0.240209in}}%
\pgfpathlineto{\pgfqpoint{0.449561in}{2.069791in}}%
\pgfusepath{stroke}%
\end{pgfscope}%
\begin{pgfscope}%
\pgfsetrectcap%
\pgfsetmiterjoin%
\pgfsetlinewidth{0.803000pt}%
\definecolor{currentstroke}{rgb}{0.000000,0.000000,0.000000}%
\pgfsetstrokecolor{currentstroke}%
\pgfsetdash{}{0pt}%
\pgfpathmoveto{\pgfqpoint{2.279142in}{0.240209in}}%
\pgfpathlineto{\pgfqpoint{2.279142in}{2.069791in}}%
\pgfusepath{stroke}%
\end{pgfscope}%
\begin{pgfscope}%
\pgfsetrectcap%
\pgfsetmiterjoin%
\pgfsetlinewidth{0.803000pt}%
\definecolor{currentstroke}{rgb}{0.000000,0.000000,0.000000}%
\pgfsetstrokecolor{currentstroke}%
\pgfsetdash{}{0pt}%
\pgfpathmoveto{\pgfqpoint{0.449561in}{0.240209in}}%
\pgfpathlineto{\pgfqpoint{2.279142in}{0.240209in}}%
\pgfusepath{stroke}%
\end{pgfscope}%
\begin{pgfscope}%
\pgfsetrectcap%
\pgfsetmiterjoin%
\pgfsetlinewidth{0.803000pt}%
\definecolor{currentstroke}{rgb}{0.000000,0.000000,0.000000}%
\pgfsetstrokecolor{currentstroke}%
\pgfsetdash{}{0pt}%
\pgfpathmoveto{\pgfqpoint{0.449561in}{2.069791in}}%
\pgfpathlineto{\pgfqpoint{2.279142in}{2.069791in}}%
\pgfusepath{stroke}%
\end{pgfscope}%
\begin{pgfscope}%
\pgfsetroundcap%
\pgfsetroundjoin%
\pgfsetlinewidth{0.803000pt}%
\definecolor{currentstroke}{rgb}{0.000000,0.000000,0.000000}%
\pgfsetstrokecolor{currentstroke}%
\pgfsetdash{}{0pt}%
\pgfpathmoveto{\pgfqpoint{1.335930in}{1.183421in}}%
\pgfpathquadraticcurveto{\pgfqpoint{1.306020in}{1.213332in}}{\pgfqpoint{1.267325in}{1.252026in}}%
\pgfusepath{stroke}%
\end{pgfscope}%
\begin{pgfscope}%
\pgfsetroundcap%
\pgfsetroundjoin%
\definecolor{currentfill}{rgb}{0.000000,0.000000,0.000000}%
\pgfsetfillcolor{currentfill}%
\pgfsetlinewidth{0.803000pt}%
\definecolor{currentstroke}{rgb}{0.000000,0.000000,0.000000}%
\pgfsetstrokecolor{currentstroke}%
\pgfsetdash{}{0pt}%
\pgfpathmoveto{\pgfqpoint{1.316288in}{1.242347in}}%
\pgfpathlineto{\pgfqpoint{1.335930in}{1.183421in}}%
\pgfpathlineto{\pgfqpoint{1.277005in}{1.203063in}}%
\pgfpathlineto{\pgfqpoint{1.316288in}{1.242347in}}%
\pgfpathclose%
\pgfusepath{stroke,fill}%
\end{pgfscope}%
\begin{pgfscope}%
\pgfsetroundcap%
\pgfsetroundjoin%
\pgfsetlinewidth{0.803000pt}%
\definecolor{currentstroke}{rgb}{0.000000,0.000000,0.000000}%
\pgfsetstrokecolor{currentstroke}%
\pgfsetdash{}{0pt}%
\pgfpathmoveto{\pgfqpoint{1.337139in}{1.337958in}}%
\pgfpathquadraticcurveto{\pgfqpoint{1.558206in}{1.337958in}}{\pgfqpoint{1.791696in}{1.337958in}}%
\pgfusepath{stroke}%
\end{pgfscope}%
\begin{pgfscope}%
\pgfsetroundcap%
\pgfsetroundjoin%
\definecolor{currentfill}{rgb}{0.000000,0.000000,0.000000}%
\pgfsetfillcolor{currentfill}%
\pgfsetlinewidth{0.803000pt}%
\definecolor{currentstroke}{rgb}{0.000000,0.000000,0.000000}%
\pgfsetstrokecolor{currentstroke}%
\pgfsetdash{}{0pt}%
\pgfpathmoveto{\pgfqpoint{1.392694in}{1.310180in}}%
\pgfpathlineto{\pgfqpoint{1.337139in}{1.337958in}}%
\pgfpathlineto{\pgfqpoint{1.392694in}{1.365736in}}%
\pgfpathlineto{\pgfqpoint{1.392694in}{1.310180in}}%
\pgfpathclose%
\pgfusepath{stroke,fill}%
\end{pgfscope}%
\begin{pgfscope}%
\pgftext[x=1.181393in,y=1.337958in,,]{\rmfamily\fontsize{10.000000}{12.000000}\selectfont \(\displaystyle \ell_{15}\)}%
\end{pgfscope}%
\begin{pgfscope}%
\pgfsetroundcap%
\pgfsetroundjoin%
\pgfsetlinewidth{0.803000pt}%
\definecolor{currentstroke}{rgb}{0.000000,0.000000,0.000000}%
\pgfsetstrokecolor{currentstroke}%
\pgfsetdash{}{0pt}%
\pgfpathmoveto{\pgfqpoint{2.031091in}{1.220092in}}%
\pgfpathquadraticcurveto{\pgfqpoint{2.059245in}{1.191938in}}{\pgfqpoint{2.096184in}{1.155000in}}%
\pgfusepath{stroke}%
\end{pgfscope}%
\begin{pgfscope}%
\pgfsetroundcap%
\pgfsetroundjoin%
\definecolor{currentfill}{rgb}{0.000000,0.000000,0.000000}%
\pgfsetfillcolor{currentfill}%
\pgfsetlinewidth{0.803000pt}%
\definecolor{currentstroke}{rgb}{0.000000,0.000000,0.000000}%
\pgfsetstrokecolor{currentstroke}%
\pgfsetdash{}{0pt}%
\pgfpathmoveto{\pgfqpoint{2.050733in}{1.161167in}}%
\pgfpathlineto{\pgfqpoint{2.031091in}{1.220092in}}%
\pgfpathlineto{\pgfqpoint{2.090017in}{1.200450in}}%
\pgfpathlineto{\pgfqpoint{2.050733in}{1.161167in}}%
\pgfpathclose%
\pgfusepath{stroke,fill}%
\end{pgfscope}%
\begin{pgfscope}%
\pgftext[x=1.913226in,y=1.337958in,,]{\rmfamily\fontsize{10.000000}{12.000000}\selectfont \(\displaystyle m_{15}\)}%
\end{pgfscope}%
\begin{pgfscope}%
\pgfsetroundcap%
\pgfsetroundjoin%
\pgfsetlinewidth{0.803000pt}%
\definecolor{currentstroke}{rgb}{0.000000,0.000000,0.000000}%
\pgfsetstrokecolor{currentstroke}%
\pgfsetdash{}{0pt}%
\pgfpathmoveto{\pgfqpoint{1.854346in}{0.932074in}}%
\pgfpathquadraticcurveto{\pgfqpoint{1.965052in}{1.034123in}}{\pgfqpoint{2.066625in}{1.127753in}}%
\pgfusepath{stroke}%
\end{pgfscope}%
\begin{pgfscope}%
\pgfsetroundcap%
\pgfsetroundjoin%
\definecolor{currentfill}{rgb}{0.000000,0.000000,0.000000}%
\pgfsetfillcolor{currentfill}%
\pgfsetlinewidth{0.803000pt}%
\definecolor{currentstroke}{rgb}{0.000000,0.000000,0.000000}%
\pgfsetstrokecolor{currentstroke}%
\pgfsetdash{}{0pt}%
\pgfpathmoveto{\pgfqpoint{2.006949in}{1.110523in}}%
\pgfpathlineto{\pgfqpoint{2.066625in}{1.127753in}}%
\pgfpathlineto{\pgfqpoint{2.044603in}{1.069674in}}%
\pgfpathlineto{\pgfqpoint{2.006949in}{1.110523in}}%
\pgfpathclose%
\pgfusepath{stroke,fill}%
\end{pgfscope}%
\begin{pgfscope}%
\pgftext[x=1.547309in,y=0.789084in,left,base]{\rmfamily\fontsize{10.000000}{12.000000}\selectfont source}%
\end{pgfscope}%
\begin{pgfscope}%
\pgftext[x=0.815477in,y=0.935450in,,]{\rmfamily\fontsize{10.000000}{12.000000}\selectfont \(\displaystyle \ell_{3}\)}%
\end{pgfscope}%
\begin{pgfscope}%
\pgftext[x=0.815477in,y=1.374550in,,]{\rmfamily\fontsize{10.000000}{12.000000}\selectfont \(\displaystyle \ell_{15}\)}%
\end{pgfscope}%
\begin{pgfscope}%
\pgfpathrectangle{\pgfqpoint{2.372270in}{0.240209in}}{\pgfqpoint{0.091479in}{1.829581in}} %
\pgfusepath{clip}%
\pgfsetbuttcap%
\pgfsetmiterjoin%
\definecolor{currentfill}{rgb}{1.000000,1.000000,1.000000}%
\pgfsetfillcolor{currentfill}%
\pgfsetlinewidth{0.010037pt}%
\definecolor{currentstroke}{rgb}{1.000000,1.000000,1.000000}%
\pgfsetstrokecolor{currentstroke}%
\pgfsetdash{}{0pt}%
\pgfpathmoveto{\pgfqpoint{2.372270in}{0.240209in}}%
\pgfpathlineto{\pgfqpoint{2.372270in}{0.247356in}}%
\pgfpathlineto{\pgfqpoint{2.372270in}{2.062644in}}%
\pgfpathlineto{\pgfqpoint{2.372270in}{2.069791in}}%
\pgfpathlineto{\pgfqpoint{2.463749in}{2.069791in}}%
\pgfpathlineto{\pgfqpoint{2.463749in}{2.062644in}}%
\pgfpathlineto{\pgfqpoint{2.463749in}{0.247356in}}%
\pgfpathlineto{\pgfqpoint{2.463749in}{0.240209in}}%
\pgfpathclose%
\pgfusepath{stroke,fill}%
\end{pgfscope}%
\begin{pgfscope}%
\pgfsys@transformshift{2.371667in}{0.240000in}%
\pgftext[left,bottom]{\pgfimage[interpolate=true,width=0.091667in,height=1.830000in]{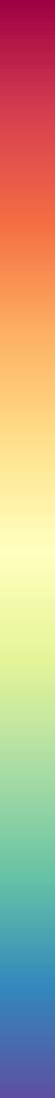}}%
\end{pgfscope}%
\begin{pgfscope}%
\pgfsetbuttcap%
\pgfsetroundjoin%
\definecolor{currentfill}{rgb}{0.000000,0.000000,0.000000}%
\pgfsetfillcolor{currentfill}%
\pgfsetlinewidth{0.803000pt}%
\definecolor{currentstroke}{rgb}{0.000000,0.000000,0.000000}%
\pgfsetstrokecolor{currentstroke}%
\pgfsetdash{}{0pt}%
\pgfsys@defobject{currentmarker}{\pgfqpoint{0.000000in}{0.000000in}}{\pgfqpoint{0.048611in}{0.000000in}}{%
\pgfpathmoveto{\pgfqpoint{0.000000in}{0.000000in}}%
\pgfpathlineto{\pgfqpoint{0.048611in}{0.000000in}}%
\pgfusepath{stroke,fill}%
}%
\begin{pgfscope}%
\pgfsys@transformshift{2.463749in}{0.240209in}%
\pgfsys@useobject{currentmarker}{}%
\end{pgfscope}%
\end{pgfscope}%
\begin{pgfscope}%
\pgftext[x=2.560971in,y=0.192015in,left,base]{\rmfamily\fontsize{10.000000}{12.000000}\selectfont −8}%
\end{pgfscope}%
\begin{pgfscope}%
\pgfsetbuttcap%
\pgfsetroundjoin%
\definecolor{currentfill}{rgb}{0.000000,0.000000,0.000000}%
\pgfsetfillcolor{currentfill}%
\pgfsetlinewidth{0.803000pt}%
\definecolor{currentstroke}{rgb}{0.000000,0.000000,0.000000}%
\pgfsetstrokecolor{currentstroke}%
\pgfsetdash{}{0pt}%
\pgfsys@defobject{currentmarker}{\pgfqpoint{0.000000in}{0.000000in}}{\pgfqpoint{0.048611in}{0.000000in}}{%
\pgfpathmoveto{\pgfqpoint{0.000000in}{0.000000in}}%
\pgfpathlineto{\pgfqpoint{0.048611in}{0.000000in}}%
\pgfusepath{stroke,fill}%
}%
\begin{pgfscope}%
\pgfsys@transformshift{2.463749in}{0.468907in}%
\pgfsys@useobject{currentmarker}{}%
\end{pgfscope}%
\end{pgfscope}%
\begin{pgfscope}%
\pgftext[x=2.560971in,y=0.420713in,left,base]{\rmfamily\fontsize{10.000000}{12.000000}\selectfont −7}%
\end{pgfscope}%
\begin{pgfscope}%
\pgfsetbuttcap%
\pgfsetroundjoin%
\definecolor{currentfill}{rgb}{0.000000,0.000000,0.000000}%
\pgfsetfillcolor{currentfill}%
\pgfsetlinewidth{0.803000pt}%
\definecolor{currentstroke}{rgb}{0.000000,0.000000,0.000000}%
\pgfsetstrokecolor{currentstroke}%
\pgfsetdash{}{0pt}%
\pgfsys@defobject{currentmarker}{\pgfqpoint{0.000000in}{0.000000in}}{\pgfqpoint{0.048611in}{0.000000in}}{%
\pgfpathmoveto{\pgfqpoint{0.000000in}{0.000000in}}%
\pgfpathlineto{\pgfqpoint{0.048611in}{0.000000in}}%
\pgfusepath{stroke,fill}%
}%
\begin{pgfscope}%
\pgfsys@transformshift{2.463749in}{0.697605in}%
\pgfsys@useobject{currentmarker}{}%
\end{pgfscope}%
\end{pgfscope}%
\begin{pgfscope}%
\pgftext[x=2.560971in,y=0.649410in,left,base]{\rmfamily\fontsize{10.000000}{12.000000}\selectfont −6}%
\end{pgfscope}%
\begin{pgfscope}%
\pgfsetbuttcap%
\pgfsetroundjoin%
\definecolor{currentfill}{rgb}{0.000000,0.000000,0.000000}%
\pgfsetfillcolor{currentfill}%
\pgfsetlinewidth{0.803000pt}%
\definecolor{currentstroke}{rgb}{0.000000,0.000000,0.000000}%
\pgfsetstrokecolor{currentstroke}%
\pgfsetdash{}{0pt}%
\pgfsys@defobject{currentmarker}{\pgfqpoint{0.000000in}{0.000000in}}{\pgfqpoint{0.048611in}{0.000000in}}{%
\pgfpathmoveto{\pgfqpoint{0.000000in}{0.000000in}}%
\pgfpathlineto{\pgfqpoint{0.048611in}{0.000000in}}%
\pgfusepath{stroke,fill}%
}%
\begin{pgfscope}%
\pgfsys@transformshift{2.463749in}{0.926302in}%
\pgfsys@useobject{currentmarker}{}%
\end{pgfscope}%
\end{pgfscope}%
\begin{pgfscope}%
\pgftext[x=2.560971in,y=0.878108in,left,base]{\rmfamily\fontsize{10.000000}{12.000000}\selectfont −5}%
\end{pgfscope}%
\begin{pgfscope}%
\pgfsetbuttcap%
\pgfsetroundjoin%
\definecolor{currentfill}{rgb}{0.000000,0.000000,0.000000}%
\pgfsetfillcolor{currentfill}%
\pgfsetlinewidth{0.803000pt}%
\definecolor{currentstroke}{rgb}{0.000000,0.000000,0.000000}%
\pgfsetstrokecolor{currentstroke}%
\pgfsetdash{}{0pt}%
\pgfsys@defobject{currentmarker}{\pgfqpoint{0.000000in}{0.000000in}}{\pgfqpoint{0.048611in}{0.000000in}}{%
\pgfpathmoveto{\pgfqpoint{0.000000in}{0.000000in}}%
\pgfpathlineto{\pgfqpoint{0.048611in}{0.000000in}}%
\pgfusepath{stroke,fill}%
}%
\begin{pgfscope}%
\pgfsys@transformshift{2.463749in}{1.155000in}%
\pgfsys@useobject{currentmarker}{}%
\end{pgfscope}%
\end{pgfscope}%
\begin{pgfscope}%
\pgftext[x=2.560971in,y=1.106806in,left,base]{\rmfamily\fontsize{10.000000}{12.000000}\selectfont −4}%
\end{pgfscope}%
\begin{pgfscope}%
\pgfsetbuttcap%
\pgfsetroundjoin%
\definecolor{currentfill}{rgb}{0.000000,0.000000,0.000000}%
\pgfsetfillcolor{currentfill}%
\pgfsetlinewidth{0.803000pt}%
\definecolor{currentstroke}{rgb}{0.000000,0.000000,0.000000}%
\pgfsetstrokecolor{currentstroke}%
\pgfsetdash{}{0pt}%
\pgfsys@defobject{currentmarker}{\pgfqpoint{0.000000in}{0.000000in}}{\pgfqpoint{0.048611in}{0.000000in}}{%
\pgfpathmoveto{\pgfqpoint{0.000000in}{0.000000in}}%
\pgfpathlineto{\pgfqpoint{0.048611in}{0.000000in}}%
\pgfusepath{stroke,fill}%
}%
\begin{pgfscope}%
\pgfsys@transformshift{2.463749in}{1.383698in}%
\pgfsys@useobject{currentmarker}{}%
\end{pgfscope}%
\end{pgfscope}%
\begin{pgfscope}%
\pgftext[x=2.560971in,y=1.335503in,left,base]{\rmfamily\fontsize{10.000000}{12.000000}\selectfont −3}%
\end{pgfscope}%
\begin{pgfscope}%
\pgfsetbuttcap%
\pgfsetroundjoin%
\definecolor{currentfill}{rgb}{0.000000,0.000000,0.000000}%
\pgfsetfillcolor{currentfill}%
\pgfsetlinewidth{0.803000pt}%
\definecolor{currentstroke}{rgb}{0.000000,0.000000,0.000000}%
\pgfsetstrokecolor{currentstroke}%
\pgfsetdash{}{0pt}%
\pgfsys@defobject{currentmarker}{\pgfqpoint{0.000000in}{0.000000in}}{\pgfqpoint{0.048611in}{0.000000in}}{%
\pgfpathmoveto{\pgfqpoint{0.000000in}{0.000000in}}%
\pgfpathlineto{\pgfqpoint{0.048611in}{0.000000in}}%
\pgfusepath{stroke,fill}%
}%
\begin{pgfscope}%
\pgfsys@transformshift{2.463749in}{1.612395in}%
\pgfsys@useobject{currentmarker}{}%
\end{pgfscope}%
\end{pgfscope}%
\begin{pgfscope}%
\pgftext[x=2.560971in,y=1.564201in,left,base]{\rmfamily\fontsize{10.000000}{12.000000}\selectfont −2}%
\end{pgfscope}%
\begin{pgfscope}%
\pgfsetbuttcap%
\pgfsetroundjoin%
\definecolor{currentfill}{rgb}{0.000000,0.000000,0.000000}%
\pgfsetfillcolor{currentfill}%
\pgfsetlinewidth{0.803000pt}%
\definecolor{currentstroke}{rgb}{0.000000,0.000000,0.000000}%
\pgfsetstrokecolor{currentstroke}%
\pgfsetdash{}{0pt}%
\pgfsys@defobject{currentmarker}{\pgfqpoint{0.000000in}{0.000000in}}{\pgfqpoint{0.048611in}{0.000000in}}{%
\pgfpathmoveto{\pgfqpoint{0.000000in}{0.000000in}}%
\pgfpathlineto{\pgfqpoint{0.048611in}{0.000000in}}%
\pgfusepath{stroke,fill}%
}%
\begin{pgfscope}%
\pgfsys@transformshift{2.463749in}{1.841093in}%
\pgfsys@useobject{currentmarker}{}%
\end{pgfscope}%
\end{pgfscope}%
\begin{pgfscope}%
\pgftext[x=2.560971in,y=1.792898in,left,base]{\rmfamily\fontsize{10.000000}{12.000000}\selectfont −1}%
\end{pgfscope}%
\begin{pgfscope}%
\pgfsetbuttcap%
\pgfsetroundjoin%
\definecolor{currentfill}{rgb}{0.000000,0.000000,0.000000}%
\pgfsetfillcolor{currentfill}%
\pgfsetlinewidth{0.803000pt}%
\definecolor{currentstroke}{rgb}{0.000000,0.000000,0.000000}%
\pgfsetstrokecolor{currentstroke}%
\pgfsetdash{}{0pt}%
\pgfsys@defobject{currentmarker}{\pgfqpoint{0.000000in}{0.000000in}}{\pgfqpoint{0.048611in}{0.000000in}}{%
\pgfpathmoveto{\pgfqpoint{0.000000in}{0.000000in}}%
\pgfpathlineto{\pgfqpoint{0.048611in}{0.000000in}}%
\pgfusepath{stroke,fill}%
}%
\begin{pgfscope}%
\pgfsys@transformshift{2.463749in}{2.069791in}%
\pgfsys@useobject{currentmarker}{}%
\end{pgfscope}%
\end{pgfscope}%
\begin{pgfscope}%
\pgftext[x=2.560971in,y=2.021596in,left,base]{\rmfamily\fontsize{10.000000}{12.000000}\selectfont 0}%
\end{pgfscope}%
\begin{pgfscope}%
\pgftext[x=2.794027in,y=1.155000in,,top,rotate=90.000000]{\rmfamily\fontsize{10.000000}{12.000000}\selectfont \(\displaystyle \log_{10} |\ell_{3,\mathrm{direct}} - \ell_{\leftarrow 3;\rightarrow 15,\mathrm{M2L}(15)}|\)}%
\end{pgfscope}%
\begin{pgfscope}%
\pgfsetbuttcap%
\pgfsetmiterjoin%
\pgfsetlinewidth{0.803000pt}%
\definecolor{currentstroke}{rgb}{0.000000,0.000000,0.000000}%
\pgfsetstrokecolor{currentstroke}%
\pgfsetdash{}{0pt}%
\pgfpathmoveto{\pgfqpoint{2.372270in}{0.240209in}}%
\pgfpathlineto{\pgfqpoint{2.372270in}{0.247356in}}%
\pgfpathlineto{\pgfqpoint{2.372270in}{2.062644in}}%
\pgfpathlineto{\pgfqpoint{2.372270in}{2.069791in}}%
\pgfpathlineto{\pgfqpoint{2.463749in}{2.069791in}}%
\pgfpathlineto{\pgfqpoint{2.463749in}{2.062644in}}%
\pgfpathlineto{\pgfqpoint{2.463749in}{0.247356in}}%
\pgfpathlineto{\pgfqpoint{2.463749in}{0.240209in}}%
\pgfpathclose%
\pgfusepath{stroke}%
\end{pgfscope}%
\end{pgfpicture}%
\makeatother%
\endgroup%

%% file: list4.pgf
\begingroup%
\makeatletter%
\begin{pgfpicture}%
\pgfpathrectangle{\pgfpointorigin}{\pgfqpoint{3.080000in}{2.310000in}}%
\pgfusepath{use as bounding box, clip}%
\begin{pgfscope}%
\pgfsetbuttcap%
\pgfsetmiterjoin%
\definecolor{currentfill}{rgb}{1.000000,1.000000,1.000000}%
\pgfsetfillcolor{currentfill}%
\pgfsetlinewidth{0.000000pt}%
\definecolor{currentstroke}{rgb}{1.000000,1.000000,1.000000}%
\pgfsetstrokecolor{currentstroke}%
\pgfsetdash{}{0pt}%
\pgfpathmoveto{\pgfqpoint{0.000000in}{0.000000in}}%
\pgfpathlineto{\pgfqpoint{3.080000in}{0.000000in}}%
\pgfpathlineto{\pgfqpoint{3.080000in}{2.310000in}}%
\pgfpathlineto{\pgfqpoint{0.000000in}{2.310000in}}%
\pgfpathclose%
\pgfusepath{fill}%
\end{pgfscope}%
\begin{pgfscope}%
\pgfsetbuttcap%
\pgfsetmiterjoin%
\definecolor{currentfill}{rgb}{1.000000,1.000000,1.000000}%
\pgfsetfillcolor{currentfill}%
\pgfsetlinewidth{0.000000pt}%
\definecolor{currentstroke}{rgb}{0.000000,0.000000,0.000000}%
\pgfsetstrokecolor{currentstroke}%
\pgfsetstrokeopacity{0.000000}%
\pgfsetdash{}{0pt}%
\pgfpathmoveto{\pgfqpoint{0.392475in}{0.225139in}}%
\pgfpathlineto{\pgfqpoint{2.279142in}{0.225139in}}%
\pgfpathlineto{\pgfqpoint{2.279142in}{2.111806in}}%
\pgfpathlineto{\pgfqpoint{0.392475in}{2.111806in}}%
\pgfpathclose%
\pgfusepath{fill}%
\end{pgfscope}%
\begin{pgfscope}%
\pgfpathrectangle{\pgfqpoint{0.392475in}{0.225139in}}{\pgfqpoint{1.886667in}{1.886667in}} %
\pgfusepath{clip}%
\pgfsys@transformshift{0.392475in}{0.225139in}%
\pgftext[left,bottom]{\pgfimage[interpolate=true,width=1.888333in,height=1.888333in]{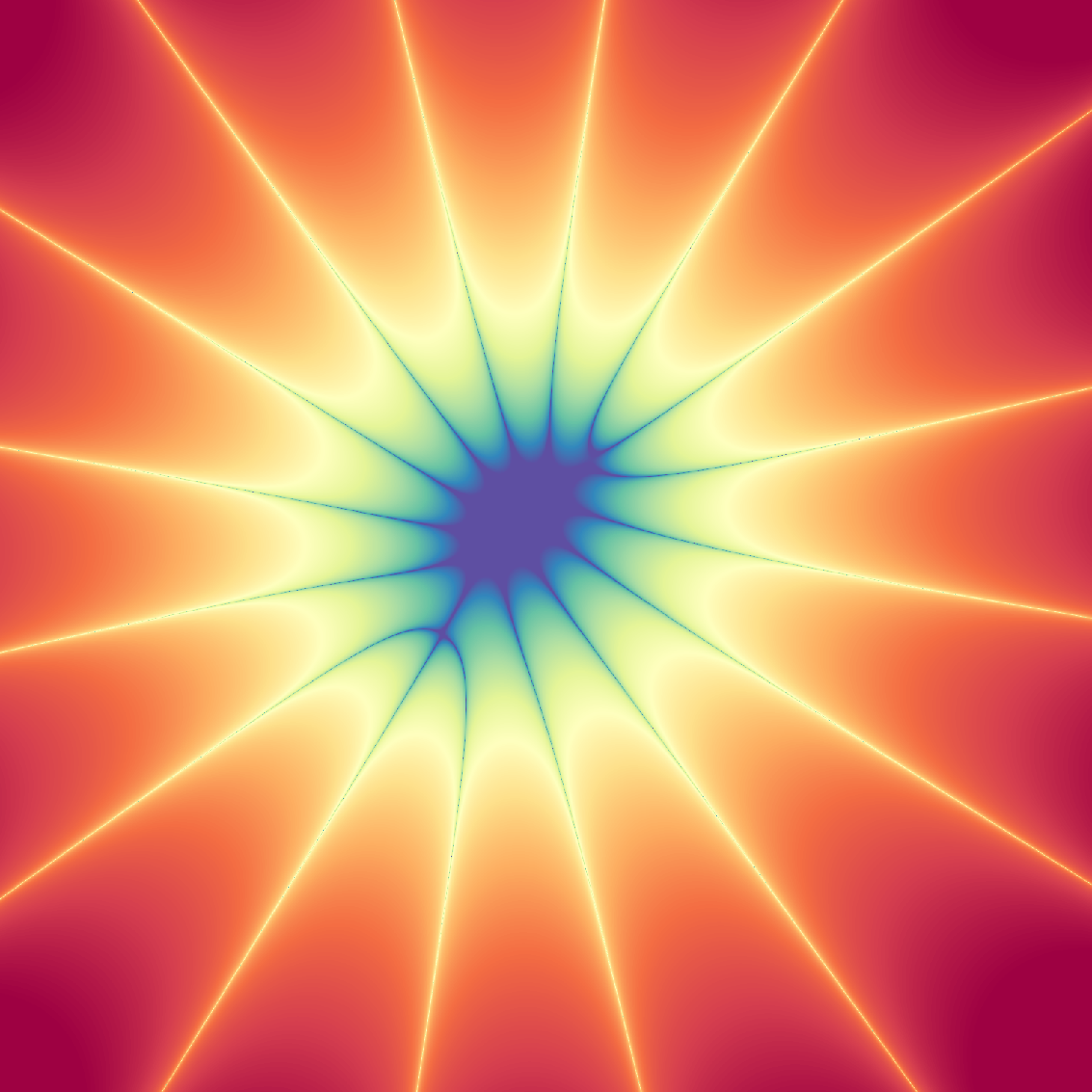}}%
\end{pgfscope}%
\begin{pgfscope}%
\pgfpathrectangle{\pgfqpoint{0.392475in}{0.225139in}}{\pgfqpoint{1.886667in}{1.886667in}} %
\pgfusepath{clip}%
\pgfsetbuttcap%
\pgfsetmiterjoin%
\pgfsetlinewidth{0.803000pt}%
\definecolor{currentstroke}{rgb}{0.000000,0.000000,0.000000}%
\pgfsetstrokecolor{currentstroke}%
\pgfsetdash{}{0pt}%
\pgfpathmoveto{\pgfqpoint{1.099975in}{1.168472in}}%
\pgfpathlineto{\pgfqpoint{1.335808in}{1.168472in}}%
\pgfpathlineto{\pgfqpoint{1.335808in}{1.404306in}}%
\pgfpathlineto{\pgfqpoint{1.099975in}{1.404306in}}%
\pgfpathclose%
\pgfusepath{stroke}%
\end{pgfscope}%
\begin{pgfscope}%
\pgfpathrectangle{\pgfqpoint{0.392475in}{0.225139in}}{\pgfqpoint{1.886667in}{1.886667in}} %
\pgfusepath{clip}%
\pgfsetbuttcap%
\pgfsetmiterjoin%
\pgfsetlinewidth{0.803000pt}%
\definecolor{currentstroke}{rgb}{0.000000,0.000000,0.000000}%
\pgfsetstrokecolor{currentstroke}%
\pgfsetdash{}{0pt}%
\pgfpathmoveto{\pgfqpoint{1.571642in}{1.168472in}}%
\pgfpathlineto{\pgfqpoint{2.043308in}{1.168472in}}%
\pgfpathlineto{\pgfqpoint{2.043308in}{1.640139in}}%
\pgfpathlineto{\pgfqpoint{1.571642in}{1.640139in}}%
\pgfpathclose%
\pgfusepath{stroke}%
\end{pgfscope}%
\begin{pgfscope}%
\pgfpathrectangle{\pgfqpoint{0.392475in}{0.225139in}}{\pgfqpoint{1.886667in}{1.886667in}} %
\pgfusepath{clip}%
\pgfsetbuttcap%
\pgfsetmiterjoin%
\pgfsetlinewidth{0.803000pt}%
\definecolor{currentstroke}{rgb}{0.000000,0.000000,0.000000}%
\pgfsetstrokecolor{currentstroke}%
\pgfsetdash{}{0pt}%
\pgfpathmoveto{\pgfqpoint{1.335808in}{0.460972in}}%
\pgfpathcurveto{\pgfqpoint{1.523440in}{0.460972in}}{\pgfqpoint{1.703411in}{0.535519in}}{\pgfqpoint{1.836086in}{0.668194in}}%
\pgfpathcurveto{\pgfqpoint{1.968762in}{0.800869in}}{\pgfqpoint{2.043308in}{0.980841in}}{\pgfqpoint{2.043308in}{1.168472in}}%
\pgfpathcurveto{\pgfqpoint{2.043308in}{1.356103in}}{\pgfqpoint{1.968762in}{1.536075in}}{\pgfqpoint{1.836086in}{1.668750in}}%
\pgfpathcurveto{\pgfqpoint{1.703411in}{1.801426in}}{\pgfqpoint{1.523440in}{1.875972in}}{\pgfqpoint{1.335808in}{1.875972in}}%
\pgfpathcurveto{\pgfqpoint{1.148177in}{1.875972in}}{\pgfqpoint{0.968206in}{1.801426in}}{\pgfqpoint{0.835530in}{1.668750in}}%
\pgfpathcurveto{\pgfqpoint{0.702855in}{1.536075in}}{\pgfqpoint{0.628308in}{1.356103in}}{\pgfqpoint{0.628308in}{1.168472in}}%
\pgfpathcurveto{\pgfqpoint{0.628308in}{0.980841in}}{\pgfqpoint{0.702855in}{0.800869in}}{\pgfqpoint{0.835530in}{0.668194in}}%
\pgfpathcurveto{\pgfqpoint{0.968206in}{0.535519in}}{\pgfqpoint{1.148177in}{0.460972in}}{\pgfqpoint{1.335808in}{0.460972in}}%
\pgfpathclose%
\pgfusepath{stroke}%
\end{pgfscope}%
\begin{pgfscope}%
\pgfpathrectangle{\pgfqpoint{0.392475in}{0.225139in}}{\pgfqpoint{1.886667in}{1.886667in}} %
\pgfusepath{clip}%
\pgfsetbuttcap%
\pgfsetroundjoin%
\definecolor{currentfill}{rgb}{0.121569,0.466667,0.705882}%
\pgfsetfillcolor{currentfill}%
\pgfsetlinewidth{1.003750pt}%
\definecolor{currentstroke}{rgb}{0.121569,0.466667,0.705882}%
\pgfsetstrokecolor{currentstroke}%
\pgfsetdash{}{0pt}%
\pgfsys@defobject{currentmarker}{\pgfqpoint{-0.041667in}{-0.041667in}}{\pgfqpoint{0.041667in}{0.041667in}}{%
\pgfpathmoveto{\pgfqpoint{0.000000in}{-0.041667in}}%
\pgfpathcurveto{\pgfqpoint{0.011050in}{-0.041667in}}{\pgfqpoint{0.021649in}{-0.037276in}}{\pgfqpoint{0.029463in}{-0.029463in}}%
\pgfpathcurveto{\pgfqpoint{0.037276in}{-0.021649in}}{\pgfqpoint{0.041667in}{-0.011050in}}{\pgfqpoint{0.041667in}{0.000000in}}%
\pgfpathcurveto{\pgfqpoint{0.041667in}{0.011050in}}{\pgfqpoint{0.037276in}{0.021649in}}{\pgfqpoint{0.029463in}{0.029463in}}%
\pgfpathcurveto{\pgfqpoint{0.021649in}{0.037276in}}{\pgfqpoint{0.011050in}{0.041667in}}{\pgfqpoint{0.000000in}{0.041667in}}%
\pgfpathcurveto{\pgfqpoint{-0.011050in}{0.041667in}}{\pgfqpoint{-0.021649in}{0.037276in}}{\pgfqpoint{-0.029463in}{0.029463in}}%
\pgfpathcurveto{\pgfqpoint{-0.037276in}{0.021649in}}{\pgfqpoint{-0.041667in}{0.011050in}}{\pgfqpoint{-0.041667in}{0.000000in}}%
\pgfpathcurveto{\pgfqpoint{-0.041667in}{-0.011050in}}{\pgfqpoint{-0.037276in}{-0.021649in}}{\pgfqpoint{-0.029463in}{-0.029463in}}%
\pgfpathcurveto{\pgfqpoint{-0.021649in}{-0.037276in}}{\pgfqpoint{-0.011050in}{-0.041667in}}{\pgfqpoint{0.000000in}{-0.041667in}}%
\pgfpathclose%
\pgfusepath{stroke,fill}%
}%
\begin{pgfscope}%
\pgfsys@transformshift{2.043308in}{1.168472in}%
\pgfsys@useobject{currentmarker}{}%
\end{pgfscope}%
\end{pgfscope}%
\begin{pgfscope}%
\pgfpathrectangle{\pgfqpoint{0.392475in}{0.225139in}}{\pgfqpoint{1.886667in}{1.886667in}} %
\pgfusepath{clip}%
\pgfsetrectcap%
\pgfsetroundjoin%
\pgfsetlinewidth{1.505625pt}%
\definecolor{currentstroke}{rgb}{1.000000,0.498039,0.054902}%
\pgfsetstrokecolor{currentstroke}%
\pgfsetdash{}{0pt}%
\pgfpathmoveto{\pgfqpoint{1.335808in}{1.168472in}}%
\pgfusepath{stroke}%
\end{pgfscope}%
\begin{pgfscope}%
\pgfpathrectangle{\pgfqpoint{0.392475in}{0.225139in}}{\pgfqpoint{1.886667in}{1.886667in}} %
\pgfusepath{clip}%
\pgfsetbuttcap%
\pgfsetroundjoin%
\definecolor{currentfill}{rgb}{1.000000,0.498039,0.054902}%
\pgfsetfillcolor{currentfill}%
\pgfsetlinewidth{1.003750pt}%
\definecolor{currentstroke}{rgb}{1.000000,0.498039,0.054902}%
\pgfsetstrokecolor{currentstroke}%
\pgfsetdash{}{0pt}%
\pgfsys@defobject{currentmarker}{\pgfqpoint{-0.041667in}{-0.041667in}}{\pgfqpoint{0.041667in}{0.041667in}}{%
\pgfpathmoveto{\pgfqpoint{0.000000in}{-0.041667in}}%
\pgfpathcurveto{\pgfqpoint{0.011050in}{-0.041667in}}{\pgfqpoint{0.021649in}{-0.037276in}}{\pgfqpoint{0.029463in}{-0.029463in}}%
\pgfpathcurveto{\pgfqpoint{0.037276in}{-0.021649in}}{\pgfqpoint{0.041667in}{-0.011050in}}{\pgfqpoint{0.041667in}{0.000000in}}%
\pgfpathcurveto{\pgfqpoint{0.041667in}{0.011050in}}{\pgfqpoint{0.037276in}{0.021649in}}{\pgfqpoint{0.029463in}{0.029463in}}%
\pgfpathcurveto{\pgfqpoint{0.021649in}{0.037276in}}{\pgfqpoint{0.011050in}{0.041667in}}{\pgfqpoint{0.000000in}{0.041667in}}%
\pgfpathcurveto{\pgfqpoint{-0.011050in}{0.041667in}}{\pgfqpoint{-0.021649in}{0.037276in}}{\pgfqpoint{-0.029463in}{0.029463in}}%
\pgfpathcurveto{\pgfqpoint{-0.037276in}{0.021649in}}{\pgfqpoint{-0.041667in}{0.011050in}}{\pgfqpoint{-0.041667in}{0.000000in}}%
\pgfpathcurveto{\pgfqpoint{-0.041667in}{-0.011050in}}{\pgfqpoint{-0.037276in}{-0.021649in}}{\pgfqpoint{-0.029463in}{-0.029463in}}%
\pgfpathcurveto{\pgfqpoint{-0.021649in}{-0.037276in}}{\pgfqpoint{-0.011050in}{-0.041667in}}{\pgfqpoint{0.000000in}{-0.041667in}}%
\pgfpathclose%
\pgfusepath{stroke,fill}%
}%
\begin{pgfscope}%
\pgfsys@transformshift{1.335808in}{1.168472in}%
\pgfsys@useobject{currentmarker}{}%
\end{pgfscope}%
\end{pgfscope}%
\begin{pgfscope}%
\pgfsetrectcap%
\pgfsetmiterjoin%
\pgfsetlinewidth{0.803000pt}%
\definecolor{currentstroke}{rgb}{0.000000,0.000000,0.000000}%
\pgfsetstrokecolor{currentstroke}%
\pgfsetdash{}{0pt}%
\pgfpathmoveto{\pgfqpoint{0.392475in}{0.225139in}}%
\pgfpathlineto{\pgfqpoint{0.392475in}{2.111806in}}%
\pgfusepath{stroke}%
\end{pgfscope}%
\begin{pgfscope}%
\pgfsetrectcap%
\pgfsetmiterjoin%
\pgfsetlinewidth{0.803000pt}%
\definecolor{currentstroke}{rgb}{0.000000,0.000000,0.000000}%
\pgfsetstrokecolor{currentstroke}%
\pgfsetdash{}{0pt}%
\pgfpathmoveto{\pgfqpoint{2.279142in}{0.225139in}}%
\pgfpathlineto{\pgfqpoint{2.279142in}{2.111806in}}%
\pgfusepath{stroke}%
\end{pgfscope}%
\begin{pgfscope}%
\pgfsetrectcap%
\pgfsetmiterjoin%
\pgfsetlinewidth{0.803000pt}%
\definecolor{currentstroke}{rgb}{0.000000,0.000000,0.000000}%
\pgfsetstrokecolor{currentstroke}%
\pgfsetdash{}{0pt}%
\pgfpathmoveto{\pgfqpoint{0.392475in}{0.225139in}}%
\pgfpathlineto{\pgfqpoint{2.279142in}{0.225139in}}%
\pgfusepath{stroke}%
\end{pgfscope}%
\begin{pgfscope}%
\pgfsetrectcap%
\pgfsetmiterjoin%
\pgfsetlinewidth{0.803000pt}%
\definecolor{currentstroke}{rgb}{0.000000,0.000000,0.000000}%
\pgfsetstrokecolor{currentstroke}%
\pgfsetdash{}{0pt}%
\pgfpathmoveto{\pgfqpoint{0.392475in}{2.111806in}}%
\pgfpathlineto{\pgfqpoint{2.279142in}{2.111806in}}%
\pgfusepath{stroke}%
\end{pgfscope}%
\begin{pgfscope}%
\pgfsetroundcap%
\pgfsetroundjoin%
\pgfsetlinewidth{0.803000pt}%
\definecolor{currentstroke}{rgb}{0.000000,0.000000,0.000000}%
\pgfsetstrokecolor{currentstroke}%
\pgfsetdash{}{0pt}%
\pgfpathmoveto{\pgfqpoint{1.307380in}{1.196901in}}%
\pgfpathquadraticcurveto{\pgfqpoint{1.309993in}{1.194287in}}{\pgfqpoint{1.303823in}{1.200458in}}%
\pgfusepath{stroke}%
\end{pgfscope}%
\begin{pgfscope}%
\pgfsetroundcap%
\pgfsetroundjoin%
\definecolor{currentfill}{rgb}{0.000000,0.000000,0.000000}%
\pgfsetfillcolor{currentfill}%
\pgfsetlinewidth{0.803000pt}%
\definecolor{currentstroke}{rgb}{0.000000,0.000000,0.000000}%
\pgfsetstrokecolor{currentstroke}%
\pgfsetdash{}{0pt}%
\pgfpathmoveto{\pgfqpoint{1.287738in}{1.255826in}}%
\pgfpathlineto{\pgfqpoint{1.307380in}{1.196901in}}%
\pgfpathlineto{\pgfqpoint{1.248454in}{1.216542in}}%
\pgfpathlineto{\pgfqpoint{1.287738in}{1.255826in}}%
\pgfpathclose%
\pgfusepath{stroke,fill}%
\end{pgfscope}%
\begin{pgfscope}%
\pgfsetroundcap%
\pgfsetroundjoin%
\pgfsetlinewidth{0.803000pt}%
\definecolor{currentstroke}{rgb}{0.000000,0.000000,0.000000}%
\pgfsetstrokecolor{currentstroke}%
\pgfsetdash{}{0pt}%
\pgfpathmoveto{\pgfqpoint{1.345557in}{1.268151in}}%
\pgfpathquadraticcurveto{\pgfqpoint{1.688284in}{1.219190in}}{\pgfqpoint{2.043308in}{1.168472in}}%
\pgfusepath{stroke}%
\end{pgfscope}%
\begin{pgfscope}%
\pgfsetroundcap%
\pgfsetroundjoin%
\definecolor{currentfill}{rgb}{0.000000,0.000000,0.000000}%
\pgfsetfillcolor{currentfill}%
\pgfsetlinewidth{0.803000pt}%
\definecolor{currentstroke}{rgb}{0.000000,0.000000,0.000000}%
\pgfsetstrokecolor{currentstroke}%
\pgfsetdash{}{0pt}%
\pgfpathmoveto{\pgfqpoint{1.396626in}{1.232796in}}%
\pgfpathlineto{\pgfqpoint{1.345557in}{1.268151in}}%
\pgfpathlineto{\pgfqpoint{1.404482in}{1.287793in}}%
\pgfpathlineto{\pgfqpoint{1.396626in}{1.232796in}}%
\pgfpathclose%
\pgfusepath{stroke,fill}%
\end{pgfscope}%
\begin{pgfscope}%
\pgftext[x=1.217892in,y=1.286389in,,]{\rmfamily\fontsize{10.000000}{12.000000}\selectfont \(\displaystyle \ell_{8}\)}%
\end{pgfscope}%
\begin{pgfscope}%
\pgfsetroundcap%
\pgfsetroundjoin%
\pgfsetlinewidth{0.803000pt}%
\definecolor{currentstroke}{rgb}{0.000000,0.000000,0.000000}%
\pgfsetstrokecolor{currentstroke}%
\pgfsetdash{}{0pt}%
\pgfpathmoveto{\pgfqpoint{1.694254in}{0.953461in}}%
\pgfpathquadraticcurveto{\pgfqpoint{1.856954in}{1.053681in}}{\pgfqpoint{2.009078in}{1.147387in}}%
\pgfusepath{stroke}%
\end{pgfscope}%
\begin{pgfscope}%
\pgfsetroundcap%
\pgfsetroundjoin%
\definecolor{currentfill}{rgb}{0.000000,0.000000,0.000000}%
\pgfsetfillcolor{currentfill}%
\pgfsetlinewidth{0.803000pt}%
\definecolor{currentstroke}{rgb}{0.000000,0.000000,0.000000}%
\pgfsetstrokecolor{currentstroke}%
\pgfsetdash{}{0pt}%
\pgfpathmoveto{\pgfqpoint{1.947208in}{1.141901in}}%
\pgfpathlineto{\pgfqpoint{2.009078in}{1.147387in}}%
\pgfpathlineto{\pgfqpoint{1.976345in}{1.094599in}}%
\pgfpathlineto{\pgfqpoint{1.947208in}{1.141901in}}%
\pgfpathclose%
\pgfusepath{stroke,fill}%
\end{pgfscope}%
\begin{pgfscope}%
\pgftext[x=1.335808in,y=0.814722in,left,base]{\rmfamily\fontsize{10.000000}{12.000000}\selectfont source}%
\end{pgfscope}%
\begin{pgfscope}%
\pgftext[x=1.123558in,y=0.814722in,,]{\rmfamily\fontsize{10.000000}{12.000000}\selectfont \(\displaystyle \ell_{8}\)}%
\end{pgfscope}%
\begin{pgfscope}%
\pgfpathrectangle{\pgfqpoint{2.372270in}{0.225139in}}{\pgfqpoint{0.094333in}{1.886667in}} %
\pgfusepath{clip}%
\pgfsetbuttcap%
\pgfsetmiterjoin%
\definecolor{currentfill}{rgb}{1.000000,1.000000,1.000000}%
\pgfsetfillcolor{currentfill}%
\pgfsetlinewidth{0.010037pt}%
\definecolor{currentstroke}{rgb}{1.000000,1.000000,1.000000}%
\pgfsetstrokecolor{currentstroke}%
\pgfsetdash{}{0pt}%
\pgfpathmoveto{\pgfqpoint{2.372270in}{0.225139in}}%
\pgfpathlineto{\pgfqpoint{2.372270in}{0.232509in}}%
\pgfpathlineto{\pgfqpoint{2.372270in}{2.104436in}}%
\pgfpathlineto{\pgfqpoint{2.372270in}{2.111806in}}%
\pgfpathlineto{\pgfqpoint{2.466603in}{2.111806in}}%
\pgfpathlineto{\pgfqpoint{2.466603in}{2.104436in}}%
\pgfpathlineto{\pgfqpoint{2.466603in}{0.232509in}}%
\pgfpathlineto{\pgfqpoint{2.466603in}{0.225139in}}%
\pgfpathclose%
\pgfusepath{stroke,fill}%
\end{pgfscope}%
\begin{pgfscope}%
\pgfsys@transformshift{2.371667in}{0.225000in}%
\pgftext[left,bottom]{\pgfimage[interpolate=true,width=0.095000in,height=1.886667in]{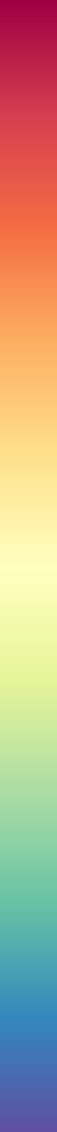}}%
\end{pgfscope}%
\begin{pgfscope}%
\pgfsetbuttcap%
\pgfsetroundjoin%
\definecolor{currentfill}{rgb}{0.000000,0.000000,0.000000}%
\pgfsetfillcolor{currentfill}%
\pgfsetlinewidth{0.803000pt}%
\definecolor{currentstroke}{rgb}{0.000000,0.000000,0.000000}%
\pgfsetstrokecolor{currentstroke}%
\pgfsetdash{}{0pt}%
\pgfsys@defobject{currentmarker}{\pgfqpoint{0.000000in}{0.000000in}}{\pgfqpoint{0.048611in}{0.000000in}}{%
\pgfpathmoveto{\pgfqpoint{0.000000in}{0.000000in}}%
\pgfpathlineto{\pgfqpoint{0.048611in}{0.000000in}}%
\pgfusepath{stroke,fill}%
}%
\begin{pgfscope}%
\pgfsys@transformshift{2.466603in}{0.225139in}%
\pgfsys@useobject{currentmarker}{}%
\end{pgfscope}%
\end{pgfscope}%
\begin{pgfscope}%
\pgftext[x=2.563825in,y=0.176944in,left,base]{\rmfamily\fontsize{10.000000}{12.000000}\selectfont −8}%
\end{pgfscope}%
\begin{pgfscope}%
\pgfsetbuttcap%
\pgfsetroundjoin%
\definecolor{currentfill}{rgb}{0.000000,0.000000,0.000000}%
\pgfsetfillcolor{currentfill}%
\pgfsetlinewidth{0.803000pt}%
\definecolor{currentstroke}{rgb}{0.000000,0.000000,0.000000}%
\pgfsetstrokecolor{currentstroke}%
\pgfsetdash{}{0pt}%
\pgfsys@defobject{currentmarker}{\pgfqpoint{0.000000in}{0.000000in}}{\pgfqpoint{0.048611in}{0.000000in}}{%
\pgfpathmoveto{\pgfqpoint{0.000000in}{0.000000in}}%
\pgfpathlineto{\pgfqpoint{0.048611in}{0.000000in}}%
\pgfusepath{stroke,fill}%
}%
\begin{pgfscope}%
\pgfsys@transformshift{2.466603in}{0.460972in}%
\pgfsys@useobject{currentmarker}{}%
\end{pgfscope}%
\end{pgfscope}%
\begin{pgfscope}%
\pgftext[x=2.563825in,y=0.412778in,left,base]{\rmfamily\fontsize{10.000000}{12.000000}\selectfont −7}%
\end{pgfscope}%
\begin{pgfscope}%
\pgfsetbuttcap%
\pgfsetroundjoin%
\definecolor{currentfill}{rgb}{0.000000,0.000000,0.000000}%
\pgfsetfillcolor{currentfill}%
\pgfsetlinewidth{0.803000pt}%
\definecolor{currentstroke}{rgb}{0.000000,0.000000,0.000000}%
\pgfsetstrokecolor{currentstroke}%
\pgfsetdash{}{0pt}%
\pgfsys@defobject{currentmarker}{\pgfqpoint{0.000000in}{0.000000in}}{\pgfqpoint{0.048611in}{0.000000in}}{%
\pgfpathmoveto{\pgfqpoint{0.000000in}{0.000000in}}%
\pgfpathlineto{\pgfqpoint{0.048611in}{0.000000in}}%
\pgfusepath{stroke,fill}%
}%
\begin{pgfscope}%
\pgfsys@transformshift{2.466603in}{0.696805in}%
\pgfsys@useobject{currentmarker}{}%
\end{pgfscope}%
\end{pgfscope}%
\begin{pgfscope}%
\pgftext[x=2.563825in,y=0.648611in,left,base]{\rmfamily\fontsize{10.000000}{12.000000}\selectfont −6}%
\end{pgfscope}%
\begin{pgfscope}%
\pgfsetbuttcap%
\pgfsetroundjoin%
\definecolor{currentfill}{rgb}{0.000000,0.000000,0.000000}%
\pgfsetfillcolor{currentfill}%
\pgfsetlinewidth{0.803000pt}%
\definecolor{currentstroke}{rgb}{0.000000,0.000000,0.000000}%
\pgfsetstrokecolor{currentstroke}%
\pgfsetdash{}{0pt}%
\pgfsys@defobject{currentmarker}{\pgfqpoint{0.000000in}{0.000000in}}{\pgfqpoint{0.048611in}{0.000000in}}{%
\pgfpathmoveto{\pgfqpoint{0.000000in}{0.000000in}}%
\pgfpathlineto{\pgfqpoint{0.048611in}{0.000000in}}%
\pgfusepath{stroke,fill}%
}%
\begin{pgfscope}%
\pgfsys@transformshift{2.466603in}{0.932639in}%
\pgfsys@useobject{currentmarker}{}%
\end{pgfscope}%
\end{pgfscope}%
\begin{pgfscope}%
\pgftext[x=2.563825in,y=0.884444in,left,base]{\rmfamily\fontsize{10.000000}{12.000000}\selectfont −5}%
\end{pgfscope}%
\begin{pgfscope}%
\pgfsetbuttcap%
\pgfsetroundjoin%
\definecolor{currentfill}{rgb}{0.000000,0.000000,0.000000}%
\pgfsetfillcolor{currentfill}%
\pgfsetlinewidth{0.803000pt}%
\definecolor{currentstroke}{rgb}{0.000000,0.000000,0.000000}%
\pgfsetstrokecolor{currentstroke}%
\pgfsetdash{}{0pt}%
\pgfsys@defobject{currentmarker}{\pgfqpoint{0.000000in}{0.000000in}}{\pgfqpoint{0.048611in}{0.000000in}}{%
\pgfpathmoveto{\pgfqpoint{0.000000in}{0.000000in}}%
\pgfpathlineto{\pgfqpoint{0.048611in}{0.000000in}}%
\pgfusepath{stroke,fill}%
}%
\begin{pgfscope}%
\pgfsys@transformshift{2.466603in}{1.168472in}%
\pgfsys@useobject{currentmarker}{}%
\end{pgfscope}%
\end{pgfscope}%
\begin{pgfscope}%
\pgftext[x=2.563825in,y=1.120278in,left,base]{\rmfamily\fontsize{10.000000}{12.000000}\selectfont −4}%
\end{pgfscope}%
\begin{pgfscope}%
\pgfsetbuttcap%
\pgfsetroundjoin%
\definecolor{currentfill}{rgb}{0.000000,0.000000,0.000000}%
\pgfsetfillcolor{currentfill}%
\pgfsetlinewidth{0.803000pt}%
\definecolor{currentstroke}{rgb}{0.000000,0.000000,0.000000}%
\pgfsetstrokecolor{currentstroke}%
\pgfsetdash{}{0pt}%
\pgfsys@defobject{currentmarker}{\pgfqpoint{0.000000in}{0.000000in}}{\pgfqpoint{0.048611in}{0.000000in}}{%
\pgfpathmoveto{\pgfqpoint{0.000000in}{0.000000in}}%
\pgfpathlineto{\pgfqpoint{0.048611in}{0.000000in}}%
\pgfusepath{stroke,fill}%
}%
\begin{pgfscope}%
\pgfsys@transformshift{2.466603in}{1.404306in}%
\pgfsys@useobject{currentmarker}{}%
\end{pgfscope}%
\end{pgfscope}%
\begin{pgfscope}%
\pgftext[x=2.563825in,y=1.356111in,left,base]{\rmfamily\fontsize{10.000000}{12.000000}\selectfont −3}%
\end{pgfscope}%
\begin{pgfscope}%
\pgfsetbuttcap%
\pgfsetroundjoin%
\definecolor{currentfill}{rgb}{0.000000,0.000000,0.000000}%
\pgfsetfillcolor{currentfill}%
\pgfsetlinewidth{0.803000pt}%
\definecolor{currentstroke}{rgb}{0.000000,0.000000,0.000000}%
\pgfsetstrokecolor{currentstroke}%
\pgfsetdash{}{0pt}%
\pgfsys@defobject{currentmarker}{\pgfqpoint{0.000000in}{0.000000in}}{\pgfqpoint{0.048611in}{0.000000in}}{%
\pgfpathmoveto{\pgfqpoint{0.000000in}{0.000000in}}%
\pgfpathlineto{\pgfqpoint{0.048611in}{0.000000in}}%
\pgfusepath{stroke,fill}%
}%
\begin{pgfscope}%
\pgfsys@transformshift{2.466603in}{1.640139in}%
\pgfsys@useobject{currentmarker}{}%
\end{pgfscope}%
\end{pgfscope}%
\begin{pgfscope}%
\pgftext[x=2.563825in,y=1.591944in,left,base]{\rmfamily\fontsize{10.000000}{12.000000}\selectfont −2}%
\end{pgfscope}%
\begin{pgfscope}%
\pgfsetbuttcap%
\pgfsetroundjoin%
\definecolor{currentfill}{rgb}{0.000000,0.000000,0.000000}%
\pgfsetfillcolor{currentfill}%
\pgfsetlinewidth{0.803000pt}%
\definecolor{currentstroke}{rgb}{0.000000,0.000000,0.000000}%
\pgfsetstrokecolor{currentstroke}%
\pgfsetdash{}{0pt}%
\pgfsys@defobject{currentmarker}{\pgfqpoint{0.000000in}{0.000000in}}{\pgfqpoint{0.048611in}{0.000000in}}{%
\pgfpathmoveto{\pgfqpoint{0.000000in}{0.000000in}}%
\pgfpathlineto{\pgfqpoint{0.048611in}{0.000000in}}%
\pgfusepath{stroke,fill}%
}%
\begin{pgfscope}%
\pgfsys@transformshift{2.466603in}{1.875972in}%
\pgfsys@useobject{currentmarker}{}%
\end{pgfscope}%
\end{pgfscope}%
\begin{pgfscope}%
\pgftext[x=2.563825in,y=1.827778in,left,base]{\rmfamily\fontsize{10.000000}{12.000000}\selectfont −1}%
\end{pgfscope}%
\begin{pgfscope}%
\pgfsetbuttcap%
\pgfsetroundjoin%
\definecolor{currentfill}{rgb}{0.000000,0.000000,0.000000}%
\pgfsetfillcolor{currentfill}%
\pgfsetlinewidth{0.803000pt}%
\definecolor{currentstroke}{rgb}{0.000000,0.000000,0.000000}%
\pgfsetstrokecolor{currentstroke}%
\pgfsetdash{}{0pt}%
\pgfsys@defobject{currentmarker}{\pgfqpoint{0.000000in}{0.000000in}}{\pgfqpoint{0.048611in}{0.000000in}}{%
\pgfpathmoveto{\pgfqpoint{0.000000in}{0.000000in}}%
\pgfpathlineto{\pgfqpoint{0.048611in}{0.000000in}}%
\pgfusepath{stroke,fill}%
}%
\begin{pgfscope}%
\pgfsys@transformshift{2.466603in}{2.111806in}%
\pgfsys@useobject{currentmarker}{}%
\end{pgfscope}%
\end{pgfscope}%
\begin{pgfscope}%
\pgftext[x=2.563825in,y=2.063611in,left,base]{\rmfamily\fontsize{10.000000}{12.000000}\selectfont 0}%
\end{pgfscope}%
\begin{pgfscope}%
\pgftext[x=2.796881in,y=1.168472in,,top,rotate=90.000000]{\rmfamily\fontsize{10.000000}{12.000000}\selectfont \(\displaystyle \log_{10} |\ell_{8,\mathrm{direct}} - \ell_{8,\mathrm{L}(8)}|\)}%
\end{pgfscope}%
\begin{pgfscope}%
\pgfsetbuttcap%
\pgfsetmiterjoin%
\pgfsetlinewidth{0.803000pt}%
\definecolor{currentstroke}{rgb}{0.000000,0.000000,0.000000}%
\pgfsetstrokecolor{currentstroke}%
\pgfsetdash{}{0pt}%
\pgfpathmoveto{\pgfqpoint{2.372270in}{0.225139in}}%
\pgfpathlineto{\pgfqpoint{2.372270in}{0.232509in}}%
\pgfpathlineto{\pgfqpoint{2.372270in}{2.104436in}}%
\pgfpathlineto{\pgfqpoint{2.372270in}{2.111806in}}%
\pgfpathlineto{\pgfqpoint{2.466603in}{2.111806in}}%
\pgfpathlineto{\pgfqpoint{2.466603in}{2.104436in}}%
\pgfpathlineto{\pgfqpoint{2.466603in}{0.232509in}}%
\pgfpathlineto{\pgfqpoint{2.466603in}{0.225139in}}%
\pgfpathclose%
\pgfusepath{stroke}%
\end{pgfscope}%
\end{pgfpicture}%
\makeatother%
\endgroup%

%% file: list2-zero-stickout.pgf
\begingroup%
\makeatletter%
\begin{pgfpicture}%
\pgfpathrectangle{\pgfpointorigin}{\pgfqpoint{3.080000in}{2.310000in}}%
\pgfusepath{use as bounding box, clip}%
\begin{pgfscope}%
\pgfsetbuttcap%
\pgfsetmiterjoin%
\definecolor{currentfill}{rgb}{1.000000,1.000000,1.000000}%
\pgfsetfillcolor{currentfill}%
\pgfsetlinewidth{0.000000pt}%
\definecolor{currentstroke}{rgb}{1.000000,1.000000,1.000000}%
\pgfsetstrokecolor{currentstroke}%
\pgfsetdash{}{0pt}%
\pgfpathmoveto{\pgfqpoint{0.000000in}{0.000000in}}%
\pgfpathlineto{\pgfqpoint{3.080000in}{0.000000in}}%
\pgfpathlineto{\pgfqpoint{3.080000in}{2.310000in}}%
\pgfpathlineto{\pgfqpoint{0.000000in}{2.310000in}}%
\pgfpathclose%
\pgfusepath{fill}%
\end{pgfscope}%
\begin{pgfscope}%
\pgfsetbuttcap%
\pgfsetmiterjoin%
\definecolor{currentfill}{rgb}{1.000000,1.000000,1.000000}%
\pgfsetfillcolor{currentfill}%
\pgfsetlinewidth{0.000000pt}%
\definecolor{currentstroke}{rgb}{0.000000,0.000000,0.000000}%
\pgfsetstrokecolor{currentstroke}%
\pgfsetstrokeopacity{0.000000}%
\pgfsetdash{}{0pt}%
\pgfpathmoveto{\pgfqpoint{0.395365in}{0.225139in}}%
\pgfpathlineto{\pgfqpoint{2.282031in}{0.225139in}}%
\pgfpathlineto{\pgfqpoint{2.282031in}{2.111806in}}%
\pgfpathlineto{\pgfqpoint{0.395365in}{2.111806in}}%
\pgfpathclose%
\pgfusepath{fill}%
\end{pgfscope}%
\begin{pgfscope}%
\pgfpathrectangle{\pgfqpoint{0.395365in}{0.225139in}}{\pgfqpoint{1.886667in}{1.886667in}} %
\pgfusepath{clip}%
\pgfsys@transformshift{0.395365in}{0.225139in}%
\pgftext[left,bottom]{\pgfimage[interpolate=true,width=1.888333in,height=1.888333in]{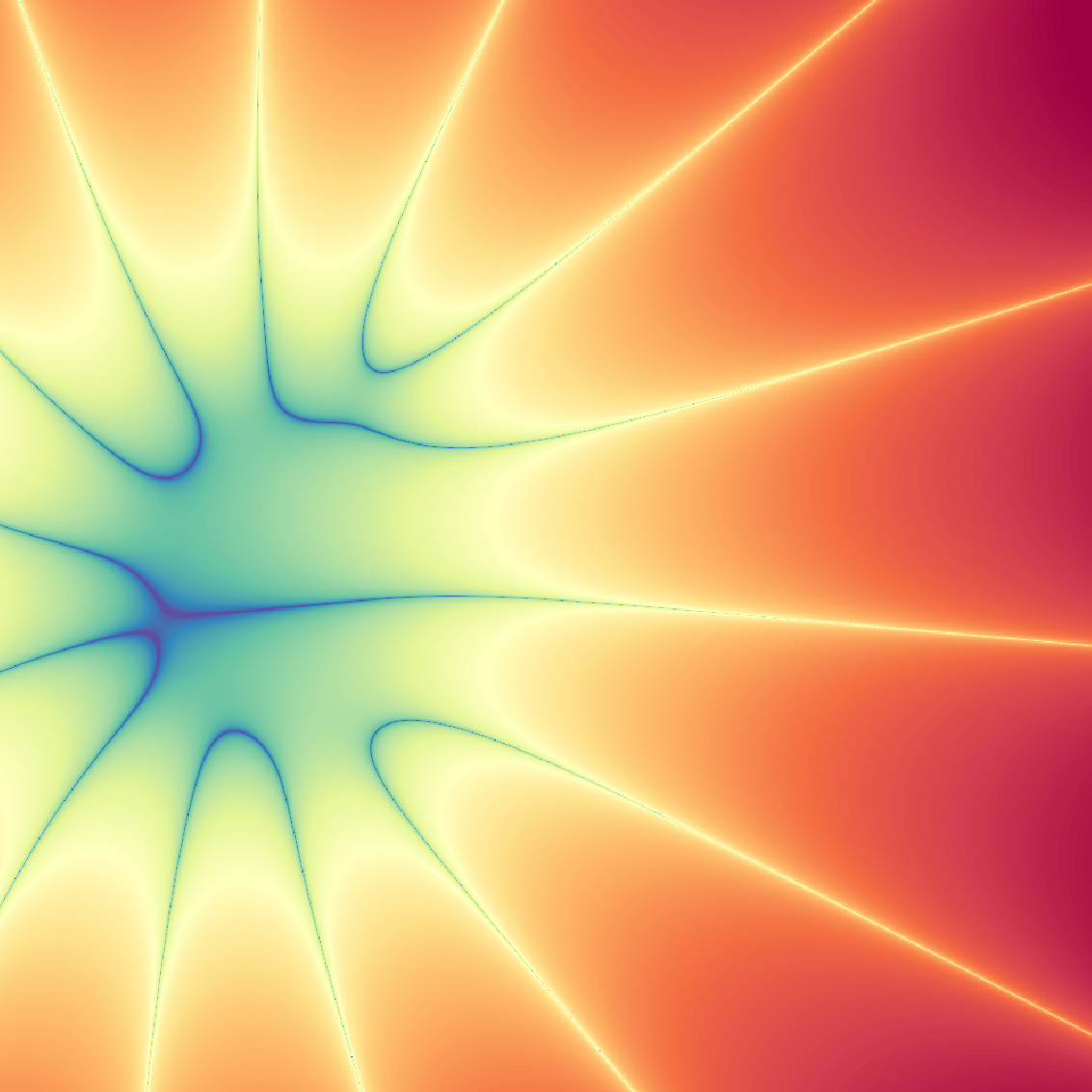}}%
\end{pgfscope}%
\begin{pgfscope}%
\pgfpathrectangle{\pgfqpoint{0.395365in}{0.225139in}}{\pgfqpoint{1.886667in}{1.886667in}} %
\pgfusepath{clip}%
\pgfsetbuttcap%
\pgfsetmiterjoin%
\pgfsetlinewidth{0.803000pt}%
\definecolor{currentstroke}{rgb}{0.000000,0.000000,0.000000}%
\pgfsetstrokecolor{currentstroke}%
\pgfsetdash{}{0pt}%
\pgfpathmoveto{\pgfqpoint{0.584031in}{0.916917in}}%
\pgfpathlineto{\pgfqpoint{1.087142in}{0.916917in}}%
\pgfpathlineto{\pgfqpoint{1.087142in}{1.420028in}}%
\pgfpathlineto{\pgfqpoint{0.584031in}{1.420028in}}%
\pgfpathclose%
\pgfusepath{stroke}%
\end{pgfscope}%
\begin{pgfscope}%
\pgfpathrectangle{\pgfqpoint{0.395365in}{0.225139in}}{\pgfqpoint{1.886667in}{1.886667in}} %
\pgfusepath{clip}%
\pgfsetbuttcap%
\pgfsetmiterjoin%
\pgfsetlinewidth{0.803000pt}%
\definecolor{currentstroke}{rgb}{0.000000,0.000000,0.000000}%
\pgfsetstrokecolor{currentstroke}%
\pgfsetdash{}{0pt}%
\pgfpathmoveto{\pgfqpoint{1.590253in}{0.916917in}}%
\pgfpathlineto{\pgfqpoint{2.093365in}{0.916917in}}%
\pgfpathlineto{\pgfqpoint{2.093365in}{1.420028in}}%
\pgfpathlineto{\pgfqpoint{1.590253in}{1.420028in}}%
\pgfpathclose%
\pgfusepath{stroke}%
\end{pgfscope}%
\begin{pgfscope}%
\pgfpathrectangle{\pgfqpoint{0.395365in}{0.225139in}}{\pgfqpoint{1.886667in}{1.886667in}} %
\pgfusepath{clip}%
\pgfsetbuttcap%
\pgfsetmiterjoin%
\pgfsetlinewidth{0.803000pt}%
\definecolor{currentstroke}{rgb}{0.000000,0.000000,0.000000}%
\pgfsetstrokecolor{currentstroke}%
\pgfsetdash{}{0pt}%
\pgfpathmoveto{\pgfqpoint{0.961365in}{0.864818in}}%
\pgfpathcurveto{\pgfqpoint{1.008538in}{0.864818in}}{\pgfqpoint{1.053786in}{0.883560in}}{\pgfqpoint{1.087142in}{0.916917in}}%
\pgfpathcurveto{\pgfqpoint{1.120499in}{0.950273in}}{\pgfqpoint{1.139241in}{0.995521in}}{\pgfqpoint{1.139241in}{1.042694in}}%
\pgfpathcurveto{\pgfqpoint{1.139241in}{1.089868in}}{\pgfqpoint{1.120499in}{1.135115in}}{\pgfqpoint{1.087142in}{1.168472in}}%
\pgfpathcurveto{\pgfqpoint{1.053786in}{1.201829in}}{\pgfqpoint{1.008538in}{1.220571in}}{\pgfqpoint{0.961365in}{1.220571in}}%
\pgfpathcurveto{\pgfqpoint{0.914191in}{1.220571in}}{\pgfqpoint{0.868943in}{1.201829in}}{\pgfqpoint{0.835587in}{1.168472in}}%
\pgfpathcurveto{\pgfqpoint{0.802230in}{1.135115in}}{\pgfqpoint{0.783488in}{1.089868in}}{\pgfqpoint{0.783488in}{1.042694in}}%
\pgfpathcurveto{\pgfqpoint{0.783488in}{0.995521in}}{\pgfqpoint{0.802230in}{0.950273in}}{\pgfqpoint{0.835587in}{0.916917in}}%
\pgfpathcurveto{\pgfqpoint{0.868943in}{0.883560in}}{\pgfqpoint{0.914191in}{0.864818in}}{\pgfqpoint{0.961365in}{0.864818in}}%
\pgfpathclose%
\pgfusepath{stroke}%
\end{pgfscope}%
\begin{pgfscope}%
\pgfpathrectangle{\pgfqpoint{0.395365in}{0.225139in}}{\pgfqpoint{1.886667in}{1.886667in}} %
\pgfusepath{clip}%
\pgfsetbuttcap%
\pgfsetmiterjoin%
\pgfsetlinewidth{0.803000pt}%
\definecolor{currentstroke}{rgb}{0.000000,0.000000,0.000000}%
\pgfsetstrokecolor{currentstroke}%
\pgfsetdash{{5.600000pt}{2.400000pt}}{0.000000pt}%
\pgfpathmoveto{\pgfqpoint{0.531932in}{0.864818in}}%
\pgfpathlineto{\pgfqpoint{1.139241in}{0.864818in}}%
\pgfpathlineto{\pgfqpoint{1.139241in}{1.472127in}}%
\pgfpathlineto{\pgfqpoint{0.531932in}{1.472127in}}%
\pgfpathclose%
\pgfusepath{stroke}%
\end{pgfscope}%
\begin{pgfscope}%
\pgfpathrectangle{\pgfqpoint{0.395365in}{0.225139in}}{\pgfqpoint{1.886667in}{1.886667in}} %
\pgfusepath{clip}%
\pgfsetbuttcap%
\pgfsetroundjoin%
\definecolor{currentfill}{rgb}{0.121569,0.466667,0.705882}%
\pgfsetfillcolor{currentfill}%
\pgfsetlinewidth{1.003750pt}%
\definecolor{currentstroke}{rgb}{0.121569,0.466667,0.705882}%
\pgfsetstrokecolor{currentstroke}%
\pgfsetdash{}{0pt}%
\pgfsys@defobject{currentmarker}{\pgfqpoint{-0.041667in}{-0.041667in}}{\pgfqpoint{0.041667in}{0.041667in}}{%
\pgfpathmoveto{\pgfqpoint{0.000000in}{-0.041667in}}%
\pgfpathcurveto{\pgfqpoint{0.011050in}{-0.041667in}}{\pgfqpoint{0.021649in}{-0.037276in}}{\pgfqpoint{0.029463in}{-0.029463in}}%
\pgfpathcurveto{\pgfqpoint{0.037276in}{-0.021649in}}{\pgfqpoint{0.041667in}{-0.011050in}}{\pgfqpoint{0.041667in}{0.000000in}}%
\pgfpathcurveto{\pgfqpoint{0.041667in}{0.011050in}}{\pgfqpoint{0.037276in}{0.021649in}}{\pgfqpoint{0.029463in}{0.029463in}}%
\pgfpathcurveto{\pgfqpoint{0.021649in}{0.037276in}}{\pgfqpoint{0.011050in}{0.041667in}}{\pgfqpoint{0.000000in}{0.041667in}}%
\pgfpathcurveto{\pgfqpoint{-0.011050in}{0.041667in}}{\pgfqpoint{-0.021649in}{0.037276in}}{\pgfqpoint{-0.029463in}{0.029463in}}%
\pgfpathcurveto{\pgfqpoint{-0.037276in}{0.021649in}}{\pgfqpoint{-0.041667in}{0.011050in}}{\pgfqpoint{-0.041667in}{0.000000in}}%
\pgfpathcurveto{\pgfqpoint{-0.041667in}{-0.011050in}}{\pgfqpoint{-0.037276in}{-0.021649in}}{\pgfqpoint{-0.029463in}{-0.029463in}}%
\pgfpathcurveto{\pgfqpoint{-0.021649in}{-0.037276in}}{\pgfqpoint{-0.011050in}{-0.041667in}}{\pgfqpoint{0.000000in}{-0.041667in}}%
\pgfpathclose%
\pgfusepath{stroke,fill}%
}%
\begin{pgfscope}%
\pgfsys@transformshift{2.093365in}{0.916917in}%
\pgfsys@useobject{currentmarker}{}%
\end{pgfscope}%
\end{pgfscope}%
\begin{pgfscope}%
\pgfpathrectangle{\pgfqpoint{0.395365in}{0.225139in}}{\pgfqpoint{1.886667in}{1.886667in}} %
\pgfusepath{clip}%
\pgfsetrectcap%
\pgfsetroundjoin%
\pgfsetlinewidth{1.505625pt}%
\definecolor{currentstroke}{rgb}{1.000000,0.498039,0.054902}%
\pgfsetstrokecolor{currentstroke}%
\pgfsetdash{}{0pt}%
\pgfpathmoveto{\pgfqpoint{0.961365in}{1.042694in}}%
\pgfusepath{stroke}%
\end{pgfscope}%
\begin{pgfscope}%
\pgfpathrectangle{\pgfqpoint{0.395365in}{0.225139in}}{\pgfqpoint{1.886667in}{1.886667in}} %
\pgfusepath{clip}%
\pgfsetbuttcap%
\pgfsetroundjoin%
\definecolor{currentfill}{rgb}{1.000000,0.498039,0.054902}%
\pgfsetfillcolor{currentfill}%
\pgfsetlinewidth{1.003750pt}%
\definecolor{currentstroke}{rgb}{1.000000,0.498039,0.054902}%
\pgfsetstrokecolor{currentstroke}%
\pgfsetdash{}{0pt}%
\pgfsys@defobject{currentmarker}{\pgfqpoint{-0.041667in}{-0.041667in}}{\pgfqpoint{0.041667in}{0.041667in}}{%
\pgfpathmoveto{\pgfqpoint{0.000000in}{-0.041667in}}%
\pgfpathcurveto{\pgfqpoint{0.011050in}{-0.041667in}}{\pgfqpoint{0.021649in}{-0.037276in}}{\pgfqpoint{0.029463in}{-0.029463in}}%
\pgfpathcurveto{\pgfqpoint{0.037276in}{-0.021649in}}{\pgfqpoint{0.041667in}{-0.011050in}}{\pgfqpoint{0.041667in}{0.000000in}}%
\pgfpathcurveto{\pgfqpoint{0.041667in}{0.011050in}}{\pgfqpoint{0.037276in}{0.021649in}}{\pgfqpoint{0.029463in}{0.029463in}}%
\pgfpathcurveto{\pgfqpoint{0.021649in}{0.037276in}}{\pgfqpoint{0.011050in}{0.041667in}}{\pgfqpoint{0.000000in}{0.041667in}}%
\pgfpathcurveto{\pgfqpoint{-0.011050in}{0.041667in}}{\pgfqpoint{-0.021649in}{0.037276in}}{\pgfqpoint{-0.029463in}{0.029463in}}%
\pgfpathcurveto{\pgfqpoint{-0.037276in}{0.021649in}}{\pgfqpoint{-0.041667in}{0.011050in}}{\pgfqpoint{-0.041667in}{0.000000in}}%
\pgfpathcurveto{\pgfqpoint{-0.041667in}{-0.011050in}}{\pgfqpoint{-0.037276in}{-0.021649in}}{\pgfqpoint{-0.029463in}{-0.029463in}}%
\pgfpathcurveto{\pgfqpoint{-0.021649in}{-0.037276in}}{\pgfqpoint{-0.011050in}{-0.041667in}}{\pgfqpoint{0.000000in}{-0.041667in}}%
\pgfpathclose%
\pgfusepath{stroke,fill}%
}%
\begin{pgfscope}%
\pgfsys@transformshift{0.961365in}{1.042694in}%
\pgfsys@useobject{currentmarker}{}%
\end{pgfscope}%
\end{pgfscope}%
\begin{pgfscope}%
\pgfsetrectcap%
\pgfsetmiterjoin%
\pgfsetlinewidth{0.803000pt}%
\definecolor{currentstroke}{rgb}{0.000000,0.000000,0.000000}%
\pgfsetstrokecolor{currentstroke}%
\pgfsetdash{}{0pt}%
\pgfpathmoveto{\pgfqpoint{0.395365in}{0.225139in}}%
\pgfpathlineto{\pgfqpoint{0.395365in}{2.111806in}}%
\pgfusepath{stroke}%
\end{pgfscope}%
\begin{pgfscope}%
\pgfsetrectcap%
\pgfsetmiterjoin%
\pgfsetlinewidth{0.803000pt}%
\definecolor{currentstroke}{rgb}{0.000000,0.000000,0.000000}%
\pgfsetstrokecolor{currentstroke}%
\pgfsetdash{}{0pt}%
\pgfpathmoveto{\pgfqpoint{2.282031in}{0.225139in}}%
\pgfpathlineto{\pgfqpoint{2.282031in}{2.111806in}}%
\pgfusepath{stroke}%
\end{pgfscope}%
\begin{pgfscope}%
\pgfsetrectcap%
\pgfsetmiterjoin%
\pgfsetlinewidth{0.803000pt}%
\definecolor{currentstroke}{rgb}{0.000000,0.000000,0.000000}%
\pgfsetstrokecolor{currentstroke}%
\pgfsetdash{}{0pt}%
\pgfpathmoveto{\pgfqpoint{0.395365in}{0.225139in}}%
\pgfpathlineto{\pgfqpoint{2.282031in}{0.225139in}}%
\pgfusepath{stroke}%
\end{pgfscope}%
\begin{pgfscope}%
\pgfsetrectcap%
\pgfsetmiterjoin%
\pgfsetlinewidth{0.803000pt}%
\definecolor{currentstroke}{rgb}{0.000000,0.000000,0.000000}%
\pgfsetstrokecolor{currentstroke}%
\pgfsetdash{}{0pt}%
\pgfpathmoveto{\pgfqpoint{0.395365in}{2.111806in}}%
\pgfpathlineto{\pgfqpoint{2.282031in}{2.111806in}}%
\pgfusepath{stroke}%
\end{pgfscope}%
\begin{pgfscope}%
\pgfsetroundcap%
\pgfsetroundjoin%
\pgfsetlinewidth{0.803000pt}%
\definecolor{currentstroke}{rgb}{0.000000,0.000000,0.000000}%
\pgfsetstrokecolor{currentstroke}%
\pgfsetdash{}{0pt}%
\pgfpathmoveto{\pgfqpoint{0.932939in}{1.071120in}}%
\pgfpathquadraticcurveto{\pgfqpoint{0.931622in}{1.072437in}}{\pgfqpoint{0.921521in}{1.082537in}}%
\pgfusepath{stroke}%
\end{pgfscope}%
\begin{pgfscope}%
\pgfsetroundcap%
\pgfsetroundjoin%
\definecolor{currentfill}{rgb}{0.000000,0.000000,0.000000}%
\pgfsetfillcolor{currentfill}%
\pgfsetlinewidth{0.803000pt}%
\definecolor{currentstroke}{rgb}{0.000000,0.000000,0.000000}%
\pgfsetstrokecolor{currentstroke}%
\pgfsetdash{}{0pt}%
\pgfpathmoveto{\pgfqpoint{0.913297in}{1.130045in}}%
\pgfpathlineto{\pgfqpoint{0.932939in}{1.071120in}}%
\pgfpathlineto{\pgfqpoint{0.874014in}{1.090762in}}%
\pgfpathlineto{\pgfqpoint{0.913297in}{1.130045in}}%
\pgfpathclose%
\pgfusepath{stroke,fill}%
\end{pgfscope}%
\begin{pgfscope}%
\pgfsetroundcap%
\pgfsetroundjoin%
\pgfsetlinewidth{0.803000pt}%
\definecolor{currentstroke}{rgb}{0.000000,0.000000,0.000000}%
\pgfsetstrokecolor{currentstroke}%
\pgfsetdash{}{0pt}%
\pgfpathmoveto{\pgfqpoint{0.963652in}{1.168472in}}%
\pgfpathquadraticcurveto{\pgfqpoint{1.335759in}{1.168472in}}{\pgfqpoint{1.720288in}{1.168472in}}%
\pgfusepath{stroke}%
\end{pgfscope}%
\begin{pgfscope}%
\pgfsetroundcap%
\pgfsetroundjoin%
\definecolor{currentfill}{rgb}{0.000000,0.000000,0.000000}%
\pgfsetfillcolor{currentfill}%
\pgfsetlinewidth{0.803000pt}%
\definecolor{currentstroke}{rgb}{0.000000,0.000000,0.000000}%
\pgfsetstrokecolor{currentstroke}%
\pgfsetdash{}{0pt}%
\pgfpathmoveto{\pgfqpoint{1.019208in}{1.140694in}}%
\pgfpathlineto{\pgfqpoint{0.963652in}{1.168472in}}%
\pgfpathlineto{\pgfqpoint{1.019208in}{1.196250in}}%
\pgfpathlineto{\pgfqpoint{1.019208in}{1.140694in}}%
\pgfpathclose%
\pgfusepath{stroke,fill}%
\end{pgfscope}%
\begin{pgfscope}%
\pgftext[x=0.835587in,y=1.168472in,,]{\rmfamily\fontsize{10.000000}{12.000000}\selectfont \(\displaystyle \ell_{8}\)}%
\end{pgfscope}%
\begin{pgfscope}%
\pgfsetroundcap%
\pgfsetroundjoin%
\pgfsetlinewidth{0.803000pt}%
\definecolor{currentstroke}{rgb}{0.000000,0.000000,0.000000}%
\pgfsetstrokecolor{currentstroke}%
\pgfsetdash{}{0pt}%
\pgfpathmoveto{\pgfqpoint{1.959685in}{1.050596in}}%
\pgfpathquadraticcurveto{\pgfqpoint{2.022133in}{0.988148in}}{\pgfqpoint{2.093365in}{0.916917in}}%
\pgfusepath{stroke}%
\end{pgfscope}%
\begin{pgfscope}%
\pgfsetroundcap%
\pgfsetroundjoin%
\definecolor{currentfill}{rgb}{0.000000,0.000000,0.000000}%
\pgfsetfillcolor{currentfill}%
\pgfsetlinewidth{0.803000pt}%
\definecolor{currentstroke}{rgb}{0.000000,0.000000,0.000000}%
\pgfsetstrokecolor{currentstroke}%
\pgfsetdash{}{0pt}%
\pgfpathmoveto{\pgfqpoint{1.979327in}{0.991670in}}%
\pgfpathlineto{\pgfqpoint{1.959685in}{1.050596in}}%
\pgfpathlineto{\pgfqpoint{2.018611in}{1.030954in}}%
\pgfpathlineto{\pgfqpoint{1.979327in}{0.991670in}}%
\pgfpathclose%
\pgfusepath{stroke,fill}%
\end{pgfscope}%
\begin{pgfscope}%
\pgftext[x=1.841809in,y=1.168472in,,]{\rmfamily\fontsize{10.000000}{12.000000}\selectfont \(\displaystyle m_{8}\)}%
\end{pgfscope}%
\begin{pgfscope}%
\pgfsetroundcap%
\pgfsetroundjoin%
\pgfsetlinewidth{0.803000pt}%
\definecolor{currentstroke}{rgb}{0.000000,0.000000,0.000000}%
\pgfsetstrokecolor{currentstroke}%
\pgfsetdash{}{0pt}%
\pgfpathmoveto{\pgfqpoint{1.657572in}{0.555695in}}%
\pgfpathquadraticcurveto{\pgfqpoint{1.864774in}{0.727441in}}{\pgfqpoint{2.062412in}{0.891260in}}%
\pgfusepath{stroke}%
\end{pgfscope}%
\begin{pgfscope}%
\pgfsetroundcap%
\pgfsetroundjoin%
\definecolor{currentfill}{rgb}{0.000000,0.000000,0.000000}%
\pgfsetfillcolor{currentfill}%
\pgfsetlinewidth{0.803000pt}%
\definecolor{currentstroke}{rgb}{0.000000,0.000000,0.000000}%
\pgfsetstrokecolor{currentstroke}%
\pgfsetdash{}{0pt}%
\pgfpathmoveto{\pgfqpoint{2.001913in}{0.877193in}}%
\pgfpathlineto{\pgfqpoint{2.062412in}{0.891260in}}%
\pgfpathlineto{\pgfqpoint{2.037366in}{0.834421in}}%
\pgfpathlineto{\pgfqpoint{2.001913in}{0.877193in}}%
\pgfpathclose%
\pgfusepath{stroke,fill}%
\end{pgfscope}%
\begin{pgfscope}%
\pgftext[x=1.338698in,y=0.413805in,left,base]{\rmfamily\fontsize{10.000000}{12.000000}\selectfont source}%
\end{pgfscope}%
\begin{pgfscope}%
\pgftext[x=0.961365in,y=0.775879in,,]{\rmfamily\fontsize{10.000000}{12.000000}\selectfont \(\displaystyle \ell_{8}\)}%
\end{pgfscope}%
\begin{pgfscope}%
\pgfpathrectangle{\pgfqpoint{2.375286in}{0.225139in}}{\pgfqpoint{0.094333in}{1.886667in}} %
\pgfusepath{clip}%
\pgfsetbuttcap%
\pgfsetmiterjoin%
\definecolor{currentfill}{rgb}{1.000000,1.000000,1.000000}%
\pgfsetfillcolor{currentfill}%
\pgfsetlinewidth{0.010037pt}%
\definecolor{currentstroke}{rgb}{1.000000,1.000000,1.000000}%
\pgfsetstrokecolor{currentstroke}%
\pgfsetdash{}{0pt}%
\pgfpathmoveto{\pgfqpoint{2.375286in}{0.225139in}}%
\pgfpathlineto{\pgfqpoint{2.375286in}{0.232509in}}%
\pgfpathlineto{\pgfqpoint{2.375286in}{2.104436in}}%
\pgfpathlineto{\pgfqpoint{2.375286in}{2.111806in}}%
\pgfpathlineto{\pgfqpoint{2.469619in}{2.111806in}}%
\pgfpathlineto{\pgfqpoint{2.469619in}{2.104436in}}%
\pgfpathlineto{\pgfqpoint{2.469619in}{0.232509in}}%
\pgfpathlineto{\pgfqpoint{2.469619in}{0.225139in}}%
\pgfpathclose%
\pgfusepath{stroke,fill}%
\end{pgfscope}%
\begin{pgfscope}%
\pgfsys@transformshift{2.375000in}{0.225000in}%
\pgftext[left,bottom]{\pgfimage[interpolate=true,width=0.095000in,height=1.886667in]{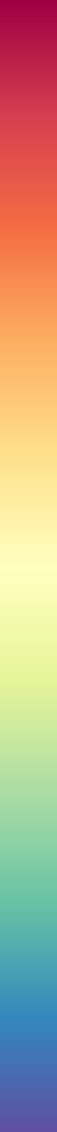}}%
\end{pgfscope}%
\begin{pgfscope}%
\pgfsetbuttcap%
\pgfsetroundjoin%
\definecolor{currentfill}{rgb}{0.000000,0.000000,0.000000}%
\pgfsetfillcolor{currentfill}%
\pgfsetlinewidth{0.803000pt}%
\definecolor{currentstroke}{rgb}{0.000000,0.000000,0.000000}%
\pgfsetstrokecolor{currentstroke}%
\pgfsetdash{}{0pt}%
\pgfsys@defobject{currentmarker}{\pgfqpoint{0.000000in}{0.000000in}}{\pgfqpoint{0.048611in}{0.000000in}}{%
\pgfpathmoveto{\pgfqpoint{0.000000in}{0.000000in}}%
\pgfpathlineto{\pgfqpoint{0.048611in}{0.000000in}}%
\pgfusepath{stroke,fill}%
}%
\begin{pgfscope}%
\pgfsys@transformshift{2.469619in}{0.225139in}%
\pgfsys@useobject{currentmarker}{}%
\end{pgfscope}%
\end{pgfscope}%
\begin{pgfscope}%
\pgftext[x=2.566841in,y=0.176944in,left,base]{\rmfamily\fontsize{10.000000}{12.000000}\selectfont −8}%
\end{pgfscope}%
\begin{pgfscope}%
\pgfsetbuttcap%
\pgfsetroundjoin%
\definecolor{currentfill}{rgb}{0.000000,0.000000,0.000000}%
\pgfsetfillcolor{currentfill}%
\pgfsetlinewidth{0.803000pt}%
\definecolor{currentstroke}{rgb}{0.000000,0.000000,0.000000}%
\pgfsetstrokecolor{currentstroke}%
\pgfsetdash{}{0pt}%
\pgfsys@defobject{currentmarker}{\pgfqpoint{0.000000in}{0.000000in}}{\pgfqpoint{0.048611in}{0.000000in}}{%
\pgfpathmoveto{\pgfqpoint{0.000000in}{0.000000in}}%
\pgfpathlineto{\pgfqpoint{0.048611in}{0.000000in}}%
\pgfusepath{stroke,fill}%
}%
\begin{pgfscope}%
\pgfsys@transformshift{2.469619in}{0.460972in}%
\pgfsys@useobject{currentmarker}{}%
\end{pgfscope}%
\end{pgfscope}%
\begin{pgfscope}%
\pgftext[x=2.566841in,y=0.412778in,left,base]{\rmfamily\fontsize{10.000000}{12.000000}\selectfont −7}%
\end{pgfscope}%
\begin{pgfscope}%
\pgfsetbuttcap%
\pgfsetroundjoin%
\definecolor{currentfill}{rgb}{0.000000,0.000000,0.000000}%
\pgfsetfillcolor{currentfill}%
\pgfsetlinewidth{0.803000pt}%
\definecolor{currentstroke}{rgb}{0.000000,0.000000,0.000000}%
\pgfsetstrokecolor{currentstroke}%
\pgfsetdash{}{0pt}%
\pgfsys@defobject{currentmarker}{\pgfqpoint{0.000000in}{0.000000in}}{\pgfqpoint{0.048611in}{0.000000in}}{%
\pgfpathmoveto{\pgfqpoint{0.000000in}{0.000000in}}%
\pgfpathlineto{\pgfqpoint{0.048611in}{0.000000in}}%
\pgfusepath{stroke,fill}%
}%
\begin{pgfscope}%
\pgfsys@transformshift{2.469619in}{0.696805in}%
\pgfsys@useobject{currentmarker}{}%
\end{pgfscope}%
\end{pgfscope}%
\begin{pgfscope}%
\pgftext[x=2.566841in,y=0.648611in,left,base]{\rmfamily\fontsize{10.000000}{12.000000}\selectfont −6}%
\end{pgfscope}%
\begin{pgfscope}%
\pgfsetbuttcap%
\pgfsetroundjoin%
\definecolor{currentfill}{rgb}{0.000000,0.000000,0.000000}%
\pgfsetfillcolor{currentfill}%
\pgfsetlinewidth{0.803000pt}%
\definecolor{currentstroke}{rgb}{0.000000,0.000000,0.000000}%
\pgfsetstrokecolor{currentstroke}%
\pgfsetdash{}{0pt}%
\pgfsys@defobject{currentmarker}{\pgfqpoint{0.000000in}{0.000000in}}{\pgfqpoint{0.048611in}{0.000000in}}{%
\pgfpathmoveto{\pgfqpoint{0.000000in}{0.000000in}}%
\pgfpathlineto{\pgfqpoint{0.048611in}{0.000000in}}%
\pgfusepath{stroke,fill}%
}%
\begin{pgfscope}%
\pgfsys@transformshift{2.469619in}{0.932639in}%
\pgfsys@useobject{currentmarker}{}%
\end{pgfscope}%
\end{pgfscope}%
\begin{pgfscope}%
\pgftext[x=2.566841in,y=0.884444in,left,base]{\rmfamily\fontsize{10.000000}{12.000000}\selectfont −5}%
\end{pgfscope}%
\begin{pgfscope}%
\pgfsetbuttcap%
\pgfsetroundjoin%
\definecolor{currentfill}{rgb}{0.000000,0.000000,0.000000}%
\pgfsetfillcolor{currentfill}%
\pgfsetlinewidth{0.803000pt}%
\definecolor{currentstroke}{rgb}{0.000000,0.000000,0.000000}%
\pgfsetstrokecolor{currentstroke}%
\pgfsetdash{}{0pt}%
\pgfsys@defobject{currentmarker}{\pgfqpoint{0.000000in}{0.000000in}}{\pgfqpoint{0.048611in}{0.000000in}}{%
\pgfpathmoveto{\pgfqpoint{0.000000in}{0.000000in}}%
\pgfpathlineto{\pgfqpoint{0.048611in}{0.000000in}}%
\pgfusepath{stroke,fill}%
}%
\begin{pgfscope}%
\pgfsys@transformshift{2.469619in}{1.168472in}%
\pgfsys@useobject{currentmarker}{}%
\end{pgfscope}%
\end{pgfscope}%
\begin{pgfscope}%
\pgftext[x=2.566841in,y=1.120278in,left,base]{\rmfamily\fontsize{10.000000}{12.000000}\selectfont −4}%
\end{pgfscope}%
\begin{pgfscope}%
\pgfsetbuttcap%
\pgfsetroundjoin%
\definecolor{currentfill}{rgb}{0.000000,0.000000,0.000000}%
\pgfsetfillcolor{currentfill}%
\pgfsetlinewidth{0.803000pt}%
\definecolor{currentstroke}{rgb}{0.000000,0.000000,0.000000}%
\pgfsetstrokecolor{currentstroke}%
\pgfsetdash{}{0pt}%
\pgfsys@defobject{currentmarker}{\pgfqpoint{0.000000in}{0.000000in}}{\pgfqpoint{0.048611in}{0.000000in}}{%
\pgfpathmoveto{\pgfqpoint{0.000000in}{0.000000in}}%
\pgfpathlineto{\pgfqpoint{0.048611in}{0.000000in}}%
\pgfusepath{stroke,fill}%
}%
\begin{pgfscope}%
\pgfsys@transformshift{2.469619in}{1.404306in}%
\pgfsys@useobject{currentmarker}{}%
\end{pgfscope}%
\end{pgfscope}%
\begin{pgfscope}%
\pgftext[x=2.566841in,y=1.356111in,left,base]{\rmfamily\fontsize{10.000000}{12.000000}\selectfont −3}%
\end{pgfscope}%
\begin{pgfscope}%
\pgfsetbuttcap%
\pgfsetroundjoin%
\definecolor{currentfill}{rgb}{0.000000,0.000000,0.000000}%
\pgfsetfillcolor{currentfill}%
\pgfsetlinewidth{0.803000pt}%
\definecolor{currentstroke}{rgb}{0.000000,0.000000,0.000000}%
\pgfsetstrokecolor{currentstroke}%
\pgfsetdash{}{0pt}%
\pgfsys@defobject{currentmarker}{\pgfqpoint{0.000000in}{0.000000in}}{\pgfqpoint{0.048611in}{0.000000in}}{%
\pgfpathmoveto{\pgfqpoint{0.000000in}{0.000000in}}%
\pgfpathlineto{\pgfqpoint{0.048611in}{0.000000in}}%
\pgfusepath{stroke,fill}%
}%
\begin{pgfscope}%
\pgfsys@transformshift{2.469619in}{1.640139in}%
\pgfsys@useobject{currentmarker}{}%
\end{pgfscope}%
\end{pgfscope}%
\begin{pgfscope}%
\pgftext[x=2.566841in,y=1.591944in,left,base]{\rmfamily\fontsize{10.000000}{12.000000}\selectfont −2}%
\end{pgfscope}%
\begin{pgfscope}%
\pgfsetbuttcap%
\pgfsetroundjoin%
\definecolor{currentfill}{rgb}{0.000000,0.000000,0.000000}%
\pgfsetfillcolor{currentfill}%
\pgfsetlinewidth{0.803000pt}%
\definecolor{currentstroke}{rgb}{0.000000,0.000000,0.000000}%
\pgfsetstrokecolor{currentstroke}%
\pgfsetdash{}{0pt}%
\pgfsys@defobject{currentmarker}{\pgfqpoint{0.000000in}{0.000000in}}{\pgfqpoint{0.048611in}{0.000000in}}{%
\pgfpathmoveto{\pgfqpoint{0.000000in}{0.000000in}}%
\pgfpathlineto{\pgfqpoint{0.048611in}{0.000000in}}%
\pgfusepath{stroke,fill}%
}%
\begin{pgfscope}%
\pgfsys@transformshift{2.469619in}{1.875972in}%
\pgfsys@useobject{currentmarker}{}%
\end{pgfscope}%
\end{pgfscope}%
\begin{pgfscope}%
\pgftext[x=2.566841in,y=1.827778in,left,base]{\rmfamily\fontsize{10.000000}{12.000000}\selectfont −1}%
\end{pgfscope}%
\begin{pgfscope}%
\pgfsetbuttcap%
\pgfsetroundjoin%
\definecolor{currentfill}{rgb}{0.000000,0.000000,0.000000}%
\pgfsetfillcolor{currentfill}%
\pgfsetlinewidth{0.803000pt}%
\definecolor{currentstroke}{rgb}{0.000000,0.000000,0.000000}%
\pgfsetstrokecolor{currentstroke}%
\pgfsetdash{}{0pt}%
\pgfsys@defobject{currentmarker}{\pgfqpoint{0.000000in}{0.000000in}}{\pgfqpoint{0.048611in}{0.000000in}}{%
\pgfpathmoveto{\pgfqpoint{0.000000in}{0.000000in}}%
\pgfpathlineto{\pgfqpoint{0.048611in}{0.000000in}}%
\pgfusepath{stroke,fill}%
}%
\begin{pgfscope}%
\pgfsys@transformshift{2.469619in}{2.111806in}%
\pgfsys@useobject{currentmarker}{}%
\end{pgfscope}%
\end{pgfscope}%
\begin{pgfscope}%
\pgftext[x=2.566841in,y=2.063611in,left,base]{\rmfamily\fontsize{10.000000}{12.000000}\selectfont 0}%
\end{pgfscope}%
\begin{pgfscope}%
\pgftext[x=2.799897in,y=1.168472in,,top,rotate=90.000000]{\rmfamily\fontsize{10.000000}{12.000000}\selectfont \(\displaystyle \log_{10} |\ell_{8,\mathrm{direct}} - \ell_{8,\mathrm{M2L}(8)}|\)}%
\end{pgfscope}%
\begin{pgfscope}%
\pgfsetbuttcap%
\pgfsetmiterjoin%
\pgfsetlinewidth{0.803000pt}%
\definecolor{currentstroke}{rgb}{0.000000,0.000000,0.000000}%
\pgfsetstrokecolor{currentstroke}%
\pgfsetdash{}{0pt}%
\pgfpathmoveto{\pgfqpoint{2.375286in}{0.225139in}}%
\pgfpathlineto{\pgfqpoint{2.375286in}{0.232509in}}%
\pgfpathlineto{\pgfqpoint{2.375286in}{2.104436in}}%
\pgfpathlineto{\pgfqpoint{2.375286in}{2.111806in}}%
\pgfpathlineto{\pgfqpoint{2.469619in}{2.111806in}}%
\pgfpathlineto{\pgfqpoint{2.469619in}{2.104436in}}%
\pgfpathlineto{\pgfqpoint{2.469619in}{0.232509in}}%
\pgfpathlineto{\pgfqpoint{2.469619in}{0.225139in}}%
\pgfpathclose%
\pgfusepath{stroke}%
\end{pgfscope}%
\end{pgfpicture}%
\makeatother%
\endgroup%

%% file: qbx-fmm.bib
@article{carrier:1988:adaptive-fmm,
  author       = {Carrier, J. and Greengard, L. and Rokhlin, V.},
  doi          = {10.1137/0909044},
  fjournal     = {Society for Industrial and Applied Mathematics. Journal on Scientific and Statistical Computing},
  % issn         = {0196-5204},
  journaltitle = {SIAM J. Sci. Statist. Comput.},
  mrclass      = {65C20 (70-08 78-08 78A35 82-08)},
  mrnumber     = {945931},
  number       = {4},
  pages        = {669--686},
  title        = {A fast adaptive multipole algorithm for particle simulations},
  % url          = {http://dx.doi.org/10.1137/0909044},
  volume       = {9},
  year         = {1988},
}

@article{epstein:2013:qbx-error-est,
  author       = {Epstein, Charles L. and Greengard, Leslie and Klöckner, Andreas},
  doi          = {10.1137/120902859},
  fjournal     = {SIAM Journal on Numerical Analysis},
  % issn         = {0036-1429},
  journaltitle = {SIAM J. Numer. Anal.},
  mrclass      = {65R20 (31B10 65N38 65N80)},
  mrnumber     = {3106484},
  mrreviewer   = {Michael J. Carley},
  number       = {5},
  pages        = {2660--2679},
  title        = {On the convergence of local expansions of layer potentials},
  % url          = {http://dx.doi.org/10.1137/120902859},
  volume       = {51},
  year         = {2013},
}

@article{greengard:1987:fmm,
  author       = {Greengard, Leslie and Rokhlin, Vladimir},
  doi          = {10.1016/0021-9991(87)90140-9},
  fjournal     = {Journal of Computational Physics},
  % issn         = {0021-9991},
  journaltitle = {J. Comput. Phys.},
  mrclass      = {82-08 (65C20 70-08 76-08 78-08)},
  mrnumber     = {918448},
  number       = {2},
  pages        = {325--348},
  title        = {A fast algorithm for particle simulations},
  % url          = {http://dx.doi.org/10.1016/0021-9991(87)90140-9},
  volume       = {73},
  year         = {1987},
}

@article{klockner:2013:qbx,
  author       = {Klöckner, Andreas and Barnett, Alexander H. and Greengard, Leslie and O'Neil, Michael},
  doi          = {10.1016/j.jcp.2013.06.027},
  fjournal     = {Journal of Computational Physics},
  % issn         = {0021-9991},
  journaltitle = {J. Comput. Phys.},
  mrclass      = {65D30 (45Exx)},
  mrnumber     = {3101510},
  pages        = {332--349},
  title        = {Quadrature by expansion: a new method for the evaluation of layer potentials},
  % url          = {http://dx.doi.org/10.1016/j.jcp.2013.06.027},
  volume       = {252},
  year         = {2013},
}

@article{rachh:2017:qbx-fmm,
  author       = {Rachh, Manas and Klöckner, Andreas and O'Neil, Michael},
  doi          = {10.1016/j.jcp.2017.04.062},
  fjournal     = {Journal of Computational Physics},
  % issn         = {0021-9991},
  journaltitle = {J. Comput. Phys.},
  mrclass      = {Prelim},
  mrnumber     = {3667635},
  pages        = {706--731},
  title        = {Fast algorithms for {Q}uadrature by {E}xpansion {I}: {G}lobally valid expansions},
  % url          = {http://dx.doi.org/10.1016/j.jcp.2017.04.062},
  volume       = {345},
  year         = {2017},
}

@article{rokhlin:1985:fmm-without-tree,
  author       = {Rokhlin, Vladimir},
  doi          = {10.1016/0021-9991(85)90002-6},
  fjournal     = {Journal of Computational Physics},
  % issn         = {0021-9991},
  journaltitle = {J. Comput. Phys.},
  mrclass      = {65N99 (65R20)},
  mrnumber     = {805870},
  number       = {2},
  pages        = {187--207},
  title        = {Rapid solution of integral equations of classical potential theory},
  % url          = {http://dx.doi.org/10.1016/0021-9991(85)90002-6},
  volume       = {60},
  year         = {1985},
}

@article{petersen:1995:fmm-error-est,
  abstract     = {The Greengard-Rokhlin algorithm is a new and interesting method for computing long-range interactions in particle systems. Although the method already has been implemented and claimed to be superior to traditional and other methods, no reliable estimates of the size of the error of the method have been given. We illustrate what the error actually is for the two-dimensional case, and derive an estimate for it. The estimate has a simple analytic form which will allow its use in tuning the algorithm for best efficiency.},
  author       = {Henrik G. Petersen and D. Soelvason and J. W. Perram and E. R. Smith},
  % issn         = {09628444},
  journaltitle = {Proceedings: Mathematical and Physical Sciences},
  number       = {1934},
  pages        = {389-400},
  publisher    = {The Royal Society},
  title        = "{Error Estimates for the Fast Multipole Method. I. The Two-Dimensional Case}",
  volume       = {448},
  year         = {1995},
  doi          = {10.1098/rspa.1995.0023},
  % url          = {http://www.jstor.org/stable/52647},
}

@article{rahimian:2017:qbkix,
  author="Rahimian, Abtin and Barnett, Alex and Zorin, Denis",
  title="Ubiquitous evaluation of layer potentials using Quadrature by Kernel-Independent Expansion",
  journal="BIT Numerical Mathematics",
  year="2017",
  month="11",
  day="06",
  issn="1572-9125",
  doi="10.1007/s10543-017-0689-2",
  %url="https://doi.org/10.1007/s10543-017-0689-2"
}

@article{barnett:2014:close-eval,
  author       = {Barnett, Alexander H.},
  doi          = {10.1137/120900253},
  % eprint       = {https://doi.org/10.1137/120900253},
  journaltitle = {SIAM Journal on Scientific Computing},
  number       = {2},
  pages        = {A427-A451},
  title        = {Evaluation of Layer Potentials Close to the Boundary for Laplace and Helmholtz Problems on Analytic Planar Domains},
  % url          = {https://doi.org/10.1137/120900253},
  volume       = {36},
  year         = {2014},
}

@article{afklinteberg:2016:quadrature-est,
  author       = {af Klinteberg, Ludvig and Tornberg, Anna-Karin},
  doi          = {10.1007/s10444-016-9484-x},
  fjournal     = {Advances in Computational Mathematics},
  % issn         = {1019-7168},
  journaltitle = {Adv. Comput. Math.},
  mrclass      = {65D30},
  mrnumber     = {3598839},
  mrreviewer   = {Alexandru Ioan Mitrea},
  number       = {1},
  pages        = {195--234},
  title        = {Error estimation for quadrature by expansion in layer potential evaluation},
  % url          = {http://dx.doi.org/10.1007/s10444-016-9484-x},
  volume       = {43},
  year         = {2017},
}

@article{afklinteberg:2017:adaptive-qbx,
  author = {Ludvig af Klinteberg and Anna-Karin Tornberg},
  title = {Adaptive Quadrature by Expansion for Layer Potential Evaluation in Two Dimensions},
  journal = {SIAM Journal on Scientific Computing},
  volume = {40},
  number = {3},
  pages = {A1225-A1249},
  year = {2018},
  doi = {10.1137/17M1121615},
  %URL = {https://doi.org/10.1137/17M1121615},
  %eprint = {https://doi.org/10.1137/17M1121615}
}

@book{greengard:1988:thesis,
  author     = {Greengard, Leslie},
  isbn       = {0-262-07110-X},
  mrclass    = {31C20 (31-04 70-08 78-04 78A30)},
  mrnumber   = {936632},
  mrreviewer = {G. R. Allcock},
  pages      = {xiv+91},
  publisher  = {MIT Press, Cambridge, MA},
  series     = {ACM Distinguished Dissertations},
  title      = {The rapid evaluation of potential fields in particle systems},
  year       = {1988},
}

@book{kress:2014:integral-equations,
  author    = {Kress, Rainer},
  doi       = {10.1007/978-1-4614-9593-2},
  edition   = {Third},
  isbn      = {978-1-4614-9592-5; 978-1-4614-9593-2},
  mrclass   = {45A05 (45-02 45L05 47G10 65R20)},
  mrnumber  = {3184286},
  pages     = {xvi+412},
  publisher = {Springer, New York},
  series    = {Applied Mathematical Sciences},
  title     = {Linear integral equations},
  % url       = {http://dx.doi.org/10.1007/978-1-4614-9593-2},
  volume    = {82},
  year      = {2014},
}

@book{conway:1978:complex-variables,
  author     = {Conway, John B.},
  edition    = {Second},
  isbn       = {0-387-90328-3},
  mrclass    = {30-01},
  mrnumber   = {503901},
  mrreviewer = {P. Lappan},
  pages      = {xiii+317},
  publisher  = {Springer-Verlag, New York-Berlin},
  series     = {Graduate Texts in Mathematics},
  title      = {Functions of one complex variable},
  volume     = {11},
  year       = {1978},
}

@article{hrycak:1998:improved-2d-fmm,
  author       = {Hrycak, Tomasz and Rokhlin, Vladimir},
  % url          = {https://doi.org/10.1137/S106482759630989X},
  date         = {1998},
  doi          = {10.1137/S106482759630989X},
  % issn         = {1064-8275},
  journaltitle = {SIAM J. Sci. Comput.},
  number       = {6},
  pages        = {1804--1826},
  title        = {An improved fast multipole algorithm for potential fields},
  volume       = {19},
}

@book{atkinson1976survey,
  author    = {Atkinson, Kendall},
  publisher = {Soc. for Industrial and Applied Mathematics},
  title     = {A survey of numerical methods for the solution of Fredholm integral equations of the second kind},
  year      = {1976},
}

@article{christlieb2006grid,
  author       = {Christlieb, Andrew J and Krasny, Robert and Verboncoeur, John P and Emhoff, Jerold W and Boyd, Iain D},
  journaltitle = {IEEE Transactions on Plasma Science},
  number       = {2},
  pages        = {149--165},
  publisher    = {IEEE},
  title        = {Grid-free plasma simulation techniques},
  volume       = {34},
  year         = {2006},
  doi          = {10.1109/TPS.2006.871104},
}

@article{barnes1986hierarchical,
  author       = {Barnes, Josh and Hut, Piet},
  journaltitle = {nature},
  number       = {6096},
  pages        = {446--449},
  publisher    = {Nature Publishing Group},
  title        = {A hierarchical O (N log N) force-calculation algorithm},
  volume       = {324},
  year         = {1986},
}

@inproceedings{moore:1995:cost-of-balancing-generalized-quadtrees,
  acmid     = {218078},
  address   = {New York, NY, USA},
  author    = {Moore, Doug},
  booktitle = {Proceedings of the Third ACM Symposium on Solid Modeling and Applications},
  doi       = {10.1145/218013.218078},
  isbn      = {0-89791-672-7},
  location  = {Salt Lake City, Utah, USA},
  numpages  = {8},
  pages     = {305--312},
  publisher = {ACM},
  series    = {SMA '95},
  title     = {The Cost of Balancing Generalized Quadtrees},
  % url       = {http://doi.acm.org/10.1145/218013.218078},
  year      = {1995},
}

@article{pouransari:2015:adaptive-fractal-fmm,
  author       = {Pouransari, Hadi and Darve, Eric},
  fjournal     = {SIAM Journal on Scientific Computing},
  % issn         = {1064-8275},
  doi          = {10.1137/140962681},
  journaltitle = {SIAM J. Sci. Comput.},
  mrclass      = {65R10 (28A80 70F10)},
  mrnumber     = {3340201},
  mrreviewer   = {Peter Junghanns},
  number       = {2},
  pages        = {A1040--A1066},
  title        = {Optimizing the adaptive fast multipole method for fractal sets},
  % url          = { https://doi.org/10.1137/140962681},
  volume       = {37},
  year         = {2015},
}

@article{yarvin_generalized_1998,
  author = {Nathan Yarvin and Vladimir Rokhlin},
  title = "{Generalized Gaussian Quadratures and Singular Value Decompositions of Integral Operators}",
  publisher = {SIAM},
  year = {1998},
  journal = {SIAM Journal on Scientific Computing},
  volume = {20},
  number = {2},
  pages = {699--718},
  doi = {10.1137/S1064827596310779}
}

@article{kussmaul1969numerisches,
  title="{Ein numerisches Verfahren zur L{\"o}sung des Neumannschen Au{\ss}enraumproblems f{\"u}r die Helmholtzsche Schwingungsgleichung}",
  author={Ku{\ss}maul, Rainer},
  journal={Computing},
  volume={4},
  number={3},
  pages={246--273},
  year={1969},
  publisher={Springer}
}

@article{martensen1963methode,
  title="{{\"U}ber eine Methode zum r{\"a}umlichen Neumannschen Problem mit einer Anwendung f{\"u}r torusartige Berandungen}",
  author={Martensen, Erich},
  journal={Acta mathematica},
  volume={109},
  number={1},
  pages={75--135},
  year={1963},
  publisher={Springer}
}

@phdthesis{rachh2015integral,
  title={Integral equation methods for problems in electrostatics, elastostatics and viscous flow},
  author={Rachh, Manas},
  year={2015},
  school={New York University}
}

@misc{rachh2018fast,
  title = {Fast Algorithms for `Quadrature by Expansion' II: Unidirectional~Interactions},
  author = {Rachh, Manas and Klöckner, Andreas and Greengard, Leslie},
  note = {In preparation},
}

@article{biros_embedded_2004,
        title = {An embedded boundary integral solver for the unsteady incompressible {Navier}-{Stokes} equations},
        journal = {J. Comput. Phys},
        author = {Biros, George and Ying, Lexing and Zorin, Denis},
        year = {2004},
        pages = {121--141},
}

@article{helsing_2008b,
  title="{Corner singularities for elliptic problems: integral equations, graded meshes, quadrature, and compressed inverse preconditioning}",
  author={Helsing, J. and Ojala, R.},
  journal = {Journal of Computational Physics},
  volume={227},
  pages={8820--8840},
  year={2008},
  doi={10.1016/j.jcp.2008.06.022},
}

@article{white_continuous_1994,
   author = {{White}, C.~A. and {Johnson}, B.~G. and {Gill}, P.~M.~W. and
        {Head-Gordon}, M.},
    title = "{The continuous fast multipole method}",
  journal = {Chemical Physics Letters},
     year = 1994,
    month = nov,
   volume = 230,
    pages = {8-16},
      doi = {10.1016/0009-2614(94)01128-1},
   adsurl = {http://adsabs.harvard.edu/abs/1994CPL...230....8W}
}

@article{siegel2017local,
  title = "A local target specific quadrature by expansion method for evaluation of layer potentials in 3D",
  journal = "Journal of Computational Physics",
  volume = "364",
  pages = "365 - 392",
  year = "2018",
  issn = "0021-9991",
  doi = "10.1016/j.jcp.2018.03.006",
  % url = "http://www.sciencedirect.com/science/article/pii/S002199911830158X",
  author = "Michael Siegel and Anna-Karin Tornberg",
  keywords = "Layer potentials, Integral equations, Quadrature by expansion, Exterior Dirichlet problem, Spherical harmonics expansions, Multiply-connected domain"
}

@article{ying2004kernel,
  author       = {Ying, Lexing and Biros, George and Zorin, Denis},
  publisher    = {Elsevier},
  date         = {2004},
  doi          = {10.1016/j.jcp.2003.11.021},
  journaltitle = {Journal of Computational Physics},
  number       = {2},
  pages        = {591--626},
  title        = {A kernel-independent adaptive fast multipole algorithm in two and three dimensions},
  volume       = {196},
}

@article{dehnen:2002:gravitational-fmm,
  author       = {Dehnen, Walter},
  doi          = {10.1006/jcph.2002.7026},
  date         = {2002},
  % issn         = {0021-9991},
  journaltitle = {Journal of Computational Physics},
  number       = {1},
  pages        = {27--42},
  title        = {A hierarchical {$O(N)$} force calculation algorithm},
  volume       = {179},
}

@article{lowengrub_1993,
  title="{High-order and efficient methods for the vorticity formulation of
the Euler equations}",
  author={Lowengrub, J. and Shelley, M. and Merriman, B.},
  journal = {{SIAM} Journal on Scientific Computing},
  volume={14},
  pages={1107--1142},
  year={1993},
  doi={10.1137/0914067},
}

@article{goodman_1990,
  title="{Convergence of the point vortex method for the 2-D Euler equations}",
  author={Goodman, J. and Hou, T. Y. and Lowengrub, J.},
  journal = {Communications on Pure and Applied Mathematics},
  volume={43},
  pages={415--430},
  year={1990},
  doi={10.1002/cpa.3160430305},
}

@article{haroldson_1998,
  title="{Numerical Calculation of Three-dimensional Interfacial Potential Flows
using the Point Vortex Method}",
  author={Haroldsen, D. J. and Meiron, D. I.},
  journal = {Communications on Pure and Applied Mathematics},
  volume={43},
  pages={415--430},
  year={1990},
  doi={10.1137/S1064827596302060},
}

@article{mayo_1985,
  title="{Fast, High-order accurate solution of Laplace's equation on Irregular Regions}",
  author={Mayo, A.},
  journal = {{SIAM} Journal on Scientific Computing},
  volume={20},
  pages={648--683},
  year={1998},
  doi={10.1137/0906012},
}

@article{beale_lai_2001,
  title="{A method for computing nearly singular integrals}",
  author={Beale, J. T. and Lai, M.-C.},
  journal = {{SIAM} Journal on Scientific Computing},
  volume={38},
  number={6},
  pages={1902--1925},
  year={2001},
  doi={10.1137/S0036142999362845},
}

@book{davis_1984,
  title="{Methods of Numerical Integration}",
  author={Davis, P. J. and Rabinowitz, P.},
  year={1984},
  publisher={Academic Press, San Diego}
}

@article{hackbusch_sauter_1994,
  title="{On numerical cubatures of nearly singular surface integrals arising in BEM collocation}",
  author={Hackbusch, W. and Sauter, S. A.},
  journal = {Computing},
  volume={52},
  number={2},
  pages={139--159},
  year={1994},
  doi={10.1007/BF02238073}
}

@article{graglia_2008,
  title="{Machine Precision Evaluation of Singular and
Nearly Singular Potential Integrals by Use of Gauss
Quadrature Formulas for Rational Functions}",
  author={Graglia, R. D. and Lombardi, G.},
  journal = {IEEE Transactions on Antennas and Propagation},
  volume={56},
  number={4},
  pages={981--998},
  year={2008},
  doi={10.1109/TAP.2008.919181},
}

@article{jarvenpaa_2003,
  title="{Singularity extraction technique for integral equation
methods with higher order basis functions on plane triangles and tetrahedra}",
  author={Jarvenpää, S. and Taskinen, M. and Yla-Oijala, P.},
  journal = {International Journal for Numerical Methods in Engineering},
  volume={58},
  pages={1149--1165},
  year={2003},
  doi={10.1002/nme.810},
}

@article{schwab_1992,
  title="{On numerical cubatures of singular surface integrals in boundary element methods}",
  author={Schwab, C. and Wendland, W. L.},
  journal = {Numerische Mathematik},
  volume={62},
  pages={342--369},
  year={1992},
  doi={10.1007/BF01396234},
}

@article{khayat_2005,
  title="{Numerical Evaluation of Singular and Near-Singular
Potential Integrals}",
  author={Khayat, M. A. and Wilton, D. R.},
  journal = {IEEE Transactions on Antennas and Propagation},
  volume={53},
  number={10},
  pages={3180--3190},
  year={2005},
  doi={10.1109/TAP.2005.856342},
}

@article{bruno_2001,
  title="{A fast, high-order algorithm for the solution of surface scattering problems:
basic implementation, tests, and applications}",
  author={Bruno, O. P. and Kunyansky, L. A.},
  journal = {Journal of Computational Physics},
  volume={169},
  pages={80--110},
  year={2001},
  doi={10.1006/jcph.2001.6714},
}

@article{ying_2006,
  title="{A high-order 3D boundary integral equation solver for elliptic PDEs in smooth domains}",
  author={Ying, L. and Biros, G. and Zorin, D.},
  journal = {Journal of Computational Physics},
  volume={219},
  pages={247--275},
  year={2006},
  doi={10.1016/j.jcp.2006.03.021},
}

@article{bremer_nonlinear_2010,
  title="{A nonlinear optimization procedure for generalized Gaussian quadratures}",
  author={Bremer, J. and Gimbutas, Z. and Rokhlin, V.},
  journal = {{SIAM} Journal on Scientific Computing},
  volume={32},
  pages={1761--1788},
  year={2010},
  doi={10.1137/080737046},
}

@article{farina_2001,
  title="{Evaluation of single layer potentials over curved surfaces}",
  author={Farina, L.},
  journal = {{SIAM} Journal on Scientific Computing},
  volume={23},
  number={1},
  pages={81--91},
  year={2001},
  doi={10.1137/S1064827599363393},
}

@article{strain_1995,
  title="{Locally-corrected multidimensional quadrature rules for singular functions}",
  author={Strain, J.},
  journal = {{SIAM} Journal on Scientific Computing},
  volume={16},
  number={4},
  pages={992--1017},
  year={1995},
  doi={10.1137/0916058},
}

@article{johnson_1989,
  title="{An Analysis of Quadrature Errors in Second-Kind Boundary Integral Methods}",
  author={Johnson, C. G. L. and Scott, L. R.},
  journal = {{SIAM} Journal on Numerical Analysis},
  volume={26},
  number={6},
  pages={1356--1382},
  year={1989},
  doi={10.1137/0726079},
}

@article{sidi_1988,
  title="{Quadrature methods for periodic singular Fredholm integral equations}",
  author={Sidi, A. and Israeli, M.},
  journal = {Journal of Scientific Computing},
  volume={3},
  pages={201--231},
  year={1988},
  doi={10.1007/BF01061258},
}

@article{carley_2007,
  title="{Numerical quadratures for singular and hypersingular integrals in boundary element methods}",
  author={Carley, M.},
  journal = {{SIAM} Journal on Scientific Computing},
  volume={29},
  number={3},
  pages={1207--1216},
  year={2007},
}

@article{helsing_2008a,
  title="{On the evaluation of layer potentials close to their sources}",
  author={Helsing, J. and Ojala, R.},
  journal = {Journal of Computational Physics},
  volume={227},
  pages={2899--2921},
  year={2008},
  doi={10.1016/j.jcp.2007.11.024},
}

@article{atkinson_1995,
  title="{Piecewise polynomial collocation for boundary integral equations}",
  author={Atkinson, K. E. and Chien, D.},
  journal = {{SIAM} Journal on Scientific Computing},
  volume={16},
  number={3},
  pages={651--681},
  year={1995},
}

@article{lyness_numerical_1967,
  title = {On Numerical Contour Integration Round a Closed Contour},
  volume = {21},
  doi = {10.2307/2005000},
  number = {100},
  urldate = {2012-12-22},
  journal = {Math. Comp.},
  author = {Lyness, J. N. and Delves, L. M.},
  month = oct,
  year = {1967},
  pages = {561--577},
}

@article{chapko_numerical_2000,
  title = {On the numerical solution of a hypersingular integral equation for elastic scattering from a planar crack},
  volume = {20},
  doi = {10.1093/imanum/20.4.601},
  language = {en},
  number = {4},
  urldate = {2012-12-22},
  journal = {{IMA} J Numer Anal},
  author = {Chapko, Roman and Kress, Rainer and Mönch, Lars},
  month = oct,
  year = {2000},
  pages = {601--619},
}

@article{saad_gmres_1986,
  title = "{GMRES: A Generalized Minimal Residual Algorithm for Solving Nonsymmetric Linear Systems}",
  volume = {7},
  doi = {10.1137/0907058},
  number = {3},
  journal = {{SIAM} Journal on Scientific and Statistical Computing},
  author = {Saad, Youcef and Schultz, Martin H.},
  month = jul,
  year = {1986},
  pages = {856--869},
}

@article{hao_high-order_2014,
        title = {High-order accurate methods for {Nyström} discretization of integral equations on smooth curves in the plane},
        volume = {40},
        issn = {1019-7168, 1572-9044},
        url = {https://link.springer.com/article/10.1007/s10444-013-9306-3},
        doi = {10.1007/s10444-013-9306-3},
        language = {en},
        number = {1},
        urldate = {2018-01-16},
        journal = {Advances in Computational Mathematics},
        author = {Hao, S. and Barnett, A. H. and Martinsson, P. G. and Young, P.},
        month = feb,
        year = {2014},
        pages = {245--272}
}

@article{bremer_nystrom_2011,
  title = "{On the Nyström discretization of integral equations on planar curves with corners}",
  journal = "Applied and Computational Harmonic Analysis",
  volume = "32",
  number = "1",
  pages = "45 - 64",
  year = "2012",
  doi = "10.1016/j.acha.2011.03.002",
  author = "James Bremer",
}
